\documentclass[reqno]{amsart}
\usepackage{mathrsfs}
\usepackage{amsfonts}
\usepackage[centertags]{amsmath}

\usepackage{amssymb}
\usepackage{amsthm}
\usepackage{graphicx}
\usepackage{caption}
\usepackage{dynkin-diagrams}
\usepackage[colorlinks]{hyperref}
\usepackage[scr=boondox,cal=esstix]{mathalpha}
\hypersetup{
bookmarksnumbered,
pdfstartview={FitH},
breaklinks=true,
linkcolor=blue,
urlcolor=blue,
citecolor=blue,
bookmarksdepth=2
}
\usepackage[sorted,compressed-cites,sorted-cites,nobysame,initials]{amsrefs}
\usepackage{enumitem}
\usepackage[margin=1in]{geometry}
\usepackage{tikz-cd}
\usepackage{rotating}
\usetikzlibrary{matrix,arrows,decorations.pathmorphing}
\usepackage{tikz-cd}
\usetikzlibrary{calc}
\usetikzlibrary{decorations.pathmorphing}

\tikzset{curve/.style={settings={#1},to path={(\tikztostart)
    .. controls ($(\tikztostart)!\pv{pos}!(\tikztotarget)!\pv{height}!270:(\tikztotarget)$)
    and ($(\tikztostart)!1-\pv{pos}!(\tikztotarget)!\pv{height}!270:(\tikztotarget)$)
    .. (\tikztotarget)\tikztonodes}},
    settings/.code={\tikzset{quiver/.cd,#1}
        \def\pv##1{\pgfkeysvalueof{/tikz/quiver/##1}}},
    quiver/.cd,pos/.initial=0.35,height/.initial=0}

\tikzset{tail reversed/.code={\pgfsetarrowsstart{tikzcd to}}}
\tikzset{2tail/.code={\pgfsetarrowsstart{Implies[reversed]}}}
\tikzset{2tail reversed/.code={\pgfsetarrowsstart{Implies}}}
\tikzset{no body/.style={/tikz/dash pattern=on 0 off 1mm}}

\usepackage{listing}
\usepackage{color}
\usepackage{float}

\DeclareMathOperator{\supp}{supp}
\newcommand{\brd}[2]{\underbrace{#1\cdots}_{#2}}

\newcommand{\CC}{{\mathbb C}}

\newcommand{\OEIS}[1]{\href{https:oeis.org/#1}{#1}}

\newcommand{\la}{\langle}
\newcommand{\ra}{\rangle}
\newcommand{\wh}[1]{\widehat{#1}}

\newcommand{\ascprod}{\mathop{\overrightarrow{\prod}}}
\newcommand{\dscprod}{\mathop{\overleftarrow{\prod}}}

\usepackage{xstring}
\makeatletter
\newcommand{\partref}[1]{\StrCut{#1}{.}\@pref@\@suff@\ref{\@pref@}\StrLen{\@suff@}[\@slen@]\ifnum\@slen@>0\ref{#1}\fi}
\makeatother

\allowdisplaybreaks[4]
\usetikzlibrary{matrix,arrows,decorations.pathmorphing}

\newtheorem{theorem}{Theorem}[section]
\newtheorem{corollary}[theorem]{Corollary}
\newtheorem{lemma}[theorem]{Lemma}
\newtheorem{proposition}[theorem]{Proposition}
\theoremstyle{definition}
\newtheorem{definition}[theorem]{Definition}
\newtheorem{remark}[theorem]{Remark}
\newtheorem{example}[theorem]{Example}

\newtheorem{problem}{Problem}

\newtheorem{conjecture}[theorem]{Conjecture}

\newcommand{\cx}[2]{c_{#1\to #2}}

\newcommand{\cxr}[2]{c_{#1\leftarrow #2}}

\newcommand{\Cx}[2]{C_{#1\to #2}}
\newcommand{\Cxr}[2]{C_{#1\leftarrow #2}}
\newcommand{\Ast}{\mathscr{Ast}}

\newenvironment{enmalph}{\begin{enumerate}[label={\rm(\alph*)},leftmargin=*]}{\end{enumerate}}

\SetLabelAlign{center}{\hfil#1\hfil}
\newenvironment{enmroman}{\begin{enumerate}[label={\rm(\roman*)},align=center,leftmargin=*]}{\end{enumerate}}

\numberwithin{equation}{section}
\newcommand{\Cox}[1]{\mathbf{Cox}(#1)}
\newcommand{\tensor}{\otimes}

\newcommand{\kk}{\Bbbk}

\newcommand{\ZZ}{{\mathbb Z}}

\DeclareMathOperator{\id}{id}

\DeclareMathOperator{\End}{End}

\DeclareMathOperator{\Br}{Br}

\DeclareMathOperator*{\Arr}{Arr}

\DeclareMathOperator{\Hom}{Hom}

\DeclareMathOperator{\Inv}{Inv}

\DeclareMathOperator{\SQF}{SQF}
\DeclareMathOperator{\Art}{\mathscr{Art}}
\DeclareMathOperator{\Heck}{\mathscr{Hec}}
\DeclareMathOperator{\CoxCat}{\mathscr{Cox}}
\makeatletter
\renewcommand*{\Hdynkin}{\Adynkin\dynkinEdgeLabel{\numexpr\dynkin@rank-1}{\dynkin@rank}{5}}
\newcommand{\plink}[1]{\hypertarget{#1}{}\label{page:#1}}
\makeatother
\pgfkeys{/Dynkin diagram,
edge length=1.0cm,text style/.style={scale=0.8},Coxeter,root radius=0.07}

\begin{document}

\title{Artin monoids, their homomorphisms and twins}
\author{Arkady Berenstein, Jacob Greenstein and Jian-Rong Li}
\address{Arkady Berenstein, Department of Mathematics, University of Oregon, Eugene, OR 97403, USA}
\email{arkadiy@math.uoregon.edu}
\address{Jacob Greenstein, Department of Mathematics, University of California, Riverside, CA 92521, USA}
\email{jacobg@ucr.edu}
\address{Jian-Rong Li, Faculty of Mathematics, University of Vienna, Oskar-Morgenstern-Platz 1, 1090 Vienna, Austria}
\email{lijr07@gmail.com}
\date{}

\thanks{This work was partially supported by the Simons Foundation Collaboration Grant no.~636972 (A.~Berenstein), the Simons
foundation collaboration grant no.~245735 (J.~Greenstein), Austrian Science Fund (FWF): P 34602, Grant DOI: 10.55776/P34602, and PAT 9039323, Grant-DOI 10.55776/PAT9039323 (J.-R. Li).
}

\begin{abstract}
Motivated by the {\em twin homomorphism problem} for Coxeter groups and the corresponding Hecke monoids, we find a large class of its solutions originating from {\em standard homomorphisms} of Artin monoids and their compositions. These homomorphisms are expected to be injective when they are optimal and injective on generators, which generalizes the homogeneous homomorphisms and the famous Tits conjecture settled by Crisp and Paris. We classify disjoint standard homomorphisms and conjecture the complete classification when the domain is of rank two.
\end{abstract}

\maketitle

\tableofcontents

\section{Introduction and main results}
This work was motivated by the following {\em twin homomorphism problem} which
in turn was motivated by geometric considerations discussed in~\cite{BGLHeck}.
\begin{problem}\label{prob:twin}
Classify all pairs of homomorphisms of Coxeter 
groups and corresponding Hecke monoids which
coincide on generators.
\end{problem}
More precisely, 
let $M$ be a Coxeter matrix over a finite set~$I$ and let $W(M)=\la s_i\,:\, i\in I\ra$ 
be the corresponding Coxeter group
with simple generators~$s_i$ satisfying~$s_i^2=1$
as well as (generalized) braid  relations (see~\S\ref{subs:Br(M)W(M)}). 
Let~$(W(M),\star)$ be the corresponding Hecke monoid, also referred to in the literature as 0-Hecke monoid, Coxeter monoid or Demazure monoid,
which has the same generators satisfying $s_i\star s_i=s_i$, $i\in I$ and the same (generalized) braid relations (see~\S\ref{subs:Hecke}). Thus, for
any homomorphism between Coxeter groups (respectively, Hecke monoids) the image 
of every~$s_i$, $i\in I$ must be an involution (respectively, an idempotent). It turns 
out (see e.g.~\cite{BGLHeck}) that 
all idempotents in~$(W(M),\star)$ are 
longest elements~$w_\circ^J$ in finite parabolic subgroups~$W_J(M)$
of~$W(M)$, $J\subset I$ and in particular
are involutions in~$W(M)$. Thus, twin homomorphisms~$W(M)\to W(M')$
are defined by assignments~$s_i\mapsto w_\circ^{J_i}$, $i\in I$ for some subsets~$J_i$ of the index set~$I'$ of~$M'$, and so the twin homomorphism problem is very natural. 

In~\cite{BGLHeck} we solved Problem~\ref{prob:twin} in an important particular case when twin homomorphisms are equal
as maps of sets. By~\cite{BGLHeck}*{Theorem~3.23}, these are precisely what we call the {\em homogeneous} homomorphisms that is, those which are compatible
with the length function on the domain and
the codomain (see~\S\ref{subs:parab hom Artin} for the details),
and to coincide with unfoldings in finite types.
In particular, that means that they lift
to well-known homomorphisms of the semisimple algebraic groups
in crystallographic cases. Also, a composition of homogeneous homomorphisms 
is again homogeneous. 
Even more importantly, they are induced from
homomorphisms of the corresponding {\em Artin monoids}. On a different note, yet conjecturally,
twin homomorphisms $S_3\to S_n$ exist only 
if~$n=3m$ for some positive~$m$ and can be described explicitly. Most of them also lift to homomorphisms of braid monoids~$\Br^+_3\to \Br^+_{3m}$ and exhaust all homomorphisms
between these monoids which induce twin homomorphisms.

The above discussion motivated us to introduce the notion of
a {\em standard homomorphism} of Artin monoids,
which play a central role in the present paper. 
Namely, for any $w\in W(M)$, denote by~$T_w$
the unique element of the Artin monoid of the same length whose canonical image in~$W(M)$ is equal to~$w$ (see~\S\ref{subs:Br(M)W(M)} for precise definitions). We say that a homomorphism of
Artin monoids~$\Phi:\Br^+(M)\to \Br^+(M')$ is {\em standard} if
$\Phi(T_i)=T_{w_\circ^{J_i}}$
for suitable subsets~$J_i$ of the index set~$I'$ of~$M'$, $i\in I$. It should be noted that 
the length homomorphism $\ell:\Br^+(M)\to(\ZZ_{\ge 0},+)\cong \Br^+_2$ or 
more generally~$\Br^+(M)\to \ZZ_{\ge0}^I$ where
$T_i$ and~$T_j$ have the same image if there is a path from~$i$ to~$j$ in the Coxeter 
graph with all edges having odd labels
(see Example~\ref{ex:char hom}) is standard.
The following is immediate from Theorem~\ref{thm:Main Thm Cox Heck}.
\begin{corollary}
A standard homomorphism $\Br^+(M)\to \Br^+(M')$ 
yields a solution of Problem~\ref{prob:twin}.
\end{corollary}
The above justifies the following
\begin{problem}\label{prob:standard}
Classify standard homomorphisms between Artin monoids. 
\end{problem}

We completely solved Problem~\ref{prob:standard} in the important case when all the~$J_i$, $i\in I'$
are pairwise disjoint and~$M$ is of type~$A$ or~$B$
(Theorems~\ref{thm:main thm adm}
and~\ref{thm:higher rank adm AB}). In the even
dihedral case, the idea is based on the observation that the assignments $T_i\mapsto X_i$, $i\in\{1,2\}$ define a homomorphism
from the Artin monoid corresponding
to the Coxeter matrix~$I_2(2m)=\left(\begin{smallmatrix}1&2m\\2m&1\end{smallmatrix}\right)$ to any
cancellative monoid~$\mathsf M$ if and only if
$z:=(X_1X_2)^m$ is central 
in the submonoid of~$\mathsf M$ generated by~$X_1$ and~$X_2$ (Lemma~\ref{lem:cradicals}). 
Centers of Artin monoids of finite type were described explicitly in classical works of Deligne (\cite{Del}) and of Brieskorn and Saito (\cites{BrSa}), and factorizations of radicals of central elements are plausible starting points for searching for homomorphisms. For instance, for homogeneous homomorphisms, $z$ generates 
the center of the codomain. 
On the other hand, for new homomorphisms that we found in this paper, $z$ is the ratio of generators of the center of the codomain and of that of its parabolic submonoid.
It is worth
mentioning that the crucial role in
proving these results is played by non-standard
homomorphisms of Artin monoids obtained by ``removing decorations'' from standard ones.
To prove that there are no other homomorphisms, we use another important observation, namely
that $z$ must be invariant with respect
to the natural anti-involution of the codomain preserving generators (Lemma~\ref{lem:I2m iff cnd}),
which means that the corresponding linear operator in the symmetrized version of Burau representation of the braid monoid (see~\S\ref{subs:Burau}) must be self-adjoint with respect to a natural bilinear form. Thus, it suffices to prove that linear operators corresponding to~$z$ for ``non-homomorphisms'' are never self-adjoint, by exhibiting a pair of  non-orthogonal eigenvectors corresponding to different eigenvalues. Since a homomorphism 
from an odd dihedral Artin monoid is always accompanied by a one with an even dihedral domain, this allows to classify odd ones very efficiently (see~\S\ref{subs:Burau}). 
Unfolding these homomorphisms via~\eqref{eq:unfold Bn Dn+1} we obtain homomorphisms with the same property for~$M=D_{n+1}$, $n\ge 3$.
We expect that, apart from very few sporadic examples, that completes the 
classification of disjoint standard 
homomorphisms (see~\S\ref{subs:rank > 2}). 

The non-disjoint case is even richer than the disjoint one. For example, all
standard homomorphisms from~$\Br^+_3$, 
apart from its automorphisms, are not disjoint (see Section~\ref{sect:non-disjoint}).
We have already discovered a plethora of families of such homomorphisms in~\S\S\ref{subs:mon braid}, \ref{subs:inf ser non-disj} and~\ref{subs:non-disj Hecke}. 
Surprisingly, we classified (yet conjecturally) {\em all} standard homomorphisms from~$\Br^+(A_2)=\Br^+_3$
and $\Br^+(B_2)$ (Theorems~\ref{thm:Hom A2}, \ref{thm:Hom B2}, \ref{thm:B2 A2n-1 spec} and Proposition~\ref{prop:Hom A2|B2 EF}) to Artin monoids of finite types.
They have rather intriguing combinatorial properties
(see Theorem~\ref{thm:combinatoric std}
and Remark~\ref{rem:OEIS} for OEIS appearances of the related numerology). For 
instance, there are precisely~$2^m$ fully supported standard homomorphisms~$\Br^+_3\to \Br^+_{3m}$, $m\ge 1$. 
Homomorphisms from~$\Br^+(B_2)\to\Br^+(M)$ are even more affluent. 
For instance,
their number grows asymptotically as $\frac14(1+\sqrt 3)^{\frac12(n+5)}$
for~$M$ of type~$A_n$, $3\cdot 2^n$
for~$M$ of type~$B_n$ and faster than~$\frac73\cdot 2^{n+1}$ if~$M$ is of type~$D_{n+1}$. There is an 
additional series of homomorphisms~$\Br^+(B_2)\to \Br^+_{2n}$ which grows at least as fast as $2^{\frac12 n}n$.
We exhibit even more conjectural non-disjoint homomorphisms in~\S\ref{subs:conj families} and
expect that our lists provide a {\em complete}
classification when the domain 
is of rank~2.
These results, in conjunction with conjectures in~\S\ref{subs:conj families}, give a hope to {\em completely} classify standard homomorphisms of Artin monoids of finite and affine types.

Compositions of standard homomorphisms are also quite interesting and recover
some well-known homomorphisms which are not standard.
\begin{example}\label{ex:1.6}
A well-known homomorphism $\Br^+(B_n)\to \Br^+(A_n)$
defined by $T'_i\mapsto T_i^{1+\delta_{i,n}}$, $i\in[1,n]$ is the composition of the standard
unfolding $\Br^+(B_n)\to \Br^+(D_{n+1})$ given by~\eqref{eq:unfold Bn Dn+1} and the standard
folding $\Br^+(D_{n+1})\to \Br^+(A_n)$
defined by $T'_i\mapsto T_{i-\delta_{i,n+1}}$, $i\in [1,n+1]$. Similarly, a well-known homomorphism
$\Br^+(G_2)\to \Br^+(A_2)$ defined by $T'_1\mapsto T_1^3$, $T'_2\mapsto T_2$ is the composition of
the standard unfolding $\Br^+(G_2)\to\Br^+(D_4)$,
$T'_1\mapsto T_1T_3 T_4$, $T'_2\mapsto T_2$
with the standard folding $\Br^+(D_4)\to \Br^+(A_2)$, $T'_i\mapsto T_1$, $i\in\{1,3,4\}$, $T'_2\mapsto T_2$
(see Example~\ref{ex:affine} for similar homomorphisms in affine types).
\end{example}
These homomorphisms are examples of
{\em Tits homomorphisms}~$\mathbf T_{\mathbf d}$, $\mathbf d=(d_i)_{i\in I}\in\ZZ_{>0}^I$ defined by~$T_i\mapsto T_i^{d_i}$, $i\in I$ (see~\S\ref{subs:Tits homs}).
We prove (Theorem~\ref{thm:Tits standard}) that all
of them are compositions of standard ones. The 
same applies (Theorem~\ref{thm:classify light Artin}) to a more general class of 
{\em light
homomorphisms} (cf.~\cite{BGLHeck}*{Section~4}) given by $T_i\mapsto T_j^{d_{ij}}$, $d_{ij}\in \ZZ_{\ge0}$ (see~\S\ref{subs:parab proj Artin} for the precise definition).

Thus, it is natural to consider the category~$\Ast$ whose 
objects are Artin monoids and morphisms 
are (composition of) standard homomorphisms. 
The following surprising result highlights
the importance of this category.
\begin{theorem}[Corollary~\ref{cor:Ast hom}]\label{thm:mainthmstd}
Any morphism in~$\Ast$ yields a solution
of Problem~\ref{prob:twin}.
\end{theorem}
This theorem is quite non-trivial because compositions of homomorphism of Artin monoids are unlikely to yield homomorphisms
of either Coxeter groups or Hecke monoids, even if individual ones do (see~\S\ref{subs:Heck Cox hom}). 
To put this in context, 
we show in~\S\ref{subs:Heck Cox hom} that any homomorphism (or even {\em any map}) $\Phi:\Br^+(M')\to \Br^+(M)$ induces 
a {\em map of sets} $\overline\Phi:W(M')\to W(M)$ (respectively, $\overline\Phi_\star:(W(M'),\star)\to (W(M),\star)$). Thus,
$\Phi$ yields a solution of Problem~\ref{prob:twin} if and only if
both $\overline\Phi$ and~$\overline\Phi_\star$
are homomorphisms of, respectively, Coxeter
groups and Hecke monoids. So far, despite all our efforts, we were 
unable to find any examples of such morphisms outside of the category~$\Ast$. Thus, we expect that the converse of Theorem~\ref{thm:mainthmstd} also holds,
namely that any homomorphism of Artin
monoids yielding a solution of Problem~\ref{prob:twin} is a morphism in~$\Ast$.

A priori, a standard homomorphism need not be injective. However, homogeneous ones, which are standard and
form a subcategory of~$\Ast$ (cf. Corollary~\ref{cor:comp homogeneous}) are known to be injective (\cite{Cri}; see Remark~\ref{rem:injectivity}). On the other hand, the famous conjecture of Tits recently proved
by Crisp and Paris (\cite{CP}) stipulates,
that any Tits homomorphism~$\mathbf T_{\mathbf d}$ such that all
the~$d_i$, $i\in I$ are strictly greater than 1 is injective. A generalization
of Tits conjecture, which involves compositions
of Tits homomorphisms with standard homomorphisms, was proposed and partially proved in~\cite{JS}.
Motivated by the above discussion we formulate the following
\begin{conjecture}\label{conj:injectivity}
Any optimal (see Definition~\ref{def:types} and Example~\ref{ex:taut not inj}) and injective on generators homomorphism in~$\Ast$ is injective.
\end{conjecture}
In particular, this would imply that {\em any} Tits homomorphism is injective (Conjecture~\ref{conj:gen Tits}). For instance,
the injectivity of the Tits homomorphism
$\Br^+(B_n)\to \Br^+(A_n)$ from Example~\ref{ex:1.6}
was established by
Crisp in~\cite{Cri}.
We provide more
supporting evidence in Section~\ref{sec:Parab proj}. It should be noted, however, that
induced homomorphisms of Coxeter groups or Hecke monoids do not have to be injective. For instance, the injective Tits endomorphism  of $\Br^+_2$, $T_1\mapsto T_1^2$, induces 
the trivial homomorphism~$S_2\to S_2$
(see also Lemma~\ref{lem:induced inj} and Remark~\ref{rem:induced do not have to be injective}). Actually, we hope that Conjecture~\ref{conj:injectivity} can be strengthened by
that for any homomorphism of Artin groups, the quotient by its kernel is again an Artin group.

\subsection*{Acknowledgments}
The main part of this work was carried out while the authors were visiting Erwin Schr\"odinger
International Institute for Theoretical Physics (ESI), Vienna, Austria,
in the framework of the ``Research in teams'' program. It is our pleasure to thank the ESI for its hospitality. This work took its present shape while the first author was visiting Max Planck Institute for Mathematics in the Sciences (MIS), Leipzig, Germany and the second author was visiting Institut des Hautes \'Etudes Scientifiques (IHES), Bures-sur-Yvette, France. The hospitality of both institutions is gratefully acknowledged.

\section{Preliminaries}\label{sec:Prelim}

\subsection{General notation}
We extend the natural order on~$\mathbb Z$ to~$\mathbb Z\cup\{\infty\}$
via $\infty>n$ for all~$n\in\mathbb Z$
and use the convention that $n \infty=n+\infty=\infty$ for all~$n\in\mathbb Z_{>0}\cup\{\infty\}$.
In particular, $\infty$ is assumed 
to be divisible by all elements of~$\mathbb Z_{>0}\cup\{\infty\}$. 
Given~\plink{bar s}$s\in\mathbb Z$, let $\bar s\in\{0,1\}$
be the remainder of~$s$ when divided by~$2$.
For any $a,b\in\mathbb Z$ we denote $[a,b]=\{ i\in\mathbb Z\,:\,
a\le i\le b\}$ and \plink{[a,b]2}$[a,b]_2=\{ k\in [a,b]\,:\, \overline{b-k}=0\}$. Given $a,b\in\ZZ$ and~$J\subset \ZZ$, set $a+b J:=\{a +b j\,:\, j\in J\}$.
The power set of a set~$S$ will be denoted~\plink{PS}$\mathscr P(S)$. Given a category~$\mathscr{C}$, we denote $\Hom_{\mathscr C}(X,Y)$ the set of morphisms from $X\in\mathscr C$ to~$Y\in\mathscr C$.

\subsection{Monoids}\label{subs:monoids}
Throughout this paper, a homomorphism of monoids is assumed to map the identity element of the domain
to the identity element of the codomain.

Let~$\mathsf M$ be a multiplicative monoid. 
Given any finite subset~$I\subset \ZZ$ and
a family $X_i$, $i\in I$ of elements of~$\mathsf M$ we set
$$\plink{ascp}
\ascprod_{i\in I} X_i=X_{i_1}\cdots X_{i_r},\qquad
\dscprod_{i\in I} X_i=X_{i_r}\cdots X_{i_1}.
$$
where $I=\{i_1,\dots,i_r\}$ with~$i_1<\cdots<i_r$. This notation
will also be used for infinite families with all but finitely many of the~$X_i$
equal to~$1$.

Given a family~$S$ of generators of~$\mathsf M$,
the length function $\ell_S:\mathsf M\to \mathbb Z_{\ge0}$ is defined by setting $\ell_S(x)$,
$x\in \mathsf M$
to be
the minimal length of a word in~$S$ which is equal
to~$x$. Clearly, $\ell_S(xy)\le \ell_S(x)+\ell_S(y)$
for all~$x,y\in\mathsf M$.

An equivalence relation~$\mathcal C\subset\mathsf M\times\mathsf M$
is called a {\em congruence relation} on~$\mathsf M$
if $(x,y),(x',y')\in\mathcal C$ implies that~$(xx',yy')\in\mathcal C$. In that case,
the set~$\mathsf M/\mathcal C$ of equivalence classes with respect to~$\mathcal C$ is also a monoid, with the multiplication defined by $[x]_{\mathcal C}[y]_{\mathcal C}=[xy]_{\mathcal C}$, $x,y\in\mathsf M$,
where~$[x]_{\mathcal C}$ is the equivalence class of~$x\in\mathsf M$ with respect to~$\mathcal C$. Furthermore, the canonical map~$\pi_{\mathcal C}:\mathsf M\to \mathsf M/\mathcal C$, $x\mapsto [x]_{\mathcal C}$,
$x\in\mathsf M$ is a surjective homomorphism of monoids.

We say that a monoid~$\mathsf M$ is {\em left} (respectively, {\em right}) {\em cancellative} if $x y=x y'$ (respectively, $y x=y' x$),
$x,x',y\in\mathsf M$ implies~$y=y'$. We say that~$\mathsf M$
is {\em cancellative} if it is left and right cancellative. 
For any $x,y\in \mathsf M$ and~$m\in\ZZ_{\ge0}$  denote
$$\plink{brd}
\brd{xy}{m}:=(xy)^{\lfloor \frac12 m\rfloor} x^{\bar m}.
$$
Thus, $\brd{xy}0=1$, $\brd{xy}{m+1}=
\brd{xy}m\,x$ if~$m$ is even, $\brd{xy}{m+1}=
\brd{xy}m\,y$ if~$m$ is odd, while $\brd{xy}{m+1}=
x\brd{yx}{m}$ for all~$m\in\mathbb Z_{\ge 0}$. Given $x,y\in\mathsf M$, define\plink{B(x,y)}
$$
B(x,y)=\{ k\in\ZZ_{>0}\,:\, \brd{xy}k=\brd{yx}k\}.
$$
\begin{lemma}\label{lem:taut homs}
Let~$\mathsf M$ be a multiplicative monoid and let
$x,y\in\mathsf M$ be such that~$B(x,y)\not=\emptyset$. 
\begin{enmalph}
    \item \label{lem:taut homs.a}
    If~$m\in B(x,y)$ then
    \begin{equation}\brd{xy}{km}=(\brd{xy}m)^k, \qquad 
    k\ge 1\label{eq:taut homs}
    \end{equation}
    and so~$\mathbb Z_{>0}m\subset B(x,y)$. 
\item\label{lem:taut homs.b} If~$\mathsf M$ is left (or right) cancellative  then~$B(x,y)=\mathbb Z_{>0}\min B(x,y)$.
\end{enmalph}
\end{lemma}
\begin{proof}
Let~$t_0=x$ and~$t_1=y$. 
Note first that for~$r\le s\in\mathbb Z_{>0}$ and
$a\in \ZZ$
\begin{equation}\label{eq:brd tail}
\brd{t_{\overline{a\vphantom1}} t_{\overline{a+1}}}{s}=\brd{t_{\overline{a\vphantom1}}t_{\overline{a+1}}}r \,\brd{t_{\overline{a+r}}t_{\overline{a+r+1}}}{s-r}.
\end{equation}

We use induction on~$k$
to prove~\eqref{eq:taut homs}, the induction base being trivial. For the inductive step we have for~$k\in\mathbb Z_{>0}$ 
\begin{alignat*}{3}
\brd{t_0t_1}{(k+1)m}=&\brd{t_0t_1}{km} \,\brd{t_{\overline{km}}t_{\overline{km+1}}}{m}
&\qquad&\text{by~\eqref{eq:brd tail}}\\
=&(\brd{t_0t_1}{m})^k \brd{t_0t_1}{m}&&\text{by the induction hypothesis and since~$m\in B(t_0,t_1)$}\\
=&(\brd{t_0t_1}{m})^{k+1}.
\end{alignat*}
The second assertion in part~\ref{lem:taut homs.a} is now immediate.

We only prove part~\ref{lem:taut homs.b} for
left cancellative monoids. Let~$m=\min B(t_0,t_1)$, $m'\in B(x,y)\setminus\{m\}$
and write~$m'=d m+r$,
$d\in\mathbb Z_{> 0}$,
$0\le r<m$. Suppose that~$r>0$. Then 
\begin{alignat*}{2}
\brd{t_1t_0}{dm}\,\brd{t_{\overline{dm+1}}t_{\overline{dm}}}{r}
=&\brd{t_1t_0}{m'}&&\text{by~\eqref{eq:brd tail}}\\
=&\brd{t_0t_1}{m'}&&\text{since $m'\in B(t_0,t_1)$}\\
=&\brd{t_0t_1}{dm}\,\brd{t_{\overline{dm}}t_{\overline{dm+1}}}{r}&&
\text{by \eqref{eq:brd tail}}\\
=&\brd{t_1t_0}{dm}\,\brd{t_{\overline{dm}}t_{\overline{dm+1}}}{r}&\qquad&
\text{by part~\ref{lem:taut homs.a}}.
\end{alignat*}
Since~$\mathsf M$ is left cancellative,
it follows that $\brd{t_{\overline{dm+1}}t_{\overline{dm}}}{r}=
\brd{t_{\overline{dm}}t_{\overline{dm+1}}}{r}$
whence~$r\in B(t_0,t_1)$ which is a contradiction. Thus, $r=0$.
\end{proof}

\subsection{Artin monoids and Coxeter groups}\label{subs:Br(M)W(M)}
Let~$I$ be a finite set and let~$M=(m_{ij})_{i,j\in I}$ be
a symmetric matrix with $m_{ii}=1$ and $m_{ij}\in\ZZ_{>1}\cup\{\infty\}$, $i\not=j$.
Such a matrix is called a {\em Coxeter matrix} (over~$I$), and we denote the set of 
all Coxeter matrices over~$I$ by~\plink{Cox I}$\Cox I$.
The Coxeter
graph~\plink{Gamma(M)}$\Gamma(M)$ associated with~$M$ is the undirected graph with vertex set~$I$
and with a unique edge connecting $i,j\in I$ if and only if~$m_{ij}>2$.
The edge is labeled with~$m_{ij}$ if~$m_{ij}>3$.

The {\em Artin monoid} $\Br^+(M)$\plink{Br+(M)} associated with~$M$ (see for example~\cites{BrSa,Del,Tits}) is generated by the~$T_i$, $i\in I$ subject to relations
$$
\brd{T_iT_j}{m_{ij}}=\brd{T_jT_i}{m_{ij}},\qquad i\not=j\in I,\, m_{ij}<\infty.
$$
The Artin group~$\Br(M)$ associated with~$M$ has the same generators satisfying the same relations. By~\cite{Par}*{Theorem~1.1}, the natural homomorphism $\Br^+(M)\to \Br(M)$
is injective. In particular, $\Br^+(M)$ is cancellative (see
also~\cite{BrSa}*{Proposition~2.3} which establishes the cancellativity
of~$\Br^+(M)$ without using the embedding into its Artin group). The following example illustrates
Lemma~\partref{lem:taut homs.b}.
\begin{example}
Let~$I=\{1,2\}$ and let~$M=\left(\begin{smallmatrix}1&3\\3&1\end{smallmatrix}\right)$. Then $\Br^+(M)$ is generated by $T_1$, $T_2$ subject to the relation $T_1T_2T_1=T_2T_1T_2$.
Yet $(T_1T_2)^2\not=(T_2T_1)^2$ for otherwise 
we would have $T_2T_1T_2^2=T_1T_2T_1T_2=T_2T_1T_2T_1$
which, since~$\Br^+(M)$ is cancellative, yields~$T_1=T_2$.
\end{example}
Since defining relations of~$\Br^+(M)$ are homogeneous in the number of generators, the
length function with respect to~$\{T_i\}_{i\in I}$
is a homomorphism of monoids~\plink{ell}$\ell:\Br^+(M)\to (\ZZ_{\ge 0},+)$. If~$|I|=1$ this homomorphism
is actually an isomorphism.

Since defining relations of~$\Br^+(M)$ are palindromic,
$\Br^+(M)$
admits a unique \plink{op}anti-involu\-tion~${}^{op}$ defined on
generators by~$(T_i)^{op}=T_i$, $i\in I$. It clearly extends to
an anti-involution of~$\Br(M)$.

The {\em Coxeter group}\plink{W(M)}
$W=W(M)$ associated with~$M$ is generated by the~$s_i$, $i\in I$ subject
to relations
$$
(s_i s_j)^{m_{ij}}=1,\qquad i,j\in I,\, m_{ij}\not=\infty.
$$
Clearly, $W(M)$ is the quotient of~$\Br(M)$ by the minimal normal
subgroup containing the $T_i^2$, $i\in I$. Let~\plink{piM}$\pi_M:\Br(M)\to W(M)$,
$T_i\mapsto s_i$, $i\in I$,
be the canonical projection, which obviously restricts to a
surjective homomorphism of monoids~$\Br^+(M)\to W(M)$. Note also that~$W(M)$ is isomorphic
to the quotient monoid of~$\Br^+(M)$ by the minimal congruence relation containing the~$(T_i^2,1)$, $i\in I$.

We denote~$\ell$ the length function for~$W(M)$
with respect to~$\{s_i\}_{i\in I}$.
An expression~$w=s_{i_1}\cdots s_{i_k}$,
$i_1,\dots,i_k\in I$ is called {\em reduced} if~$k=\ell(w)$. Clearly,
$\ell(\pi_M(T))\le \ell(T)$ for all~$T\in\Br^+(M)$ and we
set\plink{SQF}
$$
\SQF^+(M)=\{ T\in\Br^+(M)\,:\,\ell(\pi_M(T))=\ell(T)\}.
$$
Elements of~$\SQF^+(M)$ are called {\em square free}. The following is well-known.
\begin{theorem}[\cite{Tits}*{Theorem~3}] \label{thm:Tits}
\begin{enmalph}
\item\label{thm:Tits.a} $\pi_M$ restricts to a bijection~$\SQF^+(M)\to W(M)$.
\item\label{thm:Tits.b} Given $w\in W(M)$, denote $T_w$ the unique
element of~$\SQF^+(M)\cap \pi^{-1}_M(\{w\})$.
Then $T_w T_{w'}=T_{ww'}$ if and only if
$\ell(ww')=\ell(w)+\ell(w')$. In particular,
for any~$w\in W(M)$,
an expression $w=s_{i_1}\cdots s_{i_k}$, $i_1,\dots,i_k\in I$ is reduced
if and only if~$T_w=T_{i_1}\cdots T_{i_k}$.
\end{enmalph}
\end{theorem}
\begin{lemma}\label{lem:sq free fact}
Let~$M$ be a Coxeter matrix, $T\in\SQF^+(M)$ and suppose that~$T=T'T''$, $T',T''\in\Br^+(M)$.
Then both $T'$ and~$T''$ are square free.
\end{lemma}
\begin{proof}
Since $\ell(X)\ge \ell(\pi_M(X))$ for all~$X\in\Br^+(M)$,
we have~$\ell(\pi_M(T))=\ell(T)=\ell(T')+\ell(T'')\ge
\ell(\pi_M(T'))+\ell(\pi_M(T''))$.
On the other hand, since $\pi_M(T)=
\pi_M(T')\pi_M(T'')$,
$\ell(\pi_M(T))
\le \ell(\pi_M(T'))+\ell(\pi_M(T''))$, whence $\ell(\pi_M(T))=
\ell(\pi_M(T'))+\ell(\pi_M(T''))$. This forces~$\ell(T')=\ell(\pi_M(T))$
and $\ell(T'')=\ell(\pi_M(T''))$.
\end{proof}

The anti-involution~${}^{op}$ factors through to an anti-involution of~$W(M)$ which coincides with
the anti-involution $w\mapsto w^{-1}$, $w\in W(M)$. The following is
immediate.
\begin{lemma}\label{lem:can image op inv}
If~$T\in\Br^+(M)$ is ${}^{op}$-invariant then~$\pi_M(T)\in W(M)$ is an involution.
In particular, $T\in\SQF^+(M)$ is ${}^{op}$-invariant
if and only if~$\pi_M(T)$ is an involution.
\end{lemma}

\subsection{Parabolic submonoids and subgroups}\label{subs:parab}
Given~$J\subset I$, let $M_J=(m_{ij})_{i,j\in J}\in \Cox J$ be the corresponding
submatrix of~$M$. Then the submonoid \plink{Br+J(M)}$\Br^+_J(M):=\la T_j\,:\, j\in J\ra$ of~$\Br^+(M)$ is isomorphic to~$\Br^+(M_J)$. The subgroups~$\Br_J(M)$ of~$\Br(M)$ and~$W_J(M)$ of~$W(M)$ are defined similarly and are
isomorphic to respective objects corresponding to~$M_J$. Those subobjects are called {\em parabolic} submonoids (subgroups).
We will usually identify $W_J(M)$ with~$W(M_J)$ and so on and denote~\plink{iotaJ}$\iota_{J}$
the natural inclusion of~$W_J(M)$ (respectively, $\Br^+_J(M)$)
into~$W(M)$ (respectively, $\Br^+(M)$).

We say that~$J\subset I$ is of {\em finite type} if~$W(M_J)$ is finite.
The corresponding subgroups and submonoids are often referred to as being of {\em spherical type} in the literature. We denote~\plink{F(M)}$\mathscr F(M)$ the
set of all subsets of~$I$ of finite type. Clearly,
$\mathscr F(M)=\mathscr P(I)$ if and only if~$I\in\mathscr F(M)$, in which case we also say that~$M$ is of finite type. Note that~$\emptyset\in\mathscr F(M)$, the corresponding
parabolic subgroups and submonoids being trivial.

Define \plink{supp}$\supp:\Br(M)\to \mathscr P(I)$
by $$
\supp T=\bigcap_{J\subset I\,:\, T\in\Br_J(M)} J,
\qquad T\in\Br(M).
$$
The map~$\supp:W(M)\to \mathscr P(I)$ is defined
similarly.
Clearly, $\supp \pi_M(T)\subset \supp T$ for all~$T\in\Br(M)$.
Given a subset~$S$ of~$\Br(M)$ or~$W(M)$, we denote $\supp S=\bigcup_{x\in S} \supp x$. Observe that~$\supp TT'=\supp T\,\cup\,\supp T'$ for $T,T'\in\Br^+(M)$ while $\supp ww'\subset \supp w\,\cup\,\supp w'$ for $w,w'\in W(M)$. In particular, given
any expression $T=T_{i_1}\cdots T_{i_k}$
(respectively, a {\em reduced} expression
$w=s_{i_1}\cdots s_{i_k}$) where~$i_1,\dots,i_k\in I$
we have~$\supp T=\{i_1,\dots,i_k\}$ (respectively,
$\supp w=\{i_1,\dots,i_k\}$). It follows that the map
$\supp$ is surjective.
The following is well-known
(cf.~\cite{Tits}*{Theorem~3}, \cite{Bou}*{Ch. IV, \S1.5}).
\begin{lemma}\label{lem:extend supp}
Let~$w,w'\in W(M)$ with $\supp w\cap \supp w'=\emptyset$. Then
$\ell(ww')=\ell(w)+\ell(w')$ and~$\supp ww'=\supp w\cup\supp w'$.
\end{lemma}

We say that~$J,K\subset I$ are {\em orthogonal}
if $m_{jk}=2$ for all~$j\in J$, $k\in K$. We say that~$J\subset I$
is {\em self-orthogonal} if~$m_{ij}\le 2$ for all~$i,j\in J$.
A Coxeter
matrix~$M$ over~$I$ is said to be {\em irreducible} if~$I$ cannot be written as a
disjoint union of two non-empty orthogonal subsets or, equivalently,
if~$\Gamma(M)$ is connected. We denote~$\Gamma_J(M)$ the full
weighted subgraph
of~$\Gamma(M)$ with vertex set~$J$. Clearly, $\Gamma_J(M)=\Gamma(M_J)$. We say that~$J\subset I$
is {\em connected} if~$\Gamma_J(M)$ is connected as a graph or,
equivalently, if $J$ is not the disjoint union of two non-empty
orthogonal subsets. By abuse of terminology, we say that~$J\subset I$
is a {\em connected component} of~$I$ if~$\Gamma_J(M)$
is a connected component of~$\Gamma(M)$ or, equivalently, if~$J$ is a maximal connected subset of~$I$.

It is well-known (see, e.g.~\cite{Bou}*{Ch. VI, \S4, Thm.~1}) that the Coxeter group~$W(M)$ with irreducible~$M$ is finite if and only if
$\Gamma(M)$ is isomorphic to one of the following graphs 
\begin{alignat}{3}
A_n:&\dynkin[text style/.style={scale=0.8},Coxeter,root radius=0.07,expand labels={1,2,n-1,n},make indefinite edge={2-3},edge length=1.0cm]A4,&&n\ge 1,
\nonumber\\
B_n:&\dynkin[text style/.style={scale=0.8},Coxeter,root radius=0.07,expand labels={1,2,n-1,n},make indefinite edge={2-3},edge length=1.0cm]B4,&&n\ge 2,\nonumber\\
D_{n+1}: &\dynkin[text style/.style={scale=0.8},Coxeter,root radius=0.07,expand labels={1,2,n-1,n,n+1},label directions={,,right,,},make indefinite edge={2-3},edge length=1.0cm]D5,&&n\ge 3,\nonumber\\
E_n:&\dynkin[text style/.style={scale=0.8},Coxeter,ordering=Kac,root radius=0.07,expand labels={1,2,3,4,n-1,n},make indefinite edge={4-5},edge length=1.0cm]E6,&\qquad&
n\in\{6,7,8\},\nonumber\\
F_4:&\dynkin[text style/.style={scale=0.8},Coxeter,ordering=Kac,root radius=0.07,expand labels={1,2,3,4},edge length=1.0cm]F4,\nonumber\\
I_2(m):&
\dynkin[text style/.style={scale=0.8},Coxeter,ordering=Kac,root radius=0.07,expand labels={1,2},edge length=1.0cm,gonality=m]I2,&&m\ge 4,\nonumber\\
H_n:&\dynkin[text style/.style={scale=0.8},Coxeter,ordering=Kac,root radius=0.07,ordering=Adams,expand labels={1,2,n-1,n},make indefinite edge={2-3},edge length=1.0cm]H4
,&&n\in\{3,4\}.\label{eq:Coxeter graphs}
\end{alignat}
The labeling shown in~\eqref{eq:Coxeter graphs} will be used throughout the rest of
the paper unless specified otherwise.
Clearly, $I_2(3)$ (respectively,~$I_2(4)$) coincides with $A_2$ (respectively, $B_2$); the graph of type~$I_2(6)$
is traditionally denoted as~$G_2$. We will use~$X_n$ as the notation for the Coxeter matrix of the corresponding graph with the labeling as in~\eqref{eq:Coxeter graphs}. 
We will also denote $I_2(\infty)=\left(\begin{smallmatrix}1&\infty\\\infty&1\end{smallmatrix}\right)$, the corresponding
Artin monoid being just the free monoids with two generators and the associated Coxeter group
being isomorphic to the free product of two copies of~$\ZZ_2$.

An automorphism~$\sigma$ of the weighted graph~$\Gamma(M)$, or, equivalently
a permutation~$\sigma$ of~$I$ such that $m_{\sigma(i)\sigma(j)}=m_{ij}$
for all~$i,j\in I$, induces an automorphism of~$\Br^+(M)$ (respectively,
$\Br(M)$, $W(M)$), called a {\em diagram automorphism} and also denoted
by~$\sigma$, via $\sigma(T_i)=T_{\sigma(i)}$ (respectively,
$\sigma(s_i)=s_{\sigma(i)}$), $i\in I$. If~$W(M)$ is finite and~$\Gamma(M)$ is connected, diagram automorphisms of order~$2$
exist only if $\Gamma(M)$ is of type $A_n$, $n\ge 1$, $D_{n+1}$, $n\ge 3$, $F_4$, $E_6$ or~$I_2(m)$, the corresponding permutation of~$I$ being
\begin{equation}\label{eq:diag aut}
\sigma=\begin{cases}
\prod_{1\le i\le \frac12n} (i,n+1-i),& M=A_n,\, n\ge 2,\\
(n,n+1),&M=D_{n+1},\, n\ge 3,\\
(1,4)(2,3),&M=F_4,\\
(1,5)(2,4),&M=E_6.
\end{cases}
\end{equation}
In type $D_4$, there is also a diagram
automorphism of order~$3$ given by the permutation $(1,3,4)$ of~$[1,4]$ and
so the group of all diagram automorphisms of~$D_4$ is isomorphic to~$S_3$.

If~$J\in\mathscr F(M)$, then $W_J(M)$ contains the unique element~\plink{w0J}
$w_\circ^J$ of maximal length (see, e.g.~\cites{Bou,Tits}), which is obviously an involution. It is well-known (see e.g.~\cite{Bou}*{Ch. IV, Ex. 22} or~\cite{BjBr}*{Proposition~2.3.2}) that
\begin{equation}\label{eq:ell w w0}
\ell(ww_\circ^J)=\ell(w_\circ^J w)=\ell(w_\circ^J)-\ell(w),\qquad
w\in W_J(M).
\end{equation}
For~$J\subset K\in\mathscr F(M)$, we denote \plink{wJ;K}$w_{J;K}:=w_\circ^Jw_\circ^K$. 

\subsection{Hecke monoids}\label{subs:Hecke}
The {\em Hecke monoid} associated with $M$ is the quotient
of~$\Br^+(M)$ by the minimal congruence relation
containing $(T_i^2,T_i)$ for all~$i\in I$.
We denote~\plink{pi*M}$\pi^\star_M$ the canonical homomorphism from $\Br^+(M)$
to the corresponding Hecke monoid.
Thus, the Hecke monoid is generated
by the $s_i:=\pi^\star_M(T_i)$, $i\in I$
subject to relations $s_i\star s_i=s_i$, $i\in I$ and
$$
\brd{s_i\star s_j\star}{m_{ij}}=\brd{s_j\star s_i\star}{m_{ij}},\qquad
i\not=j\in I,\, m_{ij}\not=\infty.
$$
Note that~${}^{op}$ and diagram automorphisms factor through to
the Hecke monoid.
\begin{remark}
In the literature, Hecke monoids are also referred to as Coxeter monoids
(see e.g.~\cite{K14}), $0$-Hecke monoids or Demazure monoids. The latter term is due to the fact that idempotent Demazure operators provide a representation of Hecke monoids.
\end{remark}
\begin{proposition}[\cite{BGLHeck}*{Proposition~2.4}]\label{prop:prod *}
For all~$i\in I$, $w\in W(M)$
\begin{equation}\label{eq:prod *}
s_i\star w=\begin{cases}
s_i w,&\ell(s_i w)>\ell(w),\\
w,&\ell(s_i w)<\ell(w),
\end{cases}\qquad
w\star s_i=\begin{cases}
ws_i,&\ell(ws_i)>\ell(w),\\
w,&\ell(ws_i)<\ell(w),
\end{cases}
\end{equation}
where we abbreviate $w=\pi^\star_M(T_w)$.
In particular, $\pi^\star_M(\Br^+(M))$ identifies with~$W(M)$ as a set,  the restriction
of~$\pi^\star_M$ to~$\SQF^+(M)$ is a bijection onto~$W(M)$ and
$\pi^\star_M|_{\SQF^+(M)}=\pi_M|_{\SQF^+(M)}$.
\end{proposition}
It follows that
$$
\SQF^+(M)=\{ T\in\Br^+(M)\,:\,\ell(\pi^\star_M(T))=\ell(T)\}.
$$
From now on, we identify the Hecke monoid associated with
the Coxeter matrix~$M$ with the Coxeter group~$W(M)$ {\em as a set} and denote it
\plink{HeMon}$(W(M),\star)$. Note that $\supp (w\star w')=\supp w\cup\supp w'$ for all
$w,w'\in W(M)$. In fact,
Proposition~\ref{prop:prod *} can be regarded as a presentation of the Hecke monoid. Namely,
we can define it as $W(M)$, as a set, equipped with the unique associative operation~$\star$ satisfying the
first property in~\eqref{eq:prod *}. 
It follows that~$u\star v=uv$, $u,v\in W(M)$ if and only if
$\ell(uv)=\ell(u)+\ell(v)$, and in that case
we write~\plink{times}$uv=u\times v=u\star v$.

The following are immediate.
\begin{lemma}\label{lem:free product Artin}\label{lem:free product}
Let~$M\in\Cox I$, $M'\in\Cox{I'}$. 
\begin{enmalph}
\item\label{lem:free product.a}
The product in the category of monoids of~$\Br^+(M)$ and~$\Br^+(M')$
(respectively, $W(M)$ and~$W(M')$,
of~$(W(M),\star)$, $(W(M'),\star)$)
is isomorphic to~$\Br^+(M\times M')$
(respectively, $W(M\times M')$,
$(W(M\times M'),\star)$) where
$M\times M'\in\Cox{I\sqcup I'}$ is defined by
$$(M\times M')_{ij}=(M\times M')_{ji}=\begin{cases}
                                     m_{ij},&i,j\in I,\\
                                     m'_{ij},&i,j\in I',\\
                                     2,&i\in I,\,j\in I'.
                                    \end{cases}
                                    $$
\item \label{lem:free product.b}
The free product of monoids~$\Br^+(M)$ and~$\Br^+(M')$ (respectively, $W(M)$ and~$W(M')$,
$(W(M),\star)$ and $(W(M'),\star)$) is isomorphic to $\Br^+(M\coprod M')$
(respectively, $W(M\coprod M')$,
$(W(M\coprod M'),\star)$), where
$M\coprod M'\in\Cox{I\sqcup I'}$ is defined by
$$(M\textstyle\coprod M')_{ij}=(M\textstyle\coprod M')_{ji}=\begin{cases}
m_{ij},& i,j\in I,\\
m'_{ij},& i,j\in I',\\
\infty,& i\in I,\,j\in I'.
\end{cases}
$$
\end{enmalph}
\end{lemma}
\begin{lemma}\label{lem:orth factors}
Let~$M\in\Cox I$ and let $J,K\subset I$ be orthogonal.
Then
\begin{enmalph}
    \item\label{lem:orth factors.a}
    $\Br^+_{J\cup K}(M)\cong\Br^+_J(M)\times 
    \Br^+_K(M)$;
    \item\label{lem:orth factors.b}
    $W_{J\cup K}(M)\cong W_J(M)\times 
    W_K(M)$;
    \item\label{lem:orth factors.c}
    $(W_{J\cup K}(M),\star)\cong (W_J(M),\star)\times 
    (W_K(M),\star)$.
\end{enmalph}
In particular, submonoids $\Br^+_J(M)$, $\Br^+_K(M)$ (respectively,
$W_J(M)$, $W_K(M)$ and $(W_J(M),\star)$, $(W_K(M),\star)$) commute element-wise in~$\Br^+(M)$
(respectively, in~$W(M)$, $(W(M),\star)$).
\end{lemma}

Given~$w\in W(M)$, denote\plink{DL(w)}
$$
D_L(w)=\{ i\in I\,:\, \ell(s_iw)<\ell(w)\},\quad
D_R(w)=\{ i\in I\,:\, \ell(ws_i)<\ell(w)\}.
$$
Clearly, $D_R(w)=D_L(w^{-1})$. We will
now collect some facts about idempotents in Hecke monoids that will be useful in the sequel (see e.g.~\cite{BGLHeck}*{\S2.5--2.8} for the details).
\begin{proposition}\label{prop:idempot Hecke}
Let~$M\in\Cox I$, $w\in W(M)$ and~$J\subset I$.
\begin{enmalph}
\item\label{prop:idempot Hecke.0}
If~$J\in\mathscr F(M)$ then~$w_\circ^J$
is the unique element~$u$ of~$W_J(M)$ satisfying
$x\star u=u$ for all~$x\in (W_J(M),\star)$.
In particular, $w_\circ^J$ is an idempotent;
    \item\label{prop:idempot Hecke.a}
$w$ is an idempotent in~$(W(M),\star)$
if and only if~$\supp w\in\mathscr F(M)$ and~$w=w_\circ^{\supp w}$;
    \item\label{prop:idempot Hecke.b}
    Let~$J\subset I$. Then
    $J\in\mathscr F(M)$ if and only if
    $(W_J(M),\star)$ contains an idempotent~$w$ with~$\supp w=J$;
    \item\label{prop:idempot Hecke.c}
    For any~$w\in W(M)$,
    $\{x\in W(M)\,:\, x\star w=w\}$ is 
    finite and coincides with~$(W_{D_L(w)}(M),\star)$.
\end{enmalph}
\end{proposition}

\subsection{Divisibility, longest elements and Coxeter elements}\label{subs:w0J}
We say that $X\in\Br^+(M)$ is a left (respectively, right) divisor
of~$Y\in\Br^+(M)$ if $Y=XU$ (respectively, $Y=VX$) for some~$U\in\Br^+(M)$
(respectively, $V\in\Br^+(M)$). Since~$\Br^+(M)$ is cancellative,
such an element~$U$ (respectively, $V$), if exists, is unique
and will be denoted by~$(Y:X)_l$ (respectively, $(Y:X)_r$). The
following classical results will be often used in the sequel.
\begin{proposition}[\cite{BrSa}*{Lemma~5.1, Propositions~5.7,  Theorem~7.1}
and~\cite{Del}*{Theorem~4.21}]
\label{prop:fund elts BrSa}
Let~$J\in\mathscr F(M)$. Then
\begin{enmalph}
\item\label{prop:fund elts BrSa.a} $T_{w_\circ^J}$ is ${}^{op}$-invariant;
  \item\label{prop:fund elts BrSa.b} $T_{w_\circ^J}$ is the left and the right least common multiple
  of the~$T_j$, $j\in J$, that is, $T_{w_\circ^J}$ is left (respectively,
  right) divisible by all the $T_j$, $j\in J$ and is a left (respectively, right)
  divisor of every element of~$\Br^+(M)$ with that property;
\item\label{prop:fund elts BrSa.c'}
$X\in\Br^+_J(M)$ is left divisible by~$T_{w_\circ^J}$ if and
only if it is right divisible by~$T_{w_\circ^J}$;
\item\label{prop:fund elts BrSa.0}
If~$J=J_1\cup J_2$ with~$J_1$, $J_2$ orthogonal
then $T_{w_\circ^{J}}=T_{w_\circ^{J_1}}T_{w_\circ^{J_2}}$;
\item\label{prop:fund elts BrSa.c} There is a unique
diagram automorphism~$\Sigma_J$ of~$\Br^+_J(M)$  such that~$\Sigma_J^2=\id$ and
$X T_{w_\circ^J}=T_{w_\circ^J} \Sigma_J(X)$ for all~$X\in\Br^+_J(M)$;
\item\label{prop:fund elts BrSa.d}
The center of~$\Br^+_J(M)$ is generated by $T_{w_\circ^J}$ if~$\Sigma_J$
is trivial and by~$T_{w_\circ^J}^2$ otherwise;
\item\label{prop:fund elts BrSa.e}
$T_{w_\circ^J}$ is the unique element of~$\SQF^+(M)\cap \Br^+_J(M)$ of
maximal length and every square free element of~$\Br^+_J(M)$ is a
left and a right divisor of~$T_{w_\circ^J}$.
\end{enmalph}
\end{proposition}
The involution~$\Sigma_J$ for~$J\in\mathscr F(M)$ connected is
non-trivial only if~$M_J$ is of type~$A_n$, $n\ge 1$, $D_{n+1}$
with~$n$ even, $I_2(2m+1)$, $m\in\ZZ_{>0}$ or $E_6$. Note that~$\Br^+(D_{n+1})$ admits a
non-trivial diagram automorphism for all~$n\ge 3$, yet~$\Sigma$
is trivial if~$n$ is odd; likewise,
$\Br^+(F_4)$ admits a diagram automorphism, yet~$\Sigma$ is also trivial.

Given~$X\in \Br^+(M)$, define \plink{I(X)}$D_L(X)=\{ i\in I\,:\, \text{$T_i$ is
a left divisor of~$X$}\}$.
Let $i\in D_L(X)$. Then $X=T_i X'$ for some~$X'\in \Br^+(M)$
with~$\ell(X')=\ell(X)-1$ and so $$\pi_M^\star(X)=s_i\star \pi_M^\star(X')=
s_i\star s_i\star\pi^\star_M(X')=s_i\star \pi^\star_M(X),
$$
whence
$i\in D_L(\pi^\star_M(X))$ by Proposition~\ref{prop:prod *}.
Thus, $D_L(X)
\subset D_L(\pi^\star_M(X))
\in\mathscr F(M)$ by Proposition~\partref{prop:idempot Hecke.c}.
Then $X$ is left
divisible by $T_{w_\circ^{D_L(X)}}$ by Proposition~\partref{prop:fund elts BrSa.b}.
\begin{remark}
It should be noted that~$D_L(X)$ can be a proper subset of~$D_L(\pi^\star_M(X))$. For example, for~$X=T_1^3T_2^2 T_1 T_3^2T_2T_1\in \Br^+(A_3)$
we have~$D_L(X)=\{1\}$ while~$D_L(\pi^\star_{A_3}(X))=D_L(w_\circ^I)=I$. 
\end{remark}

Given~$X\in \Br^+(M)$, define inductively $D_0(X)=D_L(X)$ and $$
D_j(X)=D_L((X:\ascprod_{0\le k\le j-1} T_{w_\circ^{D_k(X)}})_l),\qquad  j\ge 1.
$$
Clearly, $D_k(X)=\emptyset$ for~$k\gg0$.
\begin{proposition}[\cite{BrSa}*{Theorem~6.3}]\label{prop:Normal form}
Let $X,Y\in\Br^+(M)$. Then
\begin{enmalph}
\item $X=\ascprod_{j\in\ZZ_{\ge 0}} T_{w_\circ^{D_j(X)}}$ (this expression for~$X$ is called its {\em normal form});
\item $X=Y$ if and only if~$D_j(X)=D_j(Y)$ for
all~$j\ge 0$.
\end{enmalph}
\end{proposition}

\begin{definition}\label{defn:weakly orthogonal}
Let~$M\in\Cox I$. We say that~$J,K\subset I$
are {\em weakly orthogonal} if 
$J\setminus K$ is orthogonal to~$K$ and $K\setminus J$ is orthogonal to~$J$. 
\end{definition}
In particular, every subset~$J$ of~$I$ is weakly
orthogonal to itself and if $J$, $K$ are orthogonal then they are weakly orthogonal.
The following Lemma is immediate from Proposition~\partref{prop:fund elts BrSa.0} and
Lemma~\partref{lem:orth factors.a}.
\begin{lemma}\label{lem:weakly orthogonal}
Let~$M\in\Cox I$, $J,K\in\mathscr F(M)$. If~$J$ and~$K$
are weakly orthogonal then $T_{w_\circ^J}T_{w_\circ^K}=
T_{w_\circ^K}T_{w_\circ^J}$. 
\end{lemma}

Let~$J\subset I$. We say that~$C\in \Br^+_J(M)$ (respectively, $c\in W_J(M)$) is a {\em Coxeter element} 
if~$\supp C=J$ (respectively, $\supp c=J$) and~$\ell(C)=|J|$ (respectively, $\ell(c)=|J|$).
In the sequel, we will often consider
special Coxeter elements corresponding to an interval~$J=[a,b]\subset I\subset \ZZ$,
namely \plink{cab}$\cx ab=\ascprod_{a\le i\le b} s_i$, $\cxr ab=\dscprod_{a\le i\le b}s_i$, $\Cx ab=T_{\cx ab}$ and~$\Cxr ab=T_{\cxr ab}=(\Cx ab)^{op}$. We will use the convention that~$\cx ij=\cxr ij=1$ if~$i>j$
and similarly for~$\Cx ij$ and~$\Cxr ij$.

It is well-known (see e.g.~\cite{Bou}*{Ch. V, \S6})
that if~$J\in\mathscr F(M)$ then all Coxeter elements~$c\in W_J(M)$
are conjugate and of the same order~\plink{h(M)}$h(M_J)$, called the {\em Coxeter number}
of~$W_J(M)$. The Coxeter number is even for all irreducible finite types
except~$I_2(2m+1)$, $m>0$ and $A_{2m}$.
Note also that if~$J\subset I$ is self-orthogonal then $T_{w_\circ^J}$ is the unique Coxeter element of~$W_J(M)$.
The following is established in~\cite{BrSa}.
\begin{proposition}[\cite{BrSa}*{\S5.8}]\label{prop:Coxeter splitting}
Let~$M$ be a Coxeter matrix and let~$J\in\mathscr F(M)$. Then for any Coxeter element~$C\in\Br^+_J(M)$
\begin{enmalph}
\item\label{prop:Coxeter splitting.a}
$T_{w_\circ^J}^2=C^{h(M_J)}$;
\item\label{prop:Coxeter splitting.b}
If $\Sigma_J$ from Proposition~\partref{prop:fund elts BrSa.c} is trivial then~$h(M_J)$ is even and $T_{w_\circ^J}=
C^{h(M_J)/2}$;
\item\label{prop:Coxeter splitting.c}
If~$M_J$ is irreducible and~$J=J'\cup J''$
is a partition of~$J$ into disjoint non-empty self-orthogonal subsets then
$$T_{w_\circ^J}=\brd{T_{w_\circ^{J'}}T_{w_\circ^{J''}}}{h(M_J)}
=\brd{T_{w_\circ^{J''}}T_{w_\circ^{J'}}}{h(M_J)}.
$$
\end{enmalph}
\end{proposition}

\section{General properties of homomorphisms
of Artin monoids}\label{sec:Gen homs}

Throughout this chapter, we denote
standard generators of~$W(M')$ or $(W(M'),\star)$
(respectively, $\Br^+(M')$) corresponding
to~$M'\in\Cox{I'}$
by~$s'_i$ (respectively, $T'_i$), $i\in I'$ and so on.
We will often use the obvious fact that a homomorphism of Artin monoids extends to a homomorphism of their ambient Artin groups.

Let~\plink{A C H}$\Art$ (respectively, $\CoxCat  $, $\Heck $) be the category whose objects are Coxeter matrices
and morphisms are homomorphisms of corresponding Artin monoids (respectively, Coxeter groups, Hecke monoids).
Parabolic submonoids and subgroups are, naturally, subobjects in these categories. By 
Lemma~\ref{lem:free product}
all these categories admit finite products and coproducts via, respectively, $(M,M')\mapsto M\times M'$ and
$(M,M')\mapsto M\coprod M'$, $M\in\Cox I$, $M'\in\Cox{I'}$ (see Remark~\ref{rem:prod coprod}).

\subsection{Homomorphisms of Artin monoids}\label{subs:Artin homs}
Let $\wh M\in\Cox{\wh I}$, $M\in\Cox I$ and $\Phi\in\Hom_{\Art}(\wh M,M)$.
We define
$[\Phi]:\wh I\to\mathscr P(I)$,
$i\mapsto \supp\Phi(\wh T_i)$, $i\in \wh I$
and extend it to a map
\plink{[Phi]}$[\Phi]:\mathscr P(\wh I)\to\mathscr P(I)$
via $[\Phi](\wh J)=\bigcup_{j\in\wh J}[\Phi](j)$, $\wh J\subset \wh I$.
\begin{definition}[cf.~\cite{BGLHeck}*{Definition~3.2}]
\label{def:types heck hom}
We say that~$\Phi\in\Hom_{\Heck }(M',M)$ is:
\begin{itemize}
\item[-] {\em disjoint} if $[\Phi](i)\cap[\Phi](j)=
\emptyset$ for all $i\not=j\in I'$;
\item[-] {\em fully supported} if $[\Phi](I')=I$;
\item[-] {\em connected} if~$[\Phi](J)$
is a connected subset of~$I$ for any connected~$J\subset I'$.
\end{itemize}
\end{definition}
\begin{lemma}\label{lem:[Phi]comp}
Let~$M\in\Cox I$, $M'\in\Cox{I'}$ and~$M''\in\Cox{I''}$.
\begin{enmalph} 
\item\label{lem:[Phi]comp.a}
$\supp\Phi(x)=[\Phi](\supp x)$
for any~$\Phi\in\Hom_{\Art}(M',M)$ and for
all~$x\in \Br^+(M')$;
\item\label{lem:[Phi]comp.b}
$[\Phi\circ\Phi']=[\Phi]\circ[\Phi']$ as maps $\mathscr P(I'')\to
\mathscr P(I)$ for any~$\Phi\in\Hom_{\Art}(M',M)$,
$\Phi'\in\Hom_{\Art}(M'',M')$;
\item\label{lem:[Phi]comp.c}
\label{lem:elem Artin hom.a}
   If~$\Phi\in \Hom_{\Art}(M',M)$ is
    disjoint and~$[\Phi](i)\not=\emptyset$ for all~$i\in\wh I$ then~$[\Phi]:\mathscr P(I')\to
    \mathscr P(I)$ is
    injective.
\item\label{lem:[Phi]comp.d} If~$\Phi\in\Hom_{\Heck }(M',M)$ is disjoint then~$\bigcap\limits_{1\le t\le r}[\Phi](J_t)=
    [\Phi](\bigcap\limits_{1\le t\le r}J_t)$ for any $\{J_t\}_{1\le t\le r}\subset \mathscr P(I')$.

\item\label{lem:[Phi]comp.e} $\Phi\in\Hom_{\Heck }(M',M)$ is
connected if and only if the~$[\Phi](i)$, $i\in I'$ are connected and~$[\Phi](i)\cup [\Phi](j)$
is connected whenever $m'_{ij}>2$, $i,j\in I'$.
\end{enmalph}
\end{lemma}
\begin{proof}
Since~$\supp(XY)=\supp X\,\cup\,\supp Y$ for all $X,Y\in \Br^+(M)$, we have for all~$T'\in\Br^+(M')$ $$\supp\Phi(T')=\bigcup_{j\in\supp T'}
\supp \Phi(T'_j)=\bigcup_{j\in\supp T'} [\Phi](j)=
[\Phi](\supp T'),$$
which proves~\ref{lem:[Phi]comp.a}.
To prove part~\ref{lem:[Phi]comp.b},
note that by part~\ref{lem:[Phi]comp.a} we have
for all~$T''\in \Br^+(M'')$
$$
[\Phi\circ\Phi'](\supp T'')=
\supp(\Phi\circ\Phi')(T'')
=[\Phi](\supp\Phi'(T''))
=[\Phi]([\Phi'](\supp T'')).
$$
Since $\supp: \Br^+(M'')\to \mathscr P(I'')$
is surjective, the assertion follows.

To prove~\ref{lem:[Phi]comp.c}, suppose that~$[\Phi](J)=[\Phi](J')$ for some~$J\not=J'$. We may assume, without loss of generality,
that $J'\not\subset J$. Let~$j\in J\setminus J'$. Then $\emptyset\not=[\Phi](j)\subset [\Phi](J)=[\Phi](J')=\bigcup_{j\in J'} [\Phi](j')$ which is a contradiction since~$[\Phi](j)\cap[\Phi](j')=\emptyset$ for all~$j'\in J'$. Finally, $[\Phi](J)\cap [\Phi](J')=
\bigcup_{j\in J,j'\in J'}[\Phi](j)\cap[\Phi](j')$.
Since~$\Phi$ is disjoint, $[\Phi](j)\cap[\Phi](j')=\emptyset$ unless~$j=j'$ and so
$[\Phi](J\cap J')=\bigcup_{j\in J\cap J'}[\Phi](j)=[\Phi](J\cap J')$. The general case in part~\ref{lem:[Phi]comp.d} follows by an obvious induction.

One direction in part~\ref{lem:[Phi]comp.e} is evident while the other follows by an obvious induction on the cardinality of~$I'$.
\end{proof}

The following is immediate from definitions.
\begin{lemma}\label{lem:fund hom}
Let~$M=(m)_{i,j\in I}$ be a Coxeter matrix, let~$\mathsf M$
be any multiplicative monoid
and let $X_i$, $i\in I$ be a collection of 
elements in~$\mathsf M$. The assignments 
$T_i\mapsto X_i$, $i\in I$ define 
a homomorphism of monoids $\Br^+(M)\to \mathsf M$
if and only if~$m_{ij}\in B(X_i,X_j)\cup\{\infty\}$
for all~$i\not=j\in I$.
\end{lemma}
\begin{example}\label{ex:char hom}
Let~$\wh M=(\wh m_{ij})_{i,j\in \wh I}\in \Cox{\wh I}$,
$M\in\Cox I$. Let~$\mathbf X=\{X_i\,:\,i\in\wh I\}\subset\Br^+(M)$, be a family of commuting elements satisfying $X_i=X_j$
whenever $\wh m_{ij}$ is odd, $i,j\in\wh I$. Then the assignments 
$\wh T_i\mapsto X_i$, $i\in\wh I$
define \plink{char hom}$\Xi_{\mathbf X}\in\Hom_{\Art}(\wh M,M)$. We call such a homomorphism a {\em character homomorphism}. The most basic
example is the generalization~$\ell_{\mathbf d}$, $\mathbf d=(d_i)_{i\in\wh I}\in\ZZ_{\ge 0}^{\wh I}$ of 
the length homomorphism $\ell:\Br^+(\wh M)\to (\ZZ_{\ge 0},+)\cong \Br^+(A_1)$ defined by $\wh T_i\mapsto d_i$,
$i\in \wh I$ where $d_i=d_j$ whenever~$\wh m_{ij}$ is odd.
\end{example}
The following Lemma guarantees that {\em any} factorization of any radical of a central element gives rise to a homomorphism of Artin monoids. 
\begin{lemma}\label{lem:cradicals}
Let~$\mathsf M$ be a left or right cancellative monoid, $X_1,X_2\in \mathsf M$ and~$m\in\ZZ_{>1}$.  The assignments~$T_i\mapsto X_i$, $i\in\{1,2\}$
define a homomorphism~$\Br^+(I_2(2m))\to\mathsf M$ if and only if 
$(X_1X_2)^m$ is central in the submonoid of~$\mathsf M$
generated by~$X_1$ and~$X_2$.
\end{lemma}
\begin{proof}
The forward direction is obvious. For the converse, let~$z=(X_1X_2)^m$ and suppose that~$X_1z=zX_1$. Then~$X_1(X_2X_1)^m=(X_1X_2)^mX_1= z X_1=X_1z$ and so~$(X_2X_1)^m=z$ since~$\mathsf M$ is 
left cancellative. The argument for a right cancellative monoid uses commutation with~$X_2$ and is omitted.
\end{proof}

\begin{definition}\label{def:types}
Let~$\wh M=(\wh m_{ij})_{i,j\in\wh I}\in\Cox{\wh I}$ and let~$M\in\Cox I$.
We say that~$\Phi:\Hom_{\Art}(\wh M,M)$ is:
\begin{itemize}
\item[-] {\em square free} if $\Phi(\wh T_i)\in \SQF^+(M)$
for all~$i\in\wh I$;
\item[-] {\em strongly square free} if $\Phi(\SQF^+(\wh M))\subset
\SQF^+(M)$.
\item[-] {\em optimal} if
$\wh m_{ij}=\min B(\Phi(\wh T_i),\Phi(\wh T_j))$ for all~$i,j\in\wh I$ such that $B(\Phi(\wh T_i),\Phi(\wh T_j))$ is non-empty and~$\Phi(\wh T_i)\not=\Phi(\wh T_j)$.
\end{itemize}
\end{definition}
\begin{lemma}\label{lem:elem Artin hom}
Let $M\in\Cox I$, $\wh M=(\wh m_{ij})_{i,j\in\wh I}\in\Cox{\wh I}$ and
let~$\Phi\in\Hom_{\Art}(\wh M,M)$.
\begin{enmalph}
  \item\label{lem:elem Artin hom.b} If~$\wh M$ is of finite type then
    $\Phi$ is strongly square free if and only if $\Phi(\wh T_{w_\circ^{\wh I}})\in\SQF^+(M)$.
    \item\label{lem:elem Artin hom.b'} $\Phi$
    commutes with~${}^{op}$ if and only
    if all the~$\Phi(\wh T_i)$, $i\in\wh I$ are
    ${}^{op}$-invariant.
    \item\label{lem:elem Artin hom.c}
    $\ell(\Phi(\wh T_i))=\ell(\Phi(\wh T_j))$
    for all $i,j\in \wh I$ such that~$\wh m_{ij}$
    is odd. In particular, if $[\Phi](i)=\emptyset$
    for some~$i\in\wh I$ then~$[\Phi](j)=\emptyset$
    for all~$j\in\wh I$ such that~$\wh m_{ij}$ is odd.
\end{enmalph}
\end{lemma}
\begin{proof}
Part~\ref{lem:elem Artin hom.b} is
immediate from Proposition~\partref{prop:fund elts BrSa.e}.
Part~\ref{lem:elem Artin hom.b'} is obvious.
To prove~\ref{lem:elem Artin hom.c}, note that
$$
\brd{\Phi(\wh T_i)\Phi(\wh T_j)}{\wh m_{ij}}=
\Phi(\brd{\wh T_i\wh T_j}{\wh m_{ij}})=\Phi(\brd{\wh T_j\wh T_i}{\wh m_{ij}})=
\brd{\Phi(\wh T_j)\Phi(\wh T_i)}{\wh m_{ij}}
$$
implies, for~$\wh m_{ij}$ odd,
that $$
\tfrac12(\wh m_{ij}-1)(\ell(\Phi(\wh T_i))
+\ell(\Phi(\wh T_j)))+\ell(\Phi(\wh T_i))=
\tfrac12(\wh m_{ij}-1)(\ell(\Phi(\wh T_j))
+\ell(\Phi(\wh T_i)))+\ell(\Phi(\wh T_j)).
$$
The assertion is now immediate.
\end{proof}

\begin{lemma}\label{lem:blowup}
Let~$M\in\Cox I$ and let
$\wh M=(d_{ij}m_{ij})_{i,j\in I}$ where
$d_{ij}=d_{ji}\in\mathbb Z_{>0}\cup\{\infty\}$ and~$d_{ii}=1$, $i,j\in I$. Then~$\wh M\in\Cox I$ and
the assignments $\wh T_i\mapsto T_i$, $i\in I$
define a homomorphism $\Br^+(\wh M)\to \Br^+(M)$.
\end{lemma}
\begin{proof}
This is immediate from Lemma~\partref{lem:taut homs.a}.
\end{proof}
We call such homomorphisms {\em tautological}, and they are, in a sense, the opposite of optimal ones.
\begin{example}\label{ex:taut not inj}
For any~$d\in\ZZ_{>0}$, $m\in\ZZ_{\ge 2}$ there is a tautological homomorphism from $\Br^+(I_2(dm))$ to~$\Br^+(I_2(m))$. Such a homomorphism
is never injective because 
$\brd{T_iT_j}{m}\not=\brd{T_jT_i}m$ in~$\Br^+(I_2(dm))$, yet these elements coincide in~$\Br^+(I_2(m))$.
\end{example}

\begin{lemma}\label{lem:factor homs}
Every homomorphism of Artin monoids is a
composition of a tautological homomorphism
with an optimal one.
\end{lemma}
\begin{proof}
Let $M=(m_{ij})_{i,j\in I}$, $\wh M=(\wh m)_{i,j\in\wh I}$ be Coxeter matrices
and let $\Phi\in\Hom_{\Art}(\wh M,M)$.
Define~$\wh M'=(\wh m'_{ij})_{i,j\in\wh I}$ as follows.
If~$\Phi(\wh T_i)=\Phi(\wh T_j)$, set 
set~$\wh m'_{ij}=\wh m_{ij}$. Otherwise,
if~$B_{ij}:=B(\Phi(\wh T_i),\Phi(\wh T_j))$ is empty then~$\wh m_{ij}=\infty$ and we set~$\wh m'_{ij}=\infty$. If~$B_{ij}\not=\emptyset$, then, since~$\Br^+(M)$ is cancellative,
$B_{ij}=
\ZZ_{>0}\wh m'_{ij}$ for some~$\wh m'_{ij}\in\ZZ_{>0}$
by Lemma~\partref{lem:taut homs.b}.
Moreover, since~$\Phi(\wh T_i)\not=\Phi(\wh T_j)$, $m'_{ij}>1$.
Since~$\wh m_{ij}\in B_{ij}\cup\{\infty\}$ by Lemma~\ref{lem:fund hom}, we have $\wh m_{ij}=d_{ij}\wh m'_{ij}$ for some~$d_{ij}\in\ZZ_{>0}\cup\{\infty\}$. Thus, $\wh M'\in\Cox{I'}$,
the assignments $\wh T_i\mapsto \wh T'_i$, $i\in \wh I$ define a tautological $\Phi'\in\Hom_{\Art}(\wh M,\wh M')$, and
the assignments $\wh T'_i\mapsto \Phi(\wh T_i)$, $i\in\wh I$, define an optimal
$\Phi''\in\Hom_{\Art}(\wh M',M)$. By construction,
$\Phi=\Phi''\circ\Phi'$.
\end{proof}

\subsection{Decorating homomorphisms from Artin
monoids}\label{subs:decor hom}
We will now discuss a machinery which allows us to produce new homomorphisms from Artin monoids to other monoids from existing ones. This construction will be used extensively in the sequel.

\begin{definition}\label{def:decoration}
Let~$M=(m_{ij})_{i,j\in I}$ be a Coxeter matrix, let~$\mathsf M$ be a multiplicative monoid and let~$\Phi:\Br^+(M)\to\mathsf M$ be a 
homomorphism of monoids. 
We say that
$\mathbf z=(z_i)_{i\in I}\in\mathsf M^I$ 
is a {\em decoration} of~$\Phi$ if 
the assignments~$T_i\mapsto \Phi(T_i)z_i$, $i\in I$
define a homomorphism~\plink{Phi z}$\Phi_{\mathbf z}:\Br^+(M)\to \mathsf M$.
\end{definition}
We will sometimes refer to~$\Phi_{\boldsymbol z}$
as a decorated companion of~$\Phi$. The following Lemma is immediate.
\begin{lemma}\label{lem:decs are invertible}
Let~$M\in\Cox I$, let $\mathsf M$ be a multiplicative
monoid and let~$\mathbf z=(z_i)_{i\in I}$ with all the~$z_i$ invertible. Then~$\mathbf z$ is a decoration of~$\Phi$ if and only if~$\mathbf z^{-1}=(z_i^{-1})_{i\in I}$ is a decoration of~$\Phi_{\mathbf z}$ and~$(\Phi_{\mathbf z})_{\mathbf z^{-1}}=\Phi$.
\end{lemma}
The following result provides a rather strong sufficient condition for the existence of a decoration of 
a given homomorphism, which will be used in multiple proofs later.
\begin{theorem}\label{thm:decoration sufficient}
Let~$M=(m_{ij})_{i,j\in I}$ be a Coxeter matrix, let~$\mathsf M$ be a multiplicative monoid and let~$\Phi:\Br^+(M)\to\mathsf M$ be a 
homomorphism of monoids. 
Suppose that for any~$i,j\in I$
we are given $z_{i,j}^{(k)}$, $k\in[1,m_{ij}]$
such that $z_{i,j}^{(1)}=z_{i,i}^{(1)}$ for all~$j\in I$
and for all~$i\not=j$ with~$m_{ij}<\infty$
\begin{enumerate}[label={$\arabic*^\circ.$},
ref={$\arabic*^\circ$}]
    \item\label{thm:decoration sufficient.1}
$z_{i,j}^{(k)}\Phi(T_j)=\Phi(T_j)z_{i,j}^{(k+1)}$
if~$k$ is odd while $z_{i,j}^{(k)}\Phi(T_i)=\Phi(T_i)z_{i,j}^{(k+1)}$ if~$k$
is even, $k\in[1,m_{ij}-1]$;
\item\label{thm:decoration sufficient.2} $z_{i,j}^{(m_{ij})}z_{j,i}^{(m_{ij}-1)}\cdots 
=z_{j,i}^{(m_{ij})}z_{i,j}^{(m_{ij}-1)}\cdots$.
\end{enumerate}
Then~$\mathbf z=(z_{i,i}^{(1)})_{i\in I}$ is a
decoration of~$\Phi$. Moreover, if~\ref{thm:decoration sufficient.1} is satisfied and~$\mathsf M$ is left 
cancellative then~$\mathbf z=(z_{i,i}^{(1)})_{i\in I}$
is a decoration of~$\Phi$ if and only if~\ref{thm:decoration sufficient.2} holds.
\end{theorem}
\begin{proof}
It suffices to prove the theorem for~$I$ with~$|I|=2$.
Let~$m=m_{ij}$, $\{i,j\}=I$ and assume that~$m<\infty$.
Abbreviate $t_i=\Phi(T_i)$, $i\in I$ and~$z_i^{(k)}=z_{i,j}^{(k)}$, $\{i,j\}=I$, $k\in[1,m]$. We will
use the convention that~$i+r=i$, $i\in I$ if~$r\in\ZZ$ is even and~$i+r=j$ if~$r$ is odd, $\{i,j\}=I$.
Using this convention,
the condition~\ref{thm:decoration sufficient.1}
can be written as
\begin{equation}\label{eq:decoration-1}
z_i^{(k)}\Phi(T_{i+k})=\Phi(T_{i+k})z_i^{(k+1)},\qquad 
i\in I,\,k\in[1,m-1],
\end{equation}
while~\ref{thm:decoration sufficient.2} becomes
\begin{equation}\label{eq:decoration-2}
    \dscprod_{1\le k\le m} z^{(k)}_{i-k}
    =\dscprod_{1\le k\le m} z^{(k)}_{i+r-k},\qquad i\in I,\,r\in\ZZ.
\end{equation}
\begin{lemma}\label{lem:decoration 1}
We have
$z_i^{(k)}\brd{t_{i+k}t_{i+k+1}}{l}=
\brd{t_{i+k}t_{i+k+1}}{l}z_i^{(k+l)}
$ for any~$i\in I$, $l\in[1,m]$ and $k\in[1,m-l]$.
\end{lemma}
\begin{proof}
We use induction on~$l$, the case~$l=1$ being just~\eqref{eq:decoration-1}.
For the inductive step, since  $\brd{t_jt_{j+1}}{r}=
t_j\,\brd{t_{j+1}t_{j+2}}{r-1}$ for any~$j\in I$, $r\in\ZZ_{>0}$, we have 
by~\eqref{eq:decoration-1} and the induction hypothesis
\begin{align*}
z_i^{(k)}\brd{t_{i+k}t_{i+k+1}}{l+1}&=
z_i^{(k)}t_{i+k}\brd{t_{i+k+1}t_{i+k+2}}{l}\\
&=t_{i+k}z_i^{(k+1)}\brd{t_{i+k+1}t_{i+k+2}}{l}
=\brd{t_{i+k}t_{i+k+1}}{l+1}z_i^{(k+l+1)}.\qedhere
\end{align*}
\end{proof}
\begin{lemma}\label{lem:decoration 2}
Let~$i\in I$. Then for any~$0\le l\le m$
\begin{align*}
\brd{(t_i z_i)(t_{i+1} z_{i+1})}{m}=\brd{(t_iz_i)(t_{i+1}z_{i+1})}{m-l}\,\brd{ t_{i+m-l}t_{i+m-l+1}}{l}\,
\dscprod_{1\le k\le l} z^{(k)}_{i+m-k}.
\end{align*}
\end{lemma}
\begin{proof}
We use induction on~$l$, the case~$l=0$ being trivial. 
For the inductive step we have by Lemma~\ref{lem:decoration 1} and the induction hypothesis
\begin{align*}
\brd{(t_i z_i)(t_{i+1} z_{i+1})}{m}
&=\brd{(t_iz_i)(t_{i+1}z_{i+1})}{m-l-1}
\, t_{i+m-l-1}z_{i+m-l-1}\,\brd{ t_{i+m-l}t_{i+m-l+1}}{l}\,
\dscprod_{1\le k\le l} z^{(k)}_{i+m-k}\\
&=\brd{(t_iz_i)(t_{i+1}z_{i+1})}{m-l-1}
\, \brd{ t_{i+m-l-1}t_{i+m-l}}{l+1}\,
z_{i+m-l-1}^{(l+1)}\dscprod_{1\le k\le l} z^{(k)}_{i+m-k}\\
&=\brd{(t_iz_i)(t_{i+1}z_{i+1})}{m-(l+1)}
\, \brd{ t_{i+m-(l+1)}t_{i+m-l}}{l+1}\,
\dscprod_{1\le k\le l+1} z^{(k)}_{i+m-k}.\qedhere
\end{align*}
\end{proof}
Using Lemma~\ref{lem:decoration 2} with~$m=l$ we obtain,
$$
\brd{(t_iz_i)(t_j z_j)}{m}=
\brd{t_i t_j}{m}\dscprod_{1\le k\le m} z^{(k)}_{i+m-k},
$$
where~$I=\{i,j\}$. Since
$\dscprod_{1\le k\le m} z^{(k)}_{i+m-k}=
\dscprod_{1\le k\le m} z^{(k)}_{i+k}=
\dscprod_{1\le k\le m} z^{(k)}_{j+k}$ by~\eqref{eq:decoration-2} while~$\Phi$ being a homomorphism yields~$\brd{t_it_j}{m}=
\brd{t_jt_i}m$, the first assertion follows.
The second assertion is now immediate.
\end{proof}

Note that if the~$\Phi(T_i)$, $i\in I$ are invertible, 
then the condition~\ref{thm:decoration sufficient.1} determines, for all~$i\not=j\in I$ with~$m_{ij}<\infty$,
the $z_{i,j}^{(k)}$, $k\in[1,m_{ij}-1]$ uniquely as
\begin{equation}\label{eq:inv decoration}
z_{i,j}^{(k+1)}=\begin{cases}
    \Phi(T_i)^{-1} z_{i,j}^{(k)}\Phi(T_i),&\bar k=0,\\
    \Phi(T_j)^{-1} z_{i,j}^{(k)}\Phi(T_j),&\bar k=1    
\end{cases}
\end{equation}
or, in the closed form $
z_{i,j}^{(k)}=\Phi(\brd{T_jT_i}{k-1})^{-1}
z_{i,i}^{(1)}\Phi(\brd{T_jT_i}{k-1})$, $k\in [1,m_{ij}-1]$.
\begin{corollary}\label{cor:dec iff}
Suppose that~$\mathsf M$ is left cancellative and
the $\Phi(T_i)$, $i\in I$ are invertible.
Then~$\mathbf z=(z_{i,i}^{(1)})_{i\in I}$  is a decoration of~$\Phi$ if and only if
the condition~\ref{thm:decoration sufficient.2} of
Theorem~\ref{thm:decoration sufficient} holds
for the~$z_{i,j}^{(k)}$ defined by~\eqref{eq:inv decoration}.
\end{corollary}

\begin{example}\label{ex:embedded std}
Let~$M=B_2$ and let~$\mathsf M=\Br(A_n)$, or~$\Br(D_{n+1})$ with~$n$ even, or~$\Br(E_6)$.
The assignments
$\wh T_1\mapsto 1$, $\wh T_2\mapsto T_{w_\circ^{I}}$ define a homomorphism~$\Phi:\Br^+(M)\to \mathsf M$
(cf. Example~\ref{ex:char hom}). Let~$\sigma$ be the involutive diagram automorphism of~$\mathsf M$ and suppose that~$J$ and~$\sigma(J)$ are weakly orthogonal.
Set
$z_{1,1}^{(1)}=T_{w_\circ^J}$ and~$z_{2,2}^{(1)}=1$. Then using~\eqref{eq:inv decoration} and
Proposition~\partref{prop:fund elts BrSa.c}
we obtain 
$z_{1,2}^{(2)}=T_{w_\circ^{I}}^{-1}T_{w_\circ^J}T_{w_\circ^{I}}=T_{w_\circ^{\sigma(J)}}$,
$z_{1,2}^{(3)}=z_{1,2}^{(2)}$ and
$z_{1,2}^{(4)}=T_{w_\circ^J}$, while~$z_{2,1}^{(k)}=1$,
$1\le k\le 4$. Since~$T_{w_\circ^J}$
commutes with~$T_{w_\circ^{\sigma(J)}}$ by Lemma~\ref{lem:weakly orthogonal} and so
$$
z_{1,2}^{(4)}z_{2,1}^{(3)}z_{1,2}^{(2)}z_{2,1}^{(1)}
=T_{w_\circ^{J}}T_{w_\circ^{\sigma(J)}}=T_{w_\circ^{\sigma(J)}}T_{w_\circ^J}=z_{2,1}^{(4)}z_{1,2}^{(3)}z_{2,1}^{(2)}z_{1,2}^{(1)}.
$$
Thus, $\mathbf z=(z_{1,1}^{(1)},z_{2,2}^{(1)})$ is a decoration of~$\Phi$
and $\Phi_{\mathbf z}:\Br^+(B_2)\to\mathsf M$ is given by $\wh T_1\mapsto T_{w_\circ^J}$, $\wh T_2\mapsto T_{w_\circ^{I}}$ and hence a homomorphism to the respective Artin monoid.

More generally, let~$M\in\Cox I$ be of finite type.
Let~$\wh I$ be a finite set.
Fix a total order~$\prec$ on~$\wh I$ and 
let~$\{J_i\}_{i\in\wh I}$ be any collection of non-empty subsets of~$I$ such $J_i\subset J_k$ whenever~$i\prec k\in \wh I$ and~$\bigcup_{i\in\wh I} J_i=I$. Let~$\sigma_i=\Sigma_{J_i}$, $i\in\wh I$, in the notation of Proposition~\partref{prop:fund elts BrSa.c}.
Define~$\wh M=(\wh m_{ik})_{i,k\in\wh I}$ by $\wh m_{ii}=1$, $i\in\wh I$
and
$$
\wh m_{ik}=\wh m_{ki}=\begin{cases}
2,& \text{$J_i=J_k$ or~$J_i=\sigma_k(J_i)$},\\
4,& \text{$J_i\subsetneq J_k$ and 
$J_i\not=\sigma_k(J_i)$ are weakly orthogonal},\\
\infty,&\text{otherwise}
\end{cases}
$$
for all~$i\prec k\in\wh I$, $i\not=k$. Then~$\wh M\in\Cox{\wh I}$.
By the above, the assignments~$\wh T_i\mapsto T_{w_\circ^{J_i}}$, $i\in\wh I$
define a homomorphism~$\Br^+(\wh M)\to \Br^+(M)$.
\end{example}

The following Lemma provides perhaps the simplest yet important example of a decoration.
\begin{lemma}\label{lem:cent decor}
Let~$M\in\Cox I$, let $\mathsf M$ be a multiplicative monoid and let
$\Phi:\Br^+(M)\to\mathsf M$ be a homomorphism for some multiplicative monoid~$\mathsf M$. Let~$\mathbf z=(z_i)_{i\in I}\in\mathsf M^I$ where 
$z_iz_j=z_jz_i$ for all~$i,j\in I$,
$z_i=z_j$ if~$m_{ij}$ is odd and all the~$z_i$, $i\in I$
are in the centralizer of the image of~$\Phi$.
Then~$\mathbf z$ is a decoration of~$\Phi$.
\end{lemma}
\begin{proof}
Let $z_{i,j}^{(k)}=z_i$
for all~$i,j\in I$, $k\in[1,m_{ij}]$. Suppose that~$i\not=j$ and~$m_{ij}<\infty$.
The condition~\ref{thm:decoration sufficient.1} of Theorem~\ref{thm:decoration sufficient} is obviously satisfied, while $z_{i,j}^{(m_{ij})}z_{j,i}^{(m_{ij}-1)}\cdots=
z_i^{\lceil \frac12 m_{ij}\rceil} z_j^{\lfloor\frac12 m_{ij}\rfloor}$ which is manifestly symmetric in~$i$ and~$j$ if~$m_{ij}$ is even and equals to $z_i^{m_{ij}}=z_j^{m_{ij}}$ if~$m_{ij}$ is odd.
\end{proof}
\begin{example}\label{ex:cent decoration}
Let~$M\in\Cox I$. Then any collection $\mathbf z=(z_i)_{i\in I}$
of central elements of~$\Br^+(M)$ satisfying
$z_i=z_j$ whenever~$m_{ij}$ is odd is a decoration of~$\id\in\Hom_{\Art}(M,M)$ and is manifestly optimal.
\end{example}

\subsection{Hecke and Coxeter type homomorphisms}\label{subs:Heck Cox hom}
Let~$\mathscr C$ be a category. For any subcategories~$\mathscr C_1$, $\mathscr C_2$
of~$\mathscr C$, define the category~$\Arr(\mathscr C_1,\mathscr C_2)$
whose objects are morphisms $f:X\to Y$
where~$X\in\mathscr C_1$ and~$Y\in\mathscr C_2$
and morphisms from~$f:X\to Y$ to~$f':X'\to Y'$
are pairs $(\phi,\psi)\in\Hom_{\mathscr C_1}(X,X')\times \Hom_{\mathscr C_2}(Y,Y')$
such that the diagram 
$$
    \begin{tikzcd}[ampersand replacement=\&]
	{X} \&\& X'\\ 
	{Y} \&\& {Y'}
	\arrow["\phi", from=1-1, to=1-3]
	\arrow["f", from=1-1, to=2-1]
	\arrow["f'", from=1-3, to=2-3]
	\arrow["\psi", from=2-1, to=2-3]
\end{tikzcd}
    $$
commutes. This generalizes the well-known
{\em arrow category} $\Arr(\mathscr C)$ of~$\mathscr C$ (see e.g.~\cite{Rom}) which coincides with~$\Arr(\mathscr C,\mathscr C)$; moreover, $\Arr(\mathscr C_1,\mathscr C_2)$ is always a subcategory 
of~$\Arr(\mathscr C)$, which 
is full if both~$\mathscr C_1$, $\mathscr C_2$ were full. The following is immediate.
\begin{lemma}\label{lem:forget}
Let~$\mathscr C_1$, $\mathscr C_2$ be 
subcategories of~$\mathscr C$. The assignments
$(f:X\to Y\in\Arr(\mathscr C_1,\mathscr C_2))\mapsto X$, $(\phi,\psi)\in
\Hom_{\Arr(\mathscr C_1,\mathscr C_2)}(f,f')\mapsto \phi$
(respectively, $(f:X\to Y)\mapsto Y$, $(\phi,\psi)\mapsto \psi$) define 
functors~$\mathcal F_i:\Arr(\mathscr C_1,\mathscr C_2)\to \mathscr C_i$,
$i\in\{1,2\}$.
\end{lemma}
\begin{lemma}\label{lem:Arr faithful}
Let~$\mathscr D$ be any subcategory of~$\Arr(\mathscr C_1,\mathscr C_2)$ 
whose objects are
epic morphisms in~$\mathscr C$. Then
the restriction of~$\mathcal F_1$ to~$\mathscr D$ is a faithful functor~$\mathscr D\to \mathscr C_1$. 
\end{lemma}
\begin{proof}
Let~$\pi_X:X\to \underline X$, $\pi_Y:Y\to \underline Y$ be objects in~$\mathscr D$
and suppose that~$(\phi,\psi),(\phi,\psi')\in\Hom_{\mathscr D}(\pi_X,\pi_Y)$.
Then $\psi\circ \pi_X=\pi_Y\circ \phi=
\psi'\circ\pi_X$ and, since~$\pi_X$ is epic,
$\psi'=\psi$. The assertion is now immediate.
\end{proof}

We now use this mini-theory in case when~$\mathscr C$ is the category of monoids~$\mathscr{Mon}$, $\mathscr C_1=\Art$, $\mathscr C_2=\Heck $ 
or~$\CoxCat$ and~$\mathscr D$ is 
the full subcategory $\mathscr D_{\mathscr{AH}}$ of~$\Arr(\Art,\Heck )$ 
(respectively the full subcategory $\mathscr D_{\mathscr{AC}}$ of
respectively $\Arr(\Art,\CoxCat)$)
whose objects are canonical surjective
homomorphisms~$\pi_M^\star:\Br^+(M)\to (W(M),\star)$ (respectively, $\pi_M:\Br^+(M)\to W(M)$). Denote by $\mathscr{AH}$
(respectively, by $\mathscr{AC}$)
the image of $\mathscr D_{\mathscr{AH}}$
(respectively, of $\mathscr D_{\mathscr{AC}}$)
under the functor~$\mathcal F_1$, which is 
faithful by Lemma~\ref{lem:Arr faithful}.
Finally, define the category~$\mathscr{ACH}$ 
with the same objects as~$\Art$ and with
$\Hom_{\mathscr{ACH}}(M,M')=
\Hom_{\mathscr{AH}}(M,M')\cap \Hom_{\mathscr{AC}}(M,M')$ for all $M,M'\in\Art$.

\begin{definition}\label{defn:Heck Coxeter}
We refer to any morphisms in~\plink{AH AC}$\mathscr{AH}$ (respectively, in~$\mathscr{AC}$, $\mathscr{ACH}$)
as homomorphisms of Artin monoids of {\em Hecke} (respectively, {\em Coxeter},
{\em Coxeter-Hecke}) type. Finally,
we say that a homomorphism of Artin monoids is {\em standard} if it is of Hecke type and is square-free.
\end{definition}
By Lemma~\ref{lem:Arr faithful}, $\mathscr{AH}$
(respectively, $\mathscr{AC}$) identifies with~$\mathscr D_{\mathscr{AH}}$ (respectively,
with~$\mathscr D_{\mathscr{AC}}$). Using the
functor~$\mathcal F_2$ from Lemma~\ref{lem:forget} we thus obtain functors~\plink{H C}$\mathsf H:\mathscr{AH}\to\Heck$ 
and $\mathsf C:\mathscr{AC}\to\CoxCat$. Explicitly, $\mathsf H(\Br^+(M))=(W(M),\star)$ (respectively, $\mathsf C(\Br^+(M))=W(M)$)
and for any morphism~$\Phi:\Br^+(\wh M)\to \Br^+(M)$ in~$\mathscr{AH}$
(respectively, in~$\mathscr{AC}$)
$\mathsf H(\Phi)$ (respectively, $\mathsf C(\Phi)$)
is the unique homomorphism
$(W(\wh M),\star)\to (W(M),\star)$ (respectively,
$W(\wh M)\to W(M)$) satisfying
$\mathsf H(\Phi)\circ\pi_{\wh M}^\star=
\pi_M^\star\circ\Phi$ (respectively, 
$\mathsf C(\Phi)\circ\pi_{\wh M}=
\pi_M\circ\Phi$).

Before providing examples, we establish a simple criterion for a homomorphism of Artin monoids to be of Coxeter or Hecke type. 
Let~$\wh M\in\Cox{\wh I}$ and~$M\in\Cox I$. Since $\pi_{\wh M}|_{\SQF^+(\wh M)}=\pi^\star_{\wh M}|_{\SQF^+(\wh M)}$ is a bijection~$\SQF^+(\wh M)\to W(\wh M)$, given
{\em any} map~$\Phi:\Br^+(\wh M)\to \Br^+(M)$ we obtain maps
\plink{barPhi}$\overline\Phi=\pi_M\circ\Phi\circ \pi_{\wh M}^{-1}:W(\wh M)\to W(M)$ and~$\overline\Phi_\star=\pi^\star_M\circ 
\Phi\circ\pi_{\wh M}^{-1}:(W(\wh M),\star)\to 
(W(M),\star)$. It should be noted that $\overline\Phi$ and~$\overline\Phi_\star$ need not be homomorphisms, and the passage from~$\Phi$
to~$\overline\Phi_\star$ or~$\overline\Phi$
is not, generally speaking, compatible with compositions. 
\begin{theorem}\label{thm:Main Thm Cox Heck}
Let~$\wh M\in\Cox{\wh I}$, $M\in\Cox I$ and let~$\Phi\in\Hom_{\Art}(\wh M,M)$. 
\begin{enmalph}
\item \label{thm:Main Thm Cox Heck.a}
The following are equivalent
\begin{enmroman}
\item \label{thm:Main Thm Cox Heck.a.i}
$\Phi$ is of Coxeter type;
\item  \label{thm:Main Thm Cox Heck.a.ii}
$\overline\Phi\in\Hom_{\CoxCat}(\wh M,M)$;
\item  \label{thm:Main Thm Cox Heck.a.iii}
$\overline\Phi(\wh s_i)$ is an involution for each~$i\in\wh I$.
\end{enmroman}
\item \label{thm:Main Thm Cox Heck.b}
The following are equivalent
\begin{enmroman}
\item \label{thm:Main Thm Cox Heck.b.i}
$\Phi$ is of Hecke type;
\item  \label{thm:Main Thm Cox Heck.b.ii}
$\overline\Phi_\star\in\Hom_{\Heck}(\wh M,M)$;
\item  \label{thm:Main Thm Cox Heck.b.iii}
$\overline\Phi_\star(\wh s_i)$ is an idempotent for each~$i\in\wh I$. Specifically, $\overline\Phi_\star(\wh s_i)=w_\circ^{[\Phi](i)}$, $i\in\wh I$.
\end{enmroman}
\item  \label{thm:Main Thm Cox Heck.c}
$\Phi$ is of Coxeter-Hecke type if and only if
$\overline\Phi(\wh s_i)=\overline\Phi_\star(\wh s_i)=w_\circ^{[\Phi](i)}$, $i\in\wh I$.
\item 
 \label{thm:Main Thm Cox Heck.d}
$\Phi$ is standard if and only if
$\Phi(\wh T_i)=T_{w_\circ^{[\Phi](i)}}$ for each~$i\in I$.
\end{enmalph}
\end{theorem}
\begin{proof}
Let~$\Phi\in\Hom_{\Art}(\wh M,M)$ be of Coxeter type. Then, since~$\pi_{\wh M}$ is surjective,
$\overline\Phi=\mathsf H(\Phi)$ and hence 
is a homomorphism of Coxeter groups $W(\wh M)\to W(M)$. In particular, this implies that~$\overline\Phi(s_i)$ is an involution for all~$i\in\wh I$. Thus, the implications 
\ref{thm:Main Thm Cox Heck.a.i}$\implies$
\ref{thm:Main Thm Cox Heck.a.ii}$\implies$\ref{thm:Main Thm Cox Heck.a.iii} in part~\ref{thm:Main Thm Cox Heck.a} are proven. 
The implications \ref{thm:Main Thm Cox Heck.b.i}$\implies$
\ref{thm:Main Thm Cox Heck.b.ii}$\implies$\ref{thm:Main Thm Cox Heck.b.iii} in part~\ref{thm:Main Thm Cox Heck.b}
are proven similarly, and the expression for~$\overline\Phi_\star(\wh s_i)$
follows from Proposition~\partref{prop:idempot Hecke.a}.
Furthermore,
if~$\overline\Phi\in\Hom_{\CoxCat}(\wh M,M)$
(respectively, $\overline\Phi_\star\in\Hom_{\Heck}(\wh M,M)$)
then $(\Phi,\overline\Phi)\in\Hom_{\Arr(\Art,\CoxCat)}(\pi_{\wh M},\pi_M)$ (respectively, $(\Phi,\overline\Phi_\star)\in\Hom_{\Arr(\Art,\CoxCat)}(\pi_{\wh M}^\star,\pi_M^\star)$)
that is, $\Phi$ is of Coxeter (respectively, Hecke) type. 

It remains to prove the implications
\ref{thm:Main Thm Cox Heck.a.iii}$\implies$\ref{thm:Main Thm Cox Heck.a.ii} in parts~\ref{thm:Main Thm Cox Heck.a} and~\ref{thm:Main Thm Cox Heck.b}. 
For, note
the following obvious
\begin{lemma}\label{lem:hom on gens}
Let~$\mathsf M$, $\mathsf M'$ be monoids and let~$S$
be a set of generators for~$\mathsf M$. The following
are equivalent for a map
$f:\mathsf M\to \mathsf M'$.
\begin{enmroman}
\item $f$ is a homomorphism of monoids;
\item\label{lem:hom on gens.b} $f(sx)=f(s)f(x)$ for all~$s\in S$, $x\in\mathsf M$.
\end{enmroman}
\end{lemma}
Let~$w\in W(\wh M)$, $i\in\wh I$. If~$\wh s_iw=\wh s_i\times w=\wh s_i\star w$ then by Theorem~\partref{thm:Tits.b}
$$
\overline\Phi(\wh s_i w)=\pi_M(\Phi(\wh T_{\wh s_iw}))
=\pi_M(\Phi(\wh T_i\wh T_w))
=\pi_M(\Phi(\wh T_{i}))\pi_M(\Phi(\wh T_w))
=\overline\Phi(\wh s_i)\overline\Phi(w),
$$
and similarly for~$\overline\Phi_\star$. Otherwise,
write~$w=\wh s_i\times w'$ for some~$w'\in W(M')$. By the above,
$\overline\Phi(w)=\overline\Phi(\wh s_i)\overline\Phi(w')$ and so,
since~$\overline\Phi(\wh s_i)$ is an involution,
$$
\overline\Phi(\wh s_iw)=\overline\Phi(w')=\overline\Phi(\wh s_i)^2\overline\Phi(w')=\overline\Phi(\wh s_i)\overline\Phi(w).
$$
Likewise,
$\overline\Phi_\star(w)=\overline\Phi_\star(\wh s_i)\star \overline\Phi_\star(w')$ whence, since~$\overline\Phi_\star(\wh s_i)$ is an idempotent,
$$
\overline\Phi_\star(\wh s_i\star w)=\overline\Phi_\star(w)=\overline\Phi_\star(\wh s_i)\star \overline\Phi_\star(w')=\overline\Phi_\star(\wh s_i)^{\star 2}\star\overline\Phi_\star(w
')=\overline\Phi_\star(\wh s_i)\star\overline\Phi_\star(w).
$$
Thus, both~$\overline\Phi$ and~$\overline\Phi_\star$ satisfy the condition~\ref{lem:hom on gens.b} from Lemma~\ref{lem:hom on gens} and hence are homomorphisms of respective monoids.

Parts~\ref{thm:Main Thm Cox Heck.c} and~\ref{thm:Main Thm Cox Heck.d} are
immediate from respective
conditions~\ref{thm:Main Thm Cox Heck.b.iii}
in parts~\ref{thm:Main Thm Cox Heck.a} and~\ref{thm:Main Thm Cox Heck.b}.
\end{proof}
\begin{corollary}\label{cor:elem prop Coxeter Hecke}
\begin{enmalph}
                                                    \item \label{cor:elem prop Coxeter Hecke.b'}
    All standard homomorphisms are of
    Coxeter-Hecke type.
    \item\label{cor:elem prop Coxeter Hecke.c}
    Square free Coxeter or Hecke type homomorphisms
    commute with~${}^{op}$.
 \end{enmalph}
 \end{corollary}
As shown in Example~\ref{ex:1.6}, the composition
of two standard homomorphisms is not necessarily standard. We now define the category~\plink{Ast}$\Ast$
as the subcategory
of~$\mathscr{ACH}$ generated by standard homomorphisms, that is, its objects are all Coxeter matrices and
morphisms are (compositions of) standard homomorphisms. The following is an immediate consequence of Theorem~\ref{thm:Main Thm Cox Heck}.
\begin{corollary}\label{cor:Ast hom}
The assignments $\Phi\mapsto \overline\Phi$
(respectively, $\Phi\mapsto \overline\Phi_\star$),
$\Phi\in\Hom_{\Ast}(\wh M,M)$, $\wh M\in\Cox{\wh I}$, $M\in\Cox I$ define functors 
$\Ast\to \CoxCat$ (respectively, $\Ast\to\Heck$.
In particular, any morphism in~$\Ast$ yields
a solution~$(\overline\Phi,\overline\Phi_\star)$ of Problem~\ref{prob:twin}.
\end{corollary}

\begin{example}\label{ex:non sqf non Cox}
One checks (it requires approximately 270 applications of braid relations) that
the assignments
$\wh T_1\mapsto T_1T_2T_1T_4$, $\wh T_2\mapsto (T_2T_3)^2$
define a homomorphism $\Br^+(I_2(10))\to \Br^+(A_4)$ which
is of Hecke type since~$\pi^\star_{A_4}(T_1T_2T_1T_4)=w_{\circ}^{\{1,2,4\}}$
and $\pi^\star_{A_4}((T_2T_3)^2)=s_3\star s_2\star s_3\star s_3=w_\circ^{\{2,3\}}$ but not of Coxeter type since $\pi_{M}((T_2T_3)^2)=s_3 s_2$
and hence is not an involution.
\end{example}

\begin{example}\label{ex: sqf cox not Hecke}
The assignments
$\wh T_i\mapsto T_{2i}T_{2i-1}T_{2i+1}T_{2i}$,
$i\in [1,n]$ define a strongly square free Coxeter type
homomorphism $\Br^+(A_n)\to \Br^+(A_{2n+1})$
(see~Theorem~\ref{thm:monomial brd}). 
However, this homomorphism is not of Hecke type since
$\pi^\star_{A_{2n+1}}(T_2 T_1 T_3 T_2)=s_2s_1s_3 s_2\not=w_\circ^{[1,3]}$.
\end{example}

\begin{example}
Let~$M$ be of finite irreducible type and let~$I=I_1\sqcup I_2$ be any partition of~$I$
into disjoint non-empty subsets.
Choose Coxeter elements~$C_j$, $j\in\{1,2\}$ in~$\Br^+_{I_j}(M)$. Then $C_1 C_2$ is a Coxeter element in~$\Br^+(M)$. Let~$m=h(M)/2$ if~$T_{w_\circ^I}$ is central in~$\Br^+(M)$ (in which case $h(M)$ is even) and~$m=h(M)$
otherwise. By Proposition~\ref{prop:Coxeter splitting},
$(C_1C_2)^{m}=T_{w_\circ}^{2m/h(M)}$ is
central in~$\Br^+(M)$. It follows from Lemma~\ref{lem:cradicals} that
the assignments $\wh T_i\mapsto C_i$,
$i\in \{1,2\}$ define $\Phi\in\Hom_{\Art}(I_2(2m),M)$ which 
is manifestly square free and, unless~$I_1$ and~$I_2$ are
self-orthogonal, is neither of Hecke nor of Coxeter type. If~$h(M)$ is even, then~$\Phi$ is also strongly square free.
\end{example}
\begin{example}
For all~$n\ge 2$, $1\le k\le n-3$ the
assignments~$\wh T_1\mapsto T_{w_{[1,k];[1,n+k]}}$,
$\wh T_2\mapsto T_{w_{[k+2,n-1];[k+2,2n-1]}}$ define 
a homomorphism~$\Br^+(B_2)\to \Br^+(A_{2n-1})$  which is square free but neither of
Hecke nor of Coxeter type (see Corollary~\ref{cor:strange homs}).
\end{example}

\begin{example}\label{ex:not sqf not cox not Hecke}
The assignments $\wh T_1\mapsto T_1T_2T_3T_2T_1T_3$,
$\wh T_2\mapsto T_2T_3T_4T_3T_2T_4$ define a homomorphism
$\Phi:\Br^+(I_2(5))\to\Br^+(A_4)$ which is not square free and is neither Hecke,
nor Coxeter type. Indeed, $\pi_M(\Phi(\wh T_1))=
s_1s_2s_3s_2s_1s_3=s_1s_3s_2s_1$ which is not an involution
since its square equals~$s_1s_2s_1s_3$,
while $\pi^\star_M(\Phi(\wh T_1))=s_1\star s_2\star s_3\star s_2\star s_1
\star s_3=
s_1s_2s_3s_2s_1\not=w_\circ^{[1,3]}$.
\end{example}
\begin{example}\label{ex:char hom heck}
The character homomorphism~$\Xi_{\mathbf X} \in\Hom_{\Art}(\wh M,M)$
from Example~\ref{ex:char hom} is of Hecke type if and only if~$\pi^\star_M(X_i)$, $i\in\wh I$
is a family of commuting idempotents in $(W(M),\star)$. In particular, this happens if~$\pi^\star_M(X_i)=w_\circ^J$ for all~$i\in\wh I$ where $J\subset I$. 
\end{example}
\begin{example}\label{ex:shift by center}
Let~$M\in\Cox I$ be of finite type and 
let~$\mathbf z=(z_i)_{i\in I}$ be
as in Example~\ref{ex:cent decoration}.
Then~$\id_{\mathbf z}$ is an endomorphism of~$\Br^+(M)$ of Hecke type. 
Indeed, it suffices to verify that for~$M$ irreducible in which case $z_i=T_{w_\circ^I}^{a_i}$ for some~$a_i\in\ZZ_{\ge 0}$ by Proposition~\partref{prop:fund elts BrSa.d}. Then
$\pi^\star_M(z_i T_i)=w_\circ^{I}$ for all~$i\in I$ such that~$a_i>0$ and~$\pi^\star_M(z_iT_i)=s_i$ otherwise. 
\end{example}
\begin{example}
Homomorphisms from Example~\ref{ex:embedded std} are
of Coxeter-Hecke type.
\end{example}

\begin{lemma}\label{lem:T_w0 divides image}
Let $\Phi\in\Hom_{\mathscr{AH}}(\wh M,M)$ and let~$J\in\mathscr F(\wh M)$. Then
\begin{enmalph}
    \item\label{lem:T_w0 divides image.b}
    $[\Phi](J)\in\mathscr F(M)$ and $\overline\Phi_\star(w_\circ^J)
    =w_\circ^{[\Phi](J)}$.
    \item\label{lem:T_w0 divides image.a} If $\Phi$ is standard then
    $\Phi(\wh T_{w_\circ^J})=
    T_{w_\circ^{[\Phi](J)}}u=\Sigma_{[\Phi](J)}(u) T_{w_\circ^{[\Phi](J)}}$ for some~$u\in\Br^+_{[\Phi](J)}(M)$.
\end{enmalph}
\end{lemma}
\begin{proof}
Denote~$w=\overline\Phi_\star(w_\circ^J)$.
Then also $w=\pi^\star_M(\Phi(\wh T_{w_\circ^J}))$, whence
$\supp w=[\Phi](J)$. Since~$w_\circ^J$
is an idempotent in~$(W(\wh M),\star)$, it follows
that~$w$ is an idempotent in~$(W(M),\star)$ and
so $\supp w=[\Phi](J)\in\mathscr F(M)$ and~$w=w_\circ^{[\Phi](J)}$ by Proposition~\partref{prop:idempot Hecke.a}.

To prove part~\ref{lem:T_w0 divides image.a}, let~$j\in J$. Then~$\wh T_{w_\circ^J}$ is left divisible by~$\wh T_j$ by Proposition~\partref{prop:fund elts BrSa.b} whence
$\Phi(\wh T_{w_\circ^J})\in\Br^+_{[\Phi](J)}(M)$ is left divisible
by $\Phi(\wh T_j)$. Since~$\Phi$ is square free, $\Phi(\wh T_j)=T_{w_\circ^{[\Phi](j)}}$ by Theorem~\partref{thm:Main Thm Cox Heck.d}
and so,
again by Proposition~\partref{prop:fund elts BrSa.b},
is left divisible by all the~$T_i$, $i\in[\Phi](j)$.
Therefore, $\Phi(\wh T_{w_\circ^J})$ is left divisible by all the~$T_i$ with~$i\in[\Phi](J)=\bigcup_{j\in J}[\Phi](j)$. Since~$[\Phi](J)\in\mathscr F(M)$ by part~\ref{lem:T_w0 divides image.b}, $\Phi(\wh T_{w_\circ^J})$
is
left divisible by~$T_{w_\circ^{[\Phi](J)}}$ by Proposition~\partref{prop:fund elts BrSa.b}. The second equality follows
by Proposition~\partref{prop:fund elts BrSa.c}.
\end{proof}
\begin{lemma}\label{lem:induced inj}
Let~$\wh M\in\Cox{\wh I}$, $M\in\Cox I$ and
let~$\Phi\in\Hom_{\Art}(\wh M,M)$ be
strongly square free. Then~$\overline\Phi=\overline\Phi_\star$ and~$
\Phi(\wh T_w)=T_{\overline\Phi(w)}$ for all $w\in W(\wh M)$. In particular, if~$\Phi$ is injective then so is~$\overline\Phi$. 
\end{lemma}
\begin{proof}
Since~$\Phi$ is strongly square free, for any~$w\in W(\wh M)$ we have~$\Phi(\wh T_w)=T_{w'}$ for some~$w'\in W(M)$. Then~$\overline\Phi(w)=\pi_M(\Phi(\wh T_w))
=\pi_M(T_{w'})=w'$ and~$\overline\Phi_\star(w)=
\pi_M^\star(\Phi(\wh T_w))=\pi_M^\star(T_{w'})=w'$. The remaining assertions are immediate.
\end{proof}
\begin{remark}\label{rem:induced do not have to be injective}
If~$\Phi$ is not strongly square free then, even then it is injective and~$\overline\Phi$ or~$\overline\Phi_\star$ (or even both of them) are homomorphisms, they do not have to be injective. For example, the homomorphism
$\Br^+(B_2)\to\Br^+(A_2)$, $T_i\mapsto T_i^{1+\delta_{i,2}}$ is of Coxeter-Hecke type by Theorem~\ref{thm:Main Thm Cox Heck} and 
is injective by~\cite{Cri}.
Yet, $\overline\Phi$ is obviously not injective
since~$\overline\Phi(s_2)=1$. Also, $\overline\Phi_\star$ is the tautological
homomorphism $(W(B_2),\star)\to (W(A_2),\star)$
and is not injective since $s_1\star s_2\star s_1\not=s_2\star s_1\star s_2$ in~$(W(B_2),\star)$.

For a more complicated example, the assignments 
$\wh T_1\mapsto T_1T_3T_4T_3$, $\wh T_2\mapsto T_2$ define an optimal standard homomorphism~$\Br^+(I_2(8))\to \Br^+(A_4)$ (see Theorem~\ref{thm:main thm adm}), which, yet conjecturally, is expected to be injective. However,
the induced homomorphisms of Coxeter groups~$W(I_2(8))\to W(A_4)$
is not injective since $(s_2s_1s_3s_4s_3)^2=
(s_4s_3s_2)s_1s_3(s_2s_3s_4)$ whence
$(s_2s_1s_3s_4s_3)^4=1$. Likewise,
the induced homomorphism of Hecke monoids
$(W(I_2(8)),\star)\to (W(A_4),\star)$ is not injective since $s_2\star s_1s_3s_4s_3
\star s_2\star s_1s_3s_4s_3\star s_2=w_\circ^{[1,4]}$.
\end{remark}

Note some additional properties of
standard homomorphisms.
\begin{lemma}\label{lem:sqf Hecke hom}
Let~$\wh M=(\wh m_{ij})_{i,j\in\wh I}\in\Cox{\wh I}$, $M\in\Cox I$ and let~$\Phi\in\Hom_{\Art}(\wh M,M)$ be
standard. Then
\begin{enmalph}
   \item\label{lem:sqf Hecke hom.a} $\Phi$ is uniquely determined by~$[\Phi]:\wh I\to \mathscr F(M)$;
    \item\label{lem:sqf Hecke hom.b}
    If~$\wh m_{ij}$, $i\not=j\in\wh I$ is odd then
    $\ell(w_\circ^{[\Phi](i)})=\ell(w_\circ^{[\Phi](j)})$;
    \item\label{lem:sqf Hecke hom.c}
    If~$\wh m_{ij}$, $i,j\in\wh I$ is even then
$(T_{w_\circ^{[\Phi](i)}}T_{w_\circ^{[\Phi](j)}})^{\wh m_{ij}}$ is ${}^{op}$-invariant;
\end{enmalph}
\end{lemma}
\begin{proof}
Part~\ref{lem:sqf Hecke hom.a} is Theorem~\partref{thm:Main Thm Cox Heck.d}. Part~\ref{lem:sqf Hecke hom.b}
follows from Lemma~\partref{lem:elem Artin hom.c}
and Theorem~\partref{thm:Main Thm Cox Heck.d}.
Part~\ref{lem:sqf Hecke hom.c} follows from Corollary~\partref{cor:elem prop Coxeter Hecke.c}.
\end{proof}
\begin{lemma}\label{lem:parab preserving}
Let~$M\in\Cox I$, $\wh M\in\Cox{\wh I}$ 
and suppose that~$J\in\mathscr F(\wh I)$
for all~$J\subset\wh I$ with~$|J|=2$.
Suppose that~$\Phi\in\Hom_{\Art}(\wh M,M)$ satisfies $\Phi(\wh T_{w_\circ^J})=T_{w_\circ^{[\Phi](J)}}$
for all~$J\in\mathscr F(\wh M)$. Then~$\Phi$
is disjoint, of Coxeter-Hecke type,  
$\Phi(\wh T_{w_{J; K}})=
T_{w_{[\Phi](J);[\Phi](K)}}$
and also $\overline\Phi(w_{J; K})
=\overline\Phi_\star(w_{J; K})=w_{[\Phi](J);[\Phi](K)}$
for all~$J\subset K\in\mathscr F(\wh M)$.
\end{lemma}
\begin{proof}
The assumption that~$\Phi(\wh T_{w_\circ^J})=T_{w_\circ^{[\Phi](J)}}$
for any~$J\in\mathscr F(M)$ implies that~$\Phi$ is square
free and of Hecke type and standard. 
Suppose that~$J=[\Phi](i)\cap [\Phi](j)\not=\emptyset$ for some~$i\not=j\in\wh I$. Since $w_\circ^{[\Phi](i)}=
w_i^{-1}\times w_\circ^J$ and
$w_\circ^{[\Phi](j)}=w_\circ^J\times w_j$ where
$w_k=w_{J;[\Phi](k)}$, $k\in\{i,j\}$, we have
$$
T_{w_\circ^{[\Phi](i)}}T_{w_\circ^{[\Phi](j)}}=
T_{w_i}^{op} T_{w_\circ^J}^2 T_{w_j}.
$$
It follows that~$\Phi(\wh T_{w_\circ^{\{i,j\}}})=\Phi(\brd{\wh T_i\wh T_j}{\wh m_{ij}})$
is not square free and hence cannot be equal to~$T_{w_\circ^K}$
for any~$K\in\mathscr F(M)$. Thus, $\Phi$  is disjoint. 
Since~$\Phi(\wh T_{w_\circ^J})=T_{w_\circ^{[\Phi](J)}}$ for any~$J
\in\mathscr F(\wh M)$ and~$\wh T_{w_{J;K}}=\wh T_{w_\circ^J}^{-1}\wh T_{w_\circ^K}$ in~$\Br(\wh M)$,
it follows that~$\Phi(\wh T_{w_{J;K}})=
T_{w_\circ^{[\Phi](J)}}^{-1}T_{w_\circ^{[\Phi](K)}}=
T_{w_{[\Phi](J);[\Phi](K)}}$. Then~$\overline\Phi(w_{J;K})=\pi_M(\Phi(\wh T_{w_{J;K}}))
=\pi_M(T_{w_{[\Phi](J);[\Phi](K)}})=
w_{[\Phi](J);[\Phi](K)}$. The identity for~$\overline\Phi_\star$ is proved similarly.
\end{proof}

In some cases we can reconstruct a Coxeter or Hecke type homomorphism of Artin monoids from its ``shadow''.
\begin{lemma}\label{lem:lifting to Cox-Hecke}
Let~$\wh M=(\wh m_{ij})_{i,j\in\wh I}\in\Cox{\wh I}$, $M\in\Cox I$ and 
suppose that the $X_i\in\Br^+(M)$, $i\in\wh I$
satisfy 
\begin{enumerate}[label={$\arabic*^\circ.$},ref=
{$\arabic*^\circ$}]
    \item\label{lem:lifting to Cox-Hecke.1} $\brd{X_iX_j}{\wh m_{ij}}\in\SQF^+(M)$ for 
    all $i,j\in\wh I$ with~$\wh m_{ij}<\infty$;
    \item\label{lem:lifting to Cox-Hecke.2} The assignments $\wh s_i\mapsto\pi_M(X_i)$, $i\in \wh I$ define a homomorphism of Coxeter groups $W(\wh M)\to W(M)$
    
    or

\noindent
    the assignments $\wh s_i\mapsto\pi^\star_M(X_i)$, $i\in\wh I$ define a homomorphism $(W(\wh M),\star)\to (W(M),\star)$.
\end{enumerate}
Then the assignments $\wh T_i\mapsto X_i$, $i\in\wh I$
define a homomorphism $\Br^+(\wh M)\to\Br^+(M)$.
\end{lemma}
\begin{proof}
Let~$i\not=j\in\wh I$ with~$\wh m_{ij}<\infty$. Since 
$$\brd{\pi_M(X_i)\pi_M(X_j)}{\wh m_{ij}}=\pi_M(\brd{X_iX_j}{\wh m_{ij}})
$$ 
and~$\pi_M|_{SQF^+(M)}$ is a bijection onto~$W(M)$, the assumption~\ref{lem:lifting to Cox-Hecke.2} implies that~$\wh m_{ij}\in B(X_i,X_j)$. It remains to apply Lemma~\ref{lem:fund hom}. The argument in the Hecke version is identical.
\end{proof}

Let~$M$ be a Coxeter matrix over~$I$.
While for Hecke monoids the parabolic
projection
$p_J:(W(M),\star)\to (W_J(M),\star)$, $s_i\mapsto s_i$, $i\in J$ and~$s_i\mapsto 1$ if~$i\in I\setminus J$,
is well-defined for any~$J\subset I$
(see for example~\cite{BGLHeck}), in the framework
of Artin monoids and Coxeter groups analogous homomorphisms exist only in special cases.
\begin{proposition}\label{prop:parab proj Artin}
Let~$M=(m_{ij})_{i,j\in I}$ be a Coxeter matrix and
let~$J\subset I$. The assignments
$$
T_i\mapsto \begin{cases}
T_i,&i\in J,\\
1,&i\in I\setminus J,
\end{cases}
$$
define a surjective Hecke type homomorphism \plink{PJ}$P_J:\Br^+(M)\to\Br_J^+(M)$ if and only
if $m_{ij}$ is even for all $j\in J$, $i\in I\setminus J$.
In particular, if $J$ and~$I\setminus J$ are orthogonal
then both~$P_J$ and~$P_{I\setminus J}$ are homomorphisms
of respective Artin monoids and $T=P_J(T)P_{I\setminus J}(T)$ for
all~$T\in\Br^+(M)$. Moreover, the same assertions hold for Coxeter groups.
\end{proposition}
\begin{proof}
It follows from Lemma~\partref{lem:elem Artin hom.c} that if~$P_J$
is a homomorphism then~$m_{ij}$ is even for all~$i\in I\setminus J$, $j\in J$. 
For the converse, note that, for all~$i\in I\setminus J$, $j\in J$ 
we have $\brd{T_i\cdot 1}{m_{ij}}=\brd{1\cdot T_i}{m_{ij}}=
T_i^{\frac12m_{ij}}$ since~$m_{ij}$ is 
even and so~$P_J$ is a homomorphism.

If~$J$ and~$I\setminus J$ are orthogonal then $m_{ij}=2$ for all $i\in I\setminus J$, $j\in J$ and so both~$P_J$ and~$P_{I\setminus J}$ are well-defined homomorphisms.

To prove the last statement, we use induction on~$\ell(T)$, $T\in \Br^+(M)$,
the case~$\ell(T)=0$ being obvious. For the inductive step, write
$T=T_i T'$ with~$i\in I$, $\ell(T')=\ell(T)-1$. Then~$T=T_i P_J(T')
P_{I\setminus J}(T')$ by the induction hypothesis. If~$i\in J$ then
$T_i P_J(T')=P_J(T_i T')
=P_J(T)$ and~$P_{J\setminus I}(T')=P_{J\setminus I}(T_i)P_{J\setminus I}(T')=P_{J\setminus I}(T_i T')=
P_{J\setminus I}(T)$. If~$i\in I\setminus J$ then, since~$J$ and~$I\setminus J$ are orthogonal, $T_i$ commutes with~$P_J(T')$ and $P_J(T')=P_J(T_i) P_J(T')=P_J(T_iT')=P_J(T)$ while
$T_iP_{I\setminus J}(T')=p_{I\setminus J}(T_i T')=p_{I\setminus J}(T)$. Thus, we have
$T=P_J(T)P_{I\setminus J}(T)$.
\end{proof}

For irreducible Coxeter matrices of finite type, the only non-trivial examples of parabolic projections of Artin monoids are $P_{[1,n-1]}\in\Hom_{\Art}(B_n,A_{n-1})$, 
$P_{\{1,2\}}\in\Hom_{\Art}(F_4,A_2)$,
$P_{\{n\}}\in \Hom_{\Art}(B_n,\allowbreak A_1)$
and $P_{\{i\}}\in\Hom_{\Art}(I_2(2m),A_1)$, $m\ge 2$,
$i\in\{1,2\}$.

\begin{remark}\label{rem:prod coprod}
For any Coxeter matrices $M$, $M'$,
canonical morphisms $M\times M'\to M$
and~$M\times M'\to M'$ in either of categories~$\Art$, $\CoxCat  $ and~$\Heck $ involve parabolic projections in
an obvious way. 
\end{remark}

The following Lemma allows us to reduce the study of Hecke type homomorphisms of Artin monoids to that of fully supported ones with a connected codomain.
\begin{lemma}\label{lem:diagonal}
Let $M\in\Cox I$, $\wh M\in\Cox{\wh I}$ and let $\Phi\in\Hom_{\Art}(\wh M,M)$
be of Hecke type (respectively, of Coxeter type, square free). Let~$J\subset
\Phi(\wh I)$ and suppose that~$J$ and~$\Phi(\wh I)\setminus J$ are orthogonal. Then $P_J\circ\Phi$ is of Hecke type (respectively, of Coxeter type, square free).

Conversely, given pairwise orthogonal $J_1,\dots,J_k\subset I$ and
Hecke (respectively, Coxeter, square free) homomorphisms
$\Phi_t:\Br^+(\wh M)\to \Br^+_{J_t}(M)$, $1\le t\le k$, the map
$\Phi:\Br^+(\wh M)\to\Br^+(M)$ defined by
$T\mapsto \Phi_1(T)\cdots\Phi_k(T)$, $T\in\Br^+(\wh M)$ is a Hecke (respectively,
Coxeter, square free) homomorphism~$\Br^+(\wh M)\to\Br^+(M)$.
\end{lemma}
\begin{proof}
Observe that~$\pi^\star_{M_J}\circ P_J=p_J\circ \pi^\star_M$ and
$\pi_{M_J}\circ P_J=\tilde p_J\circ \pi_M$ where~$\tilde p_J:W(M)\to W_J(M)$ is the Coxeter group counterpart of~$P_J$. Thus, if $\pi^\star_M(\Phi(\wh T_i))$
(respectively, $\pi_M(\Phi(\wh T_i))$)
is
an idempotent (respectively, an involution) then so is~$\pi^\star_{M_J}(
P_J\circ \Phi(\wh T_i))$ (respectively, $\pi_{M_J}(P_J\circ \Phi(\wh T_i))$).
Furthermore, if $X_i=\Phi(\wh T_i)\in \SQF^+(M)$
then, since
$X_i=P_J(X_i)P_{I\setminus J}(X_i)$ by Proposition~\ref{prop:parab proj Artin}, it follows from Lemma~\ref{lem:sq free fact} that~$P_J(X_i)=P_J\circ\Phi(\wh T_i)$ is square free.

For the converse, since images of the~$\Phi_t$, $1\le t\le k$ commute in~$\Br^+(M)$, it follows that $\Phi(TT')=\prod_{1\le t\le k}\Phi_t(TT')
=\prod_{1\le t\le k}\Phi_t(T)\prod_{1\le k\le t}\Phi_t(T')=
\Phi(T)\Phi(T')$ and so~$\Phi$ is indeed a homomorphism. Since
the product of commuting idempotents (respectively, involutions)
is again an idempotent (respectively, an involution), it follows that
if all the~$\Phi_t$ are Hecke (respectively, Coxeter) then so is~$\Phi$.
Finally, since the product of square free elements
from~$\Br^+_{J}(M)$ and~$\Br^+_{K}(M)$ where~$J$ and~$K$ are orthogonal is obviously square free, if $\Phi_t(T_i)\in \SQF^+(M)\cap \Br^+_{J_t}(M)$ for all~$1\le t\le k$ then $\Phi(T_i)\in \SQF^+(M)$.
\end{proof}

We now discuss faithfulness and fullness of functors~$\mathsf H$
and~$\mathsf C$.

\begin{example}\label{ex:non-faithful}
Let~$M\in\Cox I$, $|I|>1$ be of finite type and irreducible. Let~$\mathbf z=(z_i)_{i\in I}$
where all~$z_i$ are central and not equal to~$1$ with
$z_i=z_j$ whenever~$m_{ij}$ is odd,
and let~$\mathbf z^2=(z_i^2)_{i\in I}$.
Then both~$\mathbf z$, $\mathbf z^2$ are decorations
of~$\id\in\Hom_{\Art}(M,M)$ and
$\id_{\mathbf z}\not=\id_{\mathbf z^2}$. 
Yet 
$\overline{\id_{\mathbf z}}_\star(s_i)=w_\circ^I=
\overline{\id_{\mathbf z^3}}_\star(s_i)$ for all~$i\in I$. Thus, $\mathsf H$ is not faithful.

Similarly, if~$M$ is as above and not of type
$A_n$, $D_{n+1}$ with~$n$ odd or~$E_6$, let
$\mathbf z=(z_i)_{i\in I}$ where~$z_i=T_{w_\circ^I}$
and let~$\mathbf z^3=(z_i^3)_{i\in I}$. Then
$\overline{\id_{\mathbf z}}(s_i)=w_\circ^I=
\overline{\id_{\mathbf z^3}}(s_i)$,
yet~$\id_{\mathbf z}\not=\id_{\mathbf z^3}$.
Thus, $\mathsf C$ is not faithful either.
\end{example}

We say that~$\phi\in\Hom_{\Heck }(\wh M,M)$
is (square free) {\em liftable}
if there exists a (square free)
$\Phi\in\Hom_{\Art}(\wh M,M)$ of Hecke type such that
$\overline\Phi_\star=\phi$. The corresponding notions
for homomorphisms of Coxeter groups are defined similarly.
Note that if~$\phi$
is square free liftable then~$\Phi(\wh T_{ w})=T_{\phi(w)}$
for all~$w\in W(\wh M)$.
\begin{lemma}\label{lem:non-lift-parab}
Let~$M\in\Cox I$.
For any~$J\subset I$, the parabolic projection $p_J:(W(M),\star)\to (W_J(M),\star)$ is liftable if and only if~$P_J$ is a well-defined homomorphism $\Br^+(M)\to\Br^+_J(M)$. 
\end{lemma}
\begin{proof}
If~$P_J$ is well-defined then, clearly, $p_J=\overline{(P_J)}_\star$. 
Suppose that~$P_J$ is not well-defined and $p_J=\overline{\Phi}_\star$ for some homomorphism $\Phi:\Br^+(M)\to \Br^+_J(M)$. By Proposition~\ref{prop:parab proj Artin}, there exist $i\in I\setminus J$, $j\in J$ such that~$m_{ij}\ge 3$ is odd. Since~$1=p_J(s_i)=\pi^\star_M(\Phi(T_i))$, it follows that~$[\Phi](i)=\emptyset$. Then by Lemma~\partref{lem:elem Artin hom.c}
we have~$[\Phi](j)=\emptyset$, whence 
$\phi(s_j)=\pi^\star(\Phi(T_j))=1$ which is a contradiction.
\end{proof}
\begin{example}\label{ex:non-liftable}
By Lemma~\ref{lem:non-lift-parab},
a parabolic projection~$p_J:(W(A_n),\star)\to (W_J(A_n),\star)$ with~$J\not=\emptyset,[1,n]$ is non-liftable.
\end{example}
\begin{example}\label{ex:non-liftable-1}
\label{ex:H3 D5}
The assignments
$
s'_1\mapsto s_4,\quad s'_2\mapsto s_1s_3$, $ s'_3\mapsto s_2 s_5
$
define a homomorphism of Hecke monoids $\phi:(W(H_3),\star)\to (W(D_5),\star)$. Indeed, $\phi$ is the composition of the homomorphism $(W(H_3),\star)\to
(W(D_6),\star)$ induced by the standard
homomorphism~\eqref{eq:unfold F4 E6} of respective Artin monoids with the
parabolic projection~$p_{[2,6]}:(W(D_6),\star)\to (W_{[2,6]}(D_6),\star)\cong (W(D_5),\star)$.

The homomorphism~$\phi$ is not liftable. Indeed, if~$\phi=\overline\Phi_\star$ for some~$\Phi\in\Hom_{\Art}(H_3,D_5)$ then $\Phi(T'_1)=T_4^{a_4}$,
$\Phi(T'_2)=T_1^{a_1}T_3^{a_3}$ and~$\Phi(T'_3)=T_2^{a_2}T_5^{a_5}$ where
$a_i\in\mathbb Z_{>0}$, $1\le i\le 5$
and $a_4=a_1+a_3=a_2+a_5$ by Lemma~\partref{lem:elem Artin hom.c}. Then the relation~$T'_1T'_2T'_1=T'_2T'_1T'_2$ in~$\Br^+(H_3)$
yields $T_4^{a_1+a_3}T_1^{a_1}T_3^{a_3}T_4^{a_1+a_3}
=T_1^{a_1}T_3^{a_3}T_4^{a_1+a_3}T_1^{a_1}T_3^{a_3}$
whence, since $\Br^+(D_5)$ is cancellative,
$$T_4^{a_1+a_3}T_3^{a_3}T_4^{a_1+a_3}
=T_3^{a_3}T_4^{a_1+a_3}T_3^{a_3}T_1^{a_1}.
$$
This forces~$a_1=0$ which is a contradiction.
\end{example}
\begin{example}
Let~$M=\tilde C_2$ (see Example~\ref{ex:affine}). Then~$p_{\{1,3\}}$ is liftable by Lemma~\ref{lem:non-lift-parab}. However, for any~$d_1,d_3\in\ZZ_{>0}$,
the assignments $T_1\mapsto T_1^{d_1}$,
$T_2\mapsto 1$, $T_3\mapsto T_3^{d_2}$
define a homomorphism $\Phi:\Br^+(M)\to \Br^+_{\{1,3\}}(M)$ such that~$\overline\Phi_\star=p_{\{1,3\}}$.
\end{example}
Thus, $\mathsf H$ is neither full nor faithful.

\begin{example}
The homomorphism of Artin monoids from Example~\ref{ex:non sqf non Cox} induces a homomorphism
of Hecke monoids~$(W(I_2(10)),\star)\to
(W(A_4),\star)$, $s'_1\mapsto s_1s_2s_1 s_4$,
$s'_2\mapsto s_2s_3s_2$
which is, therefore, liftable. However, it
is not square free liftable
since the canonical image of $X=((T_1T_2T_1T_4)(T_2T_3T_2))^5$
in~$W(A_4)$ is equal to $s_2s_1s_3s_2s_4$ which is not an involution. Then by Lemma~\ref{lem:can image op inv}, $X$
is not ${}^{op}$-invariant which contradicts Lemma~\partref{lem:sqf Hecke hom.c}.
\end{example}

\begin{remark}\label{rem:C not full}
Since every Coxeter group embeds into a symmetric group
in many different ways, it is highly unlikely that~$\mathsf C$ is full. For example, the left multiplication in~$S_3$ defines a homomorphism
$S_3=W(A_2)\to W(A_5)=S_6$ given by $\wh s_1\mapsto 
(1,2)(3,4)(5,6)=s_1s_3s_5$, $\wh s_2\mapsto(1,3)(2,5)(4,6)=s_2s_1s_4s_3s_2s_5s_4$. This
homomorphism does not have a straightforward
lifting to a homomorphism~$\Br^+(A_2)\to\Br^+(A_5)$,
and there are no reasons to expect it to have a more
sophisticated one.
\end{remark}

\begin{remark}
It is easy to see that the class of strongly square free homomorphisms of Artin monoids is closed under compositions.
\end{remark}

\subsection{Homogeneous homomorphisms of Artin monoids}\label{subs:parab hom Artin}
Let~$\wh M\in\Cox{\wh I}$, $M\in \Cox I$. Following~\cite{BGLHeck}*{\S3.2}  
we 
define $\ell_f:W(\wh M)\to \ZZ_{\ge 0}$
by~$\ell_f(w)=\ell(f(w))$, $w\in W(\wh M)$
for any map~$f:W(\wh M)\to W(M)$. We say that~$f$
is {\em homogeneous} if $\ell_f(ww')=\ell_f(w)+\ell_f(w')$ for all $w,w'\in W(\wh M)$ such that~$\ell(ww')=\ell(w)+\ell(w')$. By~\cite{BGLHeck}*{Proposition~3.20}, 
$f$ is a homomorphism of {\em both}
Hecke monoids and Coxeter groups if and only if
$f$ is homogeneous and a homomorphism
of one of these objects.
\begin{remark}
A map between two Coxeter groups can be homogeneous without being a homomorphism. For example, the 
map $f:W(A_2)\to W(A_2)$ defined
by $f(1)=1$, $f(s_i)=s_1$, $i\in\{1,2\}$,
$f(s_is_j)=s_1s_2$, $\{i,j\}=\{1,2\}$ and
$f(s_1s_2s_1)=s_1s_2s_1$ is homogeneous, as $\ell_f(w)=\ell(w)$ for all~$w\in W(A_2)$.
Yet it is not a homomorphism of groups
since~$f(s_1s_2)=s_1s_2\not=1=f(s_1)f(s_2)$.
\end{remark}

The following allows us to define homogeneous homomorphisms of Artin monoids.
\begin{theorem}\label{thm:homogeneous Artin}
Let~$\wh M\in\Cox{\wh I}$, $M\in\Cox I$
and let $f:W(\wh M)\to W(M)$ be a map.
\begin{enmalph}
   \item  \label{thm:homogeneous Artin.a}
Suppose that, for each~$i\in\wh I$, $f(\wh s_i)=w_\circ^{J_i}$ for some~$J_i\in\mathscr F(M)$ and that the assignments $\wh T_w\mapsto T_{f(w)}$, $w\in W(\wh M)$ define a homomorphism $\Phi_f:\Br^+(\wh M)\to \Br^+(M)$.
Then~$f$ is a homogeneous homomorphism of Coxeter groups (and hence of Hecke monoids).
\item\label{thm:homogeneous Artin.b} If~$f$ is a homogeneous 
homomorphism of Coxeter groups (or Hecke monoids) then the assignments $\wh T_w\mapsto T_{f(w)}$, $w\in W(\wh M)$ define a standard homomorphism of respective Artin monoids.
\end{enmalph}
\end{theorem}
\begin{proof}
If~$\Phi_f$ is a homomorphism of Artin monoids
then it is standard by assumption on~$f$ and~$\overline{\Phi}_f=f=\overline{\Phi_f}_\star$. Since $\overline\Phi\in\Hom_{\CoxCat}(\wh M,M)$ and~$\overline\Phi_\star\in\Hom_{\Heck}(\wh M,M)$ by 
Theorem~\ref{thm:Main Thm Cox Heck},
part~\ref{thm:homogeneous Artin.a} follow
from the characterization of homogeneous homomorphisms discussed above.

To prove part~\ref{thm:homogeneous Artin.b}, 
note that by Theorem~\ref{thm:Tits} it suffices 
to verify that~$T_{f(w)f(w')}=T_{f(w)}T_{f(w')}$
for all~$w,w'\in W(\wh M)$ such that~$\ell(ww')=\ell(w)+\ell(w')$. But this
is immediate since~$f$ is a homogeneous homomorphism of Coxeter groups.
\end{proof}

Henceforth we call homomorphisms~$\Phi_f$ homogeneous. The following is immediate.
\begin{corollary}\label{cor:comp homogeneous}
The class of homogeneous homomorphisms of Artin monoids is closed with respect to compositions.
\end{corollary}
The following Corollary is immediate from Theorem~\ref{thm:homogeneous Artin} and~\cite{BGLHeck}*{Theorem~3.23}. 
\begin{corollary}\label{cor:adm finite class}
Let~$\wh M\in\Cox{\wh I}$, $M\in\Cox I$ be irreducible and 
of finite type. The following is 
a complete list of fully supported homogeneous
homomorphisms
$\Phi:\Br^+(\wh M)\to \Br^+(M)$. 
\begin{enmalph}
\item\label{cor:adm finite class.unfold}
For~$\wh M=B_n$, $n\ge 2$, 
\begin{alignat}{4}
&M=A_{2n-1}:\label{eq:unfold Bn A2n-1}
&\qquad&\Phi(\wh T_i)=T_i T_{2n-i},&\quad &i\in [1,n-1],&\quad& 
\Phi(\wh T_n)=T_n,
\\
&M=A_{2n}:\label{eq:unfold Bn A2n}
&& \Phi(\wh T_i)=T_i T_{2n+1-i},&& i\in [1,n-1],&& \Phi(\wh T_n)=T_n T_{n+1}T_n,
\\
&M=D_{n+1}:\label{eq:unfold Bn Dn+1}
&& \Phi(\wh T_i)=T_i,&& i\in [1,n-1],&& \Phi(\wh T_n)=T_n T_{n+1};
\end{alignat}
\item\label{cor:adm finite class.unfold F4}
For~$\wh M=F_4$, $M=E_6$ 
\begin{align}
\label{eq:unfold F4 E6}
\Phi(\wh T_1)=T_1 T_5,\quad \Phi(\wh T_2)
=T_2 T_4,\quad \Phi(\wh T_3)=T_3,\quad 
\Phi(\wh T_4)=T_6
\end{align}
\item \label{cor:adm finite class.odd}
For~$\wh M=I_2(2m+1)$, $m>0$, $M=A_{2m}$
$$\Phi(\wh T_i)=\prod_{j\in [1,2m+1-i]_2} T_j, \qquad i\in\{1,2\};
$$
\item\label{cor:adm finite class.even}
For $\wh M=I_2(2m)$, $m>1$,
any~$M$ with~$h(M)=2m$ and any
partition $I=I_1\sqcup I_2$ of~$I$ into
non-empty self-orthogonal subsets 
$$
\Phi(\wh T_j)=T_{w_\circ^{I_j}}=
\prod_{i\in I_j} T_i, \qquad j\in\{1,2\};
$$
\item\label{cor:adm finite class.I8 F4} For~$\wh M=I_2(8)$, $M=F_4$,
$$
\Phi(\wh T_1)=T_1 T_4,\quad \Phi(\wh T_2)=T_2 T_3 T_2;
$$

\item\label{cor:adm finite class.H3}
For~$\wh M=H_3$, $M=D_6$, 
\begin{equation}\label{eq:unfold H3D6}
\Phi(\wh T_1)=T_1T_5,\quad \Phi(\wh T_2)=T_2T_4,\quad \Phi(\wh T_3)=T_3T_6;
\end{equation}

\item\label{cor:adm finite class.H4}
For~$\wh M=H_4$, $M=E_8$,
\begin{equation}\label{eq:unfold H4E8}
\Phi(\wh T_1)=T_1T_7,\quad \Phi(\wh T_2)=T_2T_6,\quad 
\Phi(\wh T_3)=T_3T_5,\quad \Phi(\wh T_4)=
T_4T_8.
\end{equation}
\end{enmalph}
\end{corollary}
\begin{remark}\label{rem:injectivity}
This list coincides with that of LCM homomorphisms studied, in particular, 
in~\cites{Cri,God,Cas,Mue}.
By~\cites{Cri}, all
homomorphisms listed above are injective. 
Moreover,
homomorphisms from parts~\ref{cor:adm finite class.unfold} and~\ref{cor:adm finite class.unfold F4} are isomorphisms onto submonoids of~$\Br^+(M)$ fixed by respective diagram automorphisms.
\end{remark}

\section{Light homomorphisms of Artin monoids}
\label{sec:Parab proj}\label{sec:light}

In this section we describe a class of homomorphisms which unifies parabolic projections,
natural inclusions of parabolic submonoids and
tautological homomorphisms. 

\subsection{Light homomorphisms of Artin monoids}\label{subs:parab proj Artin}
Similarly to the case of Hecke monoids (cf.~\cite{BGLHeck}*{Definition~4.1}),
we now introduce the notion of light
homomorphisms of Artin monoids.
\begin{definition}\label{defn:light Artin}
Let~$\wh M\in\Cox{\wh I}$, $M\in\Cox I$.
We say that~$\Phi\in\Hom_{\Art}(\wh M,M)$ is {\em light}
if~$|[\Phi](i)|\le 1$ for all~$i\in \wh I$.
\end{definition}
A light homomorphism is manifestly of Coxeter-Hecke type.
A composition of light homomorphisms is light, as well as their free product.
Clearly, natural inclusions of parabolic submonoids,
parabolic projections,
tautological homomorphisms and diagram automorphisms are light. Also, if we write a light homomorphism
$\Phi$ as $\Phi''\circ\Phi'$ where~$\Phi''$ is optimal and~$\Phi'$ is tautological (Lemma~\ref{lem:factor homs}) then~$\Phi''$ is also light.

The following is immediate.
\begin{lemma}\label{lem:H functor light}
Let~$\wh M\in\Cox{\wh I}$, $M\in\Cox I$ and let
$\Phi\in\Hom_{\Art}(\wh M,M)$ be light. Then~$\overline\Phi_\star$ is light.
In other words, the functor~$\mathsf H$ restricts to the functor~$\mathsf H_{light}$ from the light Artin category to the light Hecke category.
\end{lemma}

Another class of light homomorphisms is provided by foldings.
\begin{definition}\label{defn:foldable}
Let~$\varpi:I\to J$ be a surjective map.
We say that~$M\in\Cox I$ is {\em foldable along~$\varpi$} if
 $m_{ii'}=m_{ii''}$ for all
 $i,i',i''\in I$ with $\varpi(i')=\varpi(i'')\not=
\varpi(i)$.
\end{definition}
Note that any group~$G$ of automorphisms of~$\Gamma(M)$
gives rise to a map~$\varpi_G:I\to I/G$
such that~$M$ is foldable along~$\varpi_G$.

If~$M$ is foldable along~$\varpi$,
define~\plink{MII}$M^{\varpi}$ to be the matrix over~$J$ with~$(M^\varpi)_{jj}=1$, $j\in J$ and
$(M^{\varpi})_{jj'}=m_{ii'}$
for any~$i\in \varpi^{-1}(j)$, $i'\in \varpi^{-1}(j')$, $j\not=j'\in J$.
Clearly, $M^{\varpi}$
is a Coxeter matrix.

\begin{lemma}\label{lem:orbits fold}
Let~$\varpi:I\to J$ be surjective and suppose that~$M$ is foldable along~$\varpi$. The
assignments~$T_i\mapsto T^{\varpi}_{\varpi(j)}$, $i\in I$, where the~$T^\varpi_j$, $j\in J$
are generators of~$\Br^+(M^\varpi)$,
define a surjective light homomorphism of monoids \plink{F wpi}$\mathbf F_{\varpi}:\Br^+(M)\to \Br^+(M^{\varpi})$.
\end{lemma}
\begin{proof}
Let~$i\not=i'\in I$ with~$m_{ii'}<\infty$. If~$\varpi(i)=\varpi(i')$
then $\brd{T^\varpi_{\varpi(i)}T^\varpi_{\varpi(i')}}{m_{ii'}}=(T^\varpi_{\varpi(i)})^{m_{ii'}}=
\brd{T^\varpi_{\varpi(i')}T^\varpi_{\varpi(i)}}{m_{ii'}}.
$
Otherwise, $(M^\varpi)_{\varpi(i)\varpi(i')}=m_{ii'}$
and so $\brd{T^{\varpi}_{\varpi(i)}T^{\varpi}_{
\varpi(i')}}{m_{ii'}}
=\brd{T^{\varpi}_{\varpi(i')}T^{\varpi}_{\varpi(i)}}{m_{ii'}}$. The last assertion is obvious.
\end{proof}
\begin{example}
For any~$M\in\Cox I$, $\ell:\Br^+(M)\to (\ZZ_{\ge 0},+)\cong\Br^+(A_1)$
identifies with~$\mathbf F_\varpi$ where~$\varpi$ is the unique map~$I\to\{1\}$.
\end{example}

\begin{example}\label{ex:affine}
Using Lemma~\ref{lem:orbits fold} we can obtain more homomorphisms similar to those discussed in Example~\ref{ex:1.6}. 
Consider affine Coxeter graphs labeled as follows:
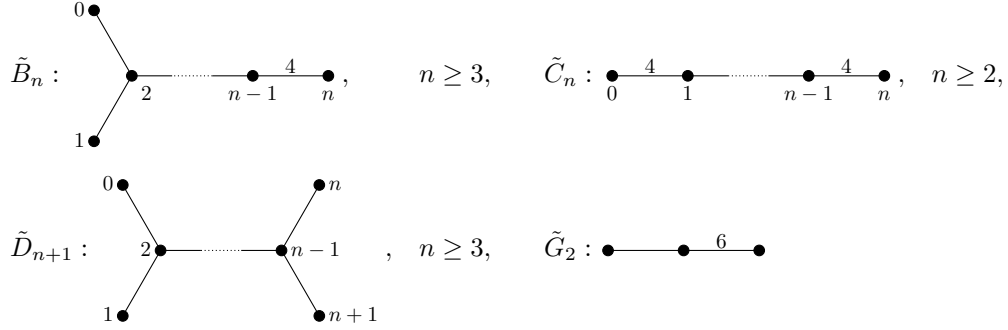
\begin{alignat*}{8}
&\tilde B_n: \dynkin[extended,o/.append style={fill=black},labels={0,1,2,n-1,n},Coxeter,edge length=1cm] B{**...**},&\quad & n\ge 3, &\qquad &\tilde C_n: \dynkin[extended,o/.append style={fill=black},labels={0,1,n-1,n},Coxeter,edge length=1cm] C{*...**},&\quad & n\ge 2,&\\
&\tilde D_{n+1}: \dynkin[extended,o/.append style={fill=black},labels={0,1,2,n-1,n,n+1},label directions={,,left,right,,},edge length=1cm] D{**...***}, &&n\ge 3,
&
&\tilde G_2: \dynkin[extended,o/.append style={fill=black},edge length=1cm] G[1]2
\end{alignat*}
We have the following unfolding homomorphisms
$$
\begin{array}{c|c|l}
\wh M&M&\multicolumn{1}{c}{\Phi}\\
\hline
\hline
\vphantom{\tilde{\tilde B}}\tilde B_n& \tilde D_{n+1}&
\wh T_i\mapsto T_i,\,i\in[0,n-1],\, \wh T_n\mapsto T_nT_{n+1}\\
\hline
\vphantom{\tilde{\tilde C}}\tilde C_n& \tilde B_{n+1}& \wh T_0\mapsto T_0T_1,\,T_i\mapsto T_{i+1},\, i\in[1,n]\\
\hline
\vphantom{\tilde{\tilde C}}\tilde G_2& \tilde D_{4}& \wh T_0\mapsto T_0,\,\wh T_1\mapsto T_2,\, \wh T_2\mapsto T_1T_3T_4\\
\hline
\end{array}
$$
while Lemma~\ref{lem:orbits fold} yields the following homomorphisms
$$
\begin{array}{c|c|l}
M&M^\varpi&\multicolumn{1}{c}{\mathbf F_\varpi}\\
\hline
\hline
\vphantom{\tilde{\tilde B}}\tilde B_n& B_n & T_i\mapsto T^{\varpi}_{i+\delta_{i,0}},\,i\in[0,n]\\
\hline
\vphantom{\tilde{\tilde B}}\tilde D_{n+1}& D_{n+1} &
T_i\mapsto T^{\varpi}_{n+1-i+\delta_{i,n+1}},\,i\in[0,n+1]\\
\hline
\vphantom{\tilde{\tilde B}}\tilde D_4& A_3 &  T_0\mapsto T^{\varpi}_1,\,  T_2\mapsto T^{\varpi}_2,\,T_i\mapsto T^{\varpi}_3,\, i\in\{1,3,4\}\\
\hline
\end{array}
$$
Their compositions yield non-standard homomorphisms of the type discussed in Example~\ref{ex:1.6}, namely
\begin{alignat}{3}
&\Br^+(\tilde C_n)\to
\Br^+(\tilde B_{n+1})\to \Br^+(B_{n+1}),&\qquad  &\wh T_i\mapsto T_{i+1}^{i+\delta_{i,0}},&\quad  &i\in [0,n],\nonumber\\
&\Br^+(\tilde B_n)\to \Br^+(\tilde D_{n+1})\to \Br^+(D_{n+1}),& &\wh T_i\mapsto T_{n+1-i}^{1+\delta_{i,n}},&&  i\in [0,n],\nonumber\\
&\Br^+(\tilde G_2)\to \Br^+(\tilde D_4)\to
\Br^+(A_3),& &\wh T_i\mapsto T_{i+1}^{1+2\delta_{i,3}},&& i\in [0,3].\label{eq:elem Tits affine}
\end{alignat}
In addition, the non-standard
homomorphism  $\Br^+(\tilde C_2)\to \Br^+(A_3)$
defined by $T'_i\mapsto T_i^{1+\delta_{i,2}}$, $i\in[1,3]$ is the composition of a standard
homomorphism $\Br^+(\tilde C_2)\to \Br^+(\tilde A_3)$,
$T'_0\mapsto T_0$, $T'_1\mapsto T_1T_3$, $T'_2\mapsto T_2$, and the standard homomorphism $
\Br^+(\tilde A_3)\to\Br^+(A_3)$, $T'_0\mapsto T_1$,
$T'_1\mapsto T_2$, $T'_2\mapsto T_3$, $T'_3\mapsto T_2$,
where Coxeter graph of type~$\tilde A_3$ are labeled as follows
$$
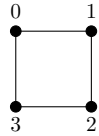
\qquad
\begin{dynkinDiagram}[name=upper,
text style/.style={scale=0.8},Coxeter,root radius=0.07,expand labels={0,1},label directions={above,above},edge length=1cm]A2
\node (current) at ($(upper root 1)+(0,-1cm)$) {};
\dynkin[at=(current),name=lower,expand labels={3,2},text style/.style={scale=0.8},Coxeter,root radius=0.07,edge length=1cm,label directions={below,below}]A2
\begin{pgfonlayer}{Dynkin behind}
\foreach \i in {1,...,2}{\draw[/Dynkin diagram]
($(upper root \i)$)
-- ($(lower root \i)$);}\end{pgfonlayer}
\end{dynkinDiagram}
$$
\end{example}

\subsection{Tits homomorphisms}\label{subs:Tits homs}
We now study a particular class of 
light homomorphisms which includes those 
from Examples~\ref{ex:1.6} and~\ref{ex:affine}. 

Given a Coxeter matrix~$M=(m_{ij})_{i,j\in I}$
and~$\mathbf d\in \ZZ_{>0}^I$, define~\plink{M(d)}$M(\mathbf d)=(m_{ij}(\mathbf d))_{i,j\in I}$ by
\begin{equation}
m_{ij}(\mathbf d)=\begin{cases}
m_{ij},&\text{$m_{ij}\le 2$ or~$d_id_j=1$},\\
2d_id_j,&m_{ij}=3,\,d_id_j\in\{2,3\},\\
\infty,&\text{otherwise}.\label{eq:Tits cover}
\end{cases}
\end{equation}
Clearly~$M(\mathbf d)\in\Cox I$.
\begin{theorem}\label{thm:Tits homomorphisms}
Let~$M,\wh M\in\Cox I$ and let~$\mathbf d\in \mathbb Z_{>0}^I$. The
assignments $\wh T_i\mapsto T_i^{d_i}$, $i\in I$ define an optimal
\plink{T d}$\mathbf T_{\mathbf d}\in\Hom_{\Art}(\wh M,M)$
if and only if~$\wh M=M(\mathbf d)$.
\end{theorem}
\begin{proof}
We may assume, without loss of generality,
that~$I=\{1,2\}$, and so $\wh M=I_2(\wh m)$, $M=I_2(m)$, $m,\wh m\in\mathbb Z_{\ge 2}\cup\{\infty\}$, 
and that~$d_1\ge d_2$.

Suppose that~$\wh M=M(\mathbf d)$.
First we prove that the assignments $\wh T_i\mapsto
T_i^{d_i}$, $i\in \{1,2\}$ define a homomorphism
$\Br^+(I_2(\wh m))\to \Br^+(I_2(m))$.
We only need to consider the case when
$m=3$ and~$d_1=d\in \{2,3\}$, $d_2=1$,
the other cases being obvious.
\begin{lemma}\label{lem:finite elem Tits}
For~$d\in\{2,3\}$, the assignments $\wh T_1\mapsto T_1^d$, $\wh T_2\mapsto T_2$ define 
an optimal homomorphism from $\Br^+(I_2(2d))$ to~$Br^+(A_2)$.
\end{lemma}
\begin{proof}
Let $M'=A_3$ if~$d=2$ and~$M'=D_4$ if~$d=3$ and let~$I'=[1,d+1]$ be its index set.
Then $h(M')=2d$
and $I'=(I'\setminus\{2\})\cup\{2\}$ is
a partition of~$I'$ into (self-orthogonal) orbits of a suitable diagram automorphism. By Corollary~\partref{cor:adm finite class.even}, the assignments
$\wh T_1\mapsto \prod_{k\in I'\setminus \{2\}}T_k$,
$\wh T_2\mapsto T_2$ define a homomorphism
$\Br^+(I_2(2d))\to \Br^+(M')$. Furthermore, $M'$
is foldable along the map $\varpi:I'\to\{1,2\}$,
$\varpi(i)=1$, $i\in I'\setminus\{2\}$, $\varpi(2)=2$,
and the corresponding homomorphism
$\Br^+(M')\to\Br^+_{\{1,2\}}(M')\cong \Br^+(A_2)$
maps $T_i$, $i\in I'\setminus\{2\}$ to~$T_1$
and $T_2$ to itself. Their composition
is the desired homomorphism $\Br^+(I_2(2d))\to
\Br^+(A_2)$.
To prove that it is optimal, note that~$2d\in B(T_1^d,T_2)$ and so~$k=\min B(T_1^d,T_2)$ must
divide~$2d$ by Lemma~\partref{lem:taut homs.b}. Yet $T_1^d$
and~$T_2$ do not commute for otherwise
we would have~$T_2T_1^{d+1}=T_1^d T_2 T_1=T_2 T_1 T_2^d$
which by the cancellativity of~$\Br^+(A_2)$ yields a non-existent relation $T_1^d=T_2^d$. Thus, $k>2$. Since~$k$
is even by Lemma~\partref{lem:elem Artin hom.c} and~$2d\in\{4,6\}$ it follows that~$k=2d$.
\end{proof}

It remains to show that~$\mathbf T_{\mathbf d}$
is optimal, which also amounts
to proving the converse. This is obvious for the
case when~$d_1=d_2=1$ or~$m_{ij}=2$ 
and has already been proven in Lemma~\ref{lem:finite elem Tits} for~$d=d_1d_2\in\{2,3\}$ and~$m=3$.
If~$d_2>1$, then
by~\cite{CP} the submonoid of~$\Br^+(I_2(m))$
generated by $T_1^{d_1}$, $T_2^{d_2}$ is free
and so~$B(T_1^{d_1},T_2^{d_2})=\emptyset$.

It remains to consider the case when~$d=d_1>d_2=1$ and either~$m_{ij}>3$ or~$d>3$.
\begin{proposition}\label{prop:key converse Tits}
Let~$m>2$ and suppose that~$\mathbf T_{(d,1)}
\in\Hom_{\Art}(I_2(\wh m),I_2(m))$ with~$d\ge 2(1+\delta_{m,3})$.
Then~$\wh m=\infty$.
\end{proposition}
\begin{proof}
Let $q\in\mathbb R\setminus\{0\}$,
$z\in\mathbb C\setminus\{0\}$.
Define
$$
\check T_1=\begin{pmatrix}1-q^2&q\\q&0\end{pmatrix},\qquad \check T_2(z)=\begin{pmatrix}0&q z\\q z^*&1-q^2\end{pmatrix},
$$
where~$z^*$ denotes the complex conjugate of~$z$.
\begin{lemma}
Let~$\zeta_l$, $l\ge 3$ be an $l$th primitive complex root of unity.
Then for any~$q\in\mathbb R\setminus\{0\}$,
the assignments $T_1\mapsto \check T_1$,
$T_2\mapsto \check T_2(\zeta_l)$, $i\in\{1,2\}$ define a representation of~$\Br^+(I_2(l))$ on~$\mathbb C^2$. Moreover, in this representation
the matrix of $T^{op}$ is the adjoint of that of~$T\in\Br^+(I_2(l))$ with respect to the
standard hermitian product on~$\mathbb C^2$.
\end{lemma}
\begin{proof}
We have
\begin{align}
\brd{\check T_1\check T_2(z)}n&=q^{n-1}
\begin{cases}
\begin{pmatrix}
q z^*{}^{\frac n2}&(1-q^2)(1+z)r_{\frac n2-1}(z)\\
0&q z^{\frac n2}
\end{pmatrix},& \bar n=0,\\
\\
\begin{pmatrix}
(1-q^2)(r_{\lfloor \frac n2\rfloor}(z)+r_{\lfloor \frac n2\rfloor-1}(z))&q z^*{}^{\lfloor \frac n2\rfloor}\\
q z^{\lfloor \frac n2\rfloor}&0
\end{pmatrix},&\bar n=1,
\end{cases}\label{eq:T1T2(z)}\\
\intertext{and}
\brd{\check T_2(z)\check T_1}n&=q^{n-1}
\begin{cases}
\begin{pmatrix}
q z^{\frac n2}&0\\
(1-q^2)(1+z^*)r_{\frac n2-1}(z)&q z^*{}^{\frac n2}
\end{pmatrix},& \bar n=0,\\
\\
\begin{pmatrix}
0&q z{}^{\lceil \frac n2\rceil}\\
q z^*{}^{\lceil \frac n2\rceil}&
(1-q^2)(r_{\lfloor \frac n2\rfloor}(z)+z z^*r_{\lfloor \frac n2\rfloor-1}(z))
\end{pmatrix},&\bar n=1,
\end{cases}\label{eq:T2(z)T1}
\end{align}
where
\begin{equation}\label{eq: r_k defn}
r_k(z)=\sum_{0\le j\le k}z^{k-j}z^*{}^j=
\begin{cases}
(k+1) z^k,&z\in\mathbb R,\\
\dfrac{z^{k+1}-z^*{}^{k+1}}{z-z^*},&z\in\mathbb C\setminus\mathbb R.
\end{cases}
\end{equation}
Note that $r_k(z^*)=r_k(z)=r_k(z)^*$ and
\begin{equation}\label{eq:rec r_k}
z r_k(z)+z^*{}^{k+1}=r_{k+1}(z).
\end{equation}
Indeed, \eqref{eq:T1T2(z)}
clearly holds for~$n=1$.
If~$n=2k\ge 2$ then by the induction hypothesis
\begin{align*}
\brd{\check T_1\check T_2(z)}{n}&=q^{2k-2}
\begin{pmatrix}
(1-q^2)(r_{k-1}(z)+r_{k-2}(z))&q z^*{}^{k-1}\\
q z^{k-1}&0
\end{pmatrix}\check T_2(z)\\
&=q^{2k-1}
\begin{pmatrix}
q z^*{}^k &(1-q^2)(z (r_{k-1}(z)+r_{k-2}(z))+z^*{}^{k-1})\\0&q z^k
\end{pmatrix}
\\
&=q^{2k-1}
\begin{pmatrix}
q z^*{}^k &(1-q^2)(1+z)r_{k-1}(z)\\0&q z^k
\end{pmatrix},
\end{align*}
where we used~\eqref{eq:rec r_k}. If~$n=2k+1$, $k>0$ then by the induction hypothesis
and~\eqref{eq:rec r_k}
\begin{align*}
\brd{\check T_1\check T_2(z)}n&=q^{2k-1}
\begin{pmatrix}
q z^*{}^{k}&(1-q^2)(1+z)r_k(z)\\
0 &qz^k
\end{pmatrix}\check T_1\\
&=q^{2k}\begin{pmatrix}
(1-q^2)(z^*{}^k+(1+z)r_{k-1}(z)&q z^*{}^k\\
q z^k&0
\end{pmatrix}\\
&=q^{2k}\begin{pmatrix}(1-q^2)(r_k(z)+r_{k-1}(z))&q z^*{}^k\\
q z^k&0
\end{pmatrix}.
\end{align*}
The identity~\eqref{eq:T2(z)T1}
is proved similarly.

Suppose now that~$z=\zeta_l$. Then~$z^*=z^{-1}$ and
$z^*{}^{\lceil \frac l2\rceil}=z^{\lfloor \frac l2\rfloor}$. Also, if~$l=2s$, $s\ge 2$ then~$r_{s-1}(z)=0$ by~\eqref{eq: r_k defn},
while for~$l=2s+1$, $s\ge 1$,
$$
r_{s}(z)+r_{s-1}(z)=\frac{z^{s+1}-z^{-s-1}+z^s-z^{-s}}{z-z^{-1}}=\frac{z^{-s}(z^{l}-1)+z^s(1-z^{-l})}{z-z^{-1}}=0.
$$
The identity~$\brd{\check T_1\check T_2(z)}{l}=
\brd{\check T_2(z)\check T_1}{l}$ is now immediate
from~\eqref{eq:T1T2(z)} and~\eqref{eq:T2(z)T1}.
The last assertion follows since $\check T_1$ is symmetric and real, while the complex conjugate of the transpose of~$\check T_2(z)$
equals~$\check T_2(z)$
for any~$z\in\mathbb C$.
\end{proof}
Denote $p_d=(1-(-q^2)^d)/(1+q^2)$, $d\ge 0$. Then $p_0=0$, $p_1=1$ and
\begin{equation}\label{eq:p_d rec}
p_{d+1}=(1-q^2)p_d+q^2 p_{d-1},\quad d\ge 1.
\end{equation}
We set~$p_{-1}=q^{-2}$ so that the above recursion holds for all~$d\ge 0$.
Note that all the~$p_d$, $d\ge 1$ are polynomials in~$q$ with
the leading term~$(-q^2)^{d-1}$. An easy
induction on~$d$ yields, using~\eqref{eq:p_d rec},
$$
\check T_1^d=\begin{pmatrix}
                     p_{d+1}&q p_d\\
                     q p_d &q^2 p_{d-1},
\end{pmatrix}
$$
whence
\begin{align*}
\check T_1^d \check T_2(\zeta_l)&=\begin{pmatrix}
                          q^2 \zeta_l^{-1}p_d& q((1-q^2)p_d+\zeta_l p_{d+1})\\ q^3 \zeta_l^{-1} p_{d-1}&q^2((1-q^2)p_{d-1}+\zeta_l p_d)
                         \end{pmatrix}\\
                         &=\begin{pmatrix}
                            q^2 \zeta_l^{-1}p_d&q ((1+\zeta_l)p_{d+1}-q^2 p_{d-1})\\ q^3 \zeta_l^{-1}p_{d-1}&q^2((1+\zeta_l)p_d-q^2 p_{d-2})
                           \end{pmatrix}.
\end{align*}
Since $\det \check T_1=\det\check T_2(\zeta_l)=-q^2$,
the characteristic polynomial of~$\check T_1^d \check T_2(\zeta_l)$ is
$$
t^2-q^2\tau_d t+(-q^2)^{d+1},\qquad \tau_d:=(\zeta_l+1+\zeta_l^{-1})p_{d-1}-q^2 p_{d-2},
$$
and so the eigenvalues of~$\check T_1^d \check T_2(\zeta_l)$ are
\begin{equation}\label{eq:lam T_1^d T_2}
\lambda_\pm(q)=\tfrac12 q^2\Big(\tau_d\pm\sqrt{\tau_d^2-4(-q^2)^{d-1}}\Big).
\end{equation}
Note that~$\zeta_l+\zeta_l^{-1}\in\mathbb R$ and so~$\tau_d\in\mathbb R$ for all~$q\in\mathbb R\setminus\{0\}$.
\begin{lemma}\label{lem:generic eigenvalues}
Let~$l\ge 3$ and~$d\ge 2(1+\delta_{l,3})$. Then for generic~$q\in\mathbb R\setminus\{0,\pm1\}$, $\lambda_+(q)^n\not=\lambda_-(q)^n$ for all~$n\ge 1$.
\end{lemma}
\begin{proof}
Abbreviate $\zeta=\zeta_l$.
Suppose first that~$l=3$ and so $\zeta+1+\zeta^{-1}=0$. Then $\tau_d=-q^2 p_{d-2}$ and, since~$d>3$, we can write
$$
\lambda_\pm(q)=-\tfrac12 q^4\Big(p_{d-2}\mp\sqrt{p_{d-2}^2-4(-q^2)^{d-3}}\Big).
$$
Since~$p_{d-2}^2-4(-q^2)^{d-3}$ is a polynomial in~$q$ with the leading term $q^{4(d-3)}$, $\lambda_+(q)\not=\lambda_-(q)$
for a generic~$q$.

Suppose that $\lambda_+(q)\not=\lambda_-(q)$ and that
$\lambda_+(q)^n=\lambda_-(q)^n$ for some~$n>1$. Then
\begin{equation}\label{eq:generic eigenvalues 1}
\sum_{k\ge 0} \binom{n}{2k+1} p_{d-2}^{n-2k-1} (p_{d-2}^2-4(-q^2)^{d-3})^k=0.
\end{equation}
Each summand in the left hand side of~\eqref{eq:generic eigenvalues 1} is a
polynomial in~$q$ with the leading term
$\binom{n}{2k+1}(-q^2)^{(d-3)(n-1)}$. Therefore, the sum in~\eqref{eq:generic eigenvalues 1}
is a polynomial in~$q$ with the leading term $(2(-q^2)^{d-3})^{n-1}$. Thus, $\lambda_+(q)^n\not=\lambda_-(q)^n$, $n\ge 1$ for all but countably many~$q\in\mathbb R\setminus\{0,\pm1\}$.

Suppose that~$l>3$ and so~$\zeta+1+\zeta^{-1}\not=0$. If~$d=2$, $\tau_d=\zeta+1+\zeta^{-1}$, $\tau_d^2-4(-q^2)^{d-1}=(\zeta+1+\zeta^{-1})^2+4q^2>0$ for all real~$q$.
If~$d=3$, $\tau_d=(\zeta+1+\zeta^{-1})1-q^2(\zeta+2+\zeta^{-1})$ and so
$$
\tau^2_d-4(-q^2)^{d-1}=(z+4)z q^4-2 (z+1)(z+2)q^2+(z+1)^2,\qquad z=\zeta+\zeta^{-1},
$$
which is clearly a non-zero polynomial in~$q$. If~$d>3$ then the degree of~$\tau_d^2$ as a polynomial in~$q$ is $4(d-2)>2(d-1)$ hence the degree of~$\tau_d^2-4(-q^2)^{d-1}$ is $4(d-2)$.
Thus, $\lambda_+(q)\not=\lambda_-(q)$ for generic~$q\in\mathbb R\setminus\{0,\pm1\}$.

Now, suppose that~$\lambda_+(q)\not=\lambda_-(q)$.
Then~$\lambda_+(q)^n=\lambda_-(q)^n$ for some~$n\ge 2$ if and only if
\begin{equation}\label{eq:pwrs eig}
\sum_{k\ge 0}\binom{n}{2k+1} \tau_d^{n-2k-1}(\tau_d^2-4(-q^2)^{d-1})^k=0.
\end{equation}
If~$d=2$ then the left hand side is equal to
$$
\sum_{k\ge 0}\binom{n}{2k+1} (\zeta+1+\zeta^{-1})^{n-2k-1}(\zeta+1+\zeta^{-1}+4q^2)^k
$$
which is equal to $2(\zeta+1+\zeta^{-1})$ if~$n=2$ and is a polynomial in~$q$ otherwise, with the leading term~$2^{n-1}q^{n-1}$ if~$n$ is odd and~$n(\zeta+1+\zeta^{-1})2^{n-2}q^{n-2}$
if~$n>2$ is even.

If~$d=3$ then the left hand side of~\eqref{eq:pwrs eig} becomes
\begin{equation}\label{eq:lhs pwrs eig}
\sum_{k\ge 0}\binom{n}{2k+1} (z+1)^{n-2k-1}((z+4)z q^4-2 (z+1)(z+2)q^2+(z+1)^2)^k,\qquad z=\zeta+\zeta^{-1}.
\end{equation}
For~$n=2$, this equals~$2(z+1)\not=0$. Suppose that~$n>2$. If~$l=4$ and so~$z=0$,
\eqref{eq:lhs pwrs eig} is a polynomial in~$q$ with the leading term $(2q)^{n-1}$ if~$n$ is odd
and $-n (2q)^{n-2}$ if~$n$ is even. If~$l>4$ then
\eqref{eq:lhs pwrs eig} is a polynomial in~$q$ with the leading term $(z(z+4))^{\frac12(n-1)}q^{2(n-1)}$ if~$n$ is odd
and~$n(z+1)(z(z+4))^{\frac12n-1} q^{2(n-2)}$ if~$n$ is even. In either case, it is a non-zero polynomial in~$q$.

Finally, if~$d>3$, each summand in the left hand side of~\eqref{eq:pwrs eig} is a polynomial in~$q$ with the leading term $\binom{n}{2k+1}(-q^2)^{2(d-2)(n-1)}$.
Thus, the leading term of the sum is $(2(-q^2)^{d-2})^{n-1}$ and so the sum is a non-zero polynomial in~$q$.

Thus, for~$l>3$ and~$d\ge 2$, the left hand side of~\eqref{eq:pwrs eig} is a non-zero polynomial in~$q$, whence 
$\lambda_+(q)^n\not=\lambda_-(q)^n$ for all~$n\ge 1$ for all but countably many~$q\in\mathbb R\setminus\{0,\pm1\}$. 
\end{proof}

Suppose that~$\wh m$ is even. Then~$\wh T_2$
commutes with $(\wh T_1\wh T_2)^{\frac12\wh m}$ and so
$T_2=\mathbf T_{(d,1)}(\wh T_2)$ must commute with~$(T_1^d T_2)^{\frac12\wh m}=\mathbf T_{(d,1)}(
(\wh T_1\wh T_2)^{\frac12\wh m})$ in~$\Br^+(I_2(m))$.
Therefore, for any~$q\in\mathbb R\setminus\{0\}$,
$\check T_2(\zeta_m)$ must commute with
$(\check T_1^d T_2(\zeta_m))^{\frac12\wh m}$

 By Lemma~\ref{lem:generic eigenvalues}, we can choose~$q\in\mathbb R\setminus\{0,\pm1\}$ so that~$(\check T_1^d \check T_2(\zeta_m))^n$ has two distinct eigenvalues~$\lambda_\pm(q)^n$ for all~$n\ge 1$. Then
 $\ker((\check T_1^d\check T_2(\zeta_m))^n-
 \lambda_\pm(q)^n\id_{\mathbb C^2})=\ker(\check T_1^d\check T_2(\zeta_m)-
 \lambda_\pm(q)\id_{\mathbb C^2})$ for all~$n\ge 1$.
 Since both~$\check T_2(\zeta_m)$ and~$(\check T_1^d\check T_2(\zeta_m))^{\frac12\wh m}$ are diagonalizable and
 commute, they are simultaneously diagonalizable, which by the above also implies that $\check T_2(\zeta_m)$
 and $\check T_1^d \check T_2(\zeta_m)$ are
 simultaneously diagonalizable and, therefore, must commute. Yet by~\eqref{eq:p_d rec}
\begin{align*}
 \check T_2(\zeta_m)&(\check T_1^d\check T_2(\zeta_m))-(\check T_1^d\check T_2(\zeta_m))\check T_2(\zeta_m)
 =
 q^2 p_d\begin{pmatrix}
  \zeta_m-\zeta_m^{-1} & (q-q^{-1}) (1+\zeta_m) \\
 (q^{-1}-q) (1+\zeta_m^{-1}) & \zeta_m^{-1}-\zeta_m
 \end{pmatrix},
 \end{align*}
which is manifestly non-zero for~$q\in\mathbb R\setminus\{0,\pm1\}$,
as $\zeta_m$ is an $m$th primitive root
of unity with~$m\ge 3$.

Finally, suppose that~$\wh m$ is odd.
Then the composition of~$\mathbf T_{(d,1)}$ with the tautological 
homomorphism $\Br^+(I_2(2\wh m))\to\Br^+(I_2(\wh m))$
yields~$\mathbf T_{(d,1)}\in\Hom_{\Art}(I_2(2\wh m),I_2(m))$
which, as we just proved, does not exist.

Thus, $\wh m=\infty$.
\end{proof}
This completes the proof of Theorem~\ref{thm:Tits homomorphisms}.
\end{proof}
\begin{lemma}\label{lem:lifts of parab projections}
Let~$M=(m_{ij})_{i,j\in I}$ be a Coxeter matrix, $J\subset I$.
Assume that~$p_J$ is liftable, that each connected component of~$\Gamma_J(M)$ has at least two vertices and that~$m_{ij}<\infty$ for all~$i,j\in J$. Then~$P_J$
is the only~$\Phi\in\Hom_{\Art}(M,M_J)$
such that~$\overline\Phi_\star=p_J$.
\end{lemma}
\begin{proof}
We may assume, without loss of generality, that~$J$ is connected and~$|J|>1$. Let~$\Phi$ be a homomorphism $\Br^+(M)\to\Br^+_J(M)$ and with~$\overline\Phi_\star=p_J$.
Then~$[\Phi](i)=\emptyset$ for all~$i\in I\setminus J$ and~$[\Phi](j)=\{j\}$ for all~$j\in J$. Thus,
$\Phi(T_i)=1$ if~$i\in I\setminus J$ and~$\Phi(T_i)=T_i^{d_i}$, $d_i\in\ZZ_{>0}$ if~$i\in J$. 

Let~$i\not=j\in J$ and suppose that~$\max(d_i,d_j)>1$.
Note that if~$m_{ij}$ is odd then~$d_i=d_j$ by Lemma~\partref{lem:elem Artin hom.c}
and so~$\min(d_i,d_j)>1$. Suppose that~$m_{ij}>2$. 
If~$\min(d_i,d_j)>1$ then~$B(T_i^{d_i},T_j^{d_j})=\emptyset$
by~\cite{CP} which is a contradiction since~$m_{ij}<\infty$. In particular,
$m_{ij}\not=3$.
Similarly,
if~$m_{ij}>3$ and~$\min(d_i,d_j)=1$ then
$B(T_i^{d_i},T_j^{d_j})=\emptyset$
by Proposition~\ref{prop:key converse Tits} which again contradicts the 
assumption that~$m_{ij}<\infty$.
Thus, if~$m_{ij}>2$ then~$d_i=d_j=1$.

Finally, if~$m_{ij}=2$ and, say, $d_i>1$ then, since~$J$ is connected,
there exists~$i'\in J$ with~$m_{ii'}\ge 3$ which leads to 
a contradiction by the above. Thus, $d_j=1$ for all~$j\in J$, that is, $\Phi=P_J$.
\end{proof}
Given~$\mathbf d\in \ZZ_{>0}^I$, let~$I_{>1}(\mathbf d)=\{i\in I\,:\, d_i>1\}$.
In reference to Tits conjecture, we call
an optimal $\mathbf T_{\mathbf d}\in\Hom_{\Art}(M(\mathbf d),M)$, $\mathbf d=(d_i)_{i\in I}\in\ZZ_{>0}^I$ satisfying $\mathbf T_{\mathbf d}(\wh T_i)=T_i^{d_i}$, $i\in I$
a {\em Tits homomorphism}.
We call a Tits homomorphism~$\mathbf T_{\mathbf d}$ {\em elementary} if~$I_{>1}(\mathbf d)=\{i\}$
for some~$i\in I$
and denote such a homomorphism by~\plink{T i,d}$\mathbf T_{i,d}$, $i\in I$, $d\in\mathbb Z_{>1}$.

\begin{proposition}\label{prop:elementary Tits}
Let~$M\in\Cox I$, $\mathbf d=(d_i)_{i\in I}\in\ZZ_{>0}^I$
and let $\mathbf T_{\mathbf d}\in\Hom_{\Art}(M(\mathbf d),M)$ be a Tits homomorphism.
Then for any total order $i_1<\cdots<i_n$ on~$I_{>1}(\mathbf d)$ there is a unique sequence of Coxeter matrices $M_0=M, M_1,\dots, M_n, M_{n+1}=
M(\mathbf d)$ over~$I$ such that
$\mathbf T_{\mathbf d}=\mathbf T_{i_1,d_{i_1}}\circ \cdots\circ \mathbf T_{i_n, d_{i_n}}$ where $\mathbf T_{i_t,d_{i_t}}\in\Hom_{\Art}(M_{t+1},M_t)$.
Thus, every Tits homomorphism is a composition of elementary Tits homomorphisms.
\end{proposition}
\begin{proof}
The argument is by induction on~$|I_{>1}(\mathbf d)|$, the case~$|I_{>1}(\mathbf d)|=1$ being trivial. For the inductive step, we need the following 
\begin{lemma}\label{lem:Tits factorization}
Let~$i\in I_{>1}(\mathbf d)$, 
$\mathbf d_i=(d_j^{\delta_{i,j}})_{j\in I},
\mathbf d^{(i)}=(d_j^{1-\delta_{i,j}})_{j\in I}
\in\ZZ_{>0}^I$ and let~$M^{(i)}=M(\mathbf d_i)$.  
Then $M^{(i)}(\mathbf d^{(i)})=
M(\mathbf d)$.
\end{lemma}
\begin{proof}
Write $M^{(i)}=(m^{(i)}_{jk})_{j,k\in I}$. Then 
$m^{(i)}_{jk}=m_{jk}$, $j,k\not=i$ and 
$$
m^{(i)}_{ij}=m^{(i)}_{ji}=\begin{cases}
             m_{ij},& m_{ij}\le 2,\\
             2d_i,& m_{ij}=3,\, d_i\in\{2,3\},\,d_j=1,\\
             \infty,& \text{otherwise}.
            \end{cases}
$$
Note that, since~$m_{jk}=m^{(i)}_{jk}$,
$j,k\not=i$, we only need to show that
$m_{ij}(\mathbf d)=m^{(i)}_{ij}(\mathbf d^{(i)})$
for~$j\not=i\in I$. By~\eqref{eq:Tits cover}, we have 
$m^{(i)}_{ij}=m_{ij}(\mathbf d)=2d_i$ if~$d_j=1$, $d_i\in\{2,3\}$ and~$m_{ij}=3$. Since~$d^{(i)}_j=
d^{(i)}_i=1$, we have~$m^{(i)}_{ij}(\mathbf d^{(i)})=m^{(i)}_{ij}$ by~\eqref{eq:Tits cover} and so $m_{ij}(\mathbf d)=m^{(i)}_{ij}(\mathbf d^{(i)})$.
Likewise, $m^{(i)}_{ij}=m_{ij}(\mathbf d)=2$ if~$m_{ij}=2$ and so~$m^{(i)}_{ij}(\mathbf d^{(i)})=2=m_{ij}(\mathbf d)$. Finally, if~$d_j=d^{(i)}_j>1$ then~$m^{(i)}_{ij}=\infty$ and
so~$m^{(i)}_{ij}(\mathbf d^{(i)})=\infty$. But in that case, as 
$d_id_j>3$, 
$m_{ij}(\mathbf d)=\infty$ by~\eqref{eq:Tits cover}. 
\end{proof}
By Lemma~\ref{lem:Tits factorization}
and Theorem~\ref{thm:Tits homomorphisms},
$\mathbf T_{i,d_i}\in\Hom_{\Art}(M^{(i)},M)$ and
$\mathbf T_{\mathbf d^{(i)}}\in\Hom_{\Art}(M(\mathbf d),M^{(i)})$
are Tits homomorphisms. By 
construction, $\mathbf T_{\mathbf d}=
\mathbf T_{i,d_i}\circ \mathbf T_{\mathbf d^{(i)}}$.
Since $|I_{>1}(\mathbf d^{(i)})|=|I_{>1}(\mathbf d)|-1$,
the induction hypothesis applies to~$\mathbf T_{\mathbf d^{(i)}}$ and so it admits
the desired factorization. The uniqueness follows 
from Theorem~\ref{thm:Tits homomorphisms}.
\end{proof}

The following extends the famous {\em Tits conjecture} and its generalization proved in~\cite{CP} as well as a special case from~\cite{Cri}.
\begin{conjecture}\label{conj:gen Tits}
All Tits homomorphisms~$\mathbf T_{\mathbf d}$ 
are injective.
\end{conjecture}
By Proposition~\ref{prop:elementary Tits}, it suffices to prove the Conjecture for elementary Tits homomorphisms. The main result of~\cite{CP}
proves the above for all~$\mathbf d\in\mathbb Z_{>1}^I$.
We now provide some supporting evidence.
The following is immediate.
\begin{corollary}\label{cor:left factor Tits}
Let~$M\in\Cox I$, $\mathbf d\in\ZZ_{>0}^I$,
$\mathbf T_{\mathbf d}\in\Hom_{\Art}(M(\mathbf d),M)$
be a Tits homomorphism
and~$\mathbf T_{\mathbf d}=\mathbf T_{i_1,d_{i_1}}\circ\cdots \circ \mathbf T_{i_n,d_{i_n}}$ be any factorization as in Proposition~\ref{prop:elementary Tits}. If~$\mathbf T_{\mathbf d}$ is injective then so is
$\mathbf T_{i_r,d_{i_r}}\circ 
\cdots \circ \mathbf T_{i_n,d_{i_n}}$ 
for any $2\le r\le n$.
\end{corollary}
\begin{example}
Let~$M=A_3$, $\mathbf d=(d,d,d)$, $d\in\{2,3\}$. Then~$M(\mathbf d)=\left(\begin{smallmatrix}
    1&\infty&2\\
    \infty&1&\infty\\
    2&\infty&1
\end{smallmatrix}\right)$
and the Tits homomorphism $\mathbf T_{\mathbf d}:\Br^+(M(\mathbf d))\to 
\Br^+(M)$ is injective by~\cite{CP}. It
factorizes into 
the following chain of elementary Tits homomorphisms
$$
\begin{dynkinDiagram}[expand labels={1,2,3}]A3
\dynkinEdgeLabel{1}{2}{\infty}
\dynkinEdgeLabel{2}{3}{\infty}
\end{dynkinDiagram}
\xrightarrow{\mathbf T_{2,d}}
\begin{dynkinDiagram}[expand labels={1,2,3}]A3
\dynkinEdgeLabel{1}{2}{2d}
\dynkinEdgeLabel{2}{3}{2d}
\end{dynkinDiagram}
\xrightarrow{\mathbf T_{1,d}}
\begin{dynkinDiagram}[expand labels={1,2,3}]A3
\dynkinEdgeLabel{2}{3}{2d}
\end{dynkinDiagram}
\xrightarrow{\mathbf T_{3,d}}\\
\dynkin A3
$$
whence~$\mathbf T_{2,d}$ and~$\mathbf T_{1,d}\circ
\mathbf T_{2,d}=\mathbf T_{(d,d,1)}$ are injective. 
\end{example}

\begin{proposition}\label{prop:partial Tits}
Let~$m\in\{3,4,6\}$. Then~$\Phi_d\in\Hom_{\Art}(I_2(\infty),
I_2(m))$ defined by $\Phi_d(\wh T_1)=T_1^d$, $
\Phi(\wh T_2)=T_2$ is injective
if and only if~$d\ge 2(1+\delta_{3,m})$, that is,
if and only if~$\Phi_d=\mathbf T_{1,d}=
\mathbf T_{(d,1)}$ is an (elementary) Tits homomorphism.
\end{proposition}
\begin{proof}
Suppose first that~$m=3$ and so~$I_2(m)=A_2$.
It is well-known (and easy to verify) that 
the assignments 
$$
T_1\mapsto \begin{pmatrix}1&u\\
0&1
\end{pmatrix},\qquad T_2\mapsto \begin{pmatrix}1&0\\
-u^{-1}&1
\end{pmatrix},\qquad u\in\mathbb C\setminus\{0\}
$$
define a homomorphism~$\rho_u:\Br(A_2)\to SL(2,\mathbb C)$.
Then~$\rho_u(T_1^d)=\left(\begin{smallmatrix}1&du\\0&1\end{smallmatrix}\right)$. By~\cite{CJR}*{Theorem~2} the 
subgroup of~$SL(2,\CC)$ generated
by $\left(\begin{smallmatrix}
1&\alpha\\
0&1
\end{smallmatrix}\right)$ and
$\left(\begin{smallmatrix}
1&0\\
\beta&1
\end{smallmatrix}\right)$, $\alpha,\beta\in\CC$
is free if and only if~$|\alpha\beta|,
|\alpha\beta\pm 2|\ge 2$. Since~$\alpha=du$ with $d\ge 4$ and~$\beta=-u^{-1}$
obviously satisfy these conditions, 
the subgroup of~$SL(2,\mathbb C)$ generated 
by $\rho_u(T_1^d)$, $d\ge 4$, and~$\rho_u(T_2)$ is free. Thus,
the subgroup of~$\Br(A_2)$ and hence the submonoid of~$\Br^+(A_2)$ generated by~$T_1^d$, $d\ge 4$
and~$T_2$ are free.

If~$m=4$ or~$m=6$ then the 
factorization of~$\mathbf T_{(d,\frac m2)}\in\Hom_{\Art}(I_2(\infty),A_2)$
as~$\mathbf T_{(d,\frac m2)}=\mathbf T_{2,\frac m2}\circ \mathbf T_{1,d}$ yields the injectivity
of~$\mathbf T_{1,d}$ by Corollary~\ref{cor:left factor Tits}.

Conversely, if~$d=1$ then the homomorphism~$\Phi_d$ is tautological and hence not injective as 
$\brd{T_1T_2}m=\brd{T_2T_1}m$. 
If~$m=3$ and~$d\in\{2,3\}$ then the homomorphism
$\Br^+(I_2(\infty))\to \Br^+(I_2(m))$ is 
again not injective as $\brd{T_1^d T_2}{2d}=
\brd{T_2T_1^{d}}{2d}$. 
\end{proof}

The following Proposition generalizes Examples~\ref{ex:1.6} and~\ref{ex:affine}.
\begin{theorem}\label{thm:Tits standard}
Every Tits homomorphism is a morphism
in~$\Ast$.
\end{theorem}
\begin{proof}
By Proposition~\ref{prop:elementary Tits}, it suffices to prove the assertion for elementary Tits homomorphisms. We need the following 
\begin{lemma}\label{lem:elem Tits std}
Let~$d\in\ZZ_{>1}$, $M\in\Cox I$, $i\in I$
and let~$\mathbf d=(d^{\delta_{i,j}})_{j\in I}$. Let~$\wh I=(I\setminus\{i\})
\sqcup S$ where~$S$ is any set with~$|S|=d$. Define~$\varpi:\wh I\to I$ 
and~$\wh M=(\wh m_{jk})_{j,k\in\wh I}$ by
\begin{align*}
&\varpi(j)=\begin{cases}
j,&j\in I\setminus\{i\},\\
i,&j\in S,
\end{cases}\\
&\wh m_{jk}=\begin{cases}
2-\delta_{j,k},&\varpi(j)=\varpi(k),\\
m_{\varpi(j)\varpi(k)},&\text{otherwise}.
\end{cases}
\end{align*}
for all $j,k\in\wh I$. 
Then
\begin{enmalph}
    \item\label{lem:elem Tits std.a} $\wh M\in\Cox{\wh I}$, is foldable along~$\varpi$, and~$\wh M^{\varpi}=M$;
    \item\label{lem:elem Tits std.b} 
    the assignments $T_k\mapsto\!\prod\limits_{j\in\varpi^{-1}(k)}\!\wh T_j$, $k\in I$, define a standard
    $\Phi\in\Hom_{\Art}(M(\mathbf d),\wh M)$;
    \item\label{lem:elem Tits std.c} $\mathbf T_{i,d}=\mathbf F_\varpi\circ \Phi\in\Hom_{\Art}(M(\mathbf d),M)$.
\end{enmalph}
\end{lemma}
\begin{proof}
Part~\ref{lem:elem Tits std.a} is obvious. 
To prove~\ref{lem:elem Tits std.b}, note first that the product~$\prod_{s\in S}\wh T_s$ is well-defined since $\wh m_{st}=2$ for all~$s\not=t\in S$. Since~$m(\mathbf d)_{jk}=m_{jk}=\wh m_{jk}$ if $j,k\in I\setminus\{i\}$,
it suffices
to prove that the assignments
\begin{equation}\label{eq:restr}
T_i\mapsto \wh T_{w_\circ^S}=\prod_{s\in S}\wh T_s,\quad T_j\mapsto \wh T_j
\end{equation}
define 
a homomorphism $\Br^+_{\{i,j\}}(M(\mathbf d))\cong \Br^+(I_2(m))\to 
\Br^+(\wh M)$ where
$$
m=m(\mathbf d)_{ij}=\begin{cases}
2,&m_{ij}=2,\\
2d,&m_{ij}=3,\,d\in\{2,3\},\\
\infty,&\text{otherwise}.
\end{cases}
$$
If~$m_{ij}=2$ then $\wh m_{sj}=m_{ij}=2$ for all~$s\in S$
and so
$\prod_{s\in S}\wh T_s$ commutes with~$\wh T_j$. 
If~$m_{ij}=3$ and~$d\in\{2,3\}$ then the assignments~\eqref{eq:restr} define 
a homomorphism $\Br^+(I_2(2d))\to \Br^+(D_{d+1})$
(see the proof of Lemma~\ref{lem:finite elem Tits}). The remaining case is obvious. Finally, $\Phi$ is
standard as~$\prod_{s\in S}\wh T_s=\wh T_{w_\circ^S}$
since all the $\wh T_s$, $s\in S$ commute.

Finally, we have~$\mathbf F_\varpi\circ\Phi(T_k)=T_k$, $k\in I\setminus\{i\}$
and~$\mathbf F_\varpi\circ\Phi(T_i)=\mathbf F_\varpi(\prod_{s\in S}\wh T_s)=T_i^d$
which proves~\ref{lem:elem Tits std.c}.
\end{proof}
Since both~$\Phi$ and~$\mathbf F_\varpi$ are standard,
the assertion follows for elementary Tits homomorphisms.
\end{proof}
\begin{remark}
The elementary Tits homomorphisms listed in Example~\ref{ex:affine} exhaust fully supported elementary Tits homomorphisms in
$\Hom_{\Art}(\wh M,M)$ where~$\wh M$ is of affine type and~$M$ is of finite or affine type.
\end{remark}

\subsection{Factorization of light homomorphisms}\label{subs:factor light}
The following is the analog of~\cite{BGLHeck}*{Proposition~4.8} for Artin monoids.
\begin{theorem}\label{thm:classify light Artin}
Every light homomorphism of Artin monoids is a composition of one or more of the following
\begin{itemize}
    \item[-] a tautological homomorphism;
    \item[-] a parabolic projection;
    \item[-] a natural inclusion;
    \item[-] a Tits homomorphism;
    \item[-] a diagram automorphism;
    \item[-] a folding along a surjective map of index sets of Coxeter matrices.
\end{itemize}
In particular, every light homomorphism of Artin monoids is a morphism in~$\Ast$.
\end{theorem}

\begin{proof}
Let~$\wh M=(\wh m_{ij})_{i,j\in \wh I}\in\Cox{\wh I}$, $M\in\Cox I$ and let~$\Phi\in\Hom_{\Art}(\wh M,M)$ be light.
By Lemma~\ref{lem:factor homs} we may assume that~$\Phi$ is optimal. 

Let~$\wh I_s=\{i\in \wh I\,:\,|[\Phi](i)|=s\}$, $s\in\{0,1\}$. Note first that~$P_{\wh I_1}$ is a well-defined homomorphism. Indeed, since~$\Phi\in\Hom_{\Art}(\wh M,M)$, by Lemma~\partref{lem:elem Artin hom.c} if~$i\in\wh I_0$ then~$j\in\wh I_0$ for all~$j\in\wh I$ with~$\wh m_{ij}$ odd. Then~$\Phi=\Phi|_{\wh I_1}\circ P_{\wh I_1}$,
so we may assume without loss of generality that~$\wh I_0=\emptyset$. Furthermore, if~$\Phi$ is not fully supported, we can write it as a composition of a
fully supported light homomorphism with a natural inclusion.

Thus, we are assuming that~$\Phi$ is optimal, fully supported and~$|[\Phi](i)|=1$ for all~$i\in \wh I$. In particular, $[\Phi]$ can be regarded as a surjective map $\wh I\to I$.
Furthermore, by Lemma~\ref{lem:diagonal} we may assume that both~$M$ and~$\wh M$ are irreducible.

Given~$i\in \wh I$, define~$d_i\in\ZZ_{>0}$ by
$\Phi(\wh T_i)=T_{[\Phi](i)}^{d_i}$. Given $j\in I$ and~$d\in\ZZ_{>0}$, let
$\wh I(\Phi,j,d)=\{ i\in[\Phi]^{-1}(j)\,:\,d_i=d\}$. Clearly, 
$\wh I(\Phi,j,d)=\emptyset$ for all
but finitely many~$d\in\ZZ_{>0}$.
Let~$N(\Phi)=|\{i\in\wh I\,:\,d_i>1\}|$.
\begin{lemma}\label{lem:all powers 1}
Suppose that~$N(\Phi)=0$. Then 
$\wh M$ is foldable along~$[\Phi]$,
$M=\wh M^{[\Phi]}$ and~$\Phi=\mathbf F_{[\Phi]}$.
\end{lemma}
\begin{proof}
Since~$N(\Phi)=0$, $d_i=1$ for all~$i\in\wh I$.  
By the optimality of~$\Phi$, $\wh m_{ij}=
m_{[\Phi](i)[\Phi](j)}$ for all $i,j\in\wh I$ such that~$[\Phi](i)\not=[\Phi](j)$. The assertion is now immediate. 
\end{proof}
\begin{lemma}\label{lem:class light Artin 2}
Let $k\in I$, $d\in\ZZ_{>1}$
and~$i\in \wh I(\Phi,k,d)$.  
Set~$\widetilde I=
(\wh I\setminus\wh I(\Phi,k,d))\cup\{i\}$ and  define~$\varpi=\varpi_{k,d}:\wh I\to \widetilde I$
and~$\widetilde M=(\widetilde m_{i'i''})_{i',i''\in\widetilde I}$
by 
\begin{align*}
\varpi(j)&=\begin{cases}
j,&j\in\wh I\setminus\wh I(\Phi,k,d),\\
i,&j\in\wh I(\Phi,k,d),
\end{cases}
\\
\widetilde m_{i'i''}&=\begin{cases}
3,&i\in\{i',i''\},\,
m_{[\Phi](i')[\Phi](i'')}=3,\,d_{i'}d_{i''}\in\{2,3\},\\
\wh m_{i'i''},&\text{otherwise}
\end{cases}
\end{align*}
for all~$j\in\wh I$ and  for all~$i',i''\in\widetilde I$.
Then 
\begin{enmalph}
\item\label{lem:class light Artin 2.i} $\wh M$ is foldable along~$\varpi$;
\item\label{lem:class light Artin 2.ii} $\widetilde M\in\Cox{\widetilde I}$ and
$\wh M^\varpi=\widetilde M(\widetilde{\mathbf d})$ where~$\widetilde{\mathbf d}=(\widetilde d_j)_{j\in\widetilde I}$
with $\widetilde d_j=d^{\delta_{i,j}}$, $j\in\widetilde I$;
\item\label{lem:class light Artin 2.iii} The assignments $\widetilde T_j\mapsto T_{[\Phi](j)}^{d_j^{1-\delta_{i,j}}}$, $j\in\widetilde I$ define
$\widetilde\Phi\in\Hom_{\Art}(\widetilde M,M)$ with~$N(\widetilde\Phi)<N(\Phi)$;
\item\label{lem:class light Artin 2.iv} $\Phi=\widetilde \Phi\circ\mathbf T_{i,d}\circ\mathbf F_\varpi$.
\end{enmalph}
\end{lemma}
\begin{proof}
Let~$j\in \wh I\setminus \wh I(\Phi,k,d)$ and let~$l=[\Phi](j)$.
We claim that for all~$i'\in \wh I(\Phi,k,d)$
\begin{equation}
\label{eq:class step 0c}
\wh m_{i'j}=\begin{cases}
2,&m_{kl}\le 2,\\
2d,&m_{kl}=3,\,d_j=1,\,d\in\{2,3\},\\
\infty,&\text{otherwise}.
\end{cases}
\end{equation}
Indeed, if~$l=k$ then, since~$\Phi(\wh T_{j})=T_k^{d_{j}}\not=T_k^{d}=\Phi(\wh T_{i'})$, $\wh m_{i'j}=2$ by the optimality of~$\Phi$. Suppose that~$l\not=k$. 
Then the restriction of~$\Phi$ to~$\Br^+_{\{i',j\}}(\wh M)\cong \Br^+(I_2(\wh m_{i'j}))$ is a homomorphism
to~$\Br^+_{\{k,l\}}(M)\cong 
\Br^+(I_2(m_{kl}))$. 
If~$m_{kl}=\infty$
then clearly
$\wh m_{i'j}=\infty$. 
If~$m_{kl}=2$ then, in particular, $\Phi(\wh T_{i
})=T_k^d\not=T_l^{d_j}=\Phi(\wh T_j)$ and so~$\wh m_{i'j}=2$ by the optimality of~$\Phi$.
Suppose that~$2<m_{kl}<\infty$.
If~$d_ {j}>1$ (respectively, if $d_{j}=1$
and either $d>3$ or~$m_{kl}>3$) then
$\wh m_{i'j}=\infty$ by~\cite{CP} (respectively, Proposition~\ref{prop:key converse Tits}). 
Finally, if $d_{j}=1$, $d\in\{2,3\}$ and~$m_{kl}=3$
then $\wh m_{i'j}=2d$ by the optimality of~$\Phi$ and Proposition~\ref{prop:key converse Tits}. 

Part~\ref{lem:class light Artin 2.i} is immediate from~\eqref{eq:class step 0c}.

The first assertion in~\ref{lem:class light Artin 2.ii} is obvious. Let~$j,j'\in\widetilde I$.
If~$i\notin\{j,j'\}$ then $\widetilde d_j=\widetilde d_{j'}=1$ and so
$\widetilde M(\widetilde{\mathbf d})_{jj'}=\widetilde m_{jj'}=\wh m_{jj'}=(\wh M^\varpi)_{jj'}$ by~\eqref{eq:Tits cover}. Suppose that, say, $j'=i\not=j$. Then~$\widetilde d_j=1$
and so by~\eqref{eq:Tits cover}
$$
\widetilde M(\widetilde{\mathbf d})_{ij}=\begin{cases}
\widetilde m_{ij},&\widetilde m_{ij}=2,\\
2d,&\widetilde m_{ij}=3,\,d\in\{2,3\},\\
\infty,&\text{otherwise}.
\end{cases}
$$
By definition of~$\widetilde M$, $\widetilde m_{ij}=2$
implies that~$\wh m_{ij}=2=(\wh M^\varpi)_{ij}$. If~$\widetilde m_{ij}=3$ then, since~$\wh m_{ij}\not=3$ by~\eqref{eq:class step 0c}, we must have
$m_{[\Phi](i)[\Phi](j)}=3$, $d_j=1$, $d\in\{2,3\}$ and then~$\wh m_{ij}=2d$ by~\eqref{eq:class step 0c}.
Finally, suppose that~$\widetilde m_{ij}>3$.
Then~$\widetilde m_{ij}=\wh m_{ij}$ by definition of~$\widetilde M$ which forces both of them to be equal to~$\infty$ by~\eqref{eq:class step 0c}. Thus, $\wh M^\varpi=\widetilde M(\widetilde{\mathbf d})$.

To prove~\ref{lem:class light Artin 2.iii}, we only need to show that 
for any~$j\in\widetilde I\setminus\{i\}$, the
assignments 
\begin{equation}\widetilde T_i\mapsto T_k,\quad \widetilde T_j\mapsto T_{l}^{d_j},\quad  l=[\Phi](j),\label{eq:Phi tild}
\end{equation}
define a homomorphism
$\Br^+_{\{i,j\}}(\widetilde M)\cong \Br^+(I_2(\widetilde m_{ij}))\to \Br^+_{\{k,l\}}(M)$.
If~$l=k$ then $\wh m_{ij}=2$ by~\eqref{eq:class step 0c},
$\widetilde m_{ij}=\wh m_{ij}=2$ and, obviously, $T_j T_j^{d_k}=
T_j^{d_k}T_j$. Suppose that~$l\not=k$
and so~$\Br^+_{\{k,l\}}(M)\cong\Br^+(I_2(m_{kl}))$.
If~$\widetilde m_{ij}=3$ then, as in the proof of~\ref{lem:class light Artin 2.ii} above, $m_{kl}=3$, $d_j=1$
and~\eqref{eq:Phi tild} defines an isomorphism of respective parabolic submonoids.
Otherwise, $\widetilde m_{ij}=\wh m_{ij}\in\{2,\infty\}$. In either case, \eqref{eq:Phi tild} defines a homomorphism of respective parabolic submonoids. By construction, $N(\Phi')
\le N(\Phi)-1<N(\Phi)$.

To prove~\ref{lem:class light Artin 2.iv}, note that for~$j\in \wh I\setminus\wh I(\Phi,k,d)$ we get $$(\widetilde\Phi\circ\mathbf T_{i,d}\circ\mathbf F_\varpi)(\wh T_j)=(\widetilde\Phi\circ\mathbf T_{i,d})(\wh T_j)=\widetilde\Phi(\widetilde T_j)=T_{[\Phi](j)}^{d_j}=\Phi(\wh T_j),$$ 
while 
for~$j\in\wh I(\Phi,k,d)$,
$$
(\widetilde\Phi\circ\mathbf T_{i,d}\circ\mathbf F_\varpi)(\wh T_j)
=(\widetilde\Phi\circ\mathbf T_{i,d})(\wh T_i)=
\widetilde \Phi(\widetilde T_i^{d})
=\widetilde \Phi(\widetilde T_i)^{d_i}
=T_{k}^{d_i}=\Phi(\wh T_j).
$$
This completes the proof of Lemma~\ref{lem:class light Artin 2}.
\end{proof}

Note that~$\widetilde\Phi$ obtained from~$\Phi$
by Lemma~\ref{lem:class light Artin 2} does not have to be optimal. However, by Lemma~\ref{lem:factor homs}, we can always factor out a tautological homomorphism. Therefore, by applying Lemma~\ref{lem:class light Artin 2} repeatedly, we conclude that~$\Phi=\Phi'\circ\Phi''$ where~$\Phi'$ is optimal with $N(\Phi')=0$ and~$\Phi''$ is a composition of elementary Tits homomorphisms, foldings along surjective maps of index sets and tautological homomorphisms. But then~$\Phi'$ is a folding by Lemma~\ref{lem:all powers 1}.

Finally, all light homomorphisms listed in Theorem~\ref{thm:classify light Artin}, with the exceptions of Tits homomorphisms, are standard, while 
Tits homomorphisms are morphisms in~$\Ast$ by Theorem~\ref{thm:Tits standard}.
\end{proof}
\begin{example}
Let $d\in\{2,3\}$, $m\in\ZZ_{>2}\cup\{\infty\}$ and 
let~$\wh M=\left(\begin{smallmatrix}
    1&\infty&2d&2d\\
    \infty&1&2&2\\
    2d&2&1&m\\
    2d&2&m&1
\end{smallmatrix}\right)$. 
It is easy to see that the assignments $\wh T_1\mapsto T_1$, $\wh T_2\mapsto T_2^{d'}$, $\wh T_3\mapsto T_2^{d}$,
$\wh T_4\mapsto T_2^{d}$, 
$d'\in\ZZ_{>3}$, define an optimal homomorphism $\Br^+(\wh M)\to \Br^+(A_2)$, which factorizes as
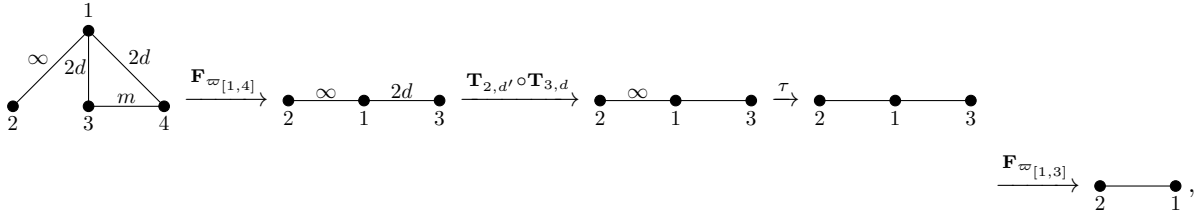
\begin{multline*}
\begin{dynkinDiagram}[name=upper,label directions={above}]A1
\node (current) at ($(upper root 1)+(-1cm,-1cm)$) {};
\dynkin[at=(current),name=left,labels={2},label directions={below}]A1
\node (current) at ($(left root 1)+(1cm,0cm)$) {};
\dynkin[at=(current),name=right,labels={3,4},label directions={below,below}]A2
\dynkinEdgeLabel{1}{2}{m}
\begin{pgfonlayer}{Dynkin behind}
\draw[/Dynkin diagram,edge]
($(upper root 1)$) -- ($(left root 1)$);
\draw[/Dynkin diagram,edge]
($(upper root 1)$) -- ($(right root 1)$);
\draw[/Dynkin diagram,edge]
($(upper root 1)$) -- ($(right root 2)$);
\end{pgfonlayer}
\draw[draw=none] (left root 1) to
                node[auto,                inner sep=0.8,                /Dynkin diagram/text style,                /Dynkin diagram/edge label]
                {\(\pgfkeys{/Dynkin diagram/label macro*=\infty}\)}                (upper root 1);\draw[draw=none] (right root 1) to
                node[auto,                inner sep=1,                /Dynkin diagram/text style,                /Dynkin diagram/edge label]
                {\(\pgfkeys{/Dynkin diagram/label macro*={2d}}\)}                (upper root 1);\draw[draw=none] (upper root 1) to
                node[auto,                inner sep=0.8,                /Dynkin diagram/text style,                /Dynkin diagram/edge label]
                {\(\pgfkeys{/Dynkin diagram/label macro*={2d}}\)}                (right root 2);\end{dynkinDiagram}
\xrightarrow{\mathbf F_{\varpi_{[1,4]}}}
\begin{dynkinDiagram}[labels={2,1,3}]A3
\dynkinEdgeLabel{1}{2}{\infty}
\dynkinEdgeLabel{2}{3}{2d}
\end{dynkinDiagram}
\xrightarrow{
\mathbf T_{2,d'}\circ\mathbf T_{3,d}}
\begin{dynkinDiagram}[labels={2,1,3}]A3
\dynkinEdgeLabel{1}{2}{\infty}
\end{dynkinDiagram}
\xrightarrow{\tau}
\dynkin[labels={2,1,3}]A3
\\
\xrightarrow{\mathbf F_{\varpi_{[1,3]}}}
\dynkin[labels={2,1}]A2,
\end{multline*}
where~$\varpi_{[1,k]}:[1,k]\to [1,k-1]$, $k>1$ is defined by $\varpi(i)=i-\delta_{i,k}$,
$i\in[1,k]$ and~$\tau$ is the tautological homomorphism.
\end{example}
\begin{example}
Let~$M=(m_{ij})_{i,j\in I}\in\Cox I$ and let~$\mathbf d=(d_i)_{i\in I}\in\ZZ_{\ge 0}^I$ with~$d_i=d_j$ whenever~$m_{ij}$
is odd. Then the character homomorphism~$\Xi_\mathbf X:\Br^+(M)\to\Br^+(A_1)$ with~$\mathbf X=(T_1^{d_i})_{i\in I}$ factorizes as follows. Let~$\sim_M$ be the transitive closure of the relation on~$I$ defined by~$i\sim_M j$ whenever~$m_{ij}$
is odd, $i,j\in I$; this is manifestly an equivalence relation. Let~$\underline I$ be the set of equivalence 
classes for~$\sim_M$ and denote the class of~$i\in I$
by~$\underline i$. Let~$\tilde M=(\tilde m_{ij})_{i,j\in I}$
with~$\tilde m_{ij}=m_{ij}$ if~$m_{ij}$ is odd and~$\tilde m_{ij}=2$
if~$m_{ij}$ is even, $i,j\in I$. Note that~$\tilde M$
is foldable along the canonical map~$\varpi:I\to\underline I$, $i\mapsto \underline i$, $i\in I$
with~$\tilde M^\varpi$ being the 
product of copies of~$A_1$ indexed by~$\underline I$
(cf. Lemma~\partref{lem:free product.a}). Then
$\Xi_{\mathbf X}=\mathbf F_{\underline\varpi}\circ 
\mathbf T_{\underline{\mathbf d}}\circ \mathbf F_{\varpi}
\circ\tau$, where~$\tau:\Br^+(M)\to\Br^+(\tilde M)$
is the tautological homomorphism, $\underline{\mathbf d}=
(d_{\underline i})_{\underline i\in\underline I}$
with~$d_{\underline i}=d_i$, $i\in I$, and~$\underline\varpi$
is the unique map~$\underline I\to \{1\}$.
\end{example}
\begin{proposition}
\label{prop:types of light}
The following exhausts
optimal fully supported light homomorphisms between Artin monoids of irreducible finite types
\begin{enmalph}

\item\label{prop:types of light.PBnAn-1} $P_{[1,n-1]}:\Br^+(B_n)\to \Br^+(A_{n-1})$;
\item\label{prop:types of light.PF4A2}  $P_{[1,2]}:\Br^+(F_4)\to \Br^+(A_2)$;
\item\label{prop:types of light.PBnA1}
$P_{\{n\}}:\Br^+(B_n)\to\Br^+(A_1)$;
\item\label{prop:types of light.PI2A1}
$P_{\{1\}}:\Br^+(I_2(2m))\to \Br^+(A_1)$;
\item\label{prop:types of light.FDn+1An} $\mathbf F_{\varpi_{(n,n+1)}}:
\Br^+(D_{n+1})\to \Br^+(A_n)$, $n\ge 2$, where 
$\varpi_{(n,n+1)}:[1,n+1]\to[1,n]$ is defined by~$\varpi_{(n,n+1)}(i)=i-\delta_{i,n+1}$, $i\in[1,n+1]$;
\item\label{prop:types of light.FD4A2} $\mathbf F_{\varpi_{(1,3,4)}}:
\Br^+(D_4)\to \Br^+(A_2)$ where 
$\varpi_{(1,3,4)}:[1,4]\to\{1,2\}$
is defined by $\varpi_{(1,3,4)}(i)=1$, $i\in\{1,3,4\}$, $\varpi_{(1,3,4)}(2)=2$;
\item\label{prop:types of light.A1A1}
$\mathbf T_{1,d}:\Br^+(A_1)\to\Br^+(A_1)$, $T_1\mapsto T_1^d$, $d\in\ZZ_{>1}$;
\item\label{prop:types of light.BnAn} $\mathbf T_{n,2}:\Br^+(B_n)\to\Br^+(A_n)$, $\wh T_i\mapsto T_i^{1+\delta_{i,n}}$, $i\in[1,n]$;
\item\label{prop:types of light.G2A2} $\mathbf T_{2,3}:\Br^+(G_2)\to \Br^+(A_2)$, $\wh T_1\mapsto T_1$, $T_2\mapsto T_2^3$;
\item\label{prop:types of light.len} The length
homomorphism~$\ell:\Br^+(M)\to (\ZZ_{\ge0},+)\cong
\Br^+(A_1)$.
\end{enmalph}

The only tautological homomorphisms
of irreducible Artin monoids of finite type are from $\Br^+(I_2(dm))$ to~$\Br^+(I_2(m))$, 
$d>0$, $m\ge 3$.

All other light homomorphisms of 
Artin monoids of finite type are obtained as compositions of the above ones, up to natural inclusions, parabolic projections onto connected components, direct products
and diagram automorphisms.
\end{proposition}
\begin{proof}
It suffices to verify that the above list exhausts all ``elementary'' light homomorphisms listed in Theorem~\ref{thm:classify light Artin}. 
 By
Proposition~\ref{prop:parab proj Artin}, \ref{prop:types of light.PBnAn-1}--\ref{prop:types of light.PI2A1} are the only parabolic projections existing in finite irreducible types. 
By Proposition~\ref{prop:key converse Tits}, \ref{prop:types of light.A1A1}--\ref{prop:types of light.G2A2} exhaust
all elementary Tits homomorphisms in finite types.
By inspection, the only existing foldings with irreducible codomain are those in~\ref{prop:types of light.FDn+1An}, \ref{prop:types of light.FD4A2}
and \ref{prop:types of light.len}.
\end{proof}
\begin{remark}
By~\cite{Cri}*{Proposition~3.5}, the homomorphism $\mathbf T_{n,2}$ from Proposition~\partref{prop:types of light.BnAn} is injective.
\end{remark}
We conclude this section with the following property of~$\mathbf F_{\varpi(n,n+1)}$.
\begin{lemma}\label{lem:Dn+1 An w0}
Let~$J\subset [1,n+1]$. Then
$$
\mathbf F_{\varpi_{(n,n+1)}}(T_{w_\circ^J})=
\begin{cases}
T_{w_\circ^J},& J\subset [1,n],\\
T_{w_\circ^{\sigma(J)}},& J\subset\sigma([1,n]),\\
T_{w_\circ^{J\setminus J'}}T_{w_\circ^{J'\setminus \{n+1\}}}^2=
T_{w_\circ^{J\setminus \{n+1\}}}
T_{w_\circ^{J'\setminus \{n+1\}}}\\
\quad =
T_{w_\circ^{J'\setminus \{n+1\}}}
T_{w_\circ^{J\setminus \{n+1\}}}
,&\{n,n+1\}\subset J,\\
\end{cases}
$$
where~$J'$ is the maximal interval $[i,n+1]$ contained in~$J$ and~$\sigma$ is the transposition~$(n,n+1)$.
\end{lemma}
\begin{proof}
The first two cases are obvious. To prove the last,
since~$J\setminus J'$ and~$J'$ are orthogonal,
it suffices to prove the first equality and hence
to consider the case when~$J=J'$. If~$J=\{n,n+1\}$
then $T_{w_\circ^J}=T_nT_{n+1}$ and so
~$\mathbf F_{\varpi_{(n,n+1)}}(T_{w_\circ^J})=T_n^2$.
Otherwise, $J=[i,n+1]$ for some~$1\le i\le n-1$ and~$T_{w_\circ^J}$ is 
central in~$\Br^+_J(M)$ if~$n-i$ is even. If~$n-i$ is odd then~$T_j T_{w_\circ^J}=T_{w_\circ^J}T_j$
for all~$j\in [i,n-1]$ while
$T_j T_{w_\circ^J}=T_{w_\circ^J}T_{2n+1-j}$, $j\in\{n,n+1\}$. It follows that
$\mathbf F_{\varpi_{(n,n+1)}}(T_{w_\circ^{[i,n+1]}})$ is central in~$\Br^+_{[i,n]}(M)\cong \Br^+(A_{n+1-i})$. Since~$\ell(\mathbf F_{\varpi_{(n,n+1)}}(T))=\ell(T)$
for all~$T\in\Br^+(M)$,
we conclude that~$\ell(\mathbf F_{\varpi_{(n,n+1)}}(T_{w_\circ^{[i,n+1]}}))=
\ell(T_{w_\circ^{[i,n+1]}})=
(n+1-i)(n+2-i)=\ell(T_{w_\circ^{[i,n]}}^2)$. 
Since~$T_{w_\circ^{[i,n]}}^2$
generates the center of~$\Br^+_{[i,n]}(M)$ by Proposition~\partref{prop:fund elts BrSa.d}, it follows that
$\mathbf F_{\varpi_{(n,n+1)}}(T_{w_\circ^{[i,n+1]}})=T_{w_\circ^{[i,n]}}^2$.
\end{proof}

\section{Classification of disjoint standard  homomorphisms}\label{sec:SQF AB}
We begin by classifying all disjoint fully supported standard homomorphisms $\Br^+(\wh M)
\to \Br^+(M)$ where $\wh M$ is irreducible and of finite type
and~$M$ is of type~$A_n$ or~$B_n$. In view of Lemma~\ref{lem:factor homs}, we will only consider
optimal homomorphisms. 

\subsection{Two families of homomorphisms in type~\texorpdfstring{$B$}{B}} 
We will often use the following Lemma.
\begin{lemma}\label{lem:I2m iff cnd}
Let~$m>1$ and~$M\in\Cox I$. Let~$X_1,X_2\in \Br^+(M)$ be ${}^{op}$-invariant. Then the 
assignments $\wh T_i\mapsto X_i$, $i\in \{1,2\}$ define a homomorphism
$\Br^+(I_2(2m))\to \Br^+(M)$ if and only if $(X_1X_2)^m$ is ${}^{op}$-invariant.
\end{lemma}
\begin{proof}
These assignments define a homomorphism~$\Br^+(I_2(2m))\to\Br^+(M)$ if and only if 
$$
(X_1X_2)^m=\brd{X_1X_2}{2m}=\brd{X_2X_1}{2m}=(X_2X_1)^m.
$$
Since~$X_1$ and~$X_2$ are ${}^{op}$-invariant, this happens if and only if $((X_1X_2)^m)^{op}=(X_2X_1)^m=(X_1X_2)^m$.
\end{proof}
We begin by constructing infinite families
of disjoint Coxeter type homomorphisms and of
standard homomorphisms $\Br^+(I_2(2m))\to \Br^+(B_n)$
for~$2\le m\le n$.
\begin{proposition}  \label{prop:admissible hom from BrI22m to BrBn}
Let $2\le m\le n$.
\begin{enmalph}
\item\label{prop:admissible hom from BrI22m to BrBn.a} 
The assignments 
$$
\wh T_1\mapsto T_{w_\circ^{[1,m-1]_2}},\quad
\wh T_2\mapsto T_{w_\circ^{[1,m-2]_2}}\Cx mn\Cxr m{(n-1)}
$$
define a disjoint~$\Phi\in\Hom_{\mathscr{AC}}(I_2(2m),B_n)$.
\item\label{prop:admissible hom from BrI22m to BrBn.b}
The assignments
$$
\wh T_1\mapsto T_{w_\circ^{[1,m-1]_2}},\quad
\wh T_2\mapsto T_{w_\circ^{[1,m-2]_2\cup [m,n]}}
$$
define a disjoint standard $\wh\Phi\in\Hom_{\Art}(I_2(2m),B_n)$.
\end{enmalph}
\end{proposition}
\begin{proof}
Abbreviate $T_{(i,n+1)}:=\Cxr in\Cxr i{(n-1)}$, $1\le 
i\le n$. Note that since $w_\circ^{[i,n]}
=\cx in\times \cxr i{(n-1)}\times w_\circ^{[i+1,n]}=
w_\circ^{[i+1,n]}\times \cx in\times \cxr i{(n-1)}$,
for all~$1\le i\le n-1$,
\begin{equation}\label{eq:B_n telescope 0}
T_{w_\circ^{[i,n]}}=T_{(i,n+1)}T_{w_\circ^{[i+1,n]}}= T_{w_\circ^{[i+1,n]}}T_{(i,n+1)}.
\end{equation}
We need the following
\begin{lemma}\label{lem:B2 I22m BrBn main id}
For all $2\le m\le n$ we have in~$\Br(B_n)$
\begin{equation}
( T_{w_\circ^{[1,m-1]_2}} T_{w_\circ^{[1,m-2]_2}} T_{(m,n+1)})^m  
=T_{w_\circ^{[m+1,n]}}^{-1} T_{w_\circ^{[1,n]}}=T_{w_\circ^{[1,n]}}
T_{w_\circ^{[m+1,n]}}^{-1} .\label{eq:prep I2m Bn}
\end{equation}
\end{lemma}
\begin{proof}
Since $m\ell(T_{w_\circ^{[1,m-1]_2}} T_{w_\circ^{[1,m-2]_2}} T_{(m,n+1)})=
(2n-m)m=n^2-(n-m)^2=\ell(T_{w_\circ^{[1,n]}})-
\ell(T_{w_\circ^{[m+1,n]}}),
$
it suffices to
prove that
\begin{equation}\label{eq:weyl grp id}
(w_\circ^{[1,m-1]_2}w_\circ^{[1,m-2]_2}\cx mn\cxr m{(n-1)})^m w_\circ^{[m+1,n]}=w_\circ^{[1,n]}
\end{equation}
in~$W(B_n)$. For that, since the reflection representation of~$W(B_n)$
on an~$n$-dimensional Euclidean space is faithful (see e.g.~\cite{Bou}*{Ch.~IV, \S4.4, Corollaire~2}),
it suffices to show that the left hand
side of~\eqref{eq:weyl grp id} acts as~$-1$ in the reflection representation since~$w_\circ^{[1,n]}$ acts this way.
Indeed, recall (see e.g.~\cite{Bou}*{Ch.~VI, \S4.5}) that~$W(B_n)$
is isomorphic to the semi-direct product of~$S_n$ with~$\mathbb Z_2^n$.
Let
$\epsilon_1,\dots,\epsilon_n$ be the standard basis of~$\mathbb R^n$.
Then~$S_n$ acts by permutations of the~$\epsilon_i$ and the~$i$th copy
of~$\mathbb Z_2$ acts as $\epsilon_j\mapsto (-1)^{\delta_{i,j}}\epsilon_j$,
$1\le j\le n$.
Then
\begin{equation}\label{eq:B refl 1}
w_\circ^{[m+1,n]}(\epsilon_i)=\begin{cases}
    \epsilon_i,&1\le i\le m\\
    -\epsilon_i,&m+1\le i\le n.
    \end{cases}
\end{equation}
while $\cx mn\cxr m{n-1}(\epsilon_i)=(-1)^{\delta_{i,m}}\epsilon_i$,
$1\le i\le n$. Since~$w_\circ^{[1,m-1]_2}w_\circ^{[1,m-2]_2}$
is a Coxeter element in the parabolic subgroup
$W_{[1,m-1]}(B_n)\cong W(A_{m-1})$,
it identifies with a cycle of length~$m=h(A_{m-1})$ permuting all the
$\epsilon_1,\dots,\epsilon_{m}$.
It follows that
\begin{equation}\label{eq:B refl 2}
(w_\circ^{[1,m-1]_2}w_\circ^{[1,m-2]_2}\cx mn\cxr m{(n-1)})^m(\epsilon_i)=
\begin{cases}-\epsilon_i,&1\le i\le m\\
\epsilon_i,&m+1\le i\le n.
\end{cases}
\end{equation}
Together~\eqref{eq:B refl 1} and~\eqref{eq:B refl 2} imply that
the left hand side of~\eqref{eq:weyl grp id} acts as~$-1$.
\end{proof}
By Proposition~\ref{prop:fund elts BrSa}, $T_{w_\circ^{[1,n]}}$ is central 
in~$\Br(B_n)$ and $T_{w_\circ^{[1,n]}}$, $T_{w_\circ^{[m+1,n]}}$ are ${}^{op}$-invariant.
It follows that the left hand side of~\eqref{eq:prep I2m Bn}
is ${}^{op}$-invariant. Since~$T_{(m,n+1)}$ as  
well as the $T_{w_\circ^{[1,m-i]_2}}$, $i\in\{1,2\}$ are ${}^{op}$-invariant 
and $T_{(m,n+1)}$ commutes with~$T_{w_\circ^{[1,m-2]_2}}$, it 
follows that~$T_{w_\circ^{[1,m-2]_2}}T_{(m,n+1)}$ is ${}^{op}$-invariant. Since~$T_{w_\circ^{[1,n-1]_2}}$ is also ${}^{op}$-invariant,
it follows from Lemma~\ref{lem:I2m iff cnd} that
the assignments 
in part~\ref{prop:admissible hom from BrI22m to BrBn.a} define a homomorphism $\Phi:\Br^+(I_2(2m))\to\Br^+(B_n)$ which is 
of Coxeter type by Theorem~\partref{thm:Main Thm Cox Heck.a}. 

To prove part~\ref{prop:admissible hom from BrI22m to BrBn.b}, let~$z_1=1$ and~$z_2=T_{w_\circ^{[m+1,n]}}$.
By~\eqref{eq:B_n telescope 0}, $T_{w_\circ^{[1,m-2]_2\cup[m,n]}}=\Phi(\wh T_2)z_2$. 
We claim that~$\mathbf z=(z_1,z_2)$ is a decoration
of~$\Phi$ and so~$\wh\Phi=\Phi_{\mathbf z}
\in\Hom_{\Art}(I_2(2m),B_n)$. Since~$z_1=1$,
by Lemma~\ref{lem:cent decor}
it suffices to prove that~$z_2$ commutes with~$\Phi(\wh T_i)$, $i\in\{1,2\}$ which is obvious for~$i=1$ and follows from~\eqref{eq:B_n telescope 0} for~$i=2$.
\end{proof}

\subsection{Key result}\label{subs:Key res}
Fix~$n>1$. We abbreviate~\plink{Brn}$\Br^+_{n+1}:=\Br^+(A_n)$, $\Br_{n+1} :=\Br(A_n)$ and~$\pi_n:=\pi_{A_n}$. Let~$\sigma$
be the diagram automorphism of~$\Br^+_{n+1}$, the corresponding
permutation of~$I=[1,n]$ being $\sigma(i)=n+1-i$, $i\in [1,n]$.
Note that if~$J=[a,b]\subset I$ satisfies~$\sigma(J)=J$ then
$b=n+1-a$ and so $n-|J|=2(a-1)$ is even.

Our present goal is to prove the following theorem which generalizes
the classical result from~\cite{BrSa}*{\S5.8} (cf. Corollary~\partref{cor:adm finite class.odd}\ref{cor:adm finite class.even} and
Proposition~\ref{prop:Coxeter splitting}).
\begin{theorem}\label{thm:main thm adm}
Let~$K\subsetneq I$ be an interval with~$|K|>1$
and let~$I'(K)\sqcup I''(K)$ be the unique partition of~$I\setminus K$ into self-orthogonal subsets such that~$I'(K)$
and~$K$ are orthogonal. Then  the assignments
$$
\wh T_1\mapsto T_{w_\circ^{I'(K)\cup K}},\qquad \wh T_2\mapsto T_{w_\circ^{I''(K)}}
$$
define an optimal (disjoint standard) homomorphism $\Br^+(I_2(2m(K)))\to \Br^+_{n+1}$
where
$$
m(K)=\begin{cases}
\frac12(n-|K|)+1,& \sigma(K)=K,\\
n-|K|+2,&\text{otherwise.}
\end{cases}
$$

Conversely, suppose that~$\Phi:\Br^+(I_2(N))\to \Br^+_{n+1}$ is an optimal disjoint fully supported standard homomorphism such that~$[\Phi](i)\not=\emptyset$,
$i\in \wh I=\{1,2\}$. Then
\begin{enmalph}
\item \label{thm:main them adm.a}
either both $[\Phi](1)$ and~$[\Phi](2)$ are self-orthogonal, or
exactly one of them contains a unique connected component
of rank~$>1$.

\item  \label{thm:main them adm.b} Suppose that
$[\Phi](1)$ and~$[\Phi](2)$ are self-orthogonal.
Then~$N=n+1$ and~$\Phi$ is the
homomorphism from Corollary~\partref{cor:adm finite class.odd} or~\ref{cor:adm finite class.even}.

\item\label{thm:main them adm.c} Suppose that precisely one
of the~$[\Phi](i)$, $i\in\{1,2\}$ contains a unique connected component~$K$
with~$|K|>1$. Then~$N=2m(K)$.
\end{enmalph}
\end{theorem}
\begin{conjecture}
Let~$\Phi:\Br^+(I_2(N))\to\Br^+_{n+1}$ and~$K$ be as in the converse of Theorem~\ref{thm:main thm adm}. Then both~$\overline\Phi$ and~$\overline\Phi_\star$ are injective
if and only if either both~$[\Phi](1)$, $[\Phi](2)$ are self-orthogonal, or~$K$ satisfies~$\sigma(K)=K$. 
\end{conjecture}

\subsection{Transpositions in braid monoids}
Given $i\le j\in [1,n]$, denote \plink{T(i,j)}$T_{(i,j+1)}$ the unique square free element of~$\Br^+_{n+1}$
corresponding to the transposition~$(i,j+1)$ in~$S_{n+1}$ which identifies with $W(A_n)$.
We use the convention that $T_{(i,j)}=1$ if~$i\ge j$.

\begin{proposition}\label{prop:elem prop transp}
Let $i\le j\in [1,n]$.
\begin{enmalph}
\item \label{prop:elem prop transp.a}$\ell(T_{(i,j+1)})=2(j-i)+1$.
    \item \label{prop:elem prop transp.b}
    $T_{(i,j+1)}=T_i T_{(i+1,j+1)}T_i=T_j T_{(i,j)} T_j$.
    \item \label{prop:elem prop transp.c}
    $T_{(i,j+1)}=\Cx ij\Cxr i{(j-1)}=\Cx i{(j-1)}\Cxr ij=
    \Cxr ij\Cx{(i+1)}j=\Cxr{(i+1)}j\Cx ij$.
    In particular, $T_{(i,j+1)}$ is ${}^{op}$-invariant.
    \item\label{prop:elem prop transp.e} $T_{k}T_{(i,j+1)}=T_{(i,j+1)}T_k$ for all $k\in [1,n]\setminus \{i-1,i,j,j+1\}$.
    \item\label{prop:elem prop transp.f}
    Let $k\le l\in[1,n]$. Then $T_{(k,l+1)}$ commutes with $T_{(i,j+1)}$
    provided that either $i<k$, $l<j$, or $l<i-1$ or~$j+1<k$.
    \item \label{prop:elem prop transp.g}
    $T_{w_\circ^{[i,j]}}=\prod_{0\le k\le \frac12(j-i)} T_{(i+k,j+1-k)}$.
\end{enmalph}
\end{proposition}
\begin{proof}
Since~$\ell((i,j+1))=2(j-i)+1$,
part~\ref{prop:elem prop transp.a} is obvious. Since
$(i,i+1)(i+1,j+1)(i,i+1)=(i,j+1)=(j,j+1)(i,j)(j,j+1)$
in~$S_{n+1}$ and
$\ell((i,j+1))=2+\ell((i+1,j+1))=2+\ell((i,j))$, part~\ref{prop:elem prop transp.b} follows. Part~\ref{prop:elem prop transp.c} follows from~\ref{prop:elem prop transp.b} by a straightforward induction.
The assertion of~\ref{prop:elem prop transp.e} is obvious
for $k\in [1,i-2]\cup[j+2,n]$. To prove it for $k\in[i+1,j-1]$, we
need the following
\begin{lemma}\label{lem:comm cox}
Let $i< j\in [1,n]$ and $k\in[i+1,j]$. Then $T_k\Cx ij=\Cx ij T_{k-1}$ and $T_{k-1}\Cxr ij=\Cxr ij T_k$.
\end{lemma}
\begin{proof}
We have
\begin{align*}
T_k \Cx ij&=\Cx i{(k-2)}T_k T_{k-1}T_k \Cx{(k+1)}j=\Cx i{(k-2)}T_{k-1}T_k T_{k-1}\Cx{(k+1)}j=\Cx ij T_{k-1}.
\end{align*}
The second identity is obtained from the first by applying~${}^{op}$.
\end{proof}
Thus, given $k\in [i+1,j-1]$ we have
$$
T_k T_{(i,j+1)}=T_k \Cx ij\Cxr i{(j-1)}=\Cx ij T_{k-1}\Cxr i{(j-1)}
=\Cx ij\Cxr i{(j-1)}T_k=T_{(i,j+1)}T_k.
$$
Part~\ref{prop:elem prop transp.f} is immediate from part~\ref{prop:elem prop transp.e}. Finally, since $w_\circ^{[i,j]}
=\prod_{0\le k\le \frac12(j-i)} (i+k,j+1-k)$ and
$\ell(w_\circ^{[i,j]})=\binom{j-i+2}2=\sum_{0\le k\le \frac12(j-i)}
(2(j-i-2k)+1)=\sum_{0\le k\le \frac12(j-i)}\ell((i+k,j+1-k))$,
part~\ref{prop:elem prop transp.g} follows.
\end{proof}

For $J=\{j_1 <j_2< \cdots <j_m\} \subset [1,n+1]$, set $$
\plink{TJ}T_J = \tilde \tau_1(J)\tilde \tau_0(J),
\qquad \tilde\tau_k(J)=\prod_{1\le r\le m\,:\, \bar{r}=k} T_{(j_r,j_{r+1})}
$$
By Proposition~\partref{prop:elem prop transp.e},
$\tilde\tau_k(J)$, $k\in\{0,1\}$ are well-defined. Note that
\begin{equation}\label{eq:len TJ}
\ell(T_J)=2\sum_{1\le r\le m-1} (j_{r+1}-j_r)-m+1=2(\max J-\min J)-|J|+1.
\end{equation}
We also set for $k\in\{0,1\}$\plink{tauk(J)}
\begin{equation}\label{eq:tau_i(J) defn}
\tau_k(J) = \prod_{1\le r\le m\,:\, \bar{r}=k} T_{w_\circ^{[j_r,j_{r+1}-1]}}.
\end{equation}
In particular, it follows from Proposition~\partref{prop:elem prop transp.g} that
\begin{equation}\label{eq:tilde tau tau}
\tau_k(J)=\tilde\tau_k(J) X_k(J),\qquad X_k(J)=\prod_{1\le r\le m\,:\,\bar r=k} T_{w_\circ^{[j_r+1,j_{r+1}-2]}},\qquad k\in\{0,1\}.
\end{equation}
Clearly, $X_1(J)$ and~$X_0(J)$ commute and also
commute with the~$\tilde\tau_k(J)$, $k\in\{0,1\}$.
The following Lemma is obvious
\begin{lemma}\label{lem:all disj TJ}
Up to renumbering of the generators, every fully supported disjoint standard homomorphism $\Psi:\Br^+(I_2(N))\to \Br^+_{n+1}$
satisfies $\Psi(\wh T_r)=\tau_{\bar r}(J)$, $r\in\{1,2\}$
for some~$\{1,n+1\}\subset J\subset [1,n+1]$;
\end{lemma}
The following Lemma is crucial for proving Theorem~\ref{thm:main thm adm}.
\begin{lemma}\label{lem:TJ hom}
The following are equivalent for~$J\subset[1,n+1]$
and~$m\ge 1$;
\begin{enmalph}
\item\label{lem:TJ hom.a}
$T_J^m$ is ${}^{op}$-invariant; 
\item\label{lem:TJ hom.b}
The assignments~$\wh T_k\mapsto \tilde\tau_{\bar k}(J)$, $k\in\{1,2\}$ define a homomorphism 
$\Br^+(I_2(2m))\to
\Br^+_{n+1}$;
\item\label{lem:TJ hom.c} The assignments 
$\wh T_k\mapsto \tau_{\bar k}(J)$, $k\in\{1,2\}$
define a homomorphism $\Br^+(I_2(2m))\to
\Br^+_{n+1}$.
\end{enmalph}
\end{lemma}
\begin{proof}
Assertions~\ref{lem:TJ hom.a} and~\ref{lem:TJ hom.b} 
are equivalent by Lemma~\ref{lem:I2m iff cnd}.
Suppose that the assignments in part~\ref{lem:TJ hom.b}
define~$\tilde\Phi\in\Hom_{\Art}(I_2(2m),A_n)$. 
Then by Lemma~\ref{lem:cent decor}, $\mathbf z=(X_1(J),X_0(J))$ is a decoration of~$\tilde\Phi$
and then~$\tilde\Phi_{\mathbf z}(\wh T_r)=\tau_{\bar r}(J)$,
$r\in\{1,2\}$ by~\eqref{eq:tilde tau tau}. 

Finally, suppose that
the assignments in~\ref{lem:TJ hom.c} define~$\Phi\in\Hom_{\Art}(I_2(2m),A_n)$ and
extend it
to a homomorphism $\Br^+(I_2(2m))\to \Br_{n+1}$. Then
$\mathbf z^{-1}=(X_1(J)^{-1},X_0(J)^{-1})$ is a decoration of~$\Phi$ by Lemma~\ref{lem:cent decor}, and
$\Phi_{\mathbf z^{-1}}(\wh T_r)=\tilde\tau_{\bar r}(J)$,
$r\in\{1,2\}$.
\end{proof}

\subsection{Symmetries and conjugation}
By Lemma~\ref{lem:TJ hom}, to prove Theorem~\ref{thm:main thm adm}
we need to find a necessary and sufficient condition
for~$T_J^m$, $\{1,n+1\}\subset J\subset [1,n+1]$, $m\in\mathbb Z_{>0}$, to be~${}^{op}$-invariant.

Given $J\subset [1,n+1]$, let $\tilde\sigma(J)=\{n+2-j\,:\,j\in J\}$.
\begin{lemma}\label{lem:diag aut TJ}
Let $J\subset [1,n+1]$ and let $\sigma$ be the diagram automorphism
of~$\Br^+_{n+1}$. Then
$$
\sigma(T_J)=\begin{cases}
T_{\tilde\sigma(J)},& \text{$|J|$ is even},\\
T_{\tilde\sigma(J)}{}^{op},& \text{$|J|$ is odd}.
\end{cases}
$$
In particular, $T_J^m$ is ${}^{op}$-invariant if and only if
$T_{\tilde\sigma(J)}^m$ is ${}^{op}$-invariant.
\end{lemma}
\begin{proof}
Write $J=\{j_1,\dots,j_k\}$ where $j_1<j_2<\cdots<j_{k-1}<j_k$. Then $\tilde\sigma(J)=\{n+2-j_k,n+2-j_{k-1},\cdots,n+2-j_2,n+2-j_1\}$ and
$$
T_J=\prod_{1\le l\le k/2} T_{(j_{2l-1},j_{2l})} \prod_{1\le l\le (k-1)/2} T_{(j_{2l},j_{2l+1})}.
$$
Suppose first that~$k=2m$. Then
$$
T_J=\prod_{1\le l\le m} T_{(j_{2l-1},j_{2l})} \prod_{1\le l\le m-1} T_{(j_{2l},j_{2l+1})}
$$
Since $\sigma(T_{(a,b)})=T_{(n+2-b,n+2-a)}$
$$
\sigma(T_J)=\prod_{1\le l\le m} T_{(n+2-j_{2l},n+2-j_{2l-1})} \prod_{1\le l\le 2m-1} T_{(n+2-j_{2l+1},n+2-j_{2l})}=T_{\tilde\sigma(J)}.
$$
If~$k=2m+1$ then
\begin{align*}
\sigma(T_J)&=\prod_{1\le l\le m} \sigma(T_{(j_{2l-1},j_{2l})}) \prod_{1\le l\le m} \sigma(T_{(j_{2l},j_{2l+1})})\\&=
\prod_{1\le l\le m} T_{(n+2-j_{2l},n+2-j_{2l-1},j_{2l})} \prod_{1\le l\le m} T_{(n+2-j_{2l+1},n+2-j_{2l})}=T_{\tilde\sigma(J)}{}^{op}.\qedhere
\end{align*}
\end{proof}

Our next goal is to show that all the~$T_J$ with~$J$ of the same
cardinality are conjugate in~$\Br_{n+1}$ (eventually we will also see that the converse is true).
\begin{proposition}\label{prop:conj}
Let $J\subset [1,n+1]$ and let
$j\in J$ with $\min J<j<\max J$ and $j-1\notin J$. Then in~$\Br_{n+1}$
$$
T_{(J\setminus\{j\})\cup\{j-1\}}=T_{j-1}^{\epsilon} T_J T_{j-1}^{-\epsilon},
$$
where $\epsilon=(-1)^{|J\cap [1,j]|+1}$.
\end{proposition}
\begin{proof}
We need the following
\begin{lemma}\label{lem:conj}
Let $i,j,k\in [1,n+1]$ with $i<j-1$ and $j<k$. Then
in~$\Br_n$ we have
$$
T_{(j-1,k)}T_{(i,j-1)}=T_{j-1} T_{(j,k)}T_{(i,j)}T_{j-1}^{-1}.
$$
and
$$
T_{(i,j-1)}T_{(j-1,k)}=T_{j-1}^{-1} T_{(i,j)}T_{(j,k)}T_{j-1}.
$$
\end{lemma}
\begin{proof}
Using Proposition~\partref{prop:elem prop transp.b}
we obtain
$$
T_{(j-1,k)}T_{(i,j-1)}
=T_{j-1} T_{(j,k)}T_{j-1} T_{(i,j-1)}
=T_{j-1} T_{(j,k)} T_{(i,j)}T_{j-1}^{-1}.
$$
The second identity follows from the first by applying~${}^{op}$.
\end{proof}
Write $J=\{j_1,\dots,j_m\}$ where~$j_1<\cdots<j_m$ and
$j_k=j$. In particular, $|J\cap [1,j]|=k$.
Since~$\min J<j<\max J$, $2\le k\le m-1$ and so $j_{k-1}\le j-2$
as $j-1\notin J$.
Let~$J'=(J\setminus \{j\})\cup \{j-1\}$.
Suppose first that~$k$ is odd and so~$\epsilon=(-1)^{k+1}=1$. Then
$T_J=X T_{(j,j_{k+1})} T_{(j_{k-1},j)} X'$ and
$T_{J'}=X T_{(j-1,j_{k+1})} T_{(j_{k-1},j-1)} X'$ where
$$
X=\prod_{\substack{t\in[1,m]\setminus\{k\}\\\bar t=1}}
T_{(j_t,j_{t+1})},\qquad
X'=\prod_{\substack{t\in[1,m]\setminus\{k-1\}\\ \bar t=0}}
T_{(j_t,j_{t+1})}.
$$
By Proposition~\partref{prop:elem prop transp.e}, $T_{j-1}$ commutes with~$X$ and~$X'$. Then
by Lemma~\ref{lem:conj}
\begin{align*}
T_{J'}&=X T_{(j-1,j_{k+1})} T_{(j_{k-1},j-1)} X'\\
&=X T_{j-1}T_{(j,j_{k+1})} T_{(j_{k-1},j)}T_{j-1}^{-1} X'\\
&=T_{j-1}XT_{(j,j_{k+1})} T_{(j_{k-1},j)}X'T_{j-1}^{-1}=
T_{j-1} T_J T_{j-1}^{-1}.
\end{align*}
Similarly, if $k$ is even, $T_J=Y T_{(j_{k-1},j)}
T_{(j,j_{k+1})}Y'$ and~$T_{J'}=Y T_{(j_{k-1},j-1)}
T_{(j-1,j_{k+1})}Y'$ where
$$
Y=\prod_{\substack{t\in[1,m]\setminus\{k-1\}\\t\equiv 1\pmod 2}}
T_{(j_t,j_{t+1})},\qquad
Y'=\prod_{\substack{t\in[1,m]\setminus\{k\}\\ t\equiv 0\pmod 2}}
T_{(j_t,j_{t+1})}.
$$
In particular, $T_{j-1}$ commutes with~$Y$ and~$Y'$ and~$\epsilon=(-1)^{k+1}=-1$.
Using Lemma~\ref{lem:conj} we obtain
\begin{align*}
T_{J'}&=Y T_{(j_{k-1},j-1)} T_{(j-1,j_{k+1})} X'\\
&=YT_{j-1}^{-1} T_{(j_{k-1},j)} T_{(j,j_{k+1})}T_{j-1} Y'\\
&=T_{j-1}^{-1}YT_{(j_{k-1},j)} T_{(j,j_{k+1})}Y'T_J T_{j-1}.\qedhere
\end{align*}
\end{proof}

Given~$J\subset[1,n+1]$, denote \plink{gJ}$g(J)=|\{ j\in J\,:\,\min J<j<\max J,
\,j-1\notin J\}|$. For example, if $J=[1,n+1]$ then~$g(J)=0$
and $g([1,a]\cup[b+1,n+1])=1$ for all~$1\le a<b\le n$. Denote
$$
\plink{Cij(a)}\Cx ij^{(a)}=\ascprod_{i\le k\le j}T^a_k,\qquad \Cxr ij^{(a)}=\dscprod_{i\le k\le j}T^a_k,\qquad a\in\mathbb Z.
$$
\begin{corollary}\label{cor:conj J}
Let $J=\{j_0,\dots,j_{m+1}\}\subset [1,n+1]$ with $1=j_0<\cdots<j_{m+1}=n+1$. Then
$$
U(J) T_J U(J)^{-1} = T_{[1,m+1]\cup \{n+1\}}
$$
where\plink{U(J)}
\begin{equation}\label{eq:U(J) defn}
U(J)=\dscprod_{k\in[1,m]} \Cx{(k+1)}{(j_k-1)}^{((-1)^k)}.
\end{equation}
In particular, if $J,J'\subset [1,n+1]$ satisfy $|J|=|J'|$, $\min J=\min J'$ and~$\max J=\max J'$ then
$T_J$ and~$T_{J'}$ are conjugate in~$\Br_n$; thus,
all $\{1,n+1\}\subset J,J'\subset [1,n+1]$ with $|J|=|J'|$ are conjugate in~$\Br_n$.
\end{corollary}
\begin{proof}
Abbreviate~$J_m=[1,m+1]\cup\{n+1\}$ and $U_k(J)=\Cx{(k+1)}{(j_k-1)}^{((-1)^k)}$.
The argument is by induction on $g(J)$. Note that
$g(J)=|\{k\in[1,m]\,:\, j_{k-1}<j_k-1\}|$.
If $g(J)=0$ then $J=J_{n}$ and~$U(J)=1$.

For the inductive step, let $k>0$ be minimal such that $j_{k-1}<
j_k-1$. Then $j_s=s+1$ for all $0\le s<k$ and so
$U_s(J)=1$ for all~$1\le s<k$. By Proposition~\ref{prop:conj}, $U_k(J) T_J U_k(J)^{-1}
=T_{J'}$ where~$J'=J\setminus \{j_k\}\cup \{k+1\}$. Since
$g(J')=g(J)-1$, $U(J')T_{J'}U(J')^{-1}=T_{J_m}$ by the induction
hypothesis. It remains to observe that
$U_s(J')=1$ for all~$1\le s\le k$ and $U_s(J)=U_s(J')$ for all~$k+1\le s\le m$, whence $U(J')U_k(J)=U(J)$.
\end{proof}
\begin{corollary}\label{cor:TJ coxeter}
Let $J\subset [1,n+1]$. Then~$\pi_n(T_J)\in S_{n+1}$ is a cycle of length~$|J|$ and, in particular,
has order~$|J|$. Moreover, if $T_J^N$ is ${}^{op}$-invariant then
$|J|$ divides~$2N$.
\end{corollary}
\begin{proof}
By Corollary~\ref{cor:conj J} it suffices to prove the first assertion for~$J=J_m=[1,m+1]\cup\{n+1\}$, $1\le m\le n-2$.
It is easy to check that
$$
\pi_n(T_{J_m})=\begin{cases}
(1,2,4,\dots,m-1,m+1,n+1,m,m-2,\dots,3),&\text{$m$ is odd}\\
(1,2,4,\dots,m-2,m,n+1,m+1,m-1,\dots,3),&\text{$m$ is even}.
\end{cases}
$$
In either case, $\pi_n(T_{J_m})$ is a cycle of length~$m+2=|J_m|$.
By Lemma~\ref{lem:can image op inv}, if~$T_J^N$ is ${}^{op}$-invariant then
$\pi_n(T_J^N)=\pi_n(T_J)^N$ is an involution. Thus, $\pi_n(T_J)^{2N}=1$ and so the order of~$\pi_n(T_J)$ divides~$2N$.
\end{proof}

\subsection{Forward direction}\label{subs:forward direction}
Our present aim is to establish the forward direction of Theorem~\ref{thm:main thm adm}.
\begin{theorem}\label{thm:adm I2m}
For any~$\{1,n+1\}\subset J\subset [1,n+1]$ with~$g(J)=1$,
the assignments $\wh T_r\mapsto \tau_{\overline r}(J)$, $r\in\{1,2\}$
define an optimal fully supported disjoint standard homomorphism
$\Br^+(I_2(2m))\to \Br^+_{n+1}$  where
$$
m=m(J)=\begin{cases}
|J|/2,& J=\tilde\sigma(J),\\
|J|,& \text{otherwise}.
\end{cases}
$$
\end{theorem}
\begin{proof} Since~$g(J)=1$, we can write $J=[1,a]\cup [b+1,n+1]$ where
$1\le a<b\le n-1$.

Suppose first that $J=\tilde\sigma(J)=\{ n+2-j\,:\,j\in J\}$. This forces $b=n+1-a$.
Applying one of the unfolding homomorphisms~\eqref{eq:unfold Bn A2n-1} or~\eqref{eq:unfold Bn A2n}, depending on the parity of~$n$, to~\eqref{eq:prep I2m Bn}
we obtain in~$\Br_{n+1}$
$$
(T_{w_\circ^{[1,a-1]_2\cup \sigma([1,a-1]_2)}} T_{w_\circ^{[1,a-2]_2\cup
\sigma([1,a-2]_2)}}T_{(a,n+2-a)})^{a} = T_{w_\circ^{[1,n]}} T_{w_\circ^{[a+1,n-a]}}^{-1}
$$
with~$\sigma(i)=n+1-i$, $1\le i\le n$,
which immediately yields the following
\begin{corollary}\label{cor:symm even J}
Let~$\tilde J_m=[1,m]\cup \tilde\sigma([1,m])$ with $2m<n+1$. Then~$|\tilde J_m|=2m$ and
$T_{\tilde J_m}^m=T_{w_\circ^{[1,n]}} T_{w_\circ^{[m+1,n-m]}}^{-1}$ in~$\Br_{n+1}$.
\end{corollary}
Since $T_{w_\circ^{[m+1,n-m]}}$ commutes with~$T_{w_\circ^{[1,n]}}$
and both are~${}^{op}$-invariant by Proposition~\ref{prop:fund elts BrSa}
\ref{prop:fund elts BrSa.a}, by Lemma~\ref{lem:TJ hom}
this establishes the first case in Theorem~\ref{thm:adm I2m}.

To establish the second case, we need to prove the following
\begin{proposition}\label{prop:g J=1}
Let $J=[1,a]\cup[b+1,n+1]$, $1\le a<b\le n$. Then~$T_J^{|J|}=T_{w_\circ^{[1,n]}}^2 T_{w_\circ^{[a+1,b-1]}}^{-2}$ in~$\Br_n$.
\end{proposition}

The following Lemma allows one to reduce Proposition~\ref{prop:g J=1}
to a specially chosen~$J$.
\begin{lemma}\label{lem:move bubble}
Let  $[a,b],[a',b']\subset [1,n+1]\setminus\{1,n+1\}$ with~$b-a=b'-a'>1$.
Let~$J=[1,n+1]\setminus[a,b]$, $J'=[1,n+1]\setminus [a',b']$. Then
$T_J^{N}=T_{w_\circ^{[1,n]}}^2
T_{w_\circ^{[a,b-1]}}^{-2}$ for some~$N\ge 1$ if and only if
$T_{J'}^{N}=T_{w_\circ^{[1,n]}}^2 T_{w_\circ^{[a',b'-1]}}^{-2}$.
\end{lemma}
\begin{proof}
It suffices to prove the Lemma when $a'=a+1$.  We have, by Proposition~\ref{prop:conj},
$$
T_{J'}=\begin{cases}
\Cx ab T_J \Cx ab^{-1},&\text{$a$ is even}\\
\Cxr ab^{-1} T_J \Cxr ab,&\text{$a$ is odd}.
\end{cases}
$$
Since
$$
\Cx ab T_{w_\circ^{[a,b-1]}}^2=T_{w_\circ^{[a,b]}} T_{w_\circ^{[a,b-1]}}
=T_{w_\circ^{[a+1,b]}} T_{w_\circ^{[a,b]}}
=T_{w_\circ^{[a+1,b]}}^2 \Cx ab,
$$
it follows that
$$
\Cx ab T_{w_\circ^{[a,b-1]}}^{-2} \Cx ab^{-1}=T_{w_\circ^{[a+1,b]}}^{-2}
=\Cxr ab^{-1} T_{w_\circ^{[a,b-1]}}^{-2} \Cxr ab,
$$
where the second equality is obtained from the first by applying~${}^{op}$.
Since~$T_{w_\circ^{[1,n]}}^2$ is central in~$\Br^+_{n+1}$ by Proposition~\partref{prop:fund elts BrSa.d}, the assertion
is now immediate.
\end{proof}

Suppose first that~$|J|$ is even with~$|J|=2m$. By Corollary~\ref{cor:symm even J}, $T_{\tilde J_m}^{2m}=T_{w_\circ^{[1,n]}}^2 T_{w_\circ^{[m+1,n-1]}}^{-2}$. Thus, Proposition~\ref{prop:g J=1} for~$|J|$ even is proved.

We will now obtain a convenient expression for
$T_{w_\circ^{[1,n]}}^2 T_{w_\circ^{[r+1,n-1]}}^{-2}$ for
$1\le r\le n$.
For that we need more properties of the~$T_{(a,b+1)}$, $a\le b\in[1,n]$.
\begin{lemma}\label{lem:transp 1}
For any $2\le a\le b$,
$
T_{(a,b+1)}T_{(a+1,b+1)}T_{a-1}T_a T_{(a+1,b+1)}
$
is ${}^{op}$-invariant.
\end{lemma}
\begin{remark}
It is easy to check that for~$b>a$ the canonical image of the above element of~$\Br^+_{n+1}$ in~$W(A_n)$ is
$(a-1,b+1)$ which is of length~$2(b-a)+3$
while the original element of~$\Br^+_{n+1}$
has length~$6(b-a)+1$. 
\end{remark}
\begin{proof}
Let $X=T_{(a,b+1)}T_{(a+1,b+1)}T_{a-1}T_a T_{(a+1,b+1)}$.
For~$b=a$ we have $X=T_a T_{a-1} T_a$ which is manifestly
${}^{op}$-invariant. Thus, we may assume, without loss of generality, that $b>a$.

For simplicity, let $a=2$. Thus, we claim that
$$
T_{(2,b+1)}T_{(3,b+1)}T_1 T_2 T_{(3,b+1)}=T_{(3,b+1)}
T_2 T_1 T_{(3,b+1)}T_{(2,b+1)}.
$$
Since $T_{(3,b+1)}=\Cx 3{(b-1)}T_b \Cxr3{(b-1)}$,
and the $T_j$ with $i<j<k-1$ commute with $T_{(i,k)}$ by Proposition~\partref{prop:elem prop transp.e}, we have
$$
X=\Cx3{(b-1)}T_{(2,b+1)}\Cxr3b \Cx1b\Cxr3{(b-1)}
$$
while
$$
X^{op}=\Cx 3{(b-1)} \Cxr1b \Cx3b T_{(2,b+1)}\Cxr3{(b-1)}.
$$
Since~$\Br^+_{n+1}$ is cancellative, it suffices to prove that
$$
T_{(2,b+1)}\Cxr3b \Cx1b=
\Cxr1b \Cx3b T_{(2,b+1)}.
$$
Since $T_{(2,b+1)}=\Cxr2b\Cx3{b}$ and $\Cxr1b=\Cxr2b T_1$, the above equality
follows once we establish that
\begin{equation}\label{eq:interm transp 1}
\Cx3b\Cxr3b \Cx1b=T_1\Cx3b T_{(2,b+1)}.
\end{equation}
But the left hand side of~\eqref{eq:interm transp 1} is equal to
\begin{equation*}
\Cx3b\Cxr3b \Cx1b=\Cx3b\Cxr3b T_1\Cx2b=T_1\Cx3b\Cxr3b\Cx2b
=T_1 \Cx3bT_{(2,b+1)}.\qedhere
\end{equation*}
\end{proof}
\begin{lemma}\label{lem:transp 2}
For all $2\le a<b\le n$ we have
$$
T_{(a,b+1)} T_{a-1} T_{(a,b+1)}T_{(a+1,b+1)}T_a= T_{(a-1,b+1)}T_{(a,b+1)}T_{(a+1,b+1)}
$$
\end{lemma}
\begin{proof}
Since
$$
T_{(a,b+1)} T_{a-1}=\Cxr{a}b \Cx{(a+1)}b T_{a-1}=
\Cxr{(a-1)}b \Cx{(a+1)}{b}.
$$
and~$\Br^+_{n+1}$ is cancellative, the assertion is equivalent to
$$
\Cx{(a+1)}b T_{(a,b+1)}T_{(a+1,b+1)}T_a=
\Cx{a}b T_{(a,b+1)}T_{(a+1,b+1)}.
$$
Now,
\begin{align*}
\Cx{a}{b}T_{(a,b+1)}=\Cx{a} b T_a T_{(a+1,b+1)} T_a 
&=T_a T_{a+1} T_a \Cx{(a+2)}b T_{(a+1,b+1)} T_a\\
&=T_{a+1}\Cx ab T_{(a+1,b+1)}T_a.
\end{align*}
Suppose we proved that
\begin{equation}\label{eq:eq transp 2}
\Cx{a}{b}T_{(a,b+1)}=\Cx{(a+1)}k \Cx ab T_{(k,b+1)}\Cxr{a}{(k-1)}.
\end{equation}
for some~$k> a$ (the case~$k=a+1$ was established above).
Then
\begin{align*}
\Cx{a}{b}T_{(a,b+1)}&=\Cx{(a+1)}k \Cx ab T_k T_{(k+1,b+1)}\Cxr{a}{k}\\
&=\Cx{(a+1)}{k} \Cx a{(k-1)}(T_k T_{k+1} T_k)\Cx{(k+2)}b T_{(k+1,b+1)}\Cxr{a}{k}\\
&=\Cx{(a+1)}{k} \Cx a{(k-1)}(T_{k+1} T_{k} T_{k+1})
\Cx{(k+2)}bT_{(k+1,b+1)}\Cxr{a}{k}\\
&=\Cx{(a+1)}{(k+1)} \Cx abT_{(k+1,b+1)}\Cxr{a}{k}.
\end{align*}
Thus, \eqref{eq:eq transp 2} holds for all~$k\ge a+1$. In particular,
for~$k=b$ we obtain
$$
\Cx{a}{b}T_{(a,b+1)}=\Cx{(a+1)}{b} \Cx abT_{(b,b+1)}\Cxr{a}{b-1}
=\Cx{(a+1)}{b} \Cx ab\Cxr{a}{b}.
$$
Therefore, it suffices to prove that
$$
T_{(a,b+1)}T_{(a+1,b+1)}T_a=\Cx ab\Cxr{a}{b}T_{(a+1,b+1)}
$$
which, since $T_{(a,b+1)}=\Cx ab\Cxr a{(b-1)}$, is equivalent to
$$
\Cxr{a}{(b-1)}T_{(a+1,b+1)}T_a=\Cxr{a}{b}T_{(a+1,b+1)}.
$$
By Proposition~\ref{prop:elem prop transp} we obtain
\begin{align*}
\Cxr{a}{(b-1)}&T_{(a+1,b+1)}T_a=\Cxr{(a+1)}{(b-1)}T_a T_{(a+1,b+1)}T_a=\Cxr{(a+1)}{(b-1)}T_{(a,b+1)}\\
&=T_{(a,b+1)}\Cxr{(a+1)}{(b-1)}
=\Cxr ab\Cx{(a+1)}b \Cxr{(a+1)}{(b-1)}=\Cxr ab T_{(a+1,b+1)}.\qedhere
\end{align*}
\end{proof}

\begin{lemma}\label{lem:cox transp exchage}We have for all $r\in[1,n-1]$
\begin{equation}\label{eq:cox transp exchange}
\Big(\ascprod_{i\in[1,r+1]} T_{(i,n+1)}\Big) \Cx1{r}=\Cx1r
T_{(r+1,n+1)}\Big(\ascprod_{i\in[1,r]} T_{(i,n+1)}\Big)
\end{equation}
\end{lemma}
\begin{proof}
Suppose we proved that for some~$i\in[1,r]$
\begin{equation}\label{eq:cox transp exchange'}
\ascprod_{j\in[1,r+1]} T_{(j,n+1)}\Cx1r=
\Cx1{(i-1)} T_{(i,n+1)}\ascprod_{j\in[1,r+1]\setminus\{i\}} T_{(j,n+1)}\Cx ir
\end{equation}
(for~$i=1$ this is trivial). Then
\begin{align*}
\ascprod_{j\in[1,r+1]} T_{(j,n+1)}\Cx1r&=\Cx1{i} T_{(i+1,n+1)}T_i\ascprod_{j\in[1,r+1]\setminus\{i\}} T_{(j,n+1)}\Cx ir\\
&=\Cx1{i} T_{(i+1,n+1)}\ascprod_{j\in[1,i-1]}T_{(j,n+1)} T_i T_{(i+1,n+1)}T_i \ascprod_{j\in[i+2,r+1]} T_{(j,n+1)}\Cx{(i+1)} r\\
&=\Cx1{i} T_{(i+1,n+1)}\ascprod_{j\in[1,r+1]\setminus \{i+1\}} T_{(j,n+1)} \Cx{(i+1)} r.
\end{align*}
Thus, \eqref{eq:cox transp exchange'} holds for all~$i\in[1,r+1]$. But for~$i=r+1$ this is precisely~\eqref{eq:cox transp exchange}.
\end{proof}
Denote
\begin{equation}\label{eq:Z_r defn}
Z_r:=\Big(\ascprod_{i\in[1,r+1]} T_{(i,n+1)}\Big) \Cx1{r}T_{(r+1,n+1)}=\Cx1r
T_{(r+1,n+1)}\Big(\ascprod_{i\in[1,r+1]} T_{(i,n+1)}\Big),
\end{equation}
the equality following from~\eqref{eq:cox transp exchange}.
\begin{proposition}\label{prop:Z_r is central}
For all $r\in[1,n-2]$ we have
$
Z_r T_{w_\circ^{[r+2,n-1]}}^2=T_{w_\circ^{[1,n]}}^2$.
\end{proposition}
\begin{proof}
We have
$$
\ell(Z_r)=\sum_{1\le i\le r+1}(2(n-i)+1)+2n-r-1=(r+2)(2n-r-1).
$$
and so
$$
\ell(Z_r T_{w_\circ^{[r+2,n-1]}}^2)=(n-r-1)(n-r-2)+(r+2)(2n-r-1)=n(n+1)
=\ell(T_{w_\circ^{[1,n]}}^2).
$$
By Proposition~\partref{prop:fund elts BrSa.d},
$T_{w_\circ^{[1,n]}}^2$ generates the center of~$\Br^+_{n+1}$,
and so
it suffices to prove that $Z_r T_{w_\circ^{[r+2,n-1]}}^2$
is central in~$\Br^+_{n+1}$.
\begin{lemma}\label{lem:comm I}
We have $T_i Z_r=Z_r T_i$ for all $i\in [1,n]\setminus \{r+1,n\}$.
\end{lemma}
\begin{proof}
By Proposition~\partref{prop:elem prop transp.c}, the $T_j$, $j\in [r+2,n-1]$ commute with the $T_{(i,n+1)}$,
$i\in[1,r+1]$. It follows that
$T_i Z_r=Z_r T_i$ for all~$i\in [r+2,n-1]$.

Since~$\Br^+_{n+1}$ embeds into~$\Br_{n+1}$ (cf.~\cites{BrSa,Del,Par}),
it suffices to prove that~$T_i^{-1} Z_r T_i=Z_r$ for all~$i\in [1,r]$.

Let~$i\in [1,r-1]$. By Proposition~\ref{prop:elem prop transp}, $T_i^{-1}$ commutes with
the $T_{(j,n+1)}$ for all $1\le j\le i-1$ and $T_i^{-1} T_{(i,n+1)}=T_{(i+1,n+1)}T_i$. Therefore, \begin{multline*}
T_i^{-1} Z_r T_i = \Big(\ascprod_{j\in[1,i-1]} T_{(j,n+1)} \Big) T_{(i+1,n+1)}T_i T_{(i+1,n+1)}T_{(i+2,n+1)}\Big(\ascprod_{j\in[i+3,r+1]} T_{(j,n+1)}\Big)\Cx1r T_i T_{(r+1,n+1)}.
\end{multline*}
Since $\Cx1r T_i=T_{i+1}\Cx1r$ by Lemma~\ref{lem:comm cox}, we obtain
\begin{align*}
T_i^{-1} &Z_r T_i = \Big(\ascprod_{j\in[1,i-1]} T_{(j,n+1)} \Big) T_{(i+1,n+1)}T_i T_{(i+1,n+1)}T_{(i+2,n+1)}T_{i+1}\Big(\ascprod_{j\in[i+3,r+1]} T_{(j,n+1)}\Big)\Cx1r T_{(r+1,n+1)}\\
&= \Big(\ascprod_{j\in[1,i-1]} T_{(j,n+1)} \Big) T_{(i,n+1)}T_{(i+1,n+1)}T_{(i+2,n+1)}
\Big(\ascprod_{j\in[i+3,r+1]} T_{(j,n+1)}\Big)\Cx1r T_{(r+1,n+1)}\\
&=Z_r,
\end{align*}
where we used Lemma~\ref{lem:transp 2}.

Let~$i=r$. Then
\begin{align*}
T_r^{-1} Z_r T_r&=\ascprod_{i\in[1,r-1]} T_{(i,n+1)} T_{(r+1,n+1)}T_r T_{(r+1,n+1)} \Cx1r T_{(r+1,n+1)} T_r\\
&=\ascprod_{i\in[1,r-1]}T_{(i,n+1)} T_{(r+1,n+1)}T_r T_{(r+1,n+1)} \Cx1{(r-1)} T_{(r,n+1)}\\
&=\ascprod_{i\in[1,r-1]}T_{(i,n+1)} \Cx1{(r-2)} T_{(r+1,n+1)}T_r T_{r-1} T_{(r+1,n+1)} T_{(r,n+1)}\\
&=\ascprod_{i\in[1,r-1]}T_{(i,n+1)} \Cx1{(r-2)} T_{(r,n+1)} T_{(r+1,n+1)}T_{r-1} T_r T_{(r+1,n+1)}\\
&=\ascprod_{i\in[1,r+1]}T_{(i,n+1)} \Cx1r T_{(r+1,n+1)}=Z_r,
\end{align*}
by Lemma~\ref{lem:transp 1}.
\end{proof}
\begin{lemma}\label{lem:comm II}
We have $T_i Z_r T_{(r+2,n)}^2=Z_r T_{(r+2,n)}^2 T_i$ for $i\in\{r+1,n\}$.
\end{lemma}
\begin{proof}
Since $T_{r+1}$ and~$T_n$ commute with $T_{(a,b)}$ for all $r+2<a<b<n-1$, $T_{w_\circ}^{[r+1,n]}=\prod\limits_{1\le i\le \frac12(n-r)+1} \!\!\! T_{(r+i,n+2-i)}$ by Proposition~\ref{prop:elem prop transp} while $T_j T_{w_\circ^{[r+1,n]}}=T_{w_\circ^{[r+1,n]}}
T_{n+r+1-j}$ for all $j\in [r+1,n]$,
we conclude that
$$
T_{r+1} T_{(r+1,n+1)}T_{(r+2,n)}=T_{(r+1,n+1)}T_{(r+2,n)} T_n
$$
and
$$
T_n T_{(r+1,n+1)}T_{(r+2,n)}=T_{(r+1,n+1)}T_{(r+2,n)} T_{r+1}.
$$
Furthermore,
note that by Proposition~\partref{prop:elem prop transp.f}, $T_{(r+2,n)}$ commutes with all factors in
the formula defining~$Z_r$.

We have
\begin{align*}
T_{r+1} Z_r T_{(r+2,n)}^2 &=\ascprod_{i\in[1,r]} T_{(i,n+1)} T_{r+1} T_{(r+1,n+1)}T_{(r+2,n)} \Cx1r  T_{(r+1,n+1)}T_{(r+2,n)}\\
&=\ascprod_{i\in[1,r]} T_{(i,n+1)} T_{(r+1,n+1)}T_{(r+2,n)} T_n\Cx1r  T_{(r+1,n+1)}T_{(r+2,n)}\\
&=\ascprod_{i\in[1,r]} T_{(i,n+1)} T_{(r+1,n+1)}T_{(r+2,n)} \Cx1r T_n  T_{(r+1,n+1)}T_{(r+2,n)}\\
&=Z_r T_{(r+2,n)}^2 T_{r+1},
\\
\intertext{and}
T_{n} Z_r T_{(r+2,n)}^2 &=\Cx1r T_n T_{(r+1,n+1)}T_{(r+2,n)}\ascprod_{i\in[1,r]} T_{(i,n+1)} T_{(r+1,n+1)}T_{(r+2,n)} \\
&=\Cx1r T_{(r+1,n+1)}T_{(r+2,n)}T_{r+1}\ascprod_{i\in[1,r]} T_{(i,n+1)} T_{(r+1,n+1)}T_{(r+2,n)} \\
&=\Cx1r T_{(r+1,n+1)}T_{(r+2,n)}\ascprod_{i\in[1,r]}
T_{(i,n+1)} T_{r+1} T_{(r+1,n+1)}T_{(r+2,n)}\\
&=\Cx1r T_{(r+1,n+1)}T_{(r+2,n)}\ascprod_{i\in[1,r]}
T_{(i,n+1)} T_{(r+1,n+1)}T_{(m+2,n)}T_n\\
&=Z_r T_{(r+2,n)}^2 T_n.\qedhere
\end{align*}
\end{proof}
Together, Lemmata~\ref{lem:comm I} and~\ref{lem:comm II} imply that
$Z_r T_{w_\circ^{[r+2,n-1]}}^2$ is central
in~$\Br^+_{n}$, which completes the proof of Proposition~\ref{prop:Z_r is central}.
\end{proof}
\begin{proposition}\label{prop:TJm power central}
For all~$m\in [1,n-2]$,
$T_{[1,m+1]\cup\{n+1\}}^{m+2}=Z_m$.
\end{proposition}
\begin{proof}
The assertion follows immediately from Proposition~\ref{prop:Z_r is central} for $m$ even. Throughout the rest of this proof, we assume that~$m$ is odd. Let~$J_m=[1,m+1]\cup\{n+1\}$.
\begin{lemma}\label{lem:1st identity}
We have
$$
T_{J_m}^{m+1}=\dscprod_{j\in [1,m]_2} T_{(j,n+1)}\ascprod_{j\in [2,m+1]_2} T_{(j,n+1)}.
$$
\end{lemma}
\begin{proof}
First we prove by descending induction that for all $k\in [1,m]_2$,
\begin{equation}\label{eq:1st rewrite}
\begin{split}
T_{J_m}^{m+1}=\dscprod_{i\in [k,m]_2} &(\Cx im T_{(m+1,n+1)}) \Big(\ascprod_{j\in[k,m]_2} T_{w_\circ^{[1,j-2]_2}} T_{w_\circ^{[2,j-1]_2}}\Big)\\
&(T_{w_\circ^{[1,m]_2}}T_{w_\circ^{[2,m-1]_2}}T_{(m+1,n+1)})^{\frac{m+k}2}.
\end{split}
\end{equation}
Indeed, since the $T_i$, $i\in [1,m-1]$ commute with $T_{(m+1,n+1)}$
by Proposition~\partref{prop:elem prop transp.e}, we have
$$
T_{J_m}^{m+1}=T_m T_{(m+1,n+1)} (T_{w_\circ^{[1,m-2]_2}}T_{w_\circ^{[2,m-1]_2}}) (T_{w_\circ^{[1,m]_2}}T_{w_\circ^{[2,m-1]_2}}T_{(m+1,n+1)})^{m},
$$
which is~\eqref{eq:1st rewrite} with~$k=m$.

For the inductive step we have
\begin{align*}
T_{J_m}^{m+1}&=\dscprod_{i\in [k,m]_2} (\Cx im T_{(m+1,n+1)}) \Big(\ascprod_{j\in[k,m]_2} T_{w_\circ^{[1,j-2]_2}} T_{w_\circ^{[2,j-1]_2}}\Big)\\
&\mskip200mu(T_{w_\circ^{[1,m]_2}}T_{w_\circ^{[2,m-1]_2}} T_{(m+1,n+1)})^{\frac{m+k}2} \\
&=\dscprod_{i\in [k,m]_2} (\Cx im T_{(m+1,n+1)}) \Big(\ascprod_{j\in[k,m]_2} T_{w_\circ^{[1,j-4]_2}} T_{w_\circ^{[2,j-3]_2}} T_{j-2} T_{j-1} \Big)
T_m
T_{(m+1,n+1)}\\&\mskip200muT_{w_\circ^{[1,m-2]_2}}T_{w_\circ^{[2,m-1]_2}}(T_{w_\circ^{[1,m]_2}}T_{w_\circ^{[2,m-1]_2}} T_{(m+1,n+1)})^{\frac{m+k}2-1} \\
&=\dscprod_{i\in [k,m]_2} (\Cx im T_{(m+1,n+1)}) \Cx{(k-2)}mT_{(m+1,n+1)}\Big(\ascprod_{j\in[k-2,m-2]_2} T_{w_\circ^{[1,j-2]_2}} T_{w_\circ^{[2,j-1]_2}}\Big)
\\&\mskip200muT_{w_\circ^{[1,m-2]_2}}T_{w_\circ^{[2,m-1]_2}} (T_{w_\circ^{[1,m]_2}}T_{w_\circ^{[2,m-1]_2}} T_{(m+1,n+1)})^{\frac{m+k}2-1} \\
&=\dscprod_{i\in [k-2,m]_2} (\Cx im T_{(m+1,n+1)})\Big(\ascprod_{j\in[k-2,m]_2} T_{w_\circ^{[1,j-2]_2}} T_{w_\circ^{[2,j-1]_2}}\Big) \\
&\mskip200mu(T_{w_\circ^{[1,m]_2}}T_{w_\circ^{[2,m-1]_2}} T_{(m+1,n+1)})^{\frac{m+k-2}2}.
\end{align*}
This proves~\eqref{eq:1st rewrite}. Taking~$k=1$ yields
\begin{equation}\label{eq:1st rewrite final}
\begin{split}
T_{J_m}^{m+1}=\dscprod_{i\in [1,m]_2} (\Cx im T_{(m+1,n+1)}) &\Big(\ascprod_{j\in[1,m-2]_2} T_{w_\circ^{[1,j]_2}} T_{w_\circ^{[2,j+1]_2}}\Big)\\
&(T_{w_\circ^{[1,m]_2}}T_{w_\circ^{[2,m-1]_2}} T_m T_{(m+1,n+1)})^{\frac{m+1}2}.
\end{split}
\end{equation}
The next step is to show that for all $k\in[1,m]_2$,
\begin{equation}\label{eq:2nd rewrite}\begin{split}
(T_{w_\circ^{[1,m]_2}}T_{w_\circ^{[2,m-1]_2}} T_m T_{(m+1,n+1)})^{\frac{m+1}2}=
\Big(\dscprod_{j\in[k,m]_2} T_{w_\circ^{[1,j]_2}}T_{w_\circ^{[2,j-1]_2}}\Big) T_{(m+1,n+1)}\\\ascprod_{j\in[k+1,m-1]_2}(\Cxr jm T_{(m+1,n+1)})
(T_{w_\circ^{[1,m]_2}}T_{w_\circ^{[2,m-1]_2}} T_{(m+1,n+1)})^{\frac{k-1}2}.
\end{split}
\end{equation}
Again, we use descending induction on~$k$, the case~$k=m$ being trivial. For the inductive step, we have
\begin{align*}
(&T_{w_\circ^{[1,m]_2}}T_{w_\circ^{[2,m-1]_2}} T_m T_{(m+1,n+1)})^{\frac{m+1}2}=
\Big(\dscprod_{j\in[k,m]_2} T_{w_\circ^{[1,j]_2}}T_{w_\circ^{[2,j-1]_2}}\Big) T_{(m+1,n+1)}\\&\mskip50mu\ascprod_{j\in[k+1,m-1]_2}(\Cxr jm T_{(m+1,n+1)})
(T_{w_\circ^{[1,m]_2}}T_{w_\circ^{[2,m-1]_2}} T_{(m+1,n+1)})^{\frac{k-1}2}\\
&=
\Big(\dscprod_{j\in[k,m]_2} T_{w_\circ^{[1,j]_2}}T_{w_\circ^{[2,j-1]_2}}\Big) T_{(m+1,n+1)}\ascprod_{j\in[k+1,m-1]_2}(\Cxr jm T_{(m+1,n+1)})
 \\
&\mskip50muT_{w_\circ^{[1,k-2]_2}} T_{w_\circ^{[2,k-3]_2}}T_{w_\circ^{[k,m]_2}} T_{w_\circ^{[k-1,m]_2}} T_{(m+1,n+1)}
(T_{w_\circ^{[1,m]_2}}T_{w_\circ^{[2,m-1]_2}} T_{(m+1,n+1)})^{\frac{k-3}2}\\
&=
\Big(\dscprod_{j\in[k-2,m]_2} T_{w_\circ^{[1,j]_2}}T_{w_\circ^{[2,j-1]_2}}\Big) T_{(m+1,n+1)}\ascprod_{j\in[k+1,m-1]_2}(\Cxr jm T_{(m+1,n+1)})\\
&\mskip50mu
\Big(\ascprod_{i\in [k,m]_2} T_i T_{i-1}\Big)
T_{(m+1,n+1)}
(T_{w_\circ^{[1,m]_2}}T_{w_\circ^{[2,m-1]_2}} T_{(m+1,n+1)})^{\frac{k-3}2}\\
&=
\Big(\dscprod_{j\in[k-2,m]_2} T_{w_\circ^{[1,j]_2}}T_{w_\circ^{[2,j-1]_2}}\Big) T_{(m+1,n+1)}\ascprod_{j\in[k+1,m-1]_2}(\Cxr{(j-2)}m T_{(m+1,n+1)})\\
&\mskip50mu
T_m T_{m-1} T_{(m+1,n+1)}
(T_{w_\circ^{[1,m]_2}}T_{w_\circ^{[2,m-1]_2}} T_{(m+1,n+1)})^{\frac{k-3}2}\\
&=
\Big(\dscprod_{j\in[k-2,m]_2} T_{w_\circ^{[1,j]_2}}T_{w_\circ^{[2,j-1]_2}}\Big) T_{(m+1,n+1)}\ascprod_{j\in[k-1,m-1]_2}(\Cxr{j}m T_{(m+1,n+1)})\\
&\mskip50mu
(T_{w_\circ^{[1,m]_2}}T_{w_\circ^{[2,m-1]_2}} T_{(m+1,n+1)})^{\frac{k-3}2}.
\end{align*}
In particular, for~$k=1$ we obtain
\begin{equation}\label{eq:2nd rewrite final}
\begin{split}
T_{J_m}^{m+1}&=\dscprod_{i\in [1,m]_2} (\Cx im T_{(m+1,n+1)}) \Big(\ascprod_{j\in[1,m-2]_2} T_{w_\circ^{[1,j]_2}} T_{w_\circ^{[2,j+1]_2}}\Big)\\
&\Big(\dscprod_{j\in[1,m]_2} T_{w_\circ^{[1,j]_2}}T_{w_\circ^{[2,j-1]_2}}\Big) T_{(m+1,n+1)}\ascprod_{j\in[2,m-1]_2}(\Cxr jm T_{(m+1,n+1)}).
\end{split}
\end{equation}
Next we claim that
$$
\Big(\ascprod_{j\in[1,m-2]_2} T_{w_\circ^{[1,j]_2}} T_{w_\circ^{[2,j+1]_2}}\Big)
\Big(\dscprod_{j\in[1,m]_2} T_{w_\circ^{[1,j]_2}}T_{w_\circ^{[2,j-1]_2}}\Big)=T_{w_\circ^{[1,m]}}.
$$
We use induction on odd~$m$, the case~$m=1$ being trivial. Note that, since
$$
T_{w_\circ^{[1,r]}}=T_{w_\circ^{[1,r-1]}}\Cxr1r
$$
and $T_{w_\circ^J}$, $J\subset[1,n]$, is ${}^{op}$-invariant by~\cite{BrSa}*{Lemma~5.1}, it follows that
\begin{equation}\label{eq:2-side w0}
T_{w_\circ^{[1,m]}}=T_{w_\circ^{[1,m-1]}}\Cxr1m=\Cx1{(m-1)}T_{w_\circ^{[1,m-2]}}\Cxr1m.
\end{equation}
Then
\begin{align*}
\Big(\ascprod_{j\in[1,m-2]_2} &T_{w_\circ^{[1,j]_2}} T_{w_\circ^{[2,j+1]_2}}\Big)
\Big(\dscprod_{j\in[1,m]_2} T_{w_\circ^{[1,j]_2}}T_{w_\circ^{[2,j-1]_2}}\Big)\\
&=\Big(\ascprod_{j\in[1,m-2]_2} T_{w_\circ^{[1,j-2]_2}} T_{j}T_{j+1} T_{w_\circ^{[2,j-1]_2}}\Big)
\Big(\dscprod_{j\in[1,m]_2} T_{w_\circ^{[1,j-2]_2}}T_j T_{j-1} T_{w_\circ^{[2,j-3]_2}}\Big)\\
&=\Cx1{(m-1)}\Big(\ascprod_{j\in[1,m-4]_2} T_{w_\circ^{[1,j]_2}} T_{w_\circ^{[2,j+1]_2}}\Big)
\Big(\dscprod_{j\in[1,m-2]_2} T_{w_\circ^{[1,j]_2}}T_{w_\circ^{[2,j-1]_2}}\Big)\Cxr1m\\
&=\Cx1{(m-1)}T_{w_\circ^{[1,m-2]}}\Cxr1m,
\end{align*}
where we used the induction hypothesis and the convention that~$T_0=1$. It remains to use~\eqref{eq:2-side w0}.
Thus,
$$
T_{J_m}^{m+1}=\dscprod_{i\in [1,m]_2} (\Cx im T_{(m+1,n+1)}) T_{w_\circ^{[1,m]}} T_{(m+1,n+1)}\ascprod_{j\in[2,m-1]_2}(\Cxr jm T_{(m+1,n+1)}).
$$
Applying the diagram automorphism of the submonoid $\Br^+_{[1,m]}(A_n)\cong \Br^+_{m+1}$ of~$\Br^+_{n+1}$, we obtain from~\eqref{eq:2-side w0}
$$
T_{w_\circ^{[1,m]}}=\Cxr2{m}T_{w_\circ^{[3,m]}}\Cx1m.
$$
Since~$T_{w_\circ^{[1,m]}}$ is ${}^{op}$-invariant by Proposition~\partref{prop:fund elts BrSa.a}, this is also equal to
$
T_{w_\circ^{[1,m]}}=\Cxr1{m}T_{w_\circ^{[1,m-2]}}\Cx2m$.
A straightforward induction now yields
$$
T_{w_\circ^{[1,m]}}=\Big(\ascprod_{j\in[1,m]_2}\Cxr jm\Big)\Big(\dscprod_{j\in[2,m-1]_2}\Cx jm\Big).
$$
Thus,
\begin{equation}\label{eq:3rd rewrite}
\begin{split}
T_{J_m}^{m+1}&=\dscprod_{i\in [1,m]_2} (\Cx im T_{(m+1,n+1)}) \Big(\ascprod_{j\in[1,m]_2}
\Cxr jm\Big)\\&\qquad\Big(\dscprod_{j\in[2,m-1]_2}\Cx jm\Big)
T_{(m+1,n+1)} \ascprod_{j\in[2,m-1]_2}(\Cxr jm T_{(m+1,n+1)}).
\end{split}
\end{equation}
Suppose we proved that, for some $k\in[1,m-2]_2$,
\begin{align}\label{eq:3rd rewrite final}
T_{J_m}^{m+1}&=\dscprod_{i\in [k,m]_2} (\Cx im T_{(m+1,n+1)}) \dscprod_{i\in [1,k-2]_2} T_{(j,n+1)}
\Big(\ascprod_{j\in[k,m]_2}
\Cxr jm\Big)\\&\qquad\Big(\dscprod_{j\in[k+1,m-1]_2}\Cx jm\Big)
\ascprod_{j\in[2,k-1]_2} T_{(j,n+1)} T_{(m+1,n+1)} \ascprod_{j\in[k+1,m-1]_2}(\Cxr jm T_{(m+1,n+1)}),\nonumber
\end{align}
the case~$k=1$ being just~\eqref{eq:3rd rewrite}. Since $\Cx im T_{(m+1,n+1)}\Cxr im=T_{(i,n+1)}$, $1\le i\le m$ and the $\Cx jm$, $i<j\le m$ commute with $T_{(i,n+1)}$, $1\le i\le m+1$,
we obtain
\begin{align*}
T_{J_m}^{m+1}&=\dscprod_{i\in [k+2,m]_2} (\Cx im T_{(m+1,n+1)}) \Cx km T_{(m+1,n+1)}\Cxr km \dscprod_{i\in [1,k-2]_2} T_{(j,n+1)}
\ascprod_{j\in[k+2,m]_2}
\Cxr jm\\&\qquad\Big(\dscprod_{j\in[k+3,m-1]_2}\Cx jm\Big)
\ascprod_{j\in[2,k-1]_2} T_{(j,n+1)}
\Cx{(k+1)}m T_{(m+1,n+1)}\Cxr{(k+1)}m\\
&\mskip300mu\ascprod_{j\in[k+3,m-1]_2}(\Cxr jm T_{(m+1,n+1)})\\
&=\dscprod_{i\in [k+2,m]_2} (\Cx im T_{(m+1,n+1)}) \dscprod_{i\in [1,k]_2} T_{(j,n+1)}
\ascprod_{j\in[k+2,m]_2}
\Cxr jm\\&\qquad\Big(\dscprod_{j\in[k+3,m-1]_2}\Cx jm\Big)
\ascprod_{j\in[2,k+1]_2} T_{(j,n+1)} T_{(m+1,n+1)} \ascprod_{j\in[k+3,m-1]_2}(\Cxr jm T_{(m+1,n+1)}).
\end{align*}
Thus, the identity~\eqref{eq:3rd rewrite final} holds for all~$k\in[1,m]_2$.
Taking~$k=m$ yields the assertion.
\end{proof}
\begin{lemma}\label{lem:2nd identity}
We have
\begin{equation}\label{eq:2nd identity}
T_{J_m}^{m+1} T_{w_\circ^{[1,m]_2}} T_{w_\circ^{[2,m-1]_2}}=\Big(\ascprod_{i\in[1,m+1]} T_{(i,n+1)}\Big) \Cx1m\end{equation}
\end{lemma}
\begin{proof}
The argument is by induction on odd~$m$, the case~$m=1$ being immediate from Lemma~\ref{lem:1st identity}. For the inductive step, by Lemma~\ref{lem:1st identity} we are proving that
$$
\dscprod_{j\in [1,m]_2} T_{(j,n+1)}\ascprod_{j\in [2,m+1]_2} T_{(j,n+1)}T_{w_\circ^{[1,m]_2}} T_{w_\circ^{[2,m-1]_2}}=\ascprod_{i\in[1,m+1]} T_{(i,n+1)} \Cx1m.
$$
The left hand side is equal to
\begin{align*}
T_{(m,n+1)} &\dscprod_{j\in [1,m-2]_2} T_{(j,n+1)}\ascprod_{j\in [2,m-1]_2} T_{(j,n+1)} T_{w_\circ^{[1,m-2]_2}} T_{w_\circ^{[2,m-3]_2}} T_{(m+1,n+1)}T_m T_{m-1}\\
&=T_{(m,n+1)} \ascprod_{i\in[1,m-1]} T_{(i,n+1)} \Cx1{(m-2)} T_{(m+1,n+1)}T_m T_{m-1}.
\end{align*}
where we used the induction hypothesis. Since $$
T_{(m-1,n+1)}T_{(m,n+1)}T_{(m+1,n+1)}=T_{(m,n+1)}T_{m-1} T_{(m,n+1)}T_{(m+1,n+1)}T_m
$$
by Lemma~\ref{lem:transp 2},
the right hand side equals to
\begin{multline*}
\ascprod_{i\in[1,m-2]} T_{(i,n+1)} T_{(m,n+1)}T_{m-1} T_{(m,n+1)}T_{(m+1,n+1)}T_m \Cx1m\\=\ascprod_{i\in[1,m-2]} T_{(i,n+1)} T_{(m,n+1)}T_{m-1} T_{(m,n+1)}
\Cx1{(m-1)} T_{(m+1,n+1)}T_m T_{m-1}.
\end{multline*}
Since the braid monoid is cancellative, it suffices to prove that
\begin{equation}
\label{eq:equiv id for reorder}
T_{(m,n+1)} \ascprod_{i\in[1,m-1]} T_{(i,n+1)} \Cx1{(m-2)}=\ascprod_{i\in[1,m-2]} T_{(i,n+1)} T_{(m,n+1)}T_{m-1} T_{(m,n+1)}
\Cx1{(m-1)}.
\end{equation}
Write $T_{(m,n+1)}=\Cx m{(n-1)}\Cxr mn$. Since~$\Cx m{(n-1)}$ commutes with the~$T_{(i,n+1)}$, $i\in [1,m-2]_2$, \eqref{eq:equiv id for reorder} is equivalent to
\begin{equation}
\label{eq:equiv id for reorder'}
\Cxr mn\ascprod_{i\in[1,m-1]} T_{(i,n+1)} \Cx1{(m-2)}=\ascprod_{i\in[1,m-2]} T_{(i,n+1)} \Cxr mn T_{m-1} T_{(m,n+1)}
\Cx1{(m-1)}.
\end{equation}
Since $\Cxr{(m-1)}n T_{(m,n+1)}=\Cxr{(m-1)}n\Cx mn\Cxr m{(n-1)}=T_{(m-1,n+1)}\Cxr m{(n-1)}$, that identity is equivalent to
\begin{equation}
\label{eq:equiv id for reorder''}
\Cxr mn\ascprod_{i\in[1,m-1]} T_{(i,n+1)} \Cx1{(m-2)}=\ascprod_{i\in[1,m-1]} T_{(i,n+1)} \Cxr m{(n-1)}
\Cx1{(m-1)}.
\end{equation}
Since $T_{(m-1,n+1)}$ commutes with $\Cxr m{(n-1)}$, we can rewrite the right hand side as
\begin{align*}
\ascprod_{i\in[1,m-2]} T_{(i,n+1)} &\Cxr m{(n-1)} T_{(m-1,n+1)}\Cx1{(m-1)}\\
&=\ascprod_{i\in[1,m-2]} T_{(i,n+1)} \Cxr m{(n-1)} T_{m-1} T_{(m,n+1)}T_{m-1}\Cx1{(m-1)}
\\
&=\ascprod_{i\in[1,m-2]} T_{(i,n+1)} \Cxr{(m-1)}{(n-1)} T_{(m,n+1)}T_{m-1}\Cx1{(m-1)}\\
&=\ascprod_{i\in[1,m-2]} T_{(i,n+1)} \Cxr{(m-1)}{(n-1)} T_{(m,n+1)}\Cx1{(m-1)}T_{m-2}.
\end{align*}
Therefore, \eqref{eq:equiv id for reorder''} is equivalent to
\begin{equation}
\label{eq:equiv id for reorder'''}
\Cxr mn\ascprod_{i\in[1,m-1]} T_{(i,n+1)} \Cx1{(m-3)}=\ascprod_{i\in[1,m-2]} T_{(i,n+1)} \Cxr{(m-1)}{(n-1)} T_{(m,n+1)}\Cx1{(m-1)}.
\end{equation}
Suppose we proved that~\eqref{eq:equiv id for reorder''} is equivalent to
\begin{equation}
\label{eq:equiv id for reorder'v}
\Cxr mn\ascprod_{j\in[1,m-1]} T_{(j,n+1)} \Cx1{(i-1)}=\ascprod_{j\in[1,i]} T_{(j,n+1)} \Cxr{(i+1)}{(n-1)} \ascprod_{j\in[i+2,m]} T_{(j,n+1)}\Cx1{(m-1)}.
\end{equation}
for some~$i\in[3,m-2]$.
We can rewrite the right hand side of~\eqref{eq:equiv id for reorder'v} as
\begin{align*}
\ascprod_{j\in[1,i-1]}& T_{(j,n+1)} \Cxr{(i+1)}{(n-1)} T_{(i,n+1)} \ascprod_{j\in[i+2,m]} T_{(j,n+1)}\Cx1{(m-1)}\\
&=\ascprod_{j\in[1,i-1]} T_{(j,n+1)} \Cxr{i}{(n-1)} T_{(i+1,n+1)} T_i \ascprod_{j\in[i+2,m]} T_{(j,n+1)}\Cx1{(m-1)}\\
&=\ascprod_{j\in[1,i-1]} T_{(j,n+1)} \Cxr{i}{(n-1)} T_{(i+1,n+1)} \ascprod_{j\in[i+2,m]} T_{(j,n+1)}T_i \Cx1{(m-1)}\\
&=\ascprod_{j\in[1,i-1]} T_{(j,n+1)} \Cxr{i}{(n-1)} T_{(i+1,n+1)} \ascprod_{j\in[i+2,m]} T_{(j,n+1)}\Cx1{(i-2)}T_i T_{i-1} T_i \Cx{(i+1)}{(m-1)}\\
&=\ascprod_{j\in[1,i-1]} T_{(j,n+1)} \Cxr{i}{(n-1)} T_{(i+1,n+1)} \ascprod_{j\in[i+2,m]} T_{(j,n+1)}\Cx1{(m-1)}T_{i-1}
\end{align*}
whence~\eqref{eq:equiv id for reorder'''} is equivalent to
$$
\Cxr mn\ascprod_{j\in[1,m-1]} T_{(j,n+1)} \Cx1{(i-2)}=\ascprod_{j\in[1,i-1]} T_{(j,n+1)} \Cxr{i}{(n-1)} \ascprod_{j\in[i+1,m]} T_{(j,n+1)}\Cx1{(m-1)}.
$$
Thus, \eqref{eq:equiv id for reorder'''} is equivalent to~\eqref{eq:equiv id for reorder'v} for all~$i\in[2,m-1]$. Taking~$i=2$
we conclude that~\eqref{eq:equiv id for reorder'''} is equivalent to
\begin{equation}\label{eq:last reduction}
\Cxr mn\ascprod_{j\in[1,m-1]} T_{(j,n+1)}=T_{(1,n+1)} \Cxr{2}{(n-1)} \ascprod_{j\in[3,m]} T_{(j,n+1)}\Cx1{(m-1)}.
\end{equation}
We now rewrite the right hand side of this identity as
\begin{align*}
T_{(1,n+1)} &\Cxr{2}{(n-1)} \ascprod_{j\in[3,m]} T_{(j,n+1)}\Cx1{(m-1)}\\
&=\Cxr1n\Cx 2n\Cxr2{(n-1)}\ascprod_{j\in[3,m]} T_{(j,n+1)}\Cx1{(m-1)}
=\Cxr1n \ascprod_{j\in[2,m]} T_{(j,n+1)}\Cx1{(m-1)}\\
&=\Cxr2n T_{(1,n+1)} \ascprod_{j\in[3,m]} T_{(j,n+1)}\Cx2{(m-1)}.
\end{align*}
Suppose we proved that for some~$i\in [2,m-1]$
\begin{align*}
T_{(1,n+1)} &\Cxr{2}{(n-1)} \ascprod_{j\in[3,m]} T_{(j,n+1)}\Cx1{(m-1)}=
\Cxr in \ascprod_{j\in[1,m]\setminus\{i\}} T_{(j,n+1)}\Cx i{(m-1)}.
\end{align*}
Then
\begin{align*}
T_{(1,n+1)} &\Cxr{2}{(n-1)} \ascprod_{j\in[3,m]} T_{(j,n+1)}\Cx1{(m-1)}\\
&=\Cxr{(i+1)} n \ascprod_{j\in[1,i-1]} T_{(j,n+1)} T_i T_{(i+1,n+1)} T_i \ascprod_{j\in[i+2,m]} T_{(j,n+1)}\Cx{(i+1)}{(m-1)}\\
&=\Cxr{(i+1)} n \ascprod_{j\in[1,i-1]} T_{(j,n+1)} T_{(i,n+1)} \ascprod_{j\in[i+2,m]} T_{(j,n+1)}\Cx{(i+1)}{(m-1)}\\
&=\Cxr{(i+1)}n \ascprod_{j\in[1,m]\setminus\{i+1\}} T_{(j,n+1)}\Cx{(i+1)}{(m-1)}.
\end{align*}
Therefore,
$$
T_{(1,n+1)} \Cxr{2}{(n-1)} \ascprod_{j\in[3,m]} T_{(j,n+1)}\Cx1{(m-1)}=\Cxr mn \ascprod_{j\in[1,m-1]} T_{(j,n+1)},
$$
which is the left-hand side of~\eqref{eq:last reduction}.
\end{proof}
Since $T_{J_m}=T_{w_\circ^{[1,m]_2}}T_{w_\circ^{[2,m-1]_2}} T_{(m+1,n+1)}$,
the assertion is immediate from Lemma~\ref{lem:2nd identity} and
the definition~\eqref{eq:Z_r defn} of~$Z_m$.
\end{proof}

Proposition~\ref{prop:g J=1} for~$J$ with~$g(J)=1$ and~$|J|$ odd follows from Proposition~\ref{prop:TJm power central} by Lemma~\ref{lem:move bubble}.

Thus, if $J=[1,a]\cup [b+1,n+1]$ with $1\le a<b\le n$,
$T_J^{|J|}$ is the product of two commuting ${}^{op}$-invariant
elements of~$\Br_n$ and, therefore, is~${}^{op}$-invariant.
It remains to apply Lemma~\ref{lem:TJ hom}. This completes the proof of
Theorem~\ref{thm:adm I2m}.
\end{proof}

\begin{proof}[Proof of forward direction in Theorem~\ref{thm:main thm adm}]
Since~$g(J)=1$, $J=[1,a]\cup [b+1,n+1]$ for some~$1\le a<b\le n+1$.
Let~$K=[a,b]$. Then, in the notation of Theorem~\ref{thm:main thm adm},
\begin{align*}
I'(K)&=[1,a-2]_2\cup \{ r\in [b+1,n+1]\,:\,
\overline{r-b}=0\},\\
I''(K)&=I\setminus (I'(K)\cup K)
=[1,a-1]_2\cup \{ r\in [b+1,n+1]\,:\, \overline{r-b}=1\},
\end{align*}
and so $\tau_{\overline a}(J)=T_{w_\circ^{I'(K)\cup K}}$,
$\tau_{\overline{a-1}}(J)=T_{w_\circ^{I''(K)}}$.
If~$\sigma(K)=K$, that is, $b=n+1-a$ then
$m(K)=\frac12(n-|K|)+1=a$ while~$|J|=2a$ and so~$m(K)=|J|/2$.
Otherwise, $m(K)=n-|K|+2=n-b+a+1=|J|$. Then
Theorem~\ref{thm:adm I2m} yields the desired homomorphism.
\end{proof}

Given $J\subset [1,n+1]$, $J=\{j_1,\dots,j_m\}$ with~$j_1<\cdots<j_m$, let $\Br^+_{n+1}[J]$ be the submonoid of~$\Br^+_{n+1}$
generated by the $T_{(j_k,j_{k+1})}$, $1\le k\le m-1$.
\begin{corollary}\label{cor:centrailty}
Let~$\{1,n+1\}\subset J\subset [1,n+1]$ with~$g(J)=1$. Then~$T_J^{|J|}$
is central in~$\Br^+_{n+1}[J]$.
\end{corollary}
\begin{proof}
Let~$J=[1,a]\cup[b+1,n+1]$, $1\le a<b\le n$. Then~$\Br^+(J)$
is generated by the $T_i$, $i\in [1,n]\setminus [a,b]$
and by $T_{(a,b+1)}$. By Proposition~\ref{prop:g J=1},
$T_J^{|J|}=T_{w_\circ^{[1,n]}}^2 T_{w_\circ^{[a+1,b-1]}}^{-2}$.
Since $T_{w_\circ^{[1,n]}}^2$ is central in~$\Br^+_{n+1}$
and $T_{w_\circ^{[a+1,b-1]}}^2$ commutes with the $T_i$,
$i\in [1,a-1]\cup[b+1,n]=[1,n]\setminus [a,b]$ and with $T_{(a,b+1)}$
by Proposition~\ref{prop:elem prop transp}, the assertion follows.
\end{proof}
\begin{corollary}\label{cor:Cox hom from TJ}
Let~$\{1,n+1\}\subset J\subset [1,n+1]$ with~$g(J)=1$ and let~$m=m(J)$ be 
as in Theorem~\ref{thm:adm I2m}. Then the
assignments $\wh T_r\mapsto \tilde\tau_{\overline r}(J)$, $r\in\{1,2\}$ define an optimal 
Coxeter type homomorphism~$\Phi:
\Br^+(I_2(2m))\to \Br^+_{n+1}$.
\end{corollary}
\begin{proof}
The above assignments define a homomorphism
by Theorem~\ref{thm:adm I2m} and Lemma~\ref{lem:TJ hom}.
Since 
$\pi_n(\tilde\tau_r(J))$, $r\in\{0,1\}$ are manifestly involutions,
being products of commuting transpositions,
$\Phi$ is of Coxeter type by Theorem~\partref{thm:Main Thm Cox Heck.a}
and is optimal by Corollary~\ref{cor:TJ coxeter}.
\end{proof}

\subsection{Symmetrized Burau representation and the converse}\label{subs:Burau} We now prove the converse in Theorem~\ref{thm:main thm adm}. The key ingredient is the following
\begin{theorem}\label{thm:adm I2m converse }
Let~$\{1,n+1\}\subset J\subset [1,n+1]$ with~$g(J)>1$. Then $T_J^m$ is not ${}^{op}$-invariant
for all~$m\in\mathbb Z_{>0}$.
\end{theorem}

To prove this Theorem, we use representation theory of braid monoids.
Let~$\kk$ be a field of characteristic zero and let~$q\in\kk^\times$
which is not a root of unity.
Let \plink{ei}$\{e_i\}_{1\le i\le n+1}$ be the standard basis of~$V=\kk^{n+1}$. Define~$T_i\in\End V$
by
$$
T_i(e_j)=e_j-(q \delta_{i,j}-\delta_{i+1,j})(q e_i-e_{i+1})
$$
for all $j\in [1,n+1]$, $i\in[1,n]$. It easy to see that the~$T_i$
are invertible with
$$
T_i^{-1}(e_j)=e_j-(\delta_{i,j}-q^{-1}\delta_{i+1,j})(e_i-q^{-1}e_{i+1}).
$$
\begin{proposition}\label{prop:Burau}
The operators $T_i$, $i\in [1,n]$ provide a representation of~$\Br_n$ on~$V$. Moreover, $T_i^2=(1-q^{2})T_i+q^{2}\id_V$, $T_i^{-1}=q^{-2} T_i+(1-q^{-2})\id_V$ and for any $T\in\Br_n$, the matrix of $T^{op}$
in the standard basis~$\{e_i\}_{1\le i\le n+1}$ is the transpose of the matrix of~$T$.
\end{proposition}
\begin{proof}
The operators~$T_i$ are easily seen to be modified Burau operators
(\cite{Bu}).
Clearly, $T_iT_j(e_k)=T_jT_i(e_k)$, $1\le k\le n+1$,
if~$|i-j|>1$. Since
\begin{align*}
T_i T_{i+1}T_i(e_j)&=T_i T_{i+1}(e_j-(\delta_{i,j}-q^{-1}\delta_{i+1,j})(e_i-q^{-1}e_{i+1})\\
&=T_i(e_j-(\delta_{i,j}-q^{-1}\delta_{i+1,j})e_i
-(\delta_{i+1,j}-q^{-1}\delta_{i+2,j})e_{i+1}
\\&\qquad+(q^{-2}\delta_{i,j}+q^{-1}(1-q^{-2})\delta_{i+1,j}-
q^{-2}\delta_{i+2,j})e_{i+2})\\
&=e_j-(\delta_{i,j}-q^{-2}\delta_{i+1,j})e_i-(1-q^{-2})(\delta _{i+1,j}-q^{-1}\delta _{i+2,j})e_{i+1}\\
&\qquad+(q^{-2}\delta_{i,j}+q^{-1}(1-q^{-2})\delta_{i+1,j}-
q^{-2}\delta_{i+2,j})e_{i+2})
\\
\intertext{and}
T_{i+1}T_iT_{i+1}(e_j)&=T_{i+1}T_i(e_j-(\delta_{i+1,j}-q^{-1}\delta_{i+2,j})(e_{i+1}-q^{-1}e_{i+2}))\\
 &=T_{i+1}(e_j-(\delta _{i,j}-q^{-2}\delta _{i+2,j})e_i+(q^{-1}\delta _{i,j}-\delta _{i+1,j}+q^{-1}(1-q^{-2})\delta _{i+2,j})e_{i+1}\\
&\qquad+q^{-1}(\delta_{i+1,j}-q^{-1}\delta_{i+2,j})e_{i+2})\\
 &=e_j-(\delta _{i,j}-q^{-2}\delta _{i+2,j})e_i
 -(1-q^{-2})(\delta _{i+1,j}-q^{-1}\delta _{i+2,j})e_{i+1}\\
 &\qquad+(q^{-2}\delta_{i,j}+q^{-1}(1-q^{-2})\delta_{i+1,j}-
 q^{-2}\delta_{i+2,j})e_{i+2}),
\end{align*}
for all $i\in[1,n-1]$, $j\in[1,n+1]$, it follows that all $T_i T_{i+1} T_i=T_{i+1}T_i T_{i+1}$.
Furthermore,
\begin{align*}
T_i^2(e_j)&=T_i(e_j-(q \delta_{i,j}-\delta_{i+1,j})(q e_i-e_{i+1})=e_j-(q\delta_{i,j}-\delta_{i+1,j})(1-q^{2})(q e_i-e_{i+1})\\
&=(1-q^2)T_i(e_j)+q^{2}e_j.
\end{align*}
The identity for~$T_i^{-1}$ is obvious. The last assertion follows since the matrix of~$T_i$ with respect to the
standard basis is obviously symmetric and thus the restriction of the transpose to the image of~$\Br_n$ in~$\End V$ coincides with~${}^{op}$.
\end{proof}

Given $i,j\in[1,n+1]$, define \plink{v[i,j]}$v_{[i,j]}=\sum_{i\le t\le j} q^t e_t$.

\begin{lemma}\label{lem:fixed point}
For all~$i\le j\in[1,n+1]$, $T(v_{[i,j]})=v_{[i,j]}$
for all $T\in\Br_n$ with $\supp(T)\subset [i,j-1]$. In particular, $T(v_{[1,n+1]})=v_{[1,n+1]}$ for all~$T\in\Br_n$.
\end{lemma}
\begin{proof}
It suffices to prove that~$T_k(v_{[i,j]})=v_{[i,j]}$ for all~$k\in [i,j-1]$. We have
\begin{equation*}
T_k(v_{[i,j]})=v_{[i,j]}-q^{k+1}(q e_k-e_{k+1})+q^{k+1}(q e_k-e_{k+1})=v_{[i,j]}.\qedhere
\end{equation*}
\end{proof}
Denote \plink{ui}$u_i=e_i-q^{-1}e_{i+1}$, $i\in [1,n]$ and let~$w^{(a)}_{[i,j]}=\sum_{i\le t\le j}q^{a t}u_t$, $a\in\mathbb Z$.
It is immediate from the definitions that for all~$i\in [1,n]$,
$T^{\pm1}_i(u_k)=u_k$ if $k\in[1,n+1]\setminus\{i-1,i,i+1\}$ while
\begin{equation}\label{eq:ind base cox w0}
\begin{array}{ll}
T^{\pm 1}_i(u_{i-1})=u_{i-1}+q^{\pm1} u_i,&  i\in [2,n],\\
T^{\pm 1}_i(u_i)=-q^{\pm 2} u_i,&i\in [1,n],\\
T^{\pm 1}_{i}(u_{i+1})=q^{\pm1} u_i+u_{i+1},&i\in[1,n-1].
\end{array}
\end{equation}

\begin{lemma}\label{lem:rev Cox act}
Let $i\le j\in [1,n]$, $k\in[1,n]$ and~$\epsilon\in\{1,-1\}$. Then
$$
\Cxr ij^{(\epsilon)}
(u_k)=\begin{cases}
u_k,&k\in[1,n]\setminus [i-1,j+1],\\
q^{\epsilon}u_{k-1}+\delta_{k,j+1}u_{j+1},&k\in[i+1,j+1],\\
q^{\epsilon(1-i)} w^{(\epsilon)}_{[i-1,j]},&k=i-1>0,\\
-q^{\epsilon(2-i)} w^{(\epsilon)}_{[i,j]},&k=i.
\end{cases}
$$
\end{lemma}
\begin{proof}
Since~$\supp(\Cxr{i}{ j}^{(\epsilon)})=[i,j]$, the assertion is obvious for~$k\in[1,n+1]\setminus [i-1,j+1]$. To prove it for~$k\in[i-1,j+1]$ we use induction on~$j-i$.
The induction base $j-i=0$ is~\eqref{eq:ind base cox w0}.
For the inductive step, for~$k\in [i+1,j]$ we have
\begin{align*}
\Cxr{i}{ j}^{(\epsilon)}(u_k)&=\Cxr{(k-1)}{ j}^{(\epsilon)}\Cxr{i}{(k-2)}^{(\epsilon)}(u_k)=\Cxr{k}{ j}^{(\epsilon)}T^{\epsilon}_{k-1}(u_k)
=\Cxr{k}{ j}^{(\epsilon)}(q^{\epsilon} u_{k-1}+u_{k})\\
&=\Cxr{(k+1)}{ j}^{(\epsilon)}(q^{\epsilon}(u_{k-1}+q^{\epsilon}u_k)-q^{2\epsilon} u_k)
=q^{\epsilon}u_{k-1}
\end{align*}
while for~$k=j+1$,
$$
\Cxr{i}{ j}^{(\epsilon)}(u_{j+1})=T_j^{\epsilon}(u_{j+1})=q^{\epsilon}u_j+u_{j+1}.
$$
For~$k=i-1$, using~\eqref{eq:ind base cox w0} and the induction hypothesis we obtain
\begin{align*}
\Cxr{i}{ j}^{(\epsilon)}(u_{i-1})&=
u_{i-1}+q^{\epsilon} \Cxr{(i+1)}{ j}^{(\epsilon)}(u_i)=u_{i-1}+q^{\epsilon(1-i)}w^{(\epsilon)}_{[i,j]}=q^{\epsilon(1-i)}w^{(\epsilon)}_{[i-1,j]}.
\end{align*}
Finally, for~$k=i$, \eqref{eq:ind base cox w0} and the induction hypothesis yield
\begin{align*}
\Cxr{i}{ j}^{(\epsilon)}(u_i)&=-q^{2\epsilon}\Cxr{(i+1)}{ j}^{(\epsilon)}(u_i)=-q^{\epsilon(2-i)} w^{(\epsilon)}_{[i,j]}.\qedhere
\end{align*}
\end{proof}
\begin{corollary}\label{cor:cox w +-}
Let $i\le j,k\le l\in [1,n]$ and $\epsilon,\epsilon'\in\{1,-1\}$. Then
$$
\Cxr ij^{(\epsilon)}(w^{(\epsilon')}_{[k,l]})=
\begin{cases}
w^{(\epsilon')}_{[k,l]},&k>j+1,\\
q^{\epsilon+\epsilon'}w^{(\epsilon')}_{[k-1,\min(l-1,j)]}+w^{(\epsilon')}_{[j+1,l]},&k\in[i+1,j+1],\\
q^{\epsilon+\epsilon'} w^{(\epsilon')}_{[i,\min(l-1,j]]}-
q^{2\epsilon+(\epsilon'-\epsilon)i}w^{(\epsilon)}_{[i,j]}
+w^{(\epsilon')}_{[j+1,l]},&k=i,\\
w^{(\epsilon')}_{[k,i-1]}+
q^{\epsilon+\epsilon'}w^{(\epsilon')}_{[i,\min(l-1,j)]}+w^{(\epsilon')}_{[j+1,l]}\\
\quad+q^{(\epsilon'-\epsilon)i}(q^{\epsilon-\epsilon'}-q^{2\epsilon}) w^{(\epsilon)}_{[i,j]},&k<i.
\end{cases}
$$
\end{corollary}
\begin{proof}
The first two cases are immediate from Lemma~\ref{lem:rev Cox act}. If~$k=i$,
\begin{align*}
\Cxr ij^{(\epsilon)}(w^{(\epsilon')}_{[i,l]})
&=q^{\epsilon' i}\Cxr ij^{(\epsilon)}(u_i)+q^{\epsilon+\epsilon'}w^{(\epsilon')}_{[i,\min(l-1,j)]}+w^{(\epsilon')}_{[j+1,l]}\\
&=-q^{(\epsilon'-\epsilon)i+2\epsilon} w^{(\epsilon)}_{[i,j]}+
q^{\epsilon+\epsilon'}w^{(\epsilon')}_{[i,\min(l-1,j)]}+w^{(\epsilon')}_{[j+1,l]}.
\end{align*}
Finally, if $k<i$,
\begin{align*}
\Cxr ij^{(\epsilon)}(w^{(\epsilon')}_{[k,l]})&=
w^{(\epsilon')}_{[k,i-2]}+q^{\epsilon'(i-1)}\Cxr ij^{(\epsilon)}(u_{i-1})-q^{(\epsilon'-\epsilon)i+2\epsilon} w^{(\epsilon)}_{[i,j]}+
q^{\epsilon+\epsilon'}w^{(\epsilon')}_{[i,\min(l-1,j)]}+w^{(\epsilon')}_{[j+1,l]}\\
&=
w^{(\epsilon')}_{[k,i-2]}+q^{(\epsilon'-\epsilon)(i-1)}
w^{(\epsilon)}_{[i-1,j]}-q^{(\epsilon'-\epsilon)i+2\epsilon} w^{(\epsilon)}_{[i,j]}+
q^{\epsilon+\epsilon'}w^{(\epsilon')}_{[i,\min(l-1,j)]}+w^{(\epsilon')}_{[j+1,l]}\\
&=w^{(\epsilon')}_{[k,i-1]}+
q^{(\epsilon'-\epsilon)i}(q^{\epsilon-\epsilon'}-q^{2\epsilon}) w^{(\epsilon)}_{[i,j]}+
q^{\epsilon+\epsilon'}w^{(\epsilon')}_{[i,\min(l-1,j)]}+w^{(\epsilon')}_{[j+1,l]}.\qedhere
\end{align*}
\end{proof}

\begin{lemma}\label{lem:Tw0ij act ui}
For all $i\le j\in[1,n]$, $k\in[1,n]$, $\epsilon\in\{1,-1\}$
\begin{equation}
T_{w_\circ^{[i,j]}}^{\epsilon}(u_k)=\begin{cases}
u_k,&k\in[1,n]\setminus[i-1,j+1],\\
-q^{\epsilon(j-i+2)}u_{i+j-k},&k\in[i,j],\\
q^{\epsilon(j+1)}w^{(-\epsilon)}_{[i,j+1]},&k=j+1,\\
q^{-\epsilon(i-1)}w^{(\epsilon)}_{[i-1,j]},&k=i-1.
\end{cases}\label{eq:Tw0 act uk}
\end{equation}
In particular,
\begin{equation}\label{eq:Tw0^2 eigen}\{ u_k\,:\, k\in[i,j]\}\subset
\ker(T_{w_\circ^{[i,j]}}^2-q^{2(j-i+2)}\id_V)
\end{equation}
and
\begin{equation}\label{eq:Tw0 eigen}
w^{(\epsilon)}_{[k,l]}\in\ker(T_{w_\circ^{[i,j]}}-\id_V),\qquad [i-1,j+1]\subset [k,l]\subset [1,n+1].
\end{equation}
\end{lemma}
\begin{proof}
The first case in~\eqref{eq:Tw0 act uk} is obvious. For the remaining cases, we use induction on~$j-i$.
The case $j=i$ has already been established in~\eqref{eq:ind base cox w0}.

Note that $T_{w_\circ^{[i,j]}}^{\epsilon}=T_{w_\circ^{[i,j]}}^{\epsilon}\Cxr{i}{ j}^{(\epsilon)}$. Indeed,
$T_{w_\circ^{[i,j]}}=T_{w_\circ^{[i,j]}}\Cxr ij$ while, since $T_{w_\circ^{[i,j]}}$ is
${}^{op}$-invariant by~\cite{BrSa}*{Lemma~5.1}, $T_{w_\circ^{[i,j]}}^{-1}=(\Cx ij T_{w_\circ^{[i,j-1]}})^{-1}=T_{w_\circ^{[i,j-1]}}^{-1}\Cx ij^{-1}$.

For the inductive step, for~$k\in[i+1,j]$
we have by Lemma~\ref{lem:rev Cox act} and the induction hypothesis
\begin{align*}
T_{w_\circ^{[i,j]}}^{\epsilon}(u_k)&=T_{w_\circ^{[i,j-1]}}^{\epsilon}\Cxr{i}{ j}^{(\epsilon)}(u_k)=q^{\epsilon} T_{w_\circ^{[i,j-1]}}^{\epsilon}(u_{k-1})=-q^{\epsilon(j-i+2)}u_{i+j-k},
\end{align*}
while for~$k=i$
\begin{align}
T_{w_\circ^{[i,j]}}^{\epsilon}(u_i)&=T_{w_\circ^{[i,j-1]}}^{\epsilon}\Cxr{i}{ j}^{(\epsilon)}(u_i)=-q^{\epsilon(2-i)} T_{w_\circ^{[i,j-1]}}^{\epsilon}(w^{(\epsilon)}_{[i,j]})
\nonumber\\
&=q^{\epsilon(2-i)}\sum_{i\le t\le j-1}q^{\epsilon(t-i+j+1)}u_{i+j-1-t}-q^{\epsilon(2+j-i)} T_{w_\circ^{[i,j-1]}}^{\epsilon}(u_j)
\nonumber\\
&=q^{\epsilon(2+2j-i)} w^{(-\epsilon1)}_{[i,j-1]}-q^{\epsilon(2+2j-i)} w^{(-\epsilon1)}_{[i,j]}=-q^{\epsilon(j-i+2)}u_{j}.\label{eq:Tw0ij ui}
\end{align}
Furthermore,
\begin{align*}
T_{w_\circ^{[i,j]}}^{\epsilon}(u_{j+1})&=T_{w_\circ^{[i,j-1]}}^{\epsilon}\Cxr{i}{ j}^{(\epsilon)}(u_{j+1})=T_{w_\circ^{[i,j-1]}}^{\epsilon}(q^{\epsilon} u_j+u_{j+1})\\
&=q^{\epsilon(j+1)}w^{(-\epsilon1)}_{[i,j]}+u_{j+1}=q^{\epsilon(j+1)} w^{(-\epsilon1)}_{[i,j+1]}.
\end{align*}
Finally, by Lemma~\ref{lem:rev Cox act}
\begin{align*}
T_{w_\circ^{[i,j]}}^{\epsilon}(u_{i-1})&=T_{w_\circ^{[i,j-1]}}^{\epsilon}\Cxr{i}{ j}^{(\epsilon)}(u_{i-1})
=q^{-\epsilon(i-1)} T_{w_\circ^{[i,j-1]}}^{\epsilon}(w^{(\epsilon)}_{[i-1,j]}).\end{align*}
As we already established in~\eqref{eq:Tw0ij ui},
$$
T_{w_\circ^{[i,j-1]}}^{\epsilon}(w^{(\epsilon)}_{[i,j]})=q^{\epsilon j} u_j.
$$
Using the induction hypothesis, we obtain
\begin{align*}
T_{w_\circ^{[i,j]}}^{\epsilon}(u_{i-1})&=T_{w_\circ^{[i,j-1]}}^{\epsilon}(u_{i-1})+q^{\epsilon(j-i+1)}u_j=q^{-\epsilon(i-1)}(w^{(\epsilon)}_{[i-1,j-1]}+q^{\epsilon j}u_j)=q^{-\epsilon(i-1)}w^{(\epsilon)}_{[i-1,j]}.
\end{align*}

The inclusion in~\eqref{eq:Tw0^2 eigen} is immediate. It is also clear that it suffices to prove~\eqref{eq:Tw0 eigen} for~$k=i-1$, $l=j+1$. Note that
$$
T_{w_\circ^{[i,j]}}^{-\epsilon}(u_{j+1})=q^{-\epsilon(j+1)}w^{(\epsilon)}_{[i-1,j+1]}-q^{-\epsilon(j-i+2)} u_{i-1}
$$
while
$$
T_{w_\circ^{[i,j]}}^{\epsilon}(u_{i-1})=q^{-\epsilon(i-1)}w^{(\epsilon)}_{[i-1,j+1]}-q^{\epsilon(j-i+2)}u_{j+1}
$$
whence
$$
w^{(\epsilon)}_{[i-1,j+1]}=q^{\epsilon(i-1)} T_{w_\circ^{[i,j]}}^{\epsilon}(u_{i-1})+q^{\epsilon(j+1)}u_{j+1}=
q^{\epsilon(j+1)}T_{w_\circ^{[i,j]}}^{-\epsilon 1}(u_{j+1})+q^{\epsilon(i-1)}u_{i-1}
$$
and so
$$
T_{w_\circ^{[i,j]}}^{\epsilon}(w^{(\epsilon)}_{[i-1,j+1]})=q^{\epsilon(j+1)}u_{j+1}+q^{\epsilon(i-1)}T_{w_\circ^{[i,j]}}^{\epsilon}(u_{i-1})
=w^{(\epsilon)}_{[i-1,j+1]},
$$
which immediately yields~\eqref{eq:Tw0 eigen}.
\end{proof}

\begin{lemma}\label{lem:Transp action}
For all $1\le i<j\le n$, $k\in[1,n]$,
$$
T_{(i,j+1)}(u_k)=\begin{cases}
				u_k,& k\in [1,n+1]\setminus [i-1,j+1],\\
				q^2 u_k,&k\in [i+1,j-1],\\
				-q^{j+2} w^{(-1)}_{[i,j-1]},&k=j,\\
				q^{j+1} w^{(-1)}_{[i,j+1]},&k=j+1,\\
				-q^{2-i} w^{(1)}_{[i+1,j]},&k=i,\\
				q^{1-i} w^{(1)}_{[i-1,j]},&k=i-1.
                 \end{cases}
$$
\end{lemma}
\begin{proof}
The assertion is obvious when $k\in[1,n+1]\setminus[i,j+1]$. For~$k\in [i-1,j+1]$
we use induction on~$j-i$. If~$j=i+1$, $T_{(i,j+1)}=T_{w_\circ^{[i,i+1]}}$
and the assertion follows from Lemma~\ref{lem:Tw0ij act ui}.

For the inductive step, recall that $T_{(i,j+1)}=T_i T_{(i+1,j+1)}T_i=T_j T_{(i,j)}T_j$.
Suppose first that~$k\in[i+2,j-1]$. Then
$$
T_{(i,j+1)}(u_k)=T_i T_{(i+1,j+1)}(u_k)=q^2 T_i(u_k)=q^2 u_k,
$$
while for~$k=i+1$ by~\eqref{eq:ind base cox w0} and the induction hypothesis
\begin{align*}
T_{(i,j+1)}(u_{i+1})&=T_i T_{(i+1,j+1)}(q u_i+u_{i+1})=T_i( q^{1-i}w^{(1)}_{[i,j]}-q^{1-i}w^{(1)}_{[i+2,j]})\\
&=T_i(q u_i+q^2 u_{i+1})=-q^3 u_i+q^2(q u_i+u_{i+1})=q^2 u_{i+1}.
\end{align*}
Furthermore, since $j>i+1$ and so $T_i(u_k)=u_k$ for~$k\in \{j,j+1\}$,
\begin{align*}
T_{(i,j+1)}(u_{j+1})&=q^{j+1} T_i(w^{(-1)}_{[i+1,j+1]})=q^{j-i} T_i(u_{i+1})+q^{j+1}w^{(-1)}_{[i+2,j+1]}\\
&=q^{j-i}(q u_i+u_{i+1})+q^{j+1}w^{(-1)}_{[i+2,j+1]}=q^{j+1} w^{(-1)}_{[i,j+1]},
\intertext{while}
T_{(i,j+1)}(u_j)&=-q^{j+2} T_i(w^{(-1)}_{[i+1,j-1]})=-q^{j+1-i} T_i(u_{i+1})-q^{j+2} w^{(-1)}_{[i+2,j-1]}\\
&=-q^{j+1-i}(q u_i+u_{i+1})-q^{j+2} w^{(-1)}_{[i+2,j-1]}=-q^{j+2} w^{(-1)}_{[i,j-1]}.
\end{align*}
Finally, for~$k\in \{i-1,i\}$, $T_j(u_k)=u_k$ and so
\begin{align*}
T_{(i,j+1)}(u_i)&=-q^{2-i} T_j ( w^{(1)}_{[i+1,j-1]})=-q^{2-i} w^{(1)}_{[i+1,j-2]}-q^{j+1-i} T_j(u_{j-1})\\
&=-q^{2-i} w^{(1)}_{[i+1,j-2]}-q^{j+1-i} (q u_{j-1}+u_j)=-q^{2-i} w^{(1)}_{[i+1,j]},
\intertext{while}
T_{(i,j+1)}(u_{i-1})&=q^{1-i} T_j (w^{(1)}_{[i-1,j-1]})=q^{1-i}w^{(1)}_{[i-1,j-2]}+q^{j-i}T_j(u_{j-1})=q^{1-i}w^{(1)}_{[i-1,j]}.\qedhere
\end{align*}
\end{proof}

Now we describe eigenspaces of~$T_J^{|J|}$ for a special choice of~$J$.
\begin{proposition}\label{prop:TJ eigenvectors}
Let $m\in [1,n-1]$ and let~$J=[1,m]\cup \{n+1\}$.
Then
\begin{enmalph}
\item\label{prop:TJ eigenvectors.a} $\{ u_i\,:\,i\in[m+1,n-1]
\subset \ker(T_J-q^{2}\id_V)$;
\item\label{prop:TJ eigenvectors.b}
$\{u_i\,:\,i\in[1,m-1]\}\cup \{w^{(\epsilon)}_{[m,n]}\,:\,\epsilon\in\{1,-1\}\}$
is a basis of~$\ker(T_J^{|J|}-q^{2(n+1)}\id_V)$;
\item\label{prop:TJ eigenvectors.c} $T_J^{|J|}$ is diagonalizable on~$V$,
$\det(t\id_V-T_J^{|J|})=(t-1)(t-q^{2(n+1)})^{m+1}(t-q^{2m})^{n-m-1}$
and $\det(t\id_V-T_J)\allowbreak=(t-1)(t^{m+1}-q^{2(n+1)})(t-q^{2m})^{n-m-1}$.
\end{enmalph}
\end{proposition}
\begin{proof}
Let $i\in[m+1,n-1]$. Then $T_{(m,n+1)}(u_i)=q^2 u_i$ by
Lemma~\ref{lem:Transp action} and $T_j(u_i)=u_i$ for all~$j\in [1,m-1]$ by~\eqref{eq:ind base cox w0}.
Since $T_J$ is the product of~$T_{(m,n+1)}$ with the $T_j$, $j\in[1,m-1]$,
part~\ref{prop:TJ eigenvectors.a} follows.

We now prove~\ref{prop:TJ eigenvectors.b}.
By Proposition~\ref{prop:g J=1}, $T_J^{|J|}=T_{w_\circ^{[1,n+1]}}^2 T_{w_\circ^{[m+1,n-1]}}^{-2}$.
By Lemma~\ref{lem:Tw0ij act ui}, $T_{w_\circ^{[1,n]}}^2(u_k)=
q^{2(n+1)}u_k$ for all~$k\in [1,n]$, while $T_{w_\circ^{[m+1,n-1]}}(u_k)=u_k$,
$k\in[1,m-1]$ and $T_{w_\circ^{[m+1,n-1]}}(w^{(\pm1)}_{[m,n]})=w^{(\pm1)}_{[m,n]}$. Therefore,
$\{u_i\,:\,i\in [1,m-1]\}\cup\{w^{(1)}_{[m,n]},w^{(-1)}_{[m,n]}\}\subset \ker(T_J^{|J|}-q^{2(n+1)}\id_V)$.

Next we claim that $\{u_i\,:\,i\in [1,m-1]\}\cup\{w^{(1)}_{[m,n]},w^{(-1)}_{[m,n]}\}$
is linearly independent. Indeed, since
$w^{(-1)}_{[m,n]}=q^{-m}e_m-q^{-n-1}e_{n+1}$, $w^{(-1)}_m$ is not
contained in the span of the~$u_i$, $i\in [1,m-1]$, which are manifestly linearly independent. Since
the coefficient of~$e_i$, $i\in[m+1,n]$ in~$w^{(1)}_{[m,n]}$
is $q^i-q^{i-2}\not=0$, while the $u_i$, $i\in[1,m-1]$
and~$w^{(-1)}_{[m,n]}$ are contained in the span of~$\{ e_j\,:\,
j\in[1,m]\cup\{n+1\}\}$,
it follows that $w^{(1)}_{[m,n]}$, is not
in the span of~$\{ u_i\,:\,i\in[1,m-1]\}\cup\{w^{(-1)}_m\}$.

In particular, $\dim\ker(T_J^{|J|}-q^{2(n+1)}\id_V)\ge m+1$.
By part~\ref{prop:TJ eigenvectors.a}, $\dim
\ker(T_J^{|J|}-q^{2(m+1)})\ge n-m-1$. By Lemma~\ref{lem:fixed point}, $\dim\ker(T_J^{|J|}-\id_V)\ge 1$. Since $(n-m-1)+(m+1)+1=n+1=\dim V$,
and the sum of $\ker(T_J^{|J|}-\lambda\id_V)$ with
$\lambda\in\{1,q^{2(m+1)},q^{2(n+1)}\}$ is direct,
we conclude that
all these inequalities are in fact equalities.
Therefore, $\{ u_i\,:\, i\in[m+1,n]\}$ is a basis
of~$\ker(T_J^{|J|}-q^{2(m+1)}\id_V)$,
$\{u_i\,:\, i\in [1,m-1]\}\cup\{w^{(1)}_{[m,n]},w^{(-1)}_{[m,n]}\}$ is a
basis of~$\ker(T_J^{|J|}-q^{2(n+1)}\id_V)$ and
$v_{[1,n+1]}$ spans~$\ker(T_J^{|J|}-\id_V)$. The remaining assertions are
now immediate.
\end{proof}
The following is an immediate consequence of Proposition~\ref{prop:TJ eigenvectors} and Corollary~\ref{cor:conj J}.
\begin{corollary}\label{cor:TJ eigenvectors}
Let $\{1,n+1\}\subset J\subset [1,n+1]$ with $|J|\le n$. Then
$T_J^{|J|}$ is diagonalizable on~$V$ and
$$
\det(t\id_V-T_J^{|J|})=(t-1)(t-q^{2(n+1)})^{|J|}(t-q^{2|J|})^{n-|J|},
$$
while
$$
\det(t\id_V-T_J)=(t-1)(t^{|J|}-q^{2(n+1)})(t-q^{2})^{n-|J|}.
$$
In particular, $T_J$ is diagonalizable on~$V$ provided that~$\kk$ contains all $|J|$th roots of~$q^{2(n+1)}$. Finally,
$T_J$ is conjugate to~$T_{J'}$ in~$\Br_n$
if and only if~$|J|=|J'|$.
\end{corollary}

Let~\plink{<.|.>}$\la\cdot\vert\cdot\ra:V\tensor V\to \kk$ be the standard symmetric bilinear form defined by $\la e_i\vert e_j\ra=\delta_{i,j}$, $i,j\in [1,n+1]$. Then
\begin{equation}\label{eq:Euclid ui uj}
\la u_i\vert u_j\ra
=\delta_{i,j}(1+q^{-2})-q^{-1}(\delta_{i+1,j}+\delta_{i,j+1}),\qquad
i,j\in [1,n].
\end{equation}
Note the following
\begin{lemma}\label{lem:orth eigenspaces}
Let $X\in \Br_n$ be ${}^{op}$-invariant and let $\lambda\not=\mu\in\kk$.
Then $$\la\ker(X-\lambda\id_V)\,|\,\ker(X-\mu\id_V)\ra=\{0\}.
$$
\end{lemma}
\begin{proof}
The adjoint operator of~$X\in \End V$ with respect to~$\la\cdot|\cdot\ra$ is~$X^T$ which coincides with~$X^{op}$ for~$X\in\Br_n$ by Proposition~\ref{prop:Burau}. The assertion is now standard.
\end{proof}
\begin{proof}[Proof of Theorem~\ref{thm:adm I2m converse }]
By Corollary~\ref{cor:TJ coxeter}, $T_J^{k}$ is not ${}^{op}$-invariant
unless~$|J|$ divides~$2k$. Note that if~$|J|$ is even and
$X=T_J^{k|J|/2}$ is ${}^{op}$-invariant then so is~$X^2=T_J^{k|J|}$.
Thus, it suffices to prove that~$T_J^{k|J|}$ is not ${}^{op}$-invariant
for all~$k\in\mathbb Z_{>0}$.
By Lemma~\ref{lem:orth eigenspaces} and Corollary~\ref{cor:TJ eigenvectors}, it suffices to prove that
for all~$k\in\mathbb Z_{>0}$
$$
\la \ker(T_J^{k|J|}-q^{2k(n+1)}\id_V)\,\vert\,\ker(T_J^{k|J|}-q^{2k|J|}\id_V)\ra\not=\{0\}.
$$
Since~$\ker(T_J^{k|J|}-\lambda^k)=
\ker(T_J^{|J|}-\lambda)$ for all~$\lambda\in\kk$ by Corollary~\ref{cor:TJ eigenvectors}, it suffices to prove that
$$
\la \ker(T_J^{|J|}-q^{2(n+1)}\id_V)\,\vert\,\ker(T_J^{|J|}-q^{2|J|}\id_V)\ra\not=\{0\}.
$$
Let~$m=|J|-1$.
Since~$U(J)T_JU(J)^{-1}=T_{[1,m]\cup\{n+1\}}$ by Corollary~\ref{cor:conj J},
where~$U(J)$ is as defined in~\eqref{eq:U(J) defn}, it suffices to
prove the following
\begin{proposition}\label{prop:red}
Let~$\{1,n+1\}\subset J\subset [1,n+1]$ with~$g(J)>1$. Then
$$
\la U(J)^{-1}(u)\,\vert\, U(J)^{-1}(v)\ra\not=0
$$
for some~$u\in \ker(T_{[1,m]\cup\{n+1\}}^{m+1}-q^{2(m+1)}\id_V)$,
$v\in \ker(T_{[1,m]\cup\{n+1\}}^{m+1}-q^{2(n+1)}\id_V)$,
where~$m=|J|-1$.
\end{proposition}
\begin{proof}
Write
$J=\{j_0=1<j_1<\cdots<j_{m-1}<j_m=n+1\}$ and define\plink{betapm}
\begin{align*}
&\beta_-(J)=\max\{k\in[0,m-1]\,:\,j_k=k+1\}+1,\\
&\beta_+(J)=\min\{k\in[1,m-1]\,:\,j_t=j_{m-1}-m+t+1\}.
\end{align*}
Thus,
$$
J=[1,\beta_-(J)]\cup \{j_{\beta_-(J)},\dots,j_{\beta_+(J)-1}\}
\cup [j_{\beta_+(J)},j_{m-1}]\cup\{n+1\}.
$$
with~$j_{\beta_-(J)}\ge \beta_-(J)+2$ and~$j_{\beta_+(J)-1}\le j_{\beta_+(J)}-2$.
Note also that~$g(J)>1$ implies that~$|J|\le n-1$.

Given~$s\in\mathbb Z$ let \plink{qs}$q_s=q^{(-1)^s}$. We have
\begin{equation}\label{eq: q r+s}
q_{r+s}=q^{(-1)^{r+s}}=q_r^{(-1)^s},\qquad r,s\in\mathbb Z.
\end{equation}
\begin{lemma}\label{lem:TJ|J| eigenvectors}
Let~$J=\{j_0=1<j_1<\cdots<j_{m-1}<j_m=n+1\}\subset [1,n+1]$ where $m=|J|-1\le n-2$.
\begin{enmalph}
\item\label{lem:TJ|J| eigenvectors.a}
For all $r\in[1,m-1]$,
\begin{equation}\label{eq:UJ-1 ur}
U(J)^{-1}(u_{r})=\sum_{\beta_-(J)\le k\le r-1}
\sum_{j_{k-1}\le t\le j_k-1} q_r^{\overline{r-k}}
(q_r^{k-t}-q_r^{t-k})u_t
+\sum_{j_{r-1}\le t\le j_r-1} q_r^{r-t}u_t,
\end{equation}

\item\label{lem:TJ|J| eigenvectors.b} Suppose that~$g(J)>1$ and
$j_{m-1}<n$. Then~$j_{m-1}\ge m+1$ and
$$
U(J)^{-1}(u_{j_{m-1}})=\sum_{j_{\beta_+(J)}-1 \le t\le j_{m-1}} q_m^{\overline{j_{m-1}-t}} u_t.
$$
\item\label{lem:TJ|J| eigenvectors.c} If $j_{m-1}=n$ then
$$
U(J)^{-1}(w^{((-1)^{m+1})}_{[m+1,n-1]})=\sum_{\beta_-(J)\le k\le \beta_+(J)}
q_{m+1}^{m-k-\overline{m-k}} w^{((-1)^{m+1})}_{[j_{k-1}+\delta_{k,\beta_-(J)},j_k-1-\delta_{k,\beta_+(J)}]}.
$$
\end{enmalph}
\end{lemma}
\begin{proof}
Note that~$k+1\le j_k\le j_{k+1}-1$, $0\le k\le m-1$, and
so if~$j_k=k+1$ for some~$k>0$ then $j_s=s+1$ for all~$0\le s\le k$. In this proof we abbreviate $\beta_\pm=\beta_\pm(J)$. By definition of~$\beta_-$,
$$
U(J)^{-1}=\ascprod_{\beta_-\le k\le m-1} \Cxr{(k+1)}{(j_k-1)}^{((-1)^{k+1})}.
$$
Denote the right hand side of~\eqref{eq:UJ-1 ur} by~$S(J,r)$.
Then by Lemma~\ref{lem:rev Cox act}
\begin{align*}
U(J)^{-1}(u_r)&=\Big(\ascprod_{\beta_-\le k\le m-1} \Cxr{(k+1)}{(j_k-1)}^{((-1)^{k+1})}\Big)(u_r)=\Big(\ascprod_{\beta_-\le k\le r} \Cxr{(k+1)}{(j_k-1)}^{((-1)^{k+1})}\Big)(u_r).
\end{align*}
If~$r<\beta_-$ then~$U(J)^{-1}(u_r)=u_r$ while
$$
S(J,r)=\sum_{j_{r-1}\le t\le j_r-1} q_r^{r-t}u_t=\sum_{r\le t\le r} q_r^{r-t}u_t=u_r.
$$
Suppose that
$r\ge \beta_-$. Then by Lemma~\ref{lem:rev Cox act}
\begin{align*}
U(J)^{-1}(u_r)&=\Big(\ascprod_{\beta_-\le k\le r-1} \Cxr{(k+1)}{(j_k-1)}^{((-1)^{k+1})}\Big)\Cxr{(r+1)}{(j_r-1)}^{((-1)^{r+1})}(u_r)\\
&=\Big(\ascprod_{\beta_-\le k\le r-1} \Cxr{(k+1)}{(j_k-1)}^{((-1)^{k+1})}\Big)(\sum_{r\le t\le j_r-1} q_r^{r-t} u_t).
\end{align*}
If~$\beta_-=r$ then~$r=j_{r-1}$ and
$$
S(J,r)=\sum_{j_{r-1}\le t\le j_r-1} q_r^{r-t}u_t=\sum_{r\le t\le j_r-1} q_r^{r-t}u_t=U(J)^{-1}(u_r).
$$
If~$\beta_-<r$ then, again by Lemma~\ref{lem:rev Cox act}
\begin{align*}
U(&J)^{-1}(u_r)=\Big(\ascprod_{\beta_-\le k\le r-2} \Cxr{(k+1)}{(j_k-1)}^{((-1)^{k+1})}\Big)\Cxr{r}{(j_{r-1}-1)}^{((-1)^{r})}(\sum_{r\le t\le j_r-1} q_r^{r-t} u_t)\\
&=\Big(\ascprod_{\beta_-\le k\le r-2} \Cxr{(k+1)}{(j_k-1)}^{((-1)^{k+1})}\Big)\Big(\sum_{r+1\le t\le j_{r-1}-1} q_r^{r+1-t} u_{t-1}\\&\qquad+
q_r^{r-j_{r-1}-1}(q_r u_{j_{r-1}-1}+u_{j_{r-1}})+\sum_{j_{r-1}+1\le t\le j_r-1} q_r^{r-t}u_t-q_r^{2-r}\sum_{r\le t\le j_{r-1}-1}q_r^{t} u_t\Big)\\
&=\Big(\ascprod_{\beta_-\le k\le r-2} \Cxr{(k+1)}{(j_k-1)}^{((-1)^{k+1})}\Big)\Big(\sum_{r\le t\le j_{r-1}-1} q_r(q_r^{r-1-t}-q_r^{t-r+1}) u_t+\sum_{j_{r-1}\le t\le j_r-1} q_r^{r-t}u_t\Big).
\end{align*}
We claim that
\begin{align*}
U(J)^{-1}(u_r)&=\Big(\ascprod_{\beta_-\le s\le k} \Cxr{(s+1)}{(j_s-1)}^{((-1)^{s+1})}\Big)(\sum_{k+2\le t\le j_{k+1}-1} q_r^{\overline{r-k-1}}(q_r^{k+1-t}-q_r^{t-k-1}) u_t\\
&\qquad +\sum_{k+2\le l\le r-1} \sum_{j_{l-1}\le t\le j_l-1} q_r^{\overline{r-l}}
(q_r^{l-t}-q_r^{t-l})u_t
+\sum_{j_{r-1}\le t\le j_r-1} q_r^{r-t}u_t
\end{align*}
for all $\beta_--1\le k\le r-2$. The case~$k=r-2$ has already been established. For the inductive step, we obtain, using Lemma~\ref{lem:rev Cox act},
\begin{align*}
&\Cxr{(k+1)}{(j_k-1)}^{((-1)^{k+1})}\Big(\sum_{k+2\le t\le j_{k+1}-1} q_r^{\overline{r-k-1}}(q_r^{k+1-t}-q_r^{t-k-1}) u_t\Big)
\\&=\sum_{k+2\le t\le j_k-1} q^{(-1)^{k+1}}q_r^{\overline{r-k-1}}(q_r^{k+1-t}-q_r^{t-k-1}) u_{t-1}+\sum_{j_k+1\le t\le j_{k+1}-1} q_r^{\overline{r-k-1}}(q_r^{k+1-t}-q_r^{t-k-1}) u_t\\
&\qquad+q_r^{\overline{r-k-1}}(q_r^{k+1-j_k}-q_r^{j_k-k-1})(q^{(-1)^{k+1}} u_{j_k-1}+u_{j_k})\\
&=\sum_{k+2\le t\le j_k} q^{(-1)^{k+1}}q_r^{\overline{r-k-1}}(q_r^{k+1-t}-q_r^{t-k-1}) u_{t-1}
+\sum_{j_k\le t\le j_{k+1}-1} q_r^{\overline{r-k-1}}(q_r^{k+1-t}-q_r^{t-k-1}) u_t.
\end{align*}
Since $(-1)^{k+1}+(-1)^r \overline{r-k-1}=(-1)^r ((-1)^{k+1-r}+\overline{r-k-1})=(-1)^r \overline{r-k}$, we obtain
\begin{align*}
\Cxr{(k+1)}{(j_k-1)}^{((-1)^{k+1})}&\Big(\sum_{k+2\le t\le j_{k+1}-1} q_r^{\overline{r-k-1}}(q_r^{k+1-t}-q_r^{t-k-1}) u_t\Big)
\\
&=\sum_{k+1\le t\le j_k-1} q_r^{\overline{r-k}}(q_r^{k-t}-q_r^{t-k}) u_t+\sum_{j_k\le t\le j_{k+1}-1} q_r^{\overline{r-k-1}}(q_r^{k+1-t}-q_r^{t-k-1}) u_t,
\end{align*}
and so
\begin{align*}
U(J)^{-1}(u_r)&=\Big(\ascprod_{\beta_-\le s\le k-1} \Cxr{(s+1)}{(j_s-1)}^{((-1)^{s+1})}\Big)\Big(\sum_{k+1\le t\le j_{k}-1} q_r^{\overline{r-k}}(q_r^{k-t}-q_r^{t-k}) u_t\\
&\qquad +\sum_{k+1\le l\le r-1} \sum_{j_{l-1}\le t\le j_l-1} q_r^{\overline{r-l}}
(q_r^{l-t}-q_r^{t-l})u_t
+\sum_{j_{r-1}\le t\le j_r-1} q_r^{r-t}u_t\Big).
\end{align*}
Taking~$k=\beta_--1$ and noting that $j_{\beta_--1}=\beta_-$, we obtain
\begin{align*}
U(J)^{-1}(u_r)&=\sum_{\beta_-+1\le t\le j_{\beta_-}-1} q_r^{\overline{r-\beta_-}}(q_r^{\beta_--t}-q_r^{t-\beta_-}) u_t\\
&\qquad +\sum_{\beta_-+1\le l\le r-1} \sum_{j_{l-1}\le t\le j_l-1} q_r^{\overline{r-l}}(q_r^{l-t}-q_r^{t-l})u_t
+\sum_{j_{r-1}\le t\le j_r-1} q_r^{r-t}u_t\\
&=\sum_{\beta_-\le t\le j_{\beta_-}-1} q_r^{\overline{r-\beta_-}}(q_r^{\beta_--t}-q_r^{t-\beta_-}) u_t\\
&\qquad +\sum_{\beta_-+1\le l\le r-1} \sum_{j_{l-1}\le t\le j_l-1} q_r^{\overline{r-l}}(q_r^{l-t}-q_r^{t-l})u_t
+\sum_{j_{r-1}\le t\le j_r-1} q_r^{r-t}u_t\\
&=\sum_{j_{\beta_--1}\le t\le j_{\beta_-}-1} q_r^{\overline{r-\beta_-}}(q_r^{\beta_--t}-q_r^{t-\beta_-}) u_t\\
&\qquad +\sum_{\beta_-+1\le l\le r-1} \sum_{j_{l-1}\le t\le j_l-1} q_r^{\overline{r-l}}(q_r^{l-t}-q_r^{t-l})u_t
+\sum_{j_{r-1}\le t\le j_r-1} q_r^{r-t}u_t\\
&=S(J,r).
\end{align*}
Part~\ref{lem:TJ|J| eigenvectors.a} is proven.

To prove part~\ref{lem:TJ|J| eigenvectors.b}, we use
induction on~$m$. Since~$g(J)>1$, the induction base is~$m=2$, that is $J=\{1,j,n+1\}$ for some~$2<j<n$.
Then $U(J)^{-1}=\Cxr2{(j-1)}$ and $U(J)^{-1}(u_j)=q u_{j-1}+u_j=q_2 u_{j-1}+u_j$
by Lemma~\ref{lem:rev Cox act}.

Suppose the claim is proven for all $J'$ with~$|J'|=m+1$, $m\ge 2$ and that~$J$ with~$g(J)>1$ satisfies~$|J|=m+2$. Then by Lemma~\ref{lem:rev Cox act}
\begin{align*}
U(J)^{-1}(u_{j_m})&=U(J\setminus\{j_m\})^{-1}\Cxr{(m+1)}{(j_m-1)}^{((-1)^{m+1})}(u_{j_m})
=U(J\setminus\{j_m\})^{-1}(q_{m+1} u_{j_m-1}+u_{j_m}).
\end{align*}
If~$j_{m-1}<j_m-1$ then $U(J\setminus\{j_m\})^{-1}(q_{m+1} u_{j_m-1}+u_{j_m})=q_{m+1} u_{j_m-1}+u_{j_m}$. But in that case~$\beta_+(J)=m$ and the formula in~\ref{lem:TJ|J| eigenvectors.b} holds. Otherwise,
$j_{m-1}=j_m-1$. If~$g(J\setminus\{j_m\})=1$ then, as $j_{m-1}<n$, it follows
that $j_i=i+1$, $0\le i\le m$ which contradicts the assumption~$g(J)>1$.
Thus, $g(J\setminus\{j_m\})>1$, $\beta_+(J\setminus\{j_m\})=\beta_+
\le m-1$ and
and so by the induction hypothesis,
\begin{align*}
U(J)^{-1}(u_{j_m})&=U(J\setminus\{j_m\})^{-1}(q_m^{-1} u_{j_{m-1}})+u_{j_m}=q_m^{-1}\sum_{j_{\beta_+}-1\le t\le j_{m-1}} q_m^{\overline{j_{m-1}-t}}u_t+u_{j_m}\\
&=q_m^{-1}\sum_{j_{\beta_+}-1\le t\le j_{m}-1} q_m^{1-\overline{j_{m}-t}}u_t+u_{j_m}=\sum_{j_{\beta_+}-1\le t\le j_m} q_{m+1}^{\overline{j_m-t}}u_t.
\end{align*}
To prove part~\ref{lem:TJ|J| eigenvectors.c}, abbreviate~$\epsilon=(-1)^{m+1}$;
thus, $q_{m+1}=q^{\epsilon}$. First we claim that for all
$\beta_+\le k\le m-1$,
$\beta_+\le k\le m-1$,
\begin{align*}
\ascprod_{k\le t\le m-1} \Cxr{(t+1)}{(j_t-1)}^{((-1)^{t+1})}(
w^{(\epsilon)}_{[m+1,n-1]})&=
\ascprod_{k\le t\le m-1} \Cxr{(t+1)}{(n-m+t)}^{((-1)^{t+1})}(
w^{(\epsilon)}_{[m+1,n-1]})\\
&=
q^{\epsilon(m-k-\overline{m-k})}
w^{(\epsilon)}_{[k+1,n-m-1+k]}
\end{align*}
Indeed, for~$k=m-1$ we have $\Cxr{m}{(n-1)}^{(-\epsilon)}(w^{(\epsilon)}_{[m+1,n-1]})=
w^{(\epsilon)}_{[m,n-2]}$
by Corollary~\ref{cor:cox w +-}.
For the inductive step, again by Corollary~\ref{cor:cox w +-}
\begin{align*}
\Cxr{(k+1)}{(n-m+k)}^{((-1)^{k+1})}(q^{\epsilon(m-k-1-\overline{m-k-1})}
w^{(\epsilon)}_{[k+2,n-m+k]})
&=q^{\epsilon(m-k+(-1)^{k-m}-\overline{m-k-1})} w^{(\epsilon)}_{[k+1,n-m+k-1]}\\
&=q^{\epsilon(m-k+(-1)^{k-m}-\overline{m-k-1})} w^{(\epsilon)}_{[k+1,n-m+k-1]},
\end{align*}
since $(-1)^{k-m}-\overline{m-k-1}=-\overline{m-k}$.

Taking~$k=\beta_+$, we obtain
$$
\ascprod_{\beta_+\le t\le m-1} \Cxr{(t+1)}{(j_t-1)}^{((-1)^{t+1})}(w^{(\epsilon)}_{[m+1,n-1]}
=
q^{\epsilon(m-\beta_+-\overline{m-\beta_+})}
w^{(\epsilon)}_{[\beta_++1,n-m-1+\beta_+]}.
$$
Since $n-m+\beta_+=j_{\beta_+}-1$, we obtain
$$
\ascprod_{\beta_+\le t\le m-1} \Cxr{(t+1)}{(j_t-1)}^{((-1)^{t+1})}(w^{(\epsilon)}_{[m+1,n-1]}=q^{\epsilon(m-\beta_+-\overline{m-\beta_+})}
w^{(\epsilon)}_{[\beta_++1,j_{\beta_+}-2]}.
$$
Finally, we prove that for all $\beta_-\le k\le \beta_+$,
\begin{align}
\ascprod_{k\le t\le \beta_+} \Cxr{(t+1)}{(j_t-1)}^{((-1)^{t+1})}(w^{(\epsilon)}_{[m+1,n-1]})&=
q^{\epsilon(m-k-\overline{m-k})} w^{(\epsilon)}_{[k+1,j_k-1-\delta_{k,\beta_+}]}\nonumber\\
&\quad +
\sum_{k+1\le t\le \beta_+} q^{\epsilon(m-t-\overline{m-t})}w^{(\epsilon)}_{[j_t,j_{t+1}-1-\delta_{t,\beta_+}]}.\label{eq:intermed UJ wm+1n-1}
\end{align}
Indeed, the case~$k=\beta_+$ has already been established. For~$k<\beta_+$ note that by definition of~$\beta-+$, $j_k<j_{k+1}-\delta_{k+1,\beta_+}$, $\beta_-+1\le k\le \beta_+-1$.
Thus,
\begin{align*}
\Cxr{(k+1)}{(j_k-1)}^{(-1)^{k+1}}&(q^{\epsilon(m-k-1-\overline{m-k-1})} w^{(\epsilon)}_{[k+2,j_{k+1}-1-\delta_{k+1,\beta_+}]})\\&
=q^{\epsilon((-1)^{k-m}+m-k-\overline{m-k-1})} w^{(\epsilon)}_{[k+1,j_k-1]}+w^{(\epsilon)}_{[j_k,j_{k+1}-1-\delta_{k+1,\beta_+}]}\\
&=q^{\epsilon(m-k-\overline{m-k})} w^{(\epsilon)}_{[k+1,j_k-1]}+w^{(\epsilon)}_{[j_k,j_{k+1}-1-\delta_{k+1,\beta_+}]}.
\end{align*}
Since $\Cxr{(k+1)}{(j_k-1)}^{(-1)^{k+1}}(w^{(\epsilon)}_{[j_t,j_{t+1}-1-\delta_{t,\beta_+}]})=w^{(\epsilon)}_{[j_t,j_{t+1}-1-\delta_{t,\beta_+}]}$,
$k+1\le t\le \beta_+-1$, \eqref{eq:intermed UJ wm+1n-1} follows. It remains to observe that this implies the assertion since $j_{\beta_--1}=\beta_-$.
\end{proof}
\begin{lemma}\label{lem:Euclid jm-1<n}
Let~$\{1,n+1\}\subset J\subset [1,n+1]$ with~$g(J)>1$. Suppose that
$j=\max(J\setminus\{n+1\})<n$. Then
$\la U(J)^{-1}(u_{m-1})\,\vert\,U(J)^{-1}(u_j)\ra\not=0$.
\end{lemma}
\begin{proof}
Let~$m=|J|-1$ and write~$J=\{j_0=1<j_1<\cdots<j_{m-1}=j<j_m=n+1\}$.
Note that by Lemma~\partref{lem:TJ|J| eigenvectors.b},
$U(J)^{-1}(u_j)$ is contained in the span of the $u_t$, $t\in[j_{\beta_+}-1,
j]$ which is orthogonal to all the $u_s$, $s\in [1,j_{\beta_+}-3]$. Thus,
we can consider $U(J)^{-1}(u_{m-1})$ modulo $V'=\sum_{1\le s\le
j_{\beta_+}-3} \kk u_s$.

Suppose first that~$\beta_+(J)=m-1$. Then~$j_{m-2}\le j_{m-1}-2$, hence
by Lemma~\ref{lem:TJ|J| eigenvectors}
$U(J)^{-1}(u_{m-1})=q_m^{j-m-1}u_{j-2}+q_m^{j-m}u_{j-1}\pmod{V'}$,
$U(J)^{-1}(u_j)=q_m u_{j-1}+u_j$ and so by~\eqref{eq:Euclid ui uj}
\begin{align*}
\langle U(J)^{-1}(u_{m-1})\,\vert\,U(J)^{-1}(u_j)\rangle&=\langle q_m^{j-m-1}u_{j-2}+q_m^{j-m}u_{j-1},q_m u_{j-1}+u_j\rangle\\
&=q_m^{j-m}( -2q^{-1}+q_m(1+q^{-2}))=q^{(-1)^m(j-m)-1}(q_m^2-1),
\end{align*}
which is manifestly a non-zero Laurent polynomial in~$q$.

Suppose now that~$\beta_+=\beta_+(J)<m-1$.
Then, in particular, $j_{m-2}=j-1$ and
\begin{align*}
U(J)^{-1}(&u_{m-1})=\sum_{\beta_+\le k\le m-2}\sum_{j_{k-1}\le t\le j_k-1} q_{m-1}^{\overline{m-1-k}}(q_{m-1}^{k-t}-q_{m-1}^{t-k})u_t
+q_{m-1}^{m-j}u_{j-1}\pmod{V'}\\
&=\sum_{j_{\beta_+-1}\le t\le j_{\beta_+}-1} q_{m-1}^{\overline{m-1-\beta_+}}(q_{m-1}^{\beta_+-t}-q_{m-1}^{t-\beta_+})u_t\\
&\quad+\sum_{\beta_++1\le k\le m-2}q_{m-1}^{\overline{m-1-k}}(q_{m-1}^{k-j_{k-1}}-q_{m-1}^{j_{k-1}-k})u_{j_{k-1}}
+q_{m-1}^{m-j}u_{j-1}\pmod{V'}
\\
&=q_{m-1}^{\overline{m-1-\beta_+}}(q_{m-1}^{m+1-j}-q_{m-1}^{j-m-1})u_{j-m+\beta_+-1}
\\
&\quad+(q_{m-1}^{m-j}-q_{m-1}^{j-m})\sum_{\beta_+\le k\le m-2}q_{m-1}^{\overline{m-1-k}}u_{j-m+k}
+q_{m-1}^{m-j}u_{j-1}\pmod{V'}\\
&=q_{m}^{-\overline{m-1-\beta_+}}(q_{m}^{j-m-1}-q_{m}^{m+1-j})u_{j-m+\beta_+-1}
\\
&\quad+(q_{m}^{j-m}-q_{m}^{m-j})\sum_{j-m+\beta_+\le t\le j-2}q_{m}^{-\overline{j-t-1}}u_t
+q_{m}^{j-m}u_{j-1}\pmod{V'}.
\end{align*}
Then, using~\eqref{eq:Euclid ui uj}, we obtain
\begin{align*}
\langle U(&J)^{-1}(u_{m-1})\,\vert\,U(J)^{-1}(u_{j_{m-1}})\rangle =
\langle q_{m}^{-\overline{m-1-\beta_+}}(q_{m}^{j-m-1}-q_{m}^{m+1-j})u_{j-m+\beta_+-1}
\\
&\qquad+(q_{m}^{j-m}-q_{m}^{m-j})\sum_{j-m+\beta_+\le t\le j-2}q_{m}^{-\overline{j-t-1}}u_t
+q_{m}^{j-m}u_{j-1}\,\vert\,
\sum_{j-m+\beta_+ \le t\le j} q_m^{\overline{j-t}} u_t\rangle\\
&=-q_{m}^{2\overline{m-\beta_+}-1}(q_{m}^{j-m-1}-q_{m}^{m+1-j})q^{-1}\\
&\qquad+(q_{m}^{j-m}-q_{m}^{m-j})\langle \sum_{j-m+\beta_+\le t\le j-2}q_{m}^{-\overline{j-t-1}}u_t
+q_{m}^{j-m}u_{j-1}\,\vert\,\sum_{j-m+\beta_+ \le t\le j} q_m^{\overline{j-t}} u_t\rangle\\
&=-q_{m}^{2\overline{m-\beta_+}-1}(q_{m}^{j-m-1}-q_{m}^{m+1-j})q^{-1}+q_m^{j-m+1}(1+q^{-2})-2q^{-1} q_m^{j-m}\\\
&\qquad+(q_{m}^{j-m}-q_{m}^{m-j})\langle \sum_{j-m+\beta_+\le t\le j-2}q_{m}^{-\overline{j-t-1}}u_t
\,\vert\,\sum_{j-m+\beta_+ \le t\le j} q_m^{\overline{j-t}} u_t\rangle\\
&=-q^{-1}(q_{m}^{2\overline{m-\beta_+}}(q_{m}^{j-m-2}-q_{m}^{m-j})-q_m^{j-m}(q_m^2-1))\\\
&\qquad+(q_{m}^{j-m}-q_{m}^{m-j})\langle \sum_{j-m+\beta_+\le t\le j-2}q_{m}^{-\overline{j-t-1}}u_t\,
\vert\,\sum_{j-m+\beta_+ \le t\le j-1} q_m^{\overline{j-t}} u_t\rangle\\
&=-q^{-1}(q_{m}^{2\overline{m-\beta_+}}(q_{m}^{j-m-2}-q_{m}^{m-j})-q_m^{j-m}(q_m^2-1))\\\
&\qquad+(q_{m}^{j-m}-q_{m}^{m-j})q_m^{-1}\langle \sum_{j-m+\beta_+\le t\le j-2}q_{m}^{\overline{j-t}}u_t\,
\vert\,\sum_{j-m+\beta_+ \le t\le j-1} q_m^{\overline{j-t}} u_t\rangle\\
&=-q^{-1}(q_{m}^{2\overline{m-\beta_+}}(q_{m}^{j-m-2}-q_{m}^{m-j})-q_m^{j-m}(q_m^2-1)+q_m^{j-m}-q_m^{m-j})\\\
&\qquad+(q_{m}^{j-m}-q_{m}^{m-j})q_m^{-1}\langle \sum_{j-m+\beta_+\le t\le j-2}q_{m}^{\overline{j-t}}u_t\,
\vert\,\sum_{j-m+\beta_+ \le t\le j-2} q_m^{\overline{j-t}} u_t\rangle\\
&=-q^{-1}(q_{m}^{2\overline{m-\beta_+}}(q_{m}^{j-m-2}-q_{m}^{m-j})-q_m^{j-m}(q_m^2-1)+q_m^{j-m}-q_m^{m-j})\\\
&\qquad+(q_{m}^{j-m}-q_{m}^{m-j})q_m^{-1} \Big((1+q^{-2})\sum_{j-m+\beta_+\le t\le j-2} q_m^{2\overline{j-t}}
-2 q^{-1} \sum_{j-m+\beta_+\le t\le j-3} q_m^{\overline{j-t}+\overline{j-t+1}}\Big)\\
&=-q^{-1}\Big(q_{m}^{2\overline{m-\beta_+}}(q_{m}^{j-m-2}-q_{m}^{m-j})-q_m^{j-m}(q_m^2-1)+(q_m^{j-m}-q_m^{m-j})(2(m-\beta_+)-3))\Big)\\
&\qquad+(q_{m}^{j-m}-q_{m}^{m-j})q_m^{-1} (1+q^{-2})(q_m^2 \lfloor\tfrac12(m-\beta_+-1)\rfloor+\lfloor\tfrac12(m-\beta_+)\rfloor).
\end{align*}
Note that, since~$g(J)>1$, $j<m$ and so $q_m^{j-m}-q_m^{m-j}\not=0$. We can rewrite the above as
$$
\langle U(J)^{-1}(u_{m-1})\,\vert\, U(J)^{-1}(u_{j_{m-1}})\rangle =q_m^{j-m} Q^+_{m-\beta_+,(-1)^m}(q)-q_m^{m-j}Q^-_{m-\beta_+,(-1)^m}(q),
$$
where~$Q^\pm_{r,\epsilon}(q)\in\mathbb Z[q,q^{-1}]$, $r\in\mathbb Z$, $\epsilon\in\{1,-1\}$, are defined by
\begin{align*}
&Q^+_{r,\epsilon}(q)=-q^{-1}(q^{2\epsilon (\overline{r}-1)}-q^{2\epsilon}+2(r-1))+(1+q^{-2})(q^\epsilon\lfloor\tfrac12(r-1)\rfloor+q^{-\epsilon}\lfloor\tfrac12 r\rfloor),
\\
&Q^-_{r,\epsilon}(q)=-q^{-1}(q^{2\epsilon \overline{r}}+2(r-1)-1)+(1+q^{-2})(q^\epsilon \lfloor\tfrac12(r-1)\rfloor+q^{-\epsilon}\lfloor\tfrac12 r\rfloor).
\end{align*}
Since~$1+q^{-2}=q^{-1}(q^\epsilon+q^{-\epsilon})$, $\epsilon\in\{1,-1\}$, we have
\begin{align*}
q &Q^+_{r,\epsilon}(q)=q^{2\epsilon}(1+\lfloor\tfrac12(r-1)\rfloor)-q^{2\epsilon(\overline r-1)}+q^{-2\epsilon}\lfloor\tfrac12 r\rfloor+
2(1-r)+\lfloor\tfrac12(r-1)\rfloor+\lfloor\tfrac12 r\rfloor\\
&=q^{2\epsilon}(1+\lfloor\tfrac12(r-1)\rfloor)+q^{-2\epsilon}\lfloor\tfrac12 r\rfloor+
1-r-q^{2\epsilon(\bar r-1)}\\
&=q^{2\epsilon}(\lfloor\tfrac12r\rfloor+\bar r)+q^{-2\epsilon}\lfloor\tfrac12 r\rfloor+
1-2\lfloor\tfrac12 r\rfloor-\bar r-q^{2\epsilon(\bar r-1)}\\
&=\lfloor\tfrac12 r\rfloor(q^\epsilon-q^{-\epsilon})^2
+\bar r(q^{2\epsilon}-1)+1-q^{2\epsilon(\bar r-1)}=(q-q^{-1})(\lfloor\tfrac12 r\rfloor(q-q^{-1})+\epsilon q^{\epsilon(2\bar r-1)}).
\end{align*}
Similarly,
\begin{align*}
q Q^-_{r,\epsilon}(q)&=(q-q^{-1})^2\lfloor\tfrac12 r\rfloor+1-q^{2\epsilon \bar r}
+(1-\bar r)(1-q^{2\epsilon})\\
&=(q-q^{-1})( \lfloor\tfrac12 r\rfloor(q-q^{-1})-\epsilon q^\epsilon).
\end{align*}
Thus, $Q^\pm_{r,\epsilon}(q)/(1-q^{-2})\in \mathbb Z[q,q^{-1}]$ and
equals~$\pm \epsilon$ at~$q=1$.
It follows that
$$\la U(J)^{-1}(u_{m-1})\,\vert\, U(J)^{-1}(u_{j_{m-1}})\ra/(1-q^{-2})
\in\mathbb Z[q,q^{-1}],
$$
equals~$2(-1)^m$ at~$q=1$ and hence
is non-zero.
\end{proof}
\begin{lemma}\label{lem:Euclid jm-1=n}
Suppose that~$g(J)>1$ and~$j_{m-1}=n$. Let~$\epsilon=(-1)^{m-1}$. Then
$$
\langle U(J)^{-1}(u_{m-1})\,\vert\, U(J)^{-1}(w^{(\epsilon)}_{[m+1,n-1]})\rangle\not=0.
$$
\end{lemma}
\begin{proof}
As before, we abbreviate~$\beta_\pm=\beta_\pm(J)$. By definition, $\beta_-\le\beta_+$.
Since $g(J)>1$, we must have $\beta_-\le \beta_+-1$ for otherwise $\beta_-=\beta_+$ and so $J=[1,\beta_--1]\cup [n-m+\beta_-+1,n+1]$.
Write
$$
U(J)^{-1}(q^{-\epsilon} u_{m-1})=\sum_{\beta_-\le t\le n-1} c_t u_t,\qquad U(J)^{-1}( q^{-\epsilon m} w^{(\epsilon)}_{[m+1,n-1]})=\sum_{\beta_-+1\le t \le j_{\beta_+}-2} c'_t u_t,
$$
where by Lemma~\ref{lem:TJ|J| eigenvectors}
$$
c_t=q^{-\epsilon\overline{m-\kappa(t)}}(q^{\epsilon(\kappa(t)-t)}-(1-\delta_{m-1,\kappa(t)})q^{\epsilon(t-\kappa(t))}),\quad
c'_t=q^{\epsilon(t-\kappa(t)-\overline{m-\kappa(t)})},
$$
and $\kappa(t)=\{ k\in[1,m]\,:\, j_{k-1}\le t\le j_k-1\}$.
Thus, by~\eqref{eq:Euclid ui uj}
\begin{align*}
q^{-\epsilon(m+1)} \langle U(J)^{-1} &u_{m-1}\,\vert\, U(J)^{-1}(w^{(\epsilon)}_{[m+1,n-1]})\rangle=(1+q^{-2})\sum_{\beta_-+1\le t \le j_{\beta_+}-2} c_t c'_t
\\&-q^{-1}\sum_{\beta_-\le t\le j_{\beta_+}-3} c_t c'_{t+1}-q^{-1}\sum_{\beta_-+1\le t\le j_{\beta_+}-2} c_{t+1}c'_t
\end{align*}
Note that $\kappa(t+1)=\kappa(t)$ unless $t=j_{\kappa(t)}-1$ in which case $\kappa(t+1)=\kappa(t)+1$.
We have
\begin{align*}
&c_t c'_t=q^{-2\epsilon \overline{m-\kappa(t)}}(1-(1-\delta_{m-1,\kappa(t)})q^{2\epsilon(t-\kappa(t))})
\intertext{
while}
&c_t c'_{t+1}=\begin{cases}
q^{\epsilon}c_t c'_t,&t\not=j_{\kappa(t)}-1,\\
q^{-\epsilon}(1-(1-\delta_{m-1,\kappa(t)})q^{2\epsilon(t-\kappa(t))}),&t=j_{\kappa(t)}-1,
\end{cases}
\\
\intertext{and}
&c'_t c_{t+1}=\begin{cases}
q^{-\epsilon(1+2\overline{m-\kappa(t)})}(1-(1-\delta_{m-1,\kappa(t)})q^{2\epsilon(1+t-\kappa(t))}),&t\not=j_{\kappa(t)}-1,\\
q^{-\epsilon}(1-(1-\delta_{m-1,\kappa(t)+1})q^{2\epsilon(t-\kappa(t))}),&t=j_{\kappa(t)}-1.
\end{cases}
\end{align*}
Thus, for~$\beta_-<k<\beta_+$,
\begin{align*}
\sum_{j_{k-1}\le t\le j_k-1} &c_t c'_t
=q^{-2\epsilon\overline{m-k}}\Big(
j_k-j_{k-1}-q^{-2\epsilon k}\frac{q^{2\epsilon j_k}-q^{2\epsilon j_{k-1}}}{q^{2\epsilon}-1}\Big),\\
\sum_{j_{k-1}\le t\le j_k-1} c_t c'_{t+1}&=
q^{\epsilon(1-2\overline{m-k})}\Big(
j_k-j_{k-1}-1-q^{-2\epsilon k}\,\frac{q^{2\epsilon (j_k-1)}-q^{2\epsilon j_{k-1}}}{q^{2\epsilon}-1}\Big)\\
&\qquad+q^{-\epsilon}(1-q^{2\epsilon(j_k-k-1)}),\\
\sum_{j_{k-1}\le t\le j_k-1} c_{t+1} c'_{t}&=
q^{-\epsilon(1+2\overline{m-k})}\Big(j_k-j_{k-1}-1-q^{-2\epsilon k}\,\frac{q^{2\epsilon j_k}-q^{2\epsilon(j_{k-1}+1)}}{q^{2\epsilon}-1}\Big)\\
&\qquad
+q^{-\epsilon}(1-(1-\delta_{m-2,k})q^{2\epsilon(j_k-1-k)}).
\end{align*}
Since $q^{-1}(q^\epsilon+q^{-\epsilon})=1+q^{-2}$, $\epsilon\in\{1,-1\}$
we obtain, for $\beta_-<k<\beta_+$,
\begin{align*}
(1+&q^{-2})\sum_{j_{k-1}\le t\le j_k-1} c_t c'_t-q^{-1}\sum_{j_{k-1}\le t\le j_k-1} (c_t c'_{t+1}+c_{t+1}c'_t)\\
&=q^{-2\epsilon\overline{m-k}}\Big( (1+q^{-2})(j_k-j_{k-1})-q^{-1}(q^\epsilon+q^{-\epsilon})(j_k-j_{k-1})+q^{-1}(q^\epsilon+q^{-\epsilon})\\
&\quad-q^{-2\epsilon k}(q^{2\epsilon}-1)^{-1}\Big( (1+q^{-2})(q^{2\epsilon j_k}-q^{2\epsilon j_{k-1}})-q^{-1+\epsilon}(q^{2\epsilon (j_k-1)}-q^{2\epsilon j_{k-1}})\\
&\quad-q^{-1-\epsilon}(q^{2\epsilon j_k}-q^{2\epsilon(j_{k-1}+1)})\Big)\Big)-q^{-1-\epsilon}(2-(2-\delta_{m-2,k})q^{2\epsilon(j_k-1-k)})\\
&=q^{-2\epsilon\overline{m-k}}\Big( (1+q^{-2})
-q^{-2\epsilon k}(q^{2\epsilon}-1)^{-1}\Big( q^{2\epsilon j_k}((1+q^{-2})-2q^{-1-\epsilon})\\
&\quad-q^{2\epsilon j_{k-1}}((1+q^{-2})-2q^{-1+\epsilon})\Big)\Big)-q^{-1-\epsilon}(2-(2-\delta_{m-2,k})q^{2\epsilon(j_k-1-k)})\\
&=q^{-2\epsilon\overline{m-k}}( (1+q^{-2})-q^{-1-\epsilon(2k+1)}(q^{2\epsilon j_k}+q^{2\epsilon j_{k-1}}))-q^{-1-\epsilon}(2-(2-\delta_{m-2,k})q^{2\epsilon(j_k-1-k)}).
\end{align*}
For $k=\beta_-$ we have
\begin{align*}
(1+&q^{-2})\sum_{\beta_-+1\le t\le j_{\beta_-}-1} c_t c'_t-q^{-1}\sum_{\beta_-\le t\le j_{\beta_-}-1} c_t c'_{t+1}-q^{-1}\sum_{\beta_-+1\le t\le j_{\beta_-}-1}
c_{t+1}c'_t)\\
&=(1+q^{-2})q^{-2\epsilon\overline{m-\beta_-}}\Big(j_{\beta_-}-\beta_--1-\frac{q^{2\epsilon(j_{\beta_-}-\beta_-)}-q^{2\epsilon}}{q^{2\epsilon}-1}\Big)
\\&\qquad-q^{-1+\epsilon(1-2\overline{m-\beta_-})}\Big( j_{\beta_-}-\beta_--1-\frac{q^{2\epsilon(j_{\beta_-}-\beta_--1)}-1}{q^{2\epsilon-1}}\Big)-q^{-1-\epsilon}(1-q^{2\epsilon(j_{\beta_-}-\beta_--1)})\\
&\qquad-q^{-1-\epsilon(1+2\overline{m-\beta_-})}\Big(j_{\beta_-}-\beta_--2-\frac{q^{2\epsilon(j_{\beta_-}-\beta_-)}-q^{4\epsilon}}{q^{2\epsilon}-1}\Big)\\
&\qquad-q^{-1-\epsilon}(1-(1-\delta_{m-2,\beta_-})q^{2\epsilon(j_{\beta_-}-1-\beta_-)})\\
&=q^{-2\epsilon\overline{m-\beta_-}}\Big(q^{-1-\epsilon}-(q^{2\epsilon}-1)^{-1}\Big( (1+q^{-2})(q^{2\epsilon(j_{\beta_-}-\beta_-)}-q^{2\epsilon})\\
&\qquad+q^{-1+\epsilon}(q^{2\epsilon(j_{\beta_-}-\beta_--1)}-1)+q^{-1-\epsilon}(q^{2\epsilon(j_{\beta_-}-\beta_-)}-q^{4\epsilon})\Big)\Big)\\
&\qquad
-q^{-1-\epsilon}(2-(2-\delta_{m-2,\beta_-})q^{2\epsilon(j_{\beta_-}-1-\beta_-)})\\
&=q^{-(1+\epsilon(1+2\epsilon\overline{m-\beta_-})}(1-q^{2\epsilon(j_{\beta_-}-\beta_-)})
-q^{-1-\epsilon}(2-(2-\delta_{m-2,\beta_-})q^{2\epsilon(j_{\beta_-}-1-\beta_-)}),
\end{align*}
while for~$k=\beta_+$,
\begin{align*}
(1&+q^{-2})\sum_{j_{\beta_+-1}\le t\le j_{\beta_+}-2} c_t c'_t-q^{-1}\sum_{j_{\beta_+-1}\le t\le j_{\beta_+}-3} c_t c'_{t+1}-q^{-1} \sum_{j_{\beta_+-1}\le t\le j_{\beta_+}-2} c_{t+1}c'_t\\
&=(1+q^{-2})q^{-2\epsilon \overline{m-\beta_+}}\Big(j_{\beta_+}-j_{\beta_+-1}-1-(1-\delta_{m-1,\beta_+})q^{-2\epsilon\beta_+}\,\frac{q^{2\epsilon (j_{\beta+}-1)}-q^{2\epsilon j_{\beta_+-1}}}{q^{2\epsilon}-1}\Big)\\
&\qquad-q^{-1+\epsilon(1-2\overline{m-\beta_+})}\Big(j_{\beta_+}-j_{\beta_+-1}-2-(1-\delta_{m-1,\beta_+})q^{-2\epsilon\beta_+}\,\frac{q^{2\epsilon (j_{\beta+}-2)}-q^{2\epsilon j_{\beta_+-1}}}{q^{2\epsilon}-1}\Big)\\
&\qquad-q^{-1-\epsilon(1+2\overline{m-\beta_+})}\Big(j_{\beta_+}-j_{\beta_+-1}-1-(1-\delta_{m-1,\beta_+})q^{-2\epsilon\beta_+}\,\frac{q^{2\epsilon j_{\beta+}}-q^{2\epsilon (j_{\beta_+-1}+1)}}{q^{2\epsilon}-1}\Big)\\
&=
-(1-\delta_{m-1,\beta_+})q^{-2\epsilon (\beta_++\overline{m-\beta_+})}(q^{2\epsilon}-1)^{-1}\Big(q^{2\epsilon j_{\beta_+}}((1+q^{-2})q^{-2\epsilon}-q^{-1-\epsilon}-q^{-1-3\epsilon})
\\
&\qquad-q^{2\epsilon j_{\beta_--1}}((1+q^{-2})-2q^{-1+\epsilon})\Big)\\
&=q^{-1-\epsilon}(q^{2\epsilon(1-\overline{m-\beta_+})}-(1-\delta_{m-1,\beta_+})q^{2\epsilon(j_{\beta_+-1}-\beta_+-\overline{m-\beta_+})}).
\end{align*}
Let~$z=q^{2\epsilon}$. Since~$\beta_-\le \beta_+-1$ we obtain
\begin{align*}
&q^{1-\epsilon m} \la U(J)^{-1} u_{m-1}\,\vert\, U(J)^{-1}(w^{(\epsilon)}_{[m+1,n-1]})\ra
=z^{1-\overline{m-\beta_+}}-(1-\delta_{m-1,\beta_+})z^{j_{\beta_+-1}-\beta_+-\overline{m-\beta_+}}\\&\quad+
z^{-\overline{m-\beta_-}}(1-z^{j_{\beta_-}-\beta_-})
-(2-(2-\delta_{m-2,\beta_-})z^{j_{\beta_-}-1-\beta_-})\\
&\quad+\sum_{\beta_-+1\le k\le \beta_+-1} \Big(z^{-\overline{m-k}}( (1+z)-z^{- k}(z^{j_k}+z^{j_{k-1}}))-2+(2-\delta_{m-2,k})z^{j_k-1-k}\Big)\\
&=z^{1-\overline{m-\beta_+}}-(1-\delta_{m-1,\beta_+})z^{j_{\beta_+-1}-\beta_+-\overline{m-\beta_+}}\\&\quad+
z^{-\overline{m-\beta_-}}(1-z^{j_{\beta_-}-\beta_-})
+(2-\delta_{m-2,\beta_-})z^{j_{\beta_-}-1-\beta_-}-2(\beta_+-\beta_-)\\
&\quad+(1+z)\Big(z^{\overline{m-\beta_-}-1}\lfloor\tfrac12(\beta_+-\beta_-)\rfloor+z^{-\overline{m-\beta_-}}\lfloor\tfrac12(\beta_+-\beta_--1)\rfloor\Big)\\
&\quad-\sum_{\beta_-+1\le k\le \beta_+-1} \Big(z^{-\overline{m-k}-k}(z^{j_k}+z^{j_{k-1}}))-(2-\delta_{m-2,k})z^{j_k-1-k}\Big)\\
&=z^{1-\overline{m-\beta_+}}+z^{-\overline{m-\beta_-}}-2(\beta_+-\beta_-)\\
&\quad+(1+z)\Big(z^{\overline{m-\beta_-}-1}\lfloor\tfrac12(\beta_+-\beta_-)\rfloor+z^{-\overline{m-\beta_-}}\lfloor\tfrac12(\beta_+-\beta_--1)\rfloor\Big)\\
&\quad-\sum_{\beta_-+1\le k\le \beta_+} (1-\delta_{m-1,k}) z^{j_{k-1}-k-\overline{m-k}}-\sum_{\beta_-\le k\le \beta_+-1} z^{j_k-k}(z^{-\overline{m-k}}-(2-\delta_{k,m-2})z^{-1})\\
&=z^{1-\overline{m-\beta_+}}+z^{-\overline{m-\beta_-}}-2(\beta_+-\beta_-)\\
&\quad+(1+z)\Big(z^{\overline{m-\beta_-}-1}\lfloor\tfrac12(\beta_+-\beta_-)\rfloor+z^{-\overline{m-\beta_-}}\lfloor\tfrac12(\beta_+-\beta_--1)\rfloor\Big)\\
&\quad-\sum_{\beta_-\le k\le \beta_+-1} z^{j_k-k}(1-z^{\overline{m-k}-1})(1-(1-\delta_{k,m-2})z^{-1})\\
&=z^{1-\overline{m-\beta_+}}+z^{-\overline{m-\beta_-}}-2(\beta_+-\beta_-)-\sum_{\beta_-\le k\le \beta_+-1} z^{j_k-k}(1-z^{\overline{m-k}-1})^{2-\delta_{k,m-2}}\\
&\quad+(1+z)\Big(z^{\overline{m-\beta_-}-1}\lfloor\tfrac12(\beta_+-\beta_-)\rfloor+z^{-\overline{m-\beta_-}}\lfloor\tfrac12(\beta_+-\beta_--1)\rfloor\Big).
\end{align*}
Denote this expression by~$Q_J(z)$. Clearly, $Q_J(z)\in\mathbb Z[z,z^{-1}]$
and, since $j_k-k\ge 2$ for all~$k\ge \beta_-$,
$z^{j_k-k}(1-z^{\overline{m-k}-1})^{2-\delta_{k,m-2}}\in
\mathbb Z[z]$ for all~$k\ge \beta_-$. Therefore,
$$
\operatorname{Res}_{z=0}Q_J(z)=
\begin{cases}
\lfloor\tfrac12(\beta_+-\beta_-)\rfloor,&\overline{m-\beta_-}=0,\\
1+\lfloor\tfrac12(\beta_+-\beta_-)\rfloor,&\overline{m-\beta_-}=1.
\end{cases}
$$
In particular, since $\beta_-\le \beta_+-1$, $\operatorname{Res}_{z=0} Q_J(z)
\not=0$, and hence~$Q_J(z)\not=0$, unless
$\beta_+=\beta_-+1$ and~$\overline{m-\beta_-}=0$. But in that case
\begin{equation*}
Q_J(z)=-z^{j_{\beta_-}-\beta_-}(1-z^{-1})^{2-\delta_{k,m-2}}\not=0.\qedhere
\end{equation*}
\end{proof}
Let~$m=|J|-1$.
By Proposition~\partref{prop:TJ eigenvectors.b}, $u_{m-1}\in\ker(T_{[1,m]\cup\{n+1\}}^{m+1}-q^{2(n+1)}\id_V)$.
If~$j_{m-1}<n$ then, since~$g(J)>1$, $j_{m-1}\ge m+1$ and so
$$u_{j_{m-1}}
\in \ker(T_{[1,m]\cup\{n+1\}}-q^2\id_V)=
\ker(T_{[1,m]\cup\{n+1\}}^{m+1}-q^{2(m+1)}\id_V)$$
by Proposition~\partref{prop:TJ eigenvectors.a}.
If~$j_{m-1}=n$ then $$w^{((-1)^m)}_{[m+1,n-1]}
\in \ker(T_{[1,m]\cup \{n+1\}}-q^2\id_V)=
\ker(T_{[1,m]\cup\{n+1\}}^{m+1}-q^{2(m+1)}\id_V)$$ by
Proposition~\partref{prop:TJ eigenvectors.a}.
Proposition~\ref{prop:red} is now immediate from Lemmata~\ref{lem:Euclid jm-1<n}
and~\ref{lem:Euclid jm-1=n}.
\end{proof}
Theorem~\ref{thm:adm I2m converse } is proven.
\end{proof}

\begin{proposition}\label{prop:eigenvectors even}
Let $\tilde J_m=[1,m]\cup[n+2-m,n+1]$, $1\le m\le \frac12 n$.
Then
\begin{enmalph}
\item\label{prop:eigenvectors even.a} $\{ u_i\,:\,m+1\le i\le n-m\}\subset \ker(T_{\tilde J_m}-q^{2}\id_V)$;
\item\label{prop:eigenvectors even.b} $\{ u_i\mp u_{n+1-i}:\,1\le i\le m-1\}\cup\{q^{n+1}w^{(-1)}_{[1,n]}\mp w^{(1)}_{[1,n]}\} \subset \ker(T_{\tilde J_m}^m\mp q^{n+1}\id_V)$,
\item\label{prop:eigenvectors even.c} $T_{\tilde J_m}^m$ is diagonalizable on~$V$ and
$$
\det(t \id_V-T_{\tilde J_m}^m)=(t-1)(t^2-q^{2(n+1)})^m (t-q^{2m})^{n-2m}.
$$
\end{enmalph}
\end{proposition}
\begin{proof}
By Lemma~\ref{lem:Transp action}, $T_{(m+1,n+2-m)}(u_i)=q^2 u_i$ for all~$i\in[m+1,n-m]$.
Since~$T_{\tilde J_m}$ is the product of~$T_{(m+1,n+2-m)}$ and the $T_j$ with $j\in[1,m]\cup[n+2-m,n]$ which fix the~$u_i$ with~$i\in[m+1,n-m]$, part~\ref{prop:eigenvectors even.a} follows.

To prove~\ref{prop:eigenvectors even.b},
recall that $
T_{\tilde J_m}^m=T_{w_\circ^{[1,n]}}T_{w_\circ^{[m+1,n-m]}}^{-1}
$
by Corollary~\ref{cor:symm even J}. By Lemma~\ref{lem:Tw0ij act ui}, $T_{w_\circ^{[1,n]}}(u_i)
=-q^{n+1}u_{n+1-i}$ for all~$i\in [1,n]$ while
$T_{w_\circ^{[m+1,n-1]}}(u_i)=u_i$ for all $i\in[1,m-1]\cup
[n+2-m,n]$. It follows that $u_i\pm u_{n+1-m}\in
\ker(T_{\tilde J_m}^m\mp q^{n+1}\id_V)$, $i\in[1,m-1]$.
Furthermore, by Lemma~\ref{lem:Tw0ij act ui}
$$
T_{w_\circ^{[1,n]}}(w^{(-1)}_{[1,n]})=
-\sum_{1\le t\le n}q^{-t+n+1}u_{n+1-t}=-w^{(1)}_{[1,n]},
$$
while
$$
T_{w_\circ^{[1,n]}}(w^{(1)}_{[1,n]})=-\sum_{1\le t\le n}q^{t+n+1}u_{n+1-t}
=-q^{2(n+1)}w^{(-1)}_{[1,n]},
$$
whence $q^{n+1}w^{(-1)}_{[1,n]}\mp w^{(1)}_{[1,n]}\in\ker(T_{w_\circ^{[1,n]}}\mp q^{n+1}\id_V)$. Since
$w^{(\pm1)}_{[1,n]}\in\ker(T_{w_\circ^{[m+1,n-m]}}-\id_V)$ by Lemma~\ref{lem:Tw0ij act ui}, part~\ref{prop:eigenvectors even.b} follows.

To prove~\ref{prop:eigenvectors even.c}, it remains to show that $q^{n+1}w^{(-1)}_{[1,n]}\mp w^{(1)}_{[1,n]}$ is not contained in the span of manifestly linearly
independent $\{ u_i\mp u_{n+1-i}\,:\, 1\le i\le m-1\}$. But we have
\begin{align*}
q^{n+1}w^{(-1)}_{[1,n]}\mp w^{(1)}_{[1,n]}&=\sum_{1\le t\le n} q^{n+1-t}u_t\mp \sum_{1\le t\le n} q^t u_t=\sum_{1\le t\le n} q^{n+1-t}(u_t\mp u_{n+1-t}).
\end{align*}
If~$n$ is odd, we obtain
$$
q^{n+1}w^{(-1)}_{[1,n]}-w^{(1)}_{[1,n]}=\sum_{1\le t\le \frac12(n-1)} (q^{n+1-t}- q^t)(u_t- u_{n+1-t})
$$
and
$$
q^{n+1}w^{(-1)}_{[1,n]}+w^{(1)}_{[1,n]}=\sum_{1\le t\le \frac12(n-1)} (q^{n+1-t}+ q^t)(u_t- u_{n+1-t})+2 q^{\frac12(n+1)} u_{\frac12(n+1)},
$$
while for~$n$ even
$$
q^{n+1}w^{(-1)}_{[1,n]}\mp w^{(1)}_{[1,n]}=\sum_{1\le t\le \frac12n} (q^{n+1-t}\mp q^t)(u_t\mp u_{n+1-t}).
$$
In either case, since all vectors appearing in the right hand side are linearly independent, it follows that $q^{n+1}w^{(-1)}_{[1,n]}\mp w^{(1)}_{[1,n]}$ is not
contained in the span of any proper subfamily of these vectors.
\end{proof}
\begin{corollary}
For any~$\{1,n+1\}\subset J\subset [1,n+1]$ with~$|J|$ even,
$T_J^{|J|/2}$ is diagonalizable on~$V$ and
$\det(t\id_V-T_J^{|J|/2})=(t-1)(t^2-q^{2(n+1)})^{|J|/2}
(t-q^{|J|})^{n-|J|}$.
\end{corollary}

\begin{proposition}\label{prop:|J| even}
Let~$\{1,n+1\}\subset J\subset [1,n+1]$ with~$2<|J|=2m<n+1$.
Then the assignments~$T_r\mapsto \tau_{\bar r}(J)$, $r\in\{1,2\}$
define a homomorphism~$\Br^+(I_2(|J|))\to\Br^+_{n+1}$
if and only if~$J=\tilde\sigma(J)$.
\end{proposition}
\begin{proof}
By Theorems~\ref{thm:adm I2m} and~\ref{thm:adm I2m converse },
it only remains to prove that~$T_J^{m}$ is not
${}^{op}$-invariant when~$g(J)=1$ and~$J\not=\tilde\sigma(J)$.
By Proposition~\ref{prop:eigenvectors even} and Lemma~\ref{lem:orth eigenspaces}, it suffices to prove that
$$
\la \ker(T_J^m-q^{n+1}\id_V)\,|\,\ker(T_J^m+q^{n+1}\id_V)\ra\not=\{0\}.
$$
By Lemma~\ref{lem:diag aut TJ}, it suffices to consider
$J=J(r,m):=[1,m-r]\cup[n+2-r-m,n+1]$ with $1\le r\le m-1$.
Denote $\tilde U(r,m):=U(\tilde J_m)^{-1}U(J(r,m))$
where~$\tilde J_m=J(0,m)=[1,m]\cup[n+2-m,n+1]=
[1,m]\cup\tilde\sigma([1,m])$.
Since by Corollary~\ref{cor:conj J}
$$
U(J(r,m))T_{J(r,m)}U(J(r,m))^{-1}=T_{[1,2m-1]\cup\{n+1\}}=U(\tilde J_m)T_{\tilde J_m}U(\tilde J_m)^{-1},
$$
we obtain
\begin{equation}\label{eq:T J(r,n) from T tilde J_m}
T_{J(r,m)}=\tilde U(r,m)^{-1} T_{\tilde J_m}\tilde U(r,m),
\end{equation}
where
$$
\tilde U(r,m):=U(\tilde J_m)^{-1}U(J(r,m))=\dscprod_{m-r+1\le k\le m} \Cx{k}{(n-2m+k)}^{((-1)^{k+1})}
$$
by~\eqref{eq:U(J) defn}.
Set
$x_{m-r}=\tilde U(r,m)^{-1}(u_{m-1})$, $1\le r\le m-1$.
Since $\tilde U(r,m)^{-1}(u_{n+2-m})=u_{n+2-m}$
by Lemma~\ref{lem:rev Cox act},
$x_{m-r}\pm u_{n+2-m}\in\ker(T_{J(r,m)}^m\pm q^{n+1}\id_V)$
by~\eqref{eq:T J(r,n) from T tilde J_m} and Proposition~\partref{prop:eigenvectors even.b}.
Therefore, it suffices to prove that
$$
\langle x_{m-r}-u_{n+2-m}\,|\, x_{m-r}+u_{n+2-m}\rangle=\langle x_{m-r}\,|\,x_{m-r}\rangle-(1+q^{-2})\not=0,\qquad 1\le r\le m-1.
$$

First, by Lemma~\ref{lem:rev Cox act}
$$
x_{m-1}=\Cxr m{(n-m)}^{((-1)^m)}(u_{m-1})=q_m^{-(m-1)} w^{((-1)^m)}_{[m-1,n-m]}=\sum_{m-1\le t\le n-m} q_m^{t-m+1}u_t.
$$
Then by~\eqref{eq:Euclid ui uj}
\begin{align*}
\langle x_{m-1}\,&|\,x_{m-1}\rangle=(1+q^{-2})\sum_{m-1\le t\le n-m} q_m^{2(t-m+1)}-2q^{-1}\sum_{m-1\le t\le n-m-1} q_m^{2(t-m+1)+1}
\\&=q^{-1}\Big((q_m+q_m^{-1})\frac{q_m^{2(n-2m+2)}-1}{q_m^2-1}-2q_m \frac{q^{2(n-2m+1)}-1}{q_m^2-1}\Big)
=q^{-1}(q_m^{-1}+q_m^{2(n-2m+1)+1}).
\end{align*}
Thus, since~$q$ is not a root of unity and~$2m<n+1$
$$
\langle x_{m-1}\,|\,x_{m-1}\rangle-(1+q^{-2})=q^{(-1)^m-1}(1-q_m^{2(n-2m+1)})\not=0.
$$
Next we claim that for~$2\le r\le m-1$,
\begin{align}
x_{m-r}&=q_m^{-\overline{r-1}}\sum_{m-r+1\le t\le n-m-r+1} (q_m^{t+r-m}-q_m^{m-r-t})u_t\nonumber\\
&\qquad+
(q_m^{n-2m+1}-q_m^{2m-n-1})\sum_{n-m-r+2\le t\le n-m-1} q_m^{-\overline{t-n+m}}u_t+q_m^{n-2m+1}u_{n-m}.\label{eq:x m-r}
\end{align}
Indeed, it is immediate from the definition of~$x_{m-r}$ that
$
x_{m-r-1}=\Cxr{m-r}{(n-m-r)}^{((-1)^{m+r})}(x_{m-r})
$.
Thus, by Lemma~\ref{lem:rev Cox act}
\begin{align*}
x_{m-2}=\Cxr{(m-1)}{(n-m-1)}^{((-1)^{m+1})}(&x_{m-1})=\sum_{m\le t\le n-m+1} q_m^{t-m}u_{t-1}
-\sum_{m-1\le t\le n-m-1} q_m^{m-t-3}\\
&=q_m^{-1}\sum_{m-1\le t\le n-m-1} (q_m^{t+2-m}-q_m^{m-t-2})u_t+q_m^{n-2m+1}u_{n-m},
\end{align*}
which is~\eqref{eq:x m-r} with~$r=2$.
For the inductive step we have
\begin{align*}
x_{m-r-1}&=\Cxr{m-r}{(n-m-r)}^{((-1)^{m+r})}\Big(q_m^{-\overline{r-1}}\sum_{m-r+1\le t\le n-m-r+1} (q_m^{t+r-m}-q_m^{m-r-t})u_t\Big)+
\\
&\qquad+(q_m^{n-2m+1}-q_m^{2m-n-1})\sum_{n-m-r+2\le t\le n-m-1} q_m^{-\overline{t-n+m}}u_t+q_m^{n-2m+1}u_{n-m}\\
&=q_m^{-\overline{r-1}}q_{m+r}\sum_{m-r+1\le t\le n-m-r} (q_m^{t+r-m}-q_m^{m-r-t})u_{t-1}\\
&\quad +q_m^{-\overline{r-1}}(q_m^{n-2m+1}-q_m^{2m-n+1})(q_{m+r}u_{n-m-r}+u_{n-m-r+1})\\
&\qquad+(q_m^{n-2m+1}-q_m^{2m-n-1})\sum_{n-m-r+2\le t\le n-m-1} q_m^{-\overline{t-n+m}}u_t+q_m^{n-2m+1}u_{n-m}.
\\
\intertext{Since $q_m^{-\overline{r-1}}q_{m+r}=q_m^{(-1)^r-\overline{r-1}}=q_m^{-\overline r}$ by~\eqref{eq: q r+s},}
x_{m-r-1}
&=q_m^{-\overline{r}}\sum_{m-r\le t\le n-m-r-1} (q_m^{t+1+r-m}-q_m^{m-r-1-t})u_{t}+q_m^{-\overline{r}}(q_m^{n-2m+1}-q_m^{2m-n+1})u_{n-m-r}\\
&\qquad+(q_m^{n-2m+1}-q_m^{2m-n-1})\sum_{n-m-r+1\le t\le n-m-1} q_m^{-\overline{t-n+m}}u_t+q_m^{n-2m+1}u_{n-m}\\
&=q_m^{-\overline{r}}\sum_{m-r\le t\le n-m-r} (q_m^{t+1+r-m}-q_m^{m-r-1-t})u_{t}\\
&\qquad+(q_m^{n-2m+1}-q_m^{2m-n-1})\sum_{n-m-r+1\le t\le n-m-1} q_m^{-\overline{t-n+m}}u_t+q_m^{n-2m+1}u_{n-m},
\end{align*}
which proves the inductive step and hence~\eqref{eq:x m-r}.
Using~\eqref{eq:Euclid ui uj} and~\eqref{eq:x m-r} we obtain
\begin{align*}
q\langle x_{m-r}&\,|\,x_{m-r}\rangle =(q_m + q_m^{-1}) \Big(\sum_{m-r+1\le t\le n-m-r+1}
     q_m^{-2 \overline{r - 1}} (q_m^{t + r - m} - q_m^{m - t - r})^2 \\
     &\quad+ (q_m^{n + 1 - 2 m} - q_m^{2 m - n - 1})^2 \sum_{n-m-r+2\le t\le n-m-1}
      q_m^{-2 \overline{t - n + m}} +
    q_m^{2 (n + 1 - 2 m)}\Big) \\
&\qquad-
 2 \Big (\sum_{m-r+1\le t\le n-m-r} q_m^{-2 \overline{r - 1}} (q_m^{t + r - m} -
        q_m^{m - t - r}) (q_m^{t + 1 + r - m} - q_m^{m - t - r - 1})\\
&\qquad\qquad + (q_m^{n + 1 - 2 m} - q_m^{2 m - n - 1})^2
      q_m^{-1}(r-2)+
    q_m^{n - 2 m} (q_m^{n + 1 - 2 m} - q_m^{2 m - n - 1})\Big)\\
&=(q_m + q_m^{-1}) \Big(\sum_{m-r+1\le t\le n-m-r+1}
     q_m^{-2 \overline{r - 1}} (q_m^{t + r - m} - q_m^{m - t - r})^2 \\
     &\qquad+ (q_m^{n + 1 - 2 m} - q_m^{2 m - n - 1})^2 \big( q_m^{-2\overline r}\lfloor\tfrac12(r-1)\rfloor+q_m^{-2\overline{r-1}}(\lfloor\tfrac12 r\rfloor-1)\big)+
    q_m^{2 (n + 1 - 2 m)}\Big) \\
&\qquad-
 2 \Big (\sum_{m-r+1\le t\le n-m-r} q_m^{-2 \overline{r - 1}} (q_m^{t + r - m} -
        q_m^{m - t - r}) (q_m^{t + 1 + r - m} - q_m^{m - t - r - 1})\\
&\qquad\qquad + (q_m^{n + 1 - 2 m} - q_m^{2 m - n - 1})^2
      q_m^{-1}(r-2)+q_m^{-1}(q_m^{2(n - 2 m+1)}-1)\Big)\\
&=(q_m + q_m^{-1}) \Big(\sum_{m-r+1\le t\le n-m-r+1}
     q_m^{-2 \overline{r - 1}} (q_m^{2(t + r - m)} + q_m^{2(m - t - r)})-2q_m^{-2\overline{r-1}} \\
     &\qquad+ (q_m^{n + 1 - 2 m} - q_m^{2 m - n - 1})^2 \big( q_m^{-2\overline r}\lfloor\tfrac12(r-1)\rfloor+q_m^{-2\overline{r-1}}(\lfloor\tfrac12 r\rfloor-1)\big)+
    q_m^{2 (n + 1 - 2 m)}\Big) \\
&\qquad-
 2 \Big (q_m^{-2 \overline{r - 1}}\sum_{m-r+1\le t\le n-m-r}  (q_m^{2(t + r - m)+1}+q_m^{2(m-r-t)-1})\\
&\qquad\qquad + (q_m^{n + 1 - 2 m} - q_m^{2 m - n - 1})^2
      q_m^{-1}(r-2) +
    q_m^{-1}(q_m^{2(n - 2 m+1)}-1)\Big)
\\
&=(q_m + q_m^{-1}) ((q_m^{n + 1 - 2 m} -
        q_m^{2 m - n - 1})^2 (q_m^{-2 \overline{r}} \lfloor\tfrac12(r-1)\rfloor +
       q_m^{-2 \overline{r - 1}}\lfloor \tfrac12 r\rfloor) +
    q_m^{2 (n + 1 - 2 m)})\\
&\qquad-
 2 ((q_m^{n + 1 - 2 m} - q_m^{2 m - n - 1)})^2 q_m^{-1} (r - 2) +
    q_m^{-1} (q_m^{2 (n + 1 - 2 m)} - 1))\\
&\qquad-(q_m - q_m^{-1}) q_m^{-2 \overline{r - 1}} \sum_{m-r+1\le t\le n-m-r}
   q_m^{2 (t + r - m)} - q_m^{2 (m - r - t)}\\
&=(q_m + q_m^{-1}) ((q_m^{n + 1 - 2 m} -
        q_m^{2 m - n - 1})^2 (q_m^{-2 \overline{r}} \lfloor \tfrac12 (r - 1) \rfloor +
       q_m^{-2 \overline{r - 1}} \lfloor \tfrac12 r \rfloor) +
    q_m^{2 (n + 1 - 2 m)}) \\
&\qquad-
 2 ((q_m^{n + 1 - 2 m} - q_m^{2 m - n - 1})^2 q_m^{-1} (r - 2) +
    q_m^{-1} (q_m^{2 (n + 1 - 2 m)} - 1))\\
&\qquad+
 q_m^{-2 \overline{r - 1}} (q_m + q_m^{-1} - q_m^{2 (2 m - n - 1) + 1} -
    q_m^{2 (n - 2 m + 1) - 1}).
\end{align*}
This can be rewritten as
\begin{align*}
q\langle x_{m-r}\,|\,x_{m-r}\rangle-(q+q^{-1})=
p_0(q_m)+p_+(q_m)q_m^{2(n-2m+1)}+p_-(q_m)q_m^{-2(n-2m+1)},
\end{align*}
where $p_0,p_\pm\in\mathbb Z[z,z^{-1}]$ are defined by
\begin{align*}
p_+(z)&=z-z^{-1}+(z+z^{-1})\Big(z^{-2\overline r}\lfloor\tfrac12(r-1)\rfloor+z^{-2\overline{r-1}}\lfloor\tfrac12r\rfloor\Big)-2z^{-1}(r-2)-z^{-1-2\overline{r-1}}\\
&=\lfloor\tfrac12(r+1)\rfloor z-\lfloor\tfrac12(r-1)\rfloor(2z^{-1}-z^{-3}),\\
p_-(z)&=(z+z^{-1})\Big(z^{-2\overline r}\lfloor\tfrac12(r-1)\rfloor+z^{-2\overline{r-1}}\lfloor\tfrac12r\rfloor\Big)-2z^{-1}(r-2)-z^{1-2\overline{r-1}}\\
&=\lfloor\tfrac12(r-2)\rfloor(z-2z^{-1})+\lfloor\tfrac12r\rfloor z^{-3},\\
p_0(z)&=-(r-1)(z+z^{-3})+2(r-2)z^{-1}.
\end{align*}
Since $n+1-2m\ge 1$ and~$r\ge 2$, it follows that $q\langle x_{m-r}\,|\,x_{m-r}\rangle-q-q^{-1}
$ is a Laurent polynomial in~$q_m$ with the leading term
$q_m^{2(n-2m+1)+1}\lfloor\tfrac12(r+1)\rfloor$ and hence is non-zero.
\end{proof}
\begin{theorem}\label{thm:adm I2m+1}
Let $\{1,n+1\}\subset J\subset [1,n+1]$.
Let~$\tau_i(J)$, $i\in\{1,2\}$ be as in~\eqref{eq:tau_i(J) defn}.
The assignments $\wh T_r\mapsto \tau_{\overline r}(J)$, $r\in\{1,2\}$
define a fully supported disjoint standard homomorphism $\Br^+(I_2(2m+1))\to \Br^+_{n+1}$, $m\ge 1$
if and only if~$J=[1,n+1]$.
\end{theorem}
\begin{proof}
The forward direction is well known (cf. Corollary~\ref{cor:adm finite class}). For the converse,
if these assignments define a homomorphism of Artin
monoids then, by Lemma~\partref{lem:elem Artin hom.c} we must have~$\ell(\tau_0(J))=\ell(\tau_1(J))$. Furthermore, Lemma~\ref{lem:TJ hom} implies that
$T_J^{2m+1}$ is ${}^{op}$-invariant. By Theorem~\ref{thm:adm I2m converse }
this forces~$g(J)=1$ that is $J=[1,a]\cup [a+r,n+1]$ for some~$1\le
a\le n-1$ and $2\le r\le n+1-a$. Then~$\ell(\tau_{\bar a}(J))=\binom{r+1}2+
k$ and~$\ell(\tau_{1-\bar a}(J))=k'$ where~$k+k'=n-r$ and $k=k'$ if~$n-r$
is even while $|k-k'|=1$ if~$n-r$ is odd.
Since~$\ell(\tau_1(J))=\ell(\tau_0(J))$ this
forces $\binom{r+1}2=k'-k$ which is impossible since~$\binom{r+1}2\ge 3$
for~$r\ge 2$.
\end{proof}

\begin{corollary}\label{cor:type B}
The homomorphisms $\Br^+(I_2(2m))\to \Br^+(B_n)$, $2\le m \le n$ from
Proposition~\ref{prop:admissible hom from BrI22m to BrBn} are
the only fully supported optimal disjoint standard homomorphisms
$I_2(N)\to \Br^+(B_n)$.
\end{corollary}
\begin{proof}
Since the composition of such a homomorphism
with the one from~\eqref{eq:unfold Bn A2n}
or~\eqref{eq:unfold Bn A2n} 
is again a homomorphism of the same type,
the assertion follows from Theorem~\ref{thm:main thm adm}.
\end{proof}

\subsection{Higher rank}
\label{subs:rank > 2}
We can now classify all fully supported 
disjoint standard homomorphisms $\Br^+(\wh M)\to\Br^+(M)$ where~$\wh M$ is irreducible
of finite type and~$M$ is of type~$A$ or~$B$.
\begin{theorem}\label{thm:higher rank adm AB}
Let~$\wh M$ be irreducible of finite type 
with~$m=|\wh I|>2$
and let~$\Phi:\Br^+(\wh M)\to \Br^+(M)$ be 
an optimal fully supported disjoint standard homomorphism. Suppose that~$M$ is of type~$A_n$ or~$B_n$.
\begin{enmalph}
 \item\label{thm:higher rank adm AB.a}
 If~$\wh M$ is not of type~$B$ then~$\wh M$ and~$M$ are both of type~$A$ and~$\Phi$ is an isomorphism;
 \item\label{thm:higher rank adm AB.b}
 If~$\wh M=B_m$ and~$M=B_n$ then $m\le n$
 and 
 $$
 \Phi(\wh T_i)=T_i,\quad  i\in [1,m-1],\qquad 
 \Phi(\wh T_m)=T_{w_\circ^{[m,n]}};
 $$
 \item\label{thm:higher rank adm AB.c}
 If~$\wh M=B_m$ and~$M=A_n$ then~$m\le \lfloor\frac n2\rfloor$ and $\Phi$ is the composition 
 of the 
 homomorphism~$\Br^+(\wh M)\to 
 \Br^+(B_{\lfloor \frac n2\rfloor})$
 from part~\ref{thm:higher rank adm AB.b}
 with the standard unfolding~$\Br^+(B_{\lfloor \frac n2\rfloor})\to 
 \Br^+(M)$ given by~\eqref{eq:unfold Bn A2n-1} or~\eqref{eq:unfold Bn A2n}, depending on
 the parity of~$n$.
 \end{enmalph}
\end{theorem}
\begin{proof}
The argument for types in which $\wh m_{ij}$
is odd for all~$i,j\in\wh I$ is the same as in classification of LCM homomorphisms.

Let~$\wh M$ be of type~$B_m$. It is easy
to see, using Proposition~\ref{prop:admissible hom from BrI22m to BrBn}, that the assignments in part~\ref{thm:higher rank adm AB.b} define 
a homomorphism $\Br^+(\wh M)\to \Br^+(B_n)$,
$n\ge m$.

Suppose first that~$M$ is of type~$A$. The restriction of~$\Phi$ to $\Br^+_{[m-1,m]}(\wh M)$ is 
a homomorphism~$\Br^+(I_2(4))\to \Br^+_J(M)$
where~$J=[\Phi](\{m-1,m\})$. By Lemma~\ref{lem:diagonal} and
Theorem~\ref{thm:main thm adm}, $J=\bigcup_{1\le i\le k} [a_i,b_i]$
where~$b_i-a_i\ge 2$, $a_i-b_{i-1}\ge 1$, $2\le i\le k$, $1\le a_1$, $b_k\le n$
and 
$$
[\Phi](m-1)=\bigcup_{1\le i\le k}\{a_i,b_i\},
\qquad [\Phi](m)=\bigcup_{1\le i\le k} [a_i+1,b_i-1].
$$
By Lemma~\ref{lem:diagonal}, Theorem~\ref{thm:main thm adm} and
Corollary~\partref{cor:adm finite class}, $[\Phi](\{m-1,m-2\})$ is the disjoint union of subsets of~$I$ corresponding to submatrices of type~$A_2$.
In particular, $[\Phi](m-2)$ and~$[\Phi](m-1)$ must be self-orthogonal.
Thus, $[\Phi](m-2)=\bigcup_{1\le i\le k}
\{a_i-1,b_i+1\}$. Continuing this way
we conclude that~$[\Phi](j)=
\bigcup_{1\le i\le k} \{ a_i-m+j+1,
b_i+m-j-1\}$, $j\in [1,m-1]$ and are self-orthogonal.
Thus, $a_i-m+2$ and~$b_{i-1}+m-2$ both belong to~$[\Phi](1)$ and so
we must have $a_i-m+2-(b_{i-1}+m-2)>1$. Yet
in that case $a_i-m+1$ does not belong to~$[\Phi](j)$ for any~$j\in [1,m]$ which 
contradicts the assumption that~$\Phi$ is 
fully supported. Thus, $k=1$,
$[\Phi](j)=\{j,n+1-j\}$, $j\in[1,m-1]$
and~$[\Phi](m)=[m,n+1-m]$. 
In particular, $\Phi$ is the composition
of the homomorphism
$\Br^+(B_m)\to \Br^+(B_{\lfloor \frac n2\rfloor})$ with one of the 
homomorphisms from~\eqref{eq:unfold Bn A2n-1},
\eqref{eq:unfold Bn A2n}, depending on 
the parity of~$n$. Furthermore, as 
any disjoint fully supported 
standard homomorphism $\Br^+(B_m)\to 
\Br^+(B_n)$ yields disjoint 
fully supported standard homomorphisms 
$\Br^+(B_m)\to \Br^+(A_{2n})$ 
and~$\Br^+(B_m)\to \Br^+(A_{2n-1})$, it 
follows that the homomorphisms 
described in parts~\ref{thm:higher rank adm AB.b} and~\ref{thm:higher rank adm AB.c} are the only ones with~$M$ of type~$A$ or~$B$.

Finally, let~$\wh M=F_4$ and let~$M=A_n$. A disjoint fully supported standard homomorphism 
$\Phi:\Br^+(\wh M)\to \Br^+(M)$ restricts 
to a disjoint standard homomorphism~$\Br^+_{[1,3]}(F_4)\to \Br^+(M)$. By Lemma~\ref{lem:diagonal}
and part~\ref{thm:higher rank adm AB.c},
$[\Phi]([1,3])=\bigcup_{1\le i\le k} [a_i,b_i]$ where~$a_i-b_{i-1}>1$, $2\le i\le k$ and
$\lceil \frac{b_i-a_i}2\rceil\ge 2$,
$1\le i\le k$. Furthermore, as
$I\setminus [\Phi]([1,3])=[\Phi](4)$ and is self-orthogonal, $b_{i-1}=
a_{i}-2$ for all~$2\le i\le k$ and 
so~$[\Phi](4)=\{a_1-1,\dots,a_k-1,b_k+1\}\cap [1,n]$. Since~$a_1,b_k\in [\Phi](1)$, we conclude that~$b_k=n$ and~$a_1=1$ for otherwise~$\Phi(\wh T_1)$
and~$\Phi(\wh T_4)$ do not commute. Then~$k>1$ and so~$b_{k-1}+1=a_k-1\in[\Phi](4)$,
$b_{k-1}=a_k-2,a_k\in [\Phi](1)$ which is 
again a contradiction. Thus, no such homomorphism exists. It remains to observe that a disjoint fully supported standard homomorphism~$\Br^+(F_4)\to\Br^+(B_n)$ yields a homomorphism~$\Br^+(F_4)\to \Br^+(A_{2n-r})$,
$r\in\{0,1\}$ of the same type by~\eqref{eq:unfold Bn A2n-1}, \eqref{eq:unfold Bn A2n}.
\end{proof}
\begin{proposition}\label{prop:Bm Bn parab}
Let~$2\le m\le n$.
The assignments $\wh T_i\mapsto T_i$, $i\in [1,m-1]$,
$\wh T_m\mapsto \Cx mn\Cxr m{(n-1)}$ define 
a Coxeter type homomorphism~$\Br^+(B_m)\to \Br^+(B_n)$,
and its composition with 
the standard unfolding $\Br^+(B_n)\to \Br^+(A_N)$,
$n=\lfloor \frac12N\rfloor$ is a  
Coxeter type homomorphism $\Br^+(B_m)\to \Br^+(A_N)$. 
\end{proposition}
\begin{proof}
This is immediate from Proposition~\partref{prop:admissible hom from BrI22m to BrBn.a}. 
\end{proof}
Taking the composition of a 
homomorphism from Theorem~\partref{thm:higher rank adm AB.b}
with the unfolding
$\Br^+(B_n)\to\Br^+(D_{n+1})$ from~\eqref{eq:unfold Bn Dn+1},
we obtain infinite families
of optimal fully supported disjoint standard homomorphisms $\Br^+(B_m)\to \Br^+(D_{n+1})$
and~$\Br^+(I_2(2m))\to \Br^+(D_{n+1})$
for~$n\ge m$. It appears that these are the only
families existing for arbitrary~$n$. For small ranks there are other sporadic examples, for instance in type~$D_5$
the assignments
$$
\wh T_1\mapsto  T_i,\qquad \wh T_2\mapsto T_{w_\circ^{[1,5]\setminus\{i\}}},\qquad i\in \{4,5\}
$$
define homomorphisms~$\Br^+(I_2(8))\to \Br^+(D_5)$
while the assignments
$$
\wh T_1\mapsto T_{w_\circ^{\{1,i\}}},
\qquad \wh T_2\mapsto T_{w_\circ^{[2,n+1]\setminus\{i\}}},
\qquad i\in \{n,n+1\}
$$
define homomorphisms~$\Br^+(I_2(10))\to \Br^+(D_{n+1})$
for~$n\in\{4,5\}$. In higher ranks no such
homomorphisms seem to exists, apart from those obtained from type~$B$ via the standard unfolding~\eqref{eq:unfold Bn Dn+1}, but there are a lot
of apparently infinite families of non-disjoint standard 
homomorphisms (see~\S\ref{subs:conj families}).

\section{Towards classification of non-disjoint homomorphisms}\label{sect:non-disjoint}

\subsection{Non-disjoint
homomorphisms in higher ranks}
\label{subs:mon braid}
The family constructed here is 
inspired by braidings of tensor powers of objects in braided monoidal
categories. We will
identify $\Br^+_k$ with the parabolic
submonoid~$\Br^+_{[1,k-1]}(A_{n-1})$ of~$\Br^+_n$ for all~$n>k$.
\begin{theorem}\label{thm:monomial brd}
Let~$m\in\ZZ_{>1}$, $n\in\ZZ_{>0}$
and let $J_i^{(m)}:=[(i-1)m+1,(i+1)m-1]$,
$i\in\mathbb Z_{>0}$.
\begin{enmalph}
\item \label{thm:monomial brd.a}
The assignments $T_i\mapsto 
T_{w_{J_i^{(m)}\setminus \{im\};J_i^{(m)}}}$,
$i\in[1,n-1]$ define a 
Coxeter type homomorphism
\plink{Phi(m)n}$\Phi_n^{(m)}:\Br^+_n\to\Br^+_{nm}$;

\item \label{thm:monomial brd.b}
The assignments $T_i\mapsto 
T_{w_\circ^{J_i^{(m)}}}$,
$i\in[1,n-1]$, define 
a standard homomorphism $\wh\Phi_n^{(m)}:\Br^+_n\to\Br^+_{nm}$.

Let~$\wh M=B_n$, $\widetilde M=B_{mn}$.
\item \label{thm:monomial brd.c}
    The assignments $\wh T_i\mapsto \widetilde T_{w_{J_i^{(m)}\setminus\{im\}};J_i^{(m)}}$,
    $i\in[1,n-1]$, $\wh T_n\mapsto \widetilde T_{w_{[(n-1)m+1,nm-1];[(n-1)m+1,nm]}}$ define a 
    Coxeter type homomorphism $\Br^+(\wh M)\to \Br^+(\widetilde M)$;
\item \label{thm:monomial brd.d}
     The assignments $\wh T_i\mapsto \widetilde T_{w_\circ^{J_i^{(m)}}}$,
    $i\in[1,n-1]$, $\wh T_n\mapsto 
    \widetilde T_{w_\circ^{[(n-1)m+1,nm]}}$ define a 
    standard homomorphism $\Br^+(\wh M)\to \Br^+(\widetilde M)$.  
\end{enmalph}
\end{theorem}
\begin{remark}
For~$m=2$, the homomorphisms from part~\ref{thm:monomial brd.c} appear in
the classification of homomorphisms from~$\Br_n$ to~$\Br_N$, $N\le 2n$ (\cite{CKM}).
\end{remark}
\begin{proof}
It will be convenient to consider all the $\Br^+_k$, $k\in\ZZ_{\ge 1}$ as parabolic submonoids 
of~$\Br^+_\infty=\Br^+(A_\infty)$, which is generated 
by the $T_i$, $i\in\mathbb Z_{>0}$ subject 
to relations $T_i T_j=T_j T_i$, $|i-j|>1$ and 
$T_i T_j T_i=T_j T_i T_j$, $|i-j|=1$ for all~$i,j\in\mathbb Z_{>0}$. Then we can consider~$\Phi^{(m)}$
as an endomorphism of~$\Br^+_\infty$. Likewise,
we consider symmetric groups~$S_n$ as parabolic 
subgroups of~$S_\infty$.

We need the following
\begin{lemma}\label{lem:prep Grassman perm}
Let~$I=[1,2m-1]$ and let
$t_m=\prod_{1\le i\le m}(i,m+i)\in S_{2m}\cong W(A_{2m-1})$.
Then 
\begin{enmalph}
 \item   \label{lem:prep Grassman perm.a} $t_m=\pi_{2m-1}(T_{w_{I\setminus\{m\};I}})$;
 \item 
 \label{lem:prep Grassman perm.b} 
 $t_m=\ascprod_{1\le i\le m}\cxr i{(m+i-1)}$
 and $T_{w_{I\setminus\{m\};I}}=
 \ascprod_{1\le i\le m}\Cxr i{(m+i-1)}$.
\end{enmalph}
\end{lemma}
\begin{proof}

Denote~$\eta_m:=t_m w_\circ^{[1,m-1]}w_\circ^{[m+1,2m-1]}$. 
Let~$i\in [1,m]$. Then $\eta_m(i) =t_m(m+1-i)=2m+1-i$. Similarly,
if~$i\in[m+1,2m]$, $\eta_m(i)=
t_m(3m+1-i)=2m+1-i$. Thus, $\eta_m=
\prod_{1\le i\le m} (i,2m+1-i)=w_\circ^{[1,2m-1]}$ and so
$t_m=w_{I\setminus\{m\};I}$. This proves~\ref{lem:prep Grassman perm.a}. 
To prove the first identity in part~\ref{lem:prep Grassman perm.b}, note that since $\cxr ij=(i,j+1,j,\dots,i+1)$  we have for~$1\le j\le m$
\begin{align*}
\Big(\ascprod_{1\le i\le m}\cxr i{(m+i-1)}\Big)(j)&=\Big(\ascprod_{1\le i\le j-1}\cxr i{(m+i-1)}\Big)
\cxr j{(m+j-1)}(j)\\
&=\Big(\ascprod_{1\le i\le j-1}
\cxr i{(m+i-1)}\Big)(m+j)=m+j
\\
\intertext{and}
\Big(\ascprod_{1\le i\le m}\cxr i{(m+i-1)}\Big)(m+j)&=
\Big(\ascprod_{1\le i\le m-1}\cxr i{(m+i-1)}\Big)(m+j-1)=\cdots
\\
&=\Big(\ascprod_{1\le i\le m-k}\cxr i{(m+i-1)}\Big)(m-k+j)=\cdots=j.
\end{align*}
The second identity follows from the first 
by Theorem~\partref{thm:Tits.b} since 
\begin{equation*}\ell(T_{w_{I\setminus\{m\};I}})=
\ell(t_m)=\ell(w_\circ^{[1,2m-1]})-
\ell(w_\circ^{[1,m-1]})-\ell(w_\circ^{[m+1,2m-1]})=m^2=\sum_{1\le i\le m}
\ell(\Cxr i{(m+i-1)}).\qedhere
\end{equation*}
\end{proof}
Note that the assignments $T_i\mapsto T_{i+1}$,
$i\in\ZZ_{>0}$ define an endomorphism~$\xi$ of~$\Br^+_\infty$
which clearly descends to the Coxeter group.
We claim that the assignments
$$
s_i\mapsto t_{i,m}:=\xi^{(i-1)m}(t_m)=\prod_{1\le j\le m} ((i-1)m+j,
i m+j),
\qquad i\in[1,n-1]
$$
define an endomorphism of~$S_\infty$ which 
restricts to homomorphisms $S_n\to S_{nm}$ for any~$n\in\ZZ_{>0}$. Indeed, since the $t_{i,m}$, $t_{j,m}$ with~$|j-i|>1$ manifestly commute, it suffices to verify that $t_{1,m}t_{2,m}t_{1,m}=t_{2,m}t_{1,m}t_{2,m}$.
Note that 
$$
t_{i,m}(j)=\begin{cases}
j+m,&j\in[(i-1)m+1,im],\\
j-m,&j\in[im+1,(i+1)m],\\
j,&\text{otherwise}.
\end{cases}
$$
It follows that for all~$j\in[1,3m]$
\begin{equation}\label{eq:mon braiding rel}
t_{1,m}t_{2,m}t_{1,m}(j)=t_{2,m}t_{1,m}t_{2,m}(j)=\begin{cases}
j+2m,& j\in[1,m],\\
j,&j\in[m+1,2m],\\
j-2m,&j\in[2m+1,3m].
\end{cases}
\end{equation}
Furthermore, it is well-known (see e.g.~\cite{BjBr}*{Proposition~1.25})
that for~$w\in W(A_{k})\cong S_{k+1}$ where the isomorphism maps
$s_i$, $i\in[1,k]$ to the transposition~$(i,i+1)$, 
we have $\ell(w)=|\Inv(w)|$ where
$$\Inv(w)=\{ (i,j)\in[1,k+1]\times[1,k+1]\,:\,
i<j,\,w(i)>w(j)\}.
$$
It is immediate
from~\eqref{eq:mon braiding rel} that
\begin{align*}
\Inv(t_{1,m}t_{2,m}t_{1,m})=([1,m]\times[m+1,2m])
\cup ([1,m]\times[2m+1,3m])\cup([m+1,2m]\times[2m+1,3m]).
\end{align*}
Therefore, $\ell(t_{1,m}t_{2,m}t_{1,m})=
3m^2=\ell(t_{1,m})+\ell(t_{2,m})+\ell(t_{1,m})$. 
Then by Lemmata~\ref{lem:lifting to Cox-Hecke} and~\partref{lem:prep Grassman perm.a} the assignments in part~\ref{thm:monomial brd.a} define a homomorphism which is of Coxeter type
by Theorem~\partref{thm:Main Thm Cox Heck.a}.
This completes the proof of part~\ref{thm:monomial brd.a}.

As~$T_{w_\circ^{J_i^{(m)}}}$, $T_{w_\circ^{J_k^{(m)}}}$
with~$|i-k|>1$ manifestly commute,
it suffices to prove part~\ref{thm:monomial brd.b} for~$n=3$. Note that, since~$T_{w_\circ^{J_i^{(m)}\setminus\{im\}}}$ is invariant with respect to the 
diagram automorphism of~$\Br^+_{J_i^{(m)}}(A_\infty)$ and,
therefore, commutes with~$T_{w_\circ^{J_i^{(m)}}}$
by Proposition~\partref{prop:fund elts BrSa.e}, we have 
$$
T_{w_\circ^{J_i^{(m)}}}=\Phi^{(m)}(T_i)z_{i,i}^{(1)},\qquad 
z_{i,i}^{(1)}:=T_{w_\circ^{J_i^{(m)}\setminus\{im\}}}=
T_{w_\circ^{[(i-1)m+1,im-1]]\cup [im+1,(i+1)m-1]}}.
$$
Using~\eqref{eq:inv decoration}
and Proposition~\partref{prop:fund elts BrSa.e} we obtain in~$\Br^+_{3m}$
\begin{align*}
z_{i,k}^{(2)}&=\Phi^{(m)}(T_k)^{-1}T_{w_\circ^{J_i^{(m)}\setminus\{im\}}}
\Phi^{(m)}(T_k)=
T_{w_\circ^{J_k^{(m)}}}^{-1}T_{w_\circ^{J_k^{(m)}\setminus\{km\}}}T_{w_\circ^{J_i^{(m)}\setminus\{im\}}}
T_{w_\circ^{J_k^{(m)}\setminus\{km\}}}^{-1}T_{w_\circ^{J_k^{(m)}}}\\
&=T_{w_\circ^{[1,m-1]\cup[2m+1,3m-1]}},\\
z_{i,k}^{(2)}&=\Phi^{(m)}(T_i)^{-1}T_{w_\circ^{[1,m-1]\cup[2m+1,3m-1]}}
\Phi^{(m)}(T_i)=T_{w_\circ^{[m+1,2m-1]\cup[(k-i+1)m+1,(k-i+2)m+1]}},
\end{align*}
where~$\{i,k\}=\{1,2\}$, whence
\begin{align*}
z_{i,k}^{(3)}z_{k,i}^{(2)}&z_{i,k}^{(1)}
=T_{w_\circ^{[m+1,2m-1]\cup[2(2-i)m+1,(2(2-i)+1)m-1]}}
T_{w_\circ^{[1,m-1]\cup[2m+1,3m-1]}}T_{w_\circ^{[(i-1)m+1,im-1]\cup[im+1,(i+1)m-1]}}\\
&=T_{w_\circ^{[1,m-1]}}^2 T_{w_\circ^{[m+1,2m-1]}}^2
T_{w_\circ^{[2m+1,3m-1]}}^2.
\end{align*}
Therefore, 
$\mathbf z=(z_{1,1}^{(1)},z_{2,2}^{(1)})$
is a decoration of~$\Phi^{(m)}_3$ by Theorem~\ref{thm:decoration sufficient}, and~$\wh \Phi^{(m)}_3=(\Phi^{(m)}_3)_{\mathbf z}$.

To prove part~\ref{thm:monomial brd.c}, we need the following
\begin{lemma}\label{lem:diag aut mon brd}
Let~$\sigma_N$ be the diagram automorphism of~$\Br^+_N$ and let~$\Br^+_N{}^{\sigma_N}$ be 
the submonoid of~$\sigma_N$-invariant elements of~$\Br^+_N$. Then
$\Phi^{(m)}_n\circ \sigma_n=\sigma_{mn}\circ\Phi^{(m)}_n$
and
$\wh \Phi^{(m)}_n\circ \sigma_n=\sigma_{mn}\circ\wh\Phi^{(m)}_n$ for all~$m,n\in\mathbb Z_{>1}$. In particular,
$\Phi^{(m)}_n$ and~$\wh \Phi^{(m)}_n$
restrict to homomorphisms $\Br^+_n{}^{\sigma_n}\to 
\Br^+_{mn}{}^{\sigma_{mn}}$. 
\end{lemma}
\begin{proof}
Since~$\sigma_N$ corresponds to the permutation
$i\mapsto N+1-i$, $i\in [1,N-1]$, we have 
\begin{align*}
\sigma_{mn}(\wh\Phi^{(m)}_n(T_i))&=
\sigma_{mn}(T_{w_\circ^{[(i-1)m+1,(i+1)m-1]}})=
T_{w_\circ^{[nm-(i+1)m+1,nm-(i-1)m-1]}}\\
&=T_{w_\circ^{[(n-i-1)m+1,(n-i+1)m-1]}}
=\wh\Phi^{(m)}_n(T_{n-i})
=\wh\Phi^{(m)}_n(\sigma_n(T_i)).
\end{align*}
The argument for~$\Phi^{(m)}_n$ is similar and is omitted.
\end{proof}
Let~$\Upsilon_k:\Br^+(B_k)\to\Br^+_{2k}$ be the 
standard unfolding~\eqref{eq:unfold Bn A2n-1} which is an isomorphism
onto~$\Br^+_{2k}{}^{\sigma_{2k}}$
by Corollary~\ref{cor:adm finite class}.
Then $\Phi^{(m)}_{2n}\circ\Upsilon_n$ is a homomorphism
$\Br^+(\wh M)\to\Br^+_{2mn}$ whose 
image is contained in~$\Br^+_{2mn}{}^{\sigma_{2mn}}$.
It follows that
$\Upsilon_{mn}^{-1}\circ\Phi^{(m)}_{2n}\circ\Upsilon_n\in\Hom_{\Art}(\wh M,\widetilde M)$. To obtain the explicit formulae,
note that we have for~$i\in[1,n-1]$
\begin{align*}
(\Phi^{(m)}_{2n}\circ\Upsilon_n)(\wh T_i)&=
T_{w_{J_i^{(m)}\setminus\{im\};J_i^{(m)}}}
T_{w_{J_{2n-i}^{(m)}\setminus\{(2n-i)m\};J_{2n-i}^{(m)}}}\\
&=T_{w_{(J_i^{(m)}\setminus\{im\})\cup (J_{2n-i}^{(m)}\setminus\{(2n-i)m\});J_i^{(m)}\cup J_{2n-i}^{(m)}}}
=\Upsilon_{mn}(\widetilde T_{w_{J_i^{(m)}\setminus\{im\};J_i^{(m)}}})
\end{align*}
while
\begin{align*}
(\Phi^{(m)}_{2n}\circ\Upsilon_n)(\wh T_n)&=
T_{w_{J_n^{(m)}\setminus\{nm\};J_n^{(m)}}}
=T_{w_{[(n-1)m+1,[nm-1]\cup[nm+1,(n+1)m-1];[(n-1)m+1,(n+1)m-1]}}\\
&=\Upsilon_{nm}(\widetilde T_{w_{[(n-1)m+1,nm-1];
[(n-1)m+1,nm-1]}}),
\end{align*}
where we used Lemma~\ref{lem:parab preserving}.
In particular, this homomorphism is of Coxeter type.
The argument in
part~\ref{thm:monomial brd.d} is similar 
and is omitted.
\end{proof}

\begin{example}
Explicitly, we have
\begin{align*}&\Phi_n^{(2)}(T'_i)=T_{2i}T_{2i-1}T_{2i+1}T_{2i},\\
&\Phi_n^{(3)}(T'_i)=T_{3i}T_{3i-1}T_{3i-2}T_{3i+1}T_{3i}T_{3i-1}
T_{3i+2}T_{3i+1}T_{3i},\qquad i\in[1,n-1].
\end{align*}
\end{example}

\subsection{More infinite series of non-disjoint standard 
homomorphisms}\label{subs:inf ser non-disj}
We now use Theorem~\ref{thm:decoration sufficient} and
homomorphisms from Theorems~\partref{thm:monomial brd.b}\ref{thm:monomial brd.d} and~\partref{thm:higher rank adm AB.b}
to obtain additional infinite families of standard homomorphisms~$\Br^+_3\to\Br^+_{3m}$, $m\ge 1$,
and~$\Br^+(B_2)\to\Br^+(M)$ where~$M$ is of type~$A_n$,
$B_n$ or~$D_{n+1}$.

\begin{theorem}\label{thm:Hom A2}
For~$m\in\ZZ_{>0}$ and~$J\subset [1,m-1]$, 
the assignments 
$$T_1\mapsto T_{w_\circ^{[1,2m-1]\cup (2m+J)}},
\qquad 
T_2\mapsto T_{w_\circ^{[m+1,3m-1]\cup J}}$$
define a homomorphism $\Br^+_3\to\Br^+_{3m}$.
\end{theorem}
\begin{proof}
 Let~$M=A_3$,  $\mathsf M=\Br_{3m}$, $m\ge 2$ and let~$\Phi=\wh\Phi_3^{(m)}$ be the homomorphism from Theorem~\partref{thm:monomial brd.b}.
Thus, $\Phi(T_i)=T_{w_\circ^{[(i-1)m+1,(i+1)m-1]}}$,
$i\in\{1,2\}$.
Denote~$\sigma_i$, $i\in\{1,2\}$ the diagram 
automorphism of $\Br^+_{[(i-1)m+1,(i+1)m-1]}(A_{3m-1})\cong \Br^+_{2m}$. By Proposition~\partref{prop:fund elts BrSa.e}, $\Phi(T_i)^{-1}(T_{w_\circ^K})\Phi(T_i)=\sigma_i(T_{w_\circ^K})=
T_{w_\circ^{2im-K}}$
for any~$K\subset[(i-1)m+1,(i+1)m-1]$. 
Let~$z_{i,i}^{(1)}=T_{w_\circ^{2m(2-i)+J}}$, 
$i\in\{1,2\}$
and define by~\eqref{eq:inv decoration}
\begin{align*}
z_{i,j}^{(2)}&=\Phi(T_j)^{-1}z_{i,i}^{(1)}\Phi(T_{j})=
T_{w_\circ^{2mj-2m(2-i)-J}}=
T_{w_\circ^{2m-J}},\\
z_{i,j}^{(3)}&=\Phi(T_{i})^{-1}z_{i,j}^{(2)}\Phi(T_{i})
=T_{w_\circ^{2(i-1)m+J}}=z_{j,j}^{(1)}
\end{align*}
where~$\{i,j\}=\{1,2\}$.
Since~$J\subset[1,m-1]$, $2m+J\subset[2m+1,3m-1]$ and~$2m-J\subset[m+1,2m-1]$, they are pairwise orthogonal. 
Then $z_{1,2}^{(3)}z_{2,1}^{(2)}z_{1,2}^{(1)}=
T_{w_\circ^{J}}T_{w_\circ^{2m-J}}T_{w_\circ^{2m+J}}
=T_{w_\circ^{2m+J}}T_{w_\circ^{2m-J}}T_{w_\circ^{J}}
=z_{2,1}^{(3)}z_{1,2}^{(2)}z_{2,1}^{(1)}$. Therefore,
$(z_{1,1}^{(1)},z_{2,2}^{(1)})=(T_{w_\circ^{2m+J}},T_{w_\circ^J})$ is a decoration of~$\Phi$ by Theorem~\ref{thm:decoration sufficient}
and the assertion follows.   
\end{proof}
\begin{theorem}\label{thm:Hom B2}
Let~$m\in\ZZ_{\ge 0}$ and~$J\subset[1,m-1]$.
\begin{enmalph}    
\item\label{thm:Hom B2.B2An}
Let~$n\ge 4m-1$ and suppose that~$K\subset[2m+1,n-2m]$
and~$n+1-K$ are weakly orthogonal.
Then the assignments $\wh T_1\mapsto T_{w_\circ^{[1,2m-1]
\cup[n+2-2m,n]\cup K}}$, $\wh T_2\mapsto T_{w_\circ^{[m+1,n-m]\cup J\cup (n+1-J)}}$ define a homomorphism~$\Br^+(B_2)\to\Br^+(A_n)$;

\item\label{thm:Hom B2.B2} Let~$n\ge 2m$ and~$K\subset[2m+1,n]$. Then
the assignments $\wh T_1\mapsto T_{w_\circ^{[1,2m-1]\cup K}}$, $\wh T_2\mapsto T_{w_\circ^{[m+1,n]\cup J}}$ define a homomorphism~$\Br^+(B_2)\to\Br^+(B_n)$;

\item\label{thm:Hom B2.B2Dn}
Let~$n\ge 2m$.
Let~$K\subset[2m+1,n+1]$ and, if~$n-m$ is even, assume in addition that~$K$ and~$\tau(K)$ are weakly orthogonal where~$\tau$
is the transposition~$(n,n+1)$.
Then the assignments 
$\wh T_1\mapsto T_{w_\circ^{[1,2m-1]\cup K}}$, $\wh T_2\mapsto T_{w_\circ^{[m+1,n+1]\cup J}}$ define a homomorphism~$\Br^+(B_2)\to\Br^+(D_{n+1})$.
\end{enmalph}
\end{theorem}
\begin{proof}
We need the following
\begin{proposition}\label{prop:basic Hom B2Bn}
Let~$m\in\ZZ_{\ge 0}$ and~$n\ge 2m$. Then the assignments
$\wh T_1\mapsto T_{w_\circ^{[1,2m-1]}}$,
$\wh T_2\mapsto T_{w_\circ^{[m+1,n]}}$ define 
a homomorphism $\Phi_n\in\Hom_{\Art}(B_2,B_n)$.
\end{proposition}
\begin{proof}
For~$m=0$, these assignments define a character homomorphism.

Suppose that~$m>1$.
We use induction on~$n-2m$. 
For~$n=2m$, this is a special case of Theorem~\partref{thm:monomial brd.d}. 
For the inductive step, we need the following
\begin{lemma}\label{lem:img longest}
Let~$\Psi_n\in\Hom_{\Art}(B_n,B_{n+1})$ be the homomorphism 
from Theorem~\partref{thm:higher rank adm AB.b}.
Then for any~$J\subset[1,n]$, $\Psi_n(\wh T_{w_\circ^J})=T_{w_\circ^{[\Psi_n](J)}}T_{n+1}^{|J'|-1}$ where~$J'$ is the connected
component of~$J$ containing~$n$.
\end{lemma}
\begin{proof}
Since~$[\Psi_n](J)=J$ if~$J\subset[1,n-1]$, $[\Psi_n](J)=J\cup\{n+1\}$ if~$n\in J$ and~$\Psi_n(\wh T_i)=T_i$,
$i\in[1,n-1]$, it follows that~$\Psi_n(\wh T_{w_\circ^{J\setminus J'}})=T_{w_\circ^{J\setminus J'}}$. Thus, it suffices to prove the lemma for~$J=J'=[1,n]$. By Proposition~\partref{prop:Coxeter splitting.b},
$\wh T_{w_\circ^{[1,n]}}=\Cx1n{}^n$ and since~$\Psi_n(T_i)=T_i$, $i\in[1,n-1]$, $\Psi_n(T_n)=T_nT_{n+1}T_nT_{n+1}$, we obtain
$\Psi_n(\wh T_{w_\circ^{[1,n]}})=(\Cx1{(n+1)}T_nT_{n+1})^n$.
We claim that for all~$1\le k\le n$ 
\begin{equation}\label{eq:Cox power Psi}
(\Cx1{(n+1)}T_nT_{n+1})^k=
(\Cx1{(n+1)})^k \Cx{(n+1-k)}{(n+1)}T_{n+1}^{k-1}.
\end{equation}
Indeed, for~$k=1$ there is nothing to prove. For the inductive step, we have 
\begin{align*}
(\Cx1{(n+1)}T_nT_{n+1})^{k+1}=
(\Cx1{(n+1)})^k \Cx{(n+1-k)}{(n+1)}T_{n+1}^{k-1}
\Cx1{(n+1)}T_n T_{n+1}.
\end{align*}
Since~$\Br^+(B_{n+1})$ is cancellative, it thus suffices to prove that 
\begin{equation}\label{eq:Cox power Psi 1}
\Cx{(n+1-k)}{(n+1)}T_{n+1}^{k-1}
\Cx1{(n+1)}T_n=\Cx1{(n+1)}\Cx{(n-k)}{(n+1)}
T_{n+1}^{k-1}.
\end{equation}
Since~$T_{n+1}^{k-1}\Cx1{(n+1)}T_n=\Cx1{(n-1)}
T_{n+1}^{k-1}T_nT_{n+1}T_n=
\Cx1{(n+1)}T_n T_{n+1}^{k-1}$, by cancellativity \eqref{eq:Cox power Psi 1}
is equivalent to
\begin{equation}\label{eq:Cox power Psi 2}
\Cx{(n+1-k)}{(n+1)}\Cx1{(n+1)}T_n=\Cx1{(n+1)}\Cx{(n-k)}{(n+1)}.
\end{equation}
We have 
\begin{align*}
\Cx{(n+1-k)}{(n+1)}\Cx1{(n+1)}T_n
&=\Cx{(n+1-k)}n\Cx1{(n-1)}T_{n+1} T_n T_{n+1}T_n\\
&=\Cx{(n+1-k)}n\Cx1{n}T_{n+1} T_{n}T_{n+1}
\end{align*}
while~$
\Cx1{(n+1)}\Cx{(n-k)}{(n+1)}=
\Cx1n\Cx{(n-k)}{(n-1)}T_{n+1}T_nT_{n+1}$.
Therefore, \eqref{eq:Cox power Psi 2} is equivalent
to
\begin{equation}\label{eq:Cox power Psi 3} 
\Cx{(n+1-k)}n\Cx1{n}=\Cx1n\Cx{(n-k)}{(n-1)}
\end{equation}
which, since both sides of~\eqref{eq:Cox power Psi 3}
are contained in~$\Br^+_{[1,n]}(B_{n+1})\cong\Br^+(A_n)$,
is immediate from Lemma~\ref{lem:comm cox}.

Taking~$k=n$ in~\eqref{eq:Cox power Psi}, we obtain
$\Psi_n(T_{w_\circ^{[1,n]}})=(\Cx1{(n+1)})^{n+1}T_{n+1}^n=
T_{w_\circ^{[1,n+1]}}T_{n+1}^n$, which completes the proof of Lemma~\ref{lem:img longest}.
\end{proof}
By induction hypothesis, the assignments $\wh T_1\mapsto
T_{w_\circ^{[1,2m-1]}}$, $\wh T_2\mapsto T_{w_\circ^{[m+1,n]}}$ define a homomorphism
$\Br^+(B_2)\to\Br^+(B_n)$. Taking its composition 
with~$\Psi_n$, we obtain a homomorphism~$\tilde\Psi_n:\Br^+(B_2)\to
\Br^+(B_{n+1})$. Let~$\mathsf M=\Br(B_{n+1})$ and
$\Phi=\tilde\Psi_n$. By Lemma~\ref{lem:img longest}, $$\Phi(\wh T_1)=
T_{w_\circ^{[1,2m-1]}},\qquad \Phi(\wh T_2)=T_{w_\circ^{[m+1,n+1]}}T_{n+1}^{n-m}.
$$
Let~$z_1=1$ and $z_2=T_{m+1}^{m-n}$. Since the~$z_i$, $\in\{1,2\}$ commute with~$\Phi(\wh T_j)$, $j\in\{1,2\}$, $\mathbf z=(z_1,z_2)$ is a decoration of~$\Phi$ by Lemma~\ref{lem:cent decor}. Then~$\Phi_{\mathbf z}$
is the desired homomorphism $\Br^+(B_2)\to\Br(B_{n+1})$.
Since its image is contained in~$\Br^+(B_{n+1})$, the assertion follows.
\end{proof}
To prove part~\ref{thm:Hom B2.B2An}, let~$N=4$, $\mathsf M=\Br^+(A_n)$, $r=\lceil \frac12 n\rceil$
and let~$\Phi:\Br^+(B_2)\to \mathsf M$ be the composition of the
homomorphism $\Br^+(B_2)\to\Br^+(B_{r})$
from Proposition~\ref{prop:basic Hom B2Bn} with
the standard unfolding $\Br^+(B_r)\to \Br^+(A_n)$ from~\eqref{eq:unfold Bn A2n-1} or~\eqref{eq:unfold Bn A2n} depending on the parity of~$n$. By Corollary~\ref{cor:adm finite class}, it follows that
$\Phi(\wh T_1)=T_{w_\circ^{[1,2m-1]\cup\tilde\sigma([1,2m-1])}}$ and~$\Phi(\wh T_2)=T_{w_\circ^{[m+1,n-m]}}$, where~$\tilde\sigma$ is  the diagram automorphism of~$\Br^+(A_n)$. In particular,
$\tilde\sigma(L)=n+1-L$ for any~$L\subset[1,n]$.
Note that in~$\Br(A_n)$ we have
$\Phi(\wh T_2)^{-1} x \Phi(\wh T_2)=\tilde\sigma(x)$
for any~$x\in\Br^+_{[m+1,n-m]}(M)$.
Let~$\sigma_1$
be the diagram automorphism of~$\Br^+_{[1,2m-1]}(A_n)$.
Then for any~$x\in \Br^+_{[1,2m-1]}(M)$, $\tilde x\in \Br^+_{[n+2-2m,n]}(M)$ we have 
$\Phi(\wh T_1)^{-1} x\Phi(\wh T_1)=
\sigma_1(x)$ and~$\Phi(\wh T_1)^{-1} \tilde x\Phi(\wh T_1)=
\tilde\sigma(\sigma_1(\tilde\sigma(\tilde x)))$.

Let~$z_{1,1}^{(1)}=T_{w_\circ^K}$ and~$z_{2,2}^{(1)}=
T_{w_\circ^{J\cup\tilde\sigma(J)}}$. 
Using~\eqref{eq:inv decoration}, we obtain
\begin{align*}
&z_{1,2}^{(2)}=\Phi(\wh T_2)^{-1}T_{w_\circ^K}
\Phi(\wh T_2)=T_{w_\circ^{\tilde\sigma(K)}},\\
&z_{1,2}^{(3)}=\Phi(\wh T_1)^{-1}T_{w_\circ^{\tilde\sigma(K)}}\Phi(\wh T_1)=
T_{w_\circ^{\tilde\sigma(K)}}=z_{1,2}^{(2)},\\
&z_{1,2}^{(4)}=\Phi(\wh T_2)^{-1}T_{w_\circ^{\tilde\sigma(K)}}\Phi(\wh T_2)
=T_{w_\circ^K}=z_{1,2}^{(1)},\\
&z_{2,1}^{(2)}=\Phi(\wh T_1)^{-1}T_{w_\circ^{J\cup\tilde\sigma(J)}}
\Phi(\wh T_1)=T_{w_\circ^{\sigma_1(J)\cup\tilde\sigma(\sigma_1(J))}},\\
&z_{2,1}^{(3)}=\Phi(\wh T_2)^{-1} T_{w_\circ^{\sigma_1(J)\cup\tilde\sigma(\sigma_1(J))}}
\Phi(\wh T_2)
=z_{2,1}^{(2)},\\
&z_{2,1}^{(4)}=\Phi(\wh T_1)^{-1}T_{w_\circ^{\sigma_1(J)\cup\tilde\sigma(\sigma_1(J)}}\Phi(\wh T_1)=
T_{w_\circ^{J\cup \tilde\sigma(J)}}=z_{2,1}^{(1)}.
\end{align*}
Then 
\begin{align*}
&z_{1,2}^{(4)}z_{2,1}^{(3)}z_{1,2}^{(2)}z_{2,1}^{(1)}
=z_{1,2}^{(1)}z_{2,1}^{(2)}z_{1,2}^{(2)}z_{2,1}^{(1)}
=T_{w_\circ^K}T_{w_\circ^{\sigma_1(J)\cup\tilde\sigma(\sigma_1(J))}}T_{w_\circ^{\tilde\sigma(K)}}T_{w_\circ^{J\cup\tilde\sigma(J)}}\\
&\qquad=T_{w_\circ^K}T_{w_\circ^{\tilde\sigma(K)}}T_{w_\circ^{\sigma_1(J)\cup\tilde\sigma(\sigma_1(J))}}T_{w_\circ^{J\cup\tilde\sigma(J)}},\\
&z_{2,1}^{(4)}z_{1,2}^{(3)}z_{2,1}^{(2)}z_{1,2}^{(1)}
=z_{2,1}^{(1)}z_{1,2}^{(2)}z_{2,1}^{(2)}z_{1,2}^{(1)}
=T_{w_\circ^{J\cup\tilde\sigma(J)}}T_{w_\circ^{\tilde\sigma(K)}}T_{w_\circ^{\sigma_1(J)\cup\tilde\sigma(\sigma_1(J))}}T_{w_\circ^K}
\\
&\qquad=T_{w_\circ^{\tilde\sigma(K)}}T_{w_\circ^K}T_{w_\circ^{J\cup\tilde\sigma(J)}}T_{w_\circ^{\sigma_1(J)\cup\tilde\sigma(\sigma_1(J))}}
=T_{w_\circ^{\tilde\sigma(K)}}T_{w_\circ^K}
T_{w_\circ^{\sigma_1(J)\cup\tilde\sigma(\sigma_1(J))}}T_{w_\circ^{J\cup\tilde\sigma(J)}}
\end{align*}
since~$\tilde\sigma(K),K\subset[2m+1,n-2m]$ while
$J\cup\tilde\sigma(J)\subset[1,m-1]\cup[n+2-m,n]$
and~$\sigma_1(J)\cup\tilde\sigma(\sigma_1(J))
\subset[m+1,2m-1]\cup [n+2-2m,n-m]$. Finally,
since~$K$ and~$n+1-K=\tilde\sigma(K)$
are weakly orthogonal, 
$T_{w_\circ^K}$ and~$T_{w_\circ^{\tilde\sigma(K)}}$ commute by Lemma~\ref{lem:weakly orthogonal}
and so the condition~\ref{thm:decoration sufficient.2} 
of Theorem~\ref{thm:decoration sufficient} holds. 
Then~$\mathbf z=(z_{1,1}^{(1)},z_{2,2}^{(1)})
=(T_{w_\circ^K},T_{w_\circ^{J\cup\tilde\sigma(J)}})$
is a decoration of~$\Phi$ and~$\Phi_{\mathbf z}$
is the desired homomorphism.

Part~\ref{thm:Hom B2.B2} follows from part~\ref{thm:Hom B2.B2An} by taking~$K=n+1-K$ and using Corollary~\ref{cor:adm finite class}.

To prove part~\ref{thm:Hom B2.B2Dn}, note first that 
the composition of the homomorphism from Proposition~\ref{prop:basic Hom B2Bn} with the 
standard unfolding~\eqref{eq:unfold Bn Dn+1} yields
a standard homomorphism~$\Phi:\Br^+(B_2)\to 
\Br^+(D_{n+1})$ satisfying $\Phi(\wh T_1)=T_{w_\circ^{[1,2m-1]}}$ and~$\Phi(\wh T_2)=
T_{w_\circ^{[m+1,n+1]}}$. Let~$\mathsf M=\Br(D_{n+1})$.
Then for any~$x\in\Br^+_{[m+1,n+1]}(M)$, 
$\Phi(\wh T_2)^{-1}x\Phi(\wh T_2)=\tau^{n-m+1}(x)$.
Let~$\sigma_1$ be as before
and let~$z_{1,1}^{(1)}=T_{w_\circ^K}$, $z_{2,2}^{(1)}=
T_{w_\circ^J}$. Then by~\eqref{eq:inv decoration}
\begin{alignat*}{2}
&z_{1,2}^{(2)}=\Phi(\wh T_2)^{-1}T_{w_\circ^K}\Phi(\wh T_2)=
T_{w_\circ^{\tau^{n-m+1}(K)}},&\quad &z_{2,1}^{(2)}=\Phi(\wh T_1)^{-1}T_{w_\circ^J}\Phi(\wh T_1)=T_{w_\circ^{2m-J}},\\
&z_{1,2}^{(3)}=\Phi(\wh T_1)^{-1}z_{1,2}^{(2)}\Phi(\wh T_1)=z_{1,2}^{(2)},&&z_{2,1}^{(3)}=\Phi(\wh T_2)^{-1}z_{2,1}^{(2)}\Phi(\wh T_2)=T_{w_\circ^{2m-J}},\\
&z_{1,2}^{(4)}=\Phi(\wh T_2)^{-1}z_{1,2}^{(2)}\Phi(\wh T_2)=z_{1,2}^{(1)},&
&z_{2,1}^{(4)}=\Phi(\wh T_1)^{-1}T_{w_\circ^{2m-J}}
\Phi(\wh T_1)=T_{w_\circ^J}=z_{2,1}^{(1)}
\end{alignat*}
and so, since~$J\subset[1,m-1]$, $2m-J\subset[m+1,2m-1]$ and~$K,\tau(K)\subset[2m+1,n+1]$,
\begin{align*}
&z_{1,2}^{(4)}z_{2,1}^{(3)}z_{1,2}^{(2)}z_{2,1}^{(1)}
=T_{w_\circ^K}T_{w_\circ^{2m-J}}z_{1,2}^{(2)}T_{w_\circ^J}
=T_{w_\circ^K}z_{1,2}^{(2)}T_{w_\circ^J}T_{w_\circ^{2m-J}},\\
&z_{2,1}^{(4)}z_{1,2}^{(3)}z_{2,1}^{(2)}z_{1,2}^{(1)}
=T_{w_\circ^J}z_{1,2}^{(2)}T_{w_\circ^{2m-J}}T_{w_\circ^K}
=z_{1,2}^{(2)}T_{w_\circ^K}z_{1,2}^{(2)}T_{w_\circ^J}T_{w_\circ^{2m-J}}.
\end{align*}
If~$n-m$ is odd then~$z_{1,2}^{(2)}=T_{w_\circ^K}$. Otherwise, as $K$ 
and~$\tau(K)$ are weakly orthogonal, 
$z_{1,2}^{(2)}$ commutes with~$T_{w_\circ^K}$ by Lemma~\ref{lem:weakly orthogonal}. In either case, the condition~\ref{thm:decoration sufficient.2} of Theorem~\ref{thm:decoration sufficient} is satisfied, whence~$\mathbf z=(z_1,z_2)=(T_{w_\circ^K},T_{w_\circ^J})$
is a decoration of~$\Phi$ and~$\Phi_{\mathbf z}$ yields the homomorphism in part~\ref{thm:Hom B2.B2Dn}.
\end{proof}
\begin{remark}\label{rem:non-std B2 An}
Composing the homomorphism from Theorem~\partref{thm:Hom B2.B2Dn} with the folding $\mathbf F_{\varpi_{(n,n+1)}}$ from Proposition~\partref{prop:types of light.FDn+1An}  we obtain a non-standard homomorphism~$\Br^+(B_2)\to 
\Br^+(A_n)$ given by $$\wh T_1\mapsto T_{w_\circ^{[1,2m-1]}}T_{w_\circ^{K\setminus K'}}T_{w_\circ^{K'\setminus\{n+1\}}}^2,
\qquad \wh T_2\mapsto T_{w_\circ^J}
T_{w_\circ^{[m+1,n]}}^2 $$ 
where~$K'$ is the maximal interval~$[i,n+1]$ contained in~$K$ (see Lemma~\ref{lem:Dn+1 An w0}). In particular,
for~$J=K=\emptyset$ we obtain the homomorphism
satisfying $\wh T_1\mapsto T_{w_\circ^{[1,2m-1]}}$,
$\wh T_2\mapsto T_{w_\circ^{[m+1,n]}}^2$. It is very
tempting to factor it as a composition of the 
Tits homomorphism~$\Br^+(B_2)\to \Br^+(A_2)$, $\wh T_i\mapsto T_i^{i}$, $i\in \{1,2\}$ and a 
homomorphism~$\Br^+(A_2)\to \Br^+(A_n)$, $T_1\mapsto 
T_{w_\circ^{[1,2m-1]}}$, $T_2\mapsto T_{w_\circ^{[m+1,n]}}$. Alas, the latter assignments define a homomorphism if and only if~$n=3m-1$ (see Theorem~\ref{thm:Hom A2}).
\end{remark}
\begin{theorem}\label{thm:B2 A2n-1 spec}
For all~$n\ge 2$, $0\le k\le n-2$, $J=k+1-J\subset[1,k]$ and~$K=3n+k+1-K\subset[n+k+2,2n-1]$, 
the assignments $\wh T_1\mapsto T_{w_\circ^{[1,n+k]\cup K}}$, $\wh T_2\mapsto 
T_{w_\circ^{[k+2,2n-1]\cup J}}$
define a homomorphism~$\Br^+(B_2)\to\Br^+_{2n}$.
\end{theorem}
\begin{proof}
First, we prove the assertion for~$J=K=\emptyset$. 
By Lemma~\ref{lem:I2m iff cnd}, that is equivalent to proving that $(T_{w_\circ^{[1,n+k]}}T_{w_\circ^{[k+2,2n-1]}})^2$ is~${}^{op}$-invariant. Abbreviate $X_{n,k}=T_{w_{[1,k];[1,n+k]}}T_{w_{[k+2,n-1];[k+2,2n-1]}}\in\Br^+_{2n}$. Then 
in~$\Br_{2n}$
\begin{align*}
X_{n,k}&=T_{w_\circ^{[1,k]}}^{-1}T_{w_\circ^{[1,n+k]}}T_{w_\circ^{[k+2,n-1]}}^{-1}T_{w_\circ^{[k+2,2n-1]}}
=(T_{w_\circ^{[1,k]}}T_{w_\circ^{[k+2,n-1]}})^{-1}T_{w_\circ^{[1,n+k]}}T_{w_\circ^{[k+2,2n-1]}}\\
&=T_{w_\circ^{[1,n+k]}}T_{w_\circ^{[k+2,2n-1]}}(T_{w_\circ^{[n+1,n+k]}}T_{w_\circ^{[n+k+2,2n-1]}})^{-1},
\end{align*}
whence 
\begin{equation}\label{eq:Xnk }
(T_{w_\circ^{[1,n+k]}}T_{w_\circ^{[k+2,2n-1]}})^2=T_{w_\circ^{[1,k]}}T_{w_\circ^{[k+2,n-1]}}X_{n,k}^2  T_{w_\circ^{[n+1,n+k]}}T_{w_\circ^{[n+k+2,2n-1]}}.   
\end{equation}
\begin{proposition}\label{prop:factor Tw0^2}
For all~$0\le k\le n-2$, $X_{n,k}^2=T_{w_\circ^{[1,2n-1]}}^2$.
\end{proposition}
\begin{proof}
Note first that, since~$\ell(w_{[a,b];[a,b+c]})=\binom{b-a+c+2}2-\binom{b-a+2}2=\frac12 c(c+2(b-a)+3)$,
$\ell(X_{n,k})=\frac12n(n+2k+1)+\frac12n(3n-2k-3)=
n(2n-1)=\ell(T_{w_\circ^{[1,2n-1]}})$. Thus, $\ell(X_{n,k}^2)=\ell(T_{w_\circ^{[1,2n-1]}}^2)$. 
Since~$T_{w_\circ^{[1,2n-1]}}^2$
generates the center of~$\Br^+_{2n}$ by Proposition~\partref{prop:fund elts BrSa.d}, it remains to prove that~$X_{n,k}^2$ commutes with the~$T_i$, $i\in[1,2n-1]$.

Suppose that~$i\in[1,2n-1]\setminus\{k+1,n,n+k+1\}$. If~$i\in[1,k]$ then using Proposition~\partref{prop:fund elts BrSa.c}
\begin{align*}
T_i X_{n,k}&=T_i T_{w_\circ^{[1,k]}}^{-1} T_{w_\circ^{[1,n+k]}}T_{w_\circ^{[k+2,n-1]}}^{-1} T_{w_\circ^{[k+2,2n-1]}}=T_{w_\circ^{[1,k]}}^{-1} T_{k+1-i} T_{w_\circ^{[1,n+k]}}T_{w_\circ^{[k+2,n-1]}}^{-1} T_{w_\circ^{[k+2,2n-1]}}\\
&=T_{w_{[1,k];[1,n+k]}} T_{n+i} T_{w_\circ^{[k+2,n-1]}}^{-1} T_{w_\circ^{[k+2,2n-1]}}=T_{w_{[1,k];[1,n+k]}}  T_{w_\circ^{[k+2,n-1]}}^{-1} T_{n+i} T_{w_\circ^{[k+2,2n-1]}}\\
&=X_{n,k} T_{n+k+1-i}
\end{align*}
with~$n+k+1-i\in[n+1,n+k]$. If~$i\in[n+1,n+k]$ then
\begin{align*}
T_i X_{n,k}&=T_i T_{w_\circ^{[1,k]}}^{-1} T_{w_\circ^{[1,n+k]}}T_{w_\circ^{[k+2,n-1]}}^{-1} T_{w_\circ^{[k+2,2n-1]}}\\
&=T_{w_\circ^{[1,k]}}^{-1} T_{w_\circ^{[1,n+k]}}T_{n+k+1-i} T_{w_\circ^{[k+2,n-1]}}^{-1} T_{w_\circ^{[k+2,2n-1]}}=X_{n,k}T_{n+k+1-i}
\end{align*}
with~$n+k+1-i\in[1,k]$. Thus, for~$i\in [1,k]\cup [n+1,n+k]$, $T_i X_{n,k}=X_{n,k}T_{n+k+1-i}$ whence~$T_i X_{n,k}^2=X_{n,k}^2 T_i$.

If~$i\in[k+2,n-1]$,
\begin{align*}
T_i X_{n,k}&=T_i T_{w_\circ^{[1,k]}}^{-1} T_{w_\circ^{[1,n+k]}}T_{w_\circ^{[k+2,n-1]}}^{-1} T_{w_\circ^{[k+2,2n-1]}}=
T_{w_\circ^{[1,k]}}^{-1} T_{w_\circ^{[1,n+k]}}T_{n+k+1-i} T_{w_\circ^{[k+2,n-1]}}^{-1} T_{w_\circ^{[k+2,2n-1]}}\\
&=T_{w_{[1,k];[1,n+k]}} T_{w_\circ^{[k+2,n-1]}}^{-1} T_i T_{w_\circ^{[k+2,2n-1]}}=X_{n,k} T_{2n+k+1-i}
\end{align*}
with~$2n+k+1-i\in [n+k+2,2n-1]$. Similarly, if~$i\in[n+k+2,2n-1]$ then
\begin{align*}
T_i X_{n,k}&=T_{w_{[1,k];[1,n+k]}} T_{w_\circ^{[k+2,n-1]}}^{-1} T_i T_{w_\circ^{[k+2,2n-1]}}=X_{n,k} T_{2n+k+1-i}
\end{align*}
with~$2n+k+1-i\in [k+2,n-1]$. Thus, $T_i X_{n,k}=X_{n,k}T_{2n+k+1-i}$ for~$i\in[k+2,n-1]\cup[n+k+2,2n-1]$ and
so~$T_iX_{n,k}^2=X_{n,k}^2 T_i$.

It remains to prove that~$X_{n,k}^2$ commutes with the $T_i$ for~$i\in\{k+1,n,n+k+1\}$. We need the following
\begin{lemma}\label{lem:move across}
Let~$M$, $i,j\in I$ be a Coxeter matrix and suppose that~$M_{[i,j]}$ is of type~$A$. Then in~$\Br_{[i,j]}(M)$
\begin{subequations}
\begin{alignat}{2}
&T_i T_{w_\circ^{[i+1,j]}}=T_{w_\circ^{[i+1,j]}}\Cx ij\Cx i{(j-1)}{}^{-1}=T_{w_\circ^{[i+1,j]}}\Cx{(i+1)}j{}^{-1}\Cx ij,&\qquad\label{eq:move across i [i+1 j]}\\
&T_{w_\circ^{[i+1,j]}}T_i=\Cxr i{(j-1)}{}^{-1}\Cxr ijT_{w_\circ^{[i+1,j]}}=\Cxr ij\Cxr{(i+1)}j{}^{-1}T_{w_\circ^{[i+1,j]}}\label{eq:move across [i+1 j] i},\\
&T_j T_{w_\circ^{[i,j-1]}}=T_{w_\circ^{[i,j-1]}}\Cxr ij\Cxr{(i+1)}{j}{}^{-1}=T_{w_\circ^{[i,j-1]}}\Cxr{i}{(j-1)}{}^{-1}\Cxr ij,&\label{eq:move across j [i j-1]}\\
&T_{w_\circ^{[i,j-1]}}T_j=\Cx{(i+1)}{j}{}^{-1}\Cx ijT_{w_\circ^{[i,j-1]}}=\Cx ij\Cx{i}{(j-1)}{}^{-1}T_{w_\circ^{[i,j-1]}},\label{eq:move across [i j-1] j}\\
&T_i^{-1} T_{w_\circ^{[i+1,j]}}=T_{w_\circ^{[i+1,j]}}\Cx i{(j-1)}\Cx ij{}^{-1}=T_{w_\circ^{[i+1,j]}}\Cx ij{}^{-1} \Cx{(i+1)}j,\label{eq:move across -i [i+1 j]}\\
&T_{w_\circ^{[i+1,j]}}T_i^{-1}=\Cxr ij{}^{-1}\Cxr i{(j-1)}T_{w_\circ^{[i+1,j]}}=\Cxr{(i+1)}j\Cxr ij{}^{-1}T_{w_\circ^{[i+1,j]}} ,\label{eq:move across [i+1 j] -i}\\
&T_j^{-1} T_{w_\circ^{[i,j-1]}}=T_{w_\circ^{[i,j-1]}}\Cxr{(i+1)}{j}\Cxr ij{}^{-1}=T_{w_\circ^{[i,j-1]}}\Cxr ij{}^{-1} \Cxr{i}{(j-1)},\label{eq:move across -j [i j-1]}\\
&T_{w_\circ^{[i,j-1]}}T_j^{-1}=\Cx ij{}^{-1}\Cx{(i+1)}{j}T_{w_\circ^{[i,j-1]}}=\Cx{i}{(j-1)}\Cx ij{}^{-1}T_{w_\circ^{[i,j-1]}} ,\label{eq:move across [i j-1] -j}
\end{alignat}
\end{subequations}
\end{lemma}
\begin{proof}
It suffices to prove~\eqref{eq:move across i [i+1 j]}. The remaining identities follow by applying~${}^{op}$ or the diagram automorphism of~$\Br^+_{[i,j]}(M)$ or taking inverses.
Note that
$$
T_{w_\circ^{[i,j]}}=T_{w_\circ^{[i+1,j]}}\Cx ij=\Cxr ij T_{w_\circ^{[i+1,j]}}.
$$
Then
\begin{equation*}
T_i T_{w_\circ^{[i+1,j]}}=T_i T_{w_\circ^{[i,j]}}\Cx ij{}^{-1}=T_{w_\circ^{[i,j]}}T_j\Cx ij{}^{-1}=T_{w_\circ^{[i+1,j]}}\Cx ij\Cx i{(j-1)}{}^{-1}.
\end{equation*}
To prove the second equality in~\eqref{eq:move across i [i+1 j]}, note that it is equivalent to $$\Cx{(i+1)}j \Cx ij=\Cx ij\Cx i{(j-1)}$$ which in turn is immediate from Lemma~\ref{lem:comm cox}.
\end{proof}
Our aim is to show that for~$i\in\{k+1,n,n+k+1\}$,
$T_i^{-1}X_{n,k}=X_{n,k}U_{i,n,k}$ while $X_{n,k}T_i=U_{i,n,k}^{-1} X_{n,k}$ for some~$U_{i,n,k}\in\Br_{2n}$. Then~$X_i^{-1}X_{n,k}^2 X_i=X_{n,k}^2$ in~$\Br_{2n}$. 

For~$i=k+1$ we obtain, using the inverse of~\eqref{eq:move across [i j-1] j}
\begin{align*}
T_i^{-1}&X_{n,k}=T_{k+1}^{-1} T_{w_\circ^{[1,k]}}^{-1} T_{w_\circ^{[1,n+k]}} T_{w_{[k+2,n-1];[k+2,2n-1]}}\\
&=T_{w_\circ^{[1,k]}}^{-1} \Cx{1}{k}\Cx 1{(k+1)}{}^{-1}T_{w_\circ^{[1,n+k]}} T_{w_{[k+2,n-1];[k+2,2n-1]}}\qquad\\
&=T_{w_\circ^{[1,k]}}^{-1} T_{w_\circ^{[1,n+k]}}\Cxr{(n+1)}{(n+k)}\Cxr n{(n+k)}{}^{-1} T_{w_{[k+2,n-1];[k+2,2n-1]}}\\
&=T_{w_{[1,k];[1,n+k]}}\Cxr{(n+1)}{(n+k)}T_n^{-1}  T_{w_\circ^{[k+2,n-1]}}^{-1}\Cxr{(n+1)}{(n+k)}{}^{-1}T_{w_\circ^{[k+2,2n-1]}}\\
&=T_{w_{[1,k];[1,n+k]}}\Cxr{(n+1)}{(n+k)}T_{w_\circ^{[k+2,n-1]}}^{-1}\Cx{(k+2)}{(n-1)}\Cx{(k+2)}n{}^{-1}\Cxr{(n+1)}{(n+k)}{}^{-1}T_{w_\circ^{[k+2,2n-1]}}\\
&=T_{w_{[1,k];[1,n+k]}}T_{w_\circ^{[k+2,n-1]}}^{-1}\Cxr{(n+1)}{(n+k)}\Cx{(k+2)}{(n-1)}\Cx{(k+2)}n{}^{-1}\Cxr{(n+1)}{(n+k)}{}^{-1}T_{w_\circ^{[k+2,2n-1]}}\\
&=X_{n,k}\Cxr{(n+k+2)}{(2n-1)}\Cx{(n+1)}{(n+k)}\Cxr{(n+k+1)}{(2n-1)}{}^{-1}\Cx{(n+1)}{(n+k)}{}^{-1}=
X_{n,k} U_{k+1,n,k},
\end{align*}
where~$U_{k+1,n,k}=\Cxr{(n+k+2)}{(2n-1)}\Cx{(n+1)}{(n+k)}\Cx{(n+1)}{(n+k+1)}{}^{-1}\Cxr{(n+k+2)}{(2n-1)}{}^{-1}$.

On the other hand,
\begin{align*}
X_{n,k}&T_{k+1}=T_{w_{[1,k];[1,n+k]}}T_{w_\circ^{[k+2,2n-1]}} T_{w_\circ^{[n+k+2,2n-1]}}^{-1} T_{k+1}
=T_{w_{[1,k];[1,n+k]}}T_{w_\circ^{[k+2,2n-1]}}T_{k+1} T_{w_\circ^{[n+k+2,2n-1]}}^{-1}\\
&=T_{w_{[1,k];[1,n+k]}}\Cxr{(k+1)}{(2n-1)}\Cxr{(k+2)}{(2n-1)}{}^{-1} T_{w_{[k+2,n-1];[k+2,2n-1]}} \\
&=T_{w_\circ^{[1,k]}}^{-1} T_{w_\circ^{[1,n+k]}} \Cxr{(n+k+1)}{(2n-1)}\Cxr{(k+1)}{(n+k)}\Cxr{(k+2)}{(2n-1)}{}^{-1}  T_{w_{[k+2,n-1];[k+2,2n-1]}}\\
&=\Cxr{(n+k+2)}{(2n-1)} T_{w_\circ^{[1,k]}}^{-1} T_{w_\circ^{[1,n+k]}}T_{n+k+1} \Cxr{(k+1)}{(n+k)}\Cxr{(k+2)}{(2n-1)}{}^{-1}  T_{w_{[k+2,n-1];[k+2,2n-1]}}\\
&=\Cxr{(n+k+2)}{(2n-1)} T_{w_\circ^{[1,k]}}^{-1} \Cx1{(n+k+1)}\Cx{1}{(n+k)}{}^{-1} T_{w_\circ^{[1,n+k]}} \\
&\qquad\qquad\qquad \Cxr{(k+1)}{(n+k)}\Cxr{(k+2)}{(2n-1)}{}^{-1}  T_{w_{[k+2,n-1];[k+2,2n-1]}}\\
&=\Cxr{(n+k+2)}{(2n-1)} T_{w_\circ^{[1,k]}}^{-1} \Cx1{(n+k+1)}\Cx{(n+1)}{(n+k)}{}^{-1} T_{w_\circ^{[1,n+k]}}\Cxr{(k+2)}{(2n-1)}{}^{-1}  T_{w_{[k+2,n-1];[k+2,2n-1]}},
\end{align*}
where we used~\eqref{eq:move across [i+1 j] i} and~\eqref{eq:move across [i j-1] j}. 
We have by~\eqref{eq:move across [i j-1] -j}
and the inverse of~\eqref{eq:move across -j [i j-1]}
\begin{align*}
T_{w_\circ^{[1,k]}}^{-1} &\Cx1{(n+k+1)}\Cx{(n+1)}{(n+k)}{}^{-1}=\Cxr1k T_{w_\circ^{[1,k]}}^{-1} T_{k+1}\Cx{(k+2)}{(n+k+1)}\Cx{(n+1)}{(n+k)}{}^{-1}\\
&=\Cxr1k \Cxr{1}{k}{}^{-1}\Cxr 1{(k+1)} \Cx{(k+2)}{(n+k+1)} T_{w_\circ^{[1,k]}}^{-1}\Cx{(n+1)}{(n+k)}{}^{-1}\\
&=\Cxr 1{(k+1)} \Cx{(k+2)}{(n+k+1)} \Cx{(n+1)}{(n+k)}{}^{-1} T_{w_\circ^{[1,k]}}^{-1}\\
&=\Cx{(k+1)}{(n+k+1)} \Cx{(n+1)}{(n+k)}{}^{-1} \Cxr1k T_{w_\circ^{[1,k]}}^{-1},
\\
T_{w_\circ^{[1,n+k]}} &\Cxr{(k+2)}{(2n-1)}{}^{-1}=\Cx{1}{(n-1)}{}^{-1} T_{w_\circ^{[1,n+k]}}T_{n+k+1}^{-1} \Cxr{(n+k+2)}{(2n-1)}{}^{-1}\\
&=\Cx{1}{(n-1)}{}^{-1} \Cx{1}{(n+k)}\Cx1{(n+k+1)}{}^{-1}\Cxr{(n+k+2)}{(2n-1)}{}^{-1} T_{w_\circ^{[1,n+k]}}\\
&=\Cx{n}{(n+k)}\Cx1{(n+k+1)}{}^{-1}\Cxr{(n+k+2)}{(2n-1)}{}^{-1} T_{w_\circ^{[1,n+k]}},
\end{align*}
whence
\begin{align*}
&T_{w_\circ^{[1,k]}}^{-1} \Cx1{(n+k+1)}\Cx{(n+1)}{(n+k)}{}^{-1}T_{w_\circ^{[1,n+k]}}\Cxr{(k+2)}{(2n-1)}{}^{-1}\\
&=\Cx{(k+1)}{(n+k+1)} \Cx{(n+1)}{(n+k)}{}^{-1}\Cx n{(n+k)} \Cxr1k T_{w_\circ^{[1,k]}}^{-1}\Cx1{(n+k+1)}{}^{-1}\Cxr{(n+k+2)}{(2n-1)}{}^{-1} T_{w_\circ^{[1,n+k]}}.
\end{align*}
Next, using the inverse~\eqref{eq:move across j [i j-1]}
\begin{align*}
\Cxr1k &T_{w_\circ^{[1,k]}}^{-1}\Cx1{(n+k+1)}{}^{-1}=\Cxr1k\Cx{(k+2)}{(n+k+1)}{}^{-1} T_{w_\circ^{[1,k]}}^{-1}T_{k+1}^{-1}\Cx1k{}^{-1}\\
&=\Cxr1k \Cx{(k+2)}{(n+k+1)}{}^{-1}\Cxr1{(k+1)}{}^{-1}\Cxr{1}{k}\Cxr1k{}^{-1} T_{w_\circ^{[1,k]}}^{-1}\\
&=\Cx{(k+2)}{(n+k+1)}{}^{-1}\Cxr1k\Cxr1{(k+1)}{}^{-1} T_{w_\circ^{[1,k]}}^{-1}=\Cx{(k+1)}{(n+k+1)}{}^{-1} T_{w_\circ^{[1,k]}}^{-1}.
\end{align*}
Putting all the above together we conclude that $$
X_{n,k}T_{k+1}=\Cxr{(n+k+2)}{(2n-1)}\tilde U_{n,k}\Cxr{(n+k+2)}{(2n-1)}{}^{-1} X_{n,k}$$ 
where
$\tilde U_{n,k}=\Cx{(k+1)}{(n+k+1)} \Cx{(n+1)}{(n+k)}{}^{-1}\Cx n{(n+k)}\Cx{(k+1)}{(n+k+1)}{}^{-1}$. 
We have, as in the proof of Lemma~\ref{lem:move across}
$$
\Cx{(n+1)}{(n+k)}{}^{-1} \Cx{n}{(n+k)}=\Cx n{(n+k)}\Cx{n}{(n+k-1)}{}^{-1},
$$
and then by Lemma~\ref{lem:comm cox}
\begin{align*}
\tilde U_{n,k}&=\Cx{(k+1)}{(n+k+1)} \Cx n{(n+k)}\Cx{n}{(n+k-1)}{}^{-1}\Cx{(k+1)}{(n+k+1)}{}^{-1}\\
&=\Cx{(n+1)}{(n+k+1)}\Cx{(n+1)}{(n+k)}{}^{-1}.
\end{align*}
Therefore, $\Cxr{(n+k+2)}{(2n-1)}\tilde U_{n,k}\Cxr{(n+k+2)}{(2n-1)}{}^{-1}=
U_{k+1,n,k}^{-1}$.

Next, for~$i=n$ we obtain, using~\eqref{eq:move across -i [i+1 j]} and the inverse of~\eqref{eq:move across [i+1 j] i}
\begin{align*}
T_i^{-1} &X_{n,k}=T_{w_{[1,k];[1,n+k]}} T_{k+1}^{-1} T_{w_{[k+2,n-1];[k+2,2n-1]}}\\
&=T_{w_{[1,k];[1,n+k]}} T_{w_\circ^{[k+2,n-1]}}^{-1}\Cxr{(k+2)}{(n-1)}\Cxr{(k+1)}{(n-1)}{}^{-1}T_{w_\circ^{[k+2,2n-1]}}\quad \\
&=T_{w_{[1,k];[1,n+k]}} T_{w_\circ^{[k+2,n-1]}}^{-1}\Cxr{(k+2)}{(n-1)}T_{k+1}^{-1}T_{w_\circ^{[k+2,2n-1]}} \Cx{(n+k+2)}{(2n-1)}{}^{-1}\\
&=T_{w_{[1,k];[1,n+k]}} T_{w_\circ^{[k+2,n-1]}}^{-1}\Cxr{(k+2)}{(n-1)}T_{w_\circ^{[k+2,2n-1]}}\\
&\qquad\qquad\qquad\Cx{(k+1)}{(2n-1)}{}^{-1} \Cx{(k+2)}{(2n-1)}\Cx{(n+k+2)}{(2n-1)}{}^{-1}=X_{n,k}U_{n,n,k},
\end{align*}
where~$U_{n,n,k}=\Cx{(k+1)}{(n+k+1)}{}^{-1} \Cx{(k+2)}{(n+k+1)}$. On the other hand,
using~\eqref{eq:move across [i j-1] j}
and the inverse of~\eqref{eq:move across -j [i j-1]}, we obtain
\begin{align*}
X_{n,k}&T_n=T_{w_{[1,k];[1,n+k]}}T_{w_\circ^{[k+2,2n-1]}}T_{w_\circ^{[k+2+n,2n-1]}}^{-1} T_n =
T_{w_{[1,k];[1,n+k]}}T_{n+k+1} T_{w_{[k+2,n-1];[k+2,2n-1]}}\\
&=T_{w_\circ^{[1,k]}}^{-1}\Cx 1{(n+k+1)}\Cx1{(n+k)}{}^{-1}T_{w_\circ^{[1,n+k]}}T_{w_{[k+2,n-1];[k+2,2n-1]}} \\
&=\Cxr1k T_{w_\circ^{[1,k]}}^{-1}\Cx{(k+1)}{(n+k+1)}\Cx1{(n+k)}{}^{-1}T_{w_\circ^{[1,n+k]}}T_{w_{[k+2,n-1];[k+2,2n-1]}}\\
&=\Cxr 1{(k+1)}\Cx{(k+2)}{(n+k+1)} T_{w_\circ^{[1,k]}}^{-1}\Cx1{(n+k)}{}^{-1}T_{w_\circ^{[1,n+k]}}T_{w_{[k+2,n-1];[k+2,2n-1]}}\\
&=\Cxr 1{(k+1)}\Cx{(k+2)}{(n+k+1)}\Cx{(k+2)}{(n+k)}{}^{-1} T_{w_\circ^{[1,k]}}^{-1}\Cx1{(k+1)}{}^{-1}T_{w_\circ^{[1,n+k]}}T_{w_{[k+2,n-1];[k+2,2n-1]}}\\
&=\Cxr 1{(k+1)}\Cx{(k+2)}{(n+k+1)}\Cx{(k+2)}{(n+k)}{}^{-1}\Cxr1{(k+1)}{}^{-1}X_{n,k}\\
&=\Cx{(k+1)}{(n+k+1)}\Cx{(k+1)}{(n+k)}{}^{-1}X_{n,k}\\
&
=\Cx{(k+2)}{(n+k+1)}{}^{-1}\Cx{(k+1)}{(n+k+1)}X_{n,k}=U_{n,n,k}^{-1} X_{n,k}.
\end{align*}
Finally, for~$i=n+k+1$ we have, by~\eqref{eq:move across -j [i j-1]},
\begin{align*}
T_i^{-1}&X_{n,k}=T_{w_\circ^{[1,k]}}^{-1} T_{n+k+1}^{-1} T_{w_\circ^{[1,n+k]}} T_{w_{[k+2,n-1];[k+2,2n-1]}}\\
&=T_{w_{[1,k];[1,n+k]}}\Cxr1{(n+k+1)}{}^{-1} \Cxr{1}{(n+k)}T_{w_{[k+2,n-1];[k+2,2n-1]}}.
\end{align*}
Furthermore, since by~\eqref{eq:move across i [i+1 j]} and the inverse of~\eqref{eq:move across [i+1 j] -i},
\begin{align*}
\Cxr{1}{(n+k)}&T_{w_{[k+2,n-1];[k+2,2n-1]}}=\Cxr{(k+1)}{(n+k)} T_{w_\circ^{[k+2,2n-1]}} T_{w_\circ^{[n+k+2,2n-1]}}^{-1} \Cxr1k\\
&=T_{w_\circ^{[k+2,2n-1]}}\Cx{(n+1)}{(2n-1)}\Cx{(k+2)}{(2n-1)}{}^{-1}\Cx{(k+1)}{(2n-1)} T_{w_\circ^{[n+k+2,2n-1]}}^{-1} \Cxr1k\quad \\
&=T_{w_\circ^{[k+2,2n-1]}}\Cx{(k+2)}{n}{}^{-1}\Cx{(k+1)}{(n+k+1)} T_{w_\circ^{[n+k+2,2n-1]}}^{-1}\Cxr{(n+k+2)}{(2n-1)} \Cxr1k\\
&=T_{w_\circ^{[k+2,2n-1]}}\Cx{(k+2)}{n}{}^{-1}\Cx{(k+1)}{(n+k)}T_{w_\circ^{[n+k+2,2n-1]}}^{-1}\Cxr{(n+k+1)}{(2n-1)} \Cxr1k\\
&=T_{w_{[k+2,n-1];[k+2,2n-1]}}\Cx{(k+2)}{n}{}^{-1}\Cx{(k+1)}{(n+k)}\Cxr{(n+k+1)}{(2n-1)} \Cxr1k
\end{align*}
while~\eqref{eq:move across -i [i+1 j]}
and the inverse of~\eqref{eq:move across [i+1 j] i} yield
\begin{align*}
&\Cxr1{(n+k+1)}{}^{-1}T_{w_{[k+2,n-1];[k+2,2n-1]}}=\Cxr1{(n+k+1)}{}^{-1}T_{w_\circ^{[k+2,2n-1]}}T_{w_\circ^{[n+k+2,2n-1]}}^{-1}
\\
&=\Cxr1{(k+1)}{}^{-1}T_{w_\circ^{[k+2,2n-1]}}\Cx n{(2n-1)}{}^{-1}T_{w_\circ^{[n+k+2,2n-1]}}^{-1}\\
&=T_{w_\circ^{[k+2,2n-1]}}\Cxr1{k}{}^{-1}\Cx{(k+1)}{(2n-1)}{}^{-1} \Cx{(k+2)}{(n-1)}T_{w_\circ^{[n+k+2,2n-1]}}^{-1}\\
&=T_{w_\circ^{[k+2,2n-1]}}\Cxr1{k}{}^{-1}\Cx{(n+k+2)}{(2n-1)}{}^{-1}T_{n+k+1}^{-1}T_{w_\circ^{[n+k+2,2n-1]}}^{-1}\Cx{(k+1)}{(n+k)}{}^{-1}\Cx{(k+2)}{(n-1)}\\
&=T_{w_\circ^{[k+2,2n-1]}}\Cxr1{k}{}^{-1}T_{w_\circ^{[n+k+2,2n-1]}}^{-1}\Cxr{(n+k+1)}{(2n-1)}{}^{-1}\Cx{(k+1)}{(n+k)}{}^{-1}\Cx{(k+2)}{(n-1)}\\
&=T_{w_{[k+2,n-1];[k+2,2n-1]}}\Cxr1{k}{}^{-1}\Cxr{(n+k+1)}{(2n-1)}{}^{-1}\Cx{(k+1)}{(n+k)}{}^{-1}\Cx{(k+2)}{(n-1)}.
\end{align*}
It follows that $T_{n+k+1}^{-1}X_{n,k}=X_{n,k}U_{n+k+1,n,k}$
where
\begin{align*}
U_{n+k+1,n,k}
&=\Cxr1{k}{}^{-1}\Cxr{(n+k+1)}{(2n-1)}{}^{-1}\Cx{(k+1)}{(n+k)}{}^{-1}\Cx{(k+2)}{(n-1)}\times\\
&\qquad\qquad \Cx{(k+2)}{n}{}^{-1}\Cx{(k+1)}{(n+k)}\Cxr{(n+k+1)}{(2n-1)} \Cxr1k\\
&=\Cxr1{k}{}^{-1}\Cx{(k+1)}{(n-2)}\Cx{(k+1)}{(n-1)}{}^{-1}\Cxr1k\\
&=\Cxr1{k}{}^{-1}\Cx{(k+1)}{(n-1)}{}^{-1}\Cxr1k\Cx{(k+2)}{(n-1)}
\end{align*}
 by Lemma~\ref{lem:comm cox}. On the other hand, by the inverse of~\eqref{eq:move across -j [i j-1]}
\begin{align*}
X_{n,k}&T_{n+k+1}=T_{w_{[1,k];[1,n+k]}}T_{w_\circ^{[k+2,n-1]}}^{-1} T_n T_{w_\circ^{[k+2,2n-1]}}\\
&=T_{w_{[1,k];[1,n+k]}}\Cxr{(k+2)}{(n-1)}{}^{-1}\Cxr{(k+2)}n T_{w_{[k+2,n-1];[k+2,2n-1]}}\\
&=\Cx{(k+2)}{(n-1)}{}^{-1} T_{w_\circ^{[1,k]}}^{-1} \Cx{(k+1)}{(n-1)} T_{w_\circ^{[1,n+k]}} T_{w_{[k+2,n-1];[k+2,2n-1]}}\\
&=\Cx{(k+2)}{(n-1)}{}^{-1}\Cxr{1}{k}{}^{-1}\Cxr1{(k+1)}\Cx{(k+2)}{(n-1)} X_{n,k}\\
&=\Cx{(k+2)}{(n-1)}{}^{-1}\Cxr{1}{k}{}^{-1}\Cx{(k+1)}{(n-1)}\Cxr1{k} X_{n,k}=U_{n+k+1,n,k}^{-1}X_{n,k}.
\end{align*}
Thus, $T_i^{-1}X_{n,k}^2T_i=X_{n,k}^2$ for $i\in\{k+1,n,n+k+1\}$ which completes the proof of Proposition~\ref{prop:factor Tw0^2}.
\end{proof}
\begin{remark}
While~$\ell(X_{n,k})=\ell(w_\circ^{[1,2n-1]})$, the~$X_{n,k}$ are not square free and hence are not equal to~$T_{w_\circ^{[1,2n-1]}}$. For example,
$X_{2,0}=T_1T_2T_1T_2T_3T_2=T_2T_1^2T_3T_2$.
\end{remark}
By~\eqref{eq:Xnk } and Proposition~\ref{prop:factor Tw0^2},
\begin{align}
(T_{w_\circ^{[1,n+k]}}T_{w_\circ^{[k+2,2n-1]}})^2&=T_{w_\circ^{[1,k]}}T_{w_\circ^{[k+2,n-1]}}T_{w_\circ^{[1,2n-1]}}^2  T_{w_\circ^{[n+1,n+k]}}T_{w_\circ^{[n+k+2,2n-1]}}
\nonumber\\
&=T_{w_\circ^{[1,2n-1]}}^2 T_{w_\circ^{[1,2n-1]\setminus\{k+1,n,n+k+1\}}},\label{eq: str hom main}
\end{align}
which is manifestly~${}^{op}$ invariant, being the product of two commuting ${}^{op}$-invariant elements of~$\Br^+_{2n}$.
\begin{corollary}\label{cor:strange homs}
For all~$n\ge 2$, $0\le k\le n-2$ the assignments 
$\wh T_1\mapsto T_{w_{[1,k];[1,n+k]}}$, 
$\wh T_2\mapsto T_{w_{[k+2,n-1];[k+2,2n-1]}}$ (respectively, $
\wh T_1\mapsto T_{w_{[n+1,n+k];[1,n+k]}},\quad  \wh T_2\mapsto T_{w_{[n+k+2,2n-1];[k+2,2n-1]}}$)
define square free homomorphisms $\Br^+(B_2)\to\Br^+_{2n}$ 
which are neither Coxeter nor Hecke type. 
\end{corollary}
\begin{proof}
The first assignments define a homomorphism by Proposition~\ref{prop:factor Tw0^2} and Lemma~\ref{lem:cradicals}. The second is obtained from the first by applying~${}^{op}$.
\end{proof}

Now let~$J=k+1-J\subset[1,k]$, $K=3n+k+1-K\subset [n+k+2,2n-1]$. Set~$z_{1,1}^{(1)}=T_{w_\circ^K}$, 
$z_{2,2}^{(1)}=T_{w_\circ^J}$. Then, using~\eqref{eq:inv decoration}, set
\begin{alignat*}{2}
&z_{1,2}^{(2)}=T_{w_\circ^{[k+2,2n-1]}}^{-1}z_{1,2}^{(1)}T_{w_\circ^{[k+2,2n-1]}}=T_{w_\circ^{2n+k+1-K}},&\quad&z_{2,1}^{(2)}=T_{w_\circ^{[1,n+k]}}^{-1}z_{2,1}^{(1)}T_{w_\circ^{[1,n+k]}}=T_{w_\circ^{n+k+1-J}},\\
&z_{1,2}^{(3)}=T_{w_\circ^{[1,n+k]}}^{-1}z_{1,2}^{(2)}T_{w_\circ^{[1,n+k]}}=
T_{w_\circ^{-n+K}},&&
z_{2,1}^{(3)}=T_{w_\circ^{[k+2,2n-1]}}^{-1}z_{2,1}^{(2)}T_{w_\circ^{[k+2,2n-1]}}=T_{w_\circ^{n+J}},\\
&z_{1,2}^{(4)}=T_{w_\circ^{[k+2,2n-1]}}^{-1}z_{2,1}^{(3)}T_{w_\circ^{[k+2,2n-1]}}
=T_{w_\circ^{3n+k+1-K}},&&
z_{2,1}^{(4)}=T_{w_\circ^{[1,n+k]}}^{-1}z_{2,1}^{(3)}T_{w_\circ^{[1,n+k]}}
=T_{w_\circ^{k+1-J}}.
\end{alignat*}
Since~$J=k+1-J$ and~$K=3n+k+1-K$, we get
\begin{align*}
z_{1,2}^{(4)}z_{2,1}^{(3)}z_{1,2}^{(2)}z_{2,1}^{(1)}&=T_{w_\circ^K}T_{w_\circ^{n+J}}
T_{w_\circ^{2n+k+1-K}}T_{w_\circ^J}
\\
&=T_{w_\circ^K} T_{w_\circ^{-n+K}} T_{w_\circ^{n+J}} T_{w_\circ^J}
=T_{w_\circ^J}T_{w_\circ^{-n+K}}T_{w_\circ^{n+k+1-J}}T_{w_\circ^K}
=z_{2,1}^{(4)}z_{1,2}^{(3)}z_{2,1}^{(2)}z_{1,2}^{(1)}.
\end{align*}
Thus, $\mathbf z=(T_{w_\circ^K},T_{w_\circ^J})$
is a decoration of the basic 
homomorphism~$\Phi:\
\Br^+(B_2)\to \Br^+_{2n}$, $\Phi(\wh T_1)=T_{w_\circ^{[1,n+k]}}$,
$\Phi(\wh T_2)=T_{w_\circ^{[k+2,2n-1]}}$,
and $\Phi_{\mathbf z}$ is the desired homomorphism. This completes 
the proof of Theorem~\ref{thm:B2 A2n-1 spec}.
\end{proof}
\subsection{Combinatorics of standard homomorphisms}\label{subs:comb std}
It is obvious that the total number of distinct homomorphisms~$\Br^+_3\to\Br^+_{3m}$ described in Theorem~\ref{thm:Hom A2} is~$2^m$, as they are parametrized, up to the diagram automorphism of~$\Br^+_3$,
by subsets of~$[1,m-1]$. 

Enumerating homomorphisms from~$\Br^+(B_2)$ requires more effort and yields some interesting sequences.
\begin{theorem}\label{thm:combinatoric std}
\begin{enmalph}
\item \label{thm:combinatoric std.a}
The number of  homomorphisms~$\Br^+(B_2)\to\Br^+(A_{n}) $ described in Theorem~\partref{thm:Hom B2.B2An} is 
equal to
\begin{equation}\label{eq:H n}
H_n:=\tfrac12 \boldsymbol w_{\lfloor \frac12(n+5)\rfloor}-\overline{\lfloor \tfrac12(n+1)\rfloor}2^{\frac12\lfloor \frac12(n-1)\rfloor},\qquad n\ge 1
\end{equation}
where~$\boldsymbol w_0=\boldsymbol w_1=1$, $\boldsymbol w_{r+1}=2(\boldsymbol w_r+\boldsymbol w_{r-1})$, $r\ge 1$ and
\begin{equation}\label{eq:w_r seq}
\boldsymbol w_r=\tfrac12((1+\sqrt3)^r+(1-\sqrt3)^r)=\sum_{k\ge 0}\binom{r}{2k}3^k,\qquad r\ge 0.
\end{equation}

\item\label{thm:combinatoric std.b}
The number of homomorphisms
$\Br^+(B_2)\to \Br^+(B_n)$, $n\ge 2$
described in Theorem \partref{thm:Hom B2.B2} is
$3\cdot 2^{n}-2^{\lceil \frac n2\rceil}$.

\item\label{thm:combinatoric std.c}
The number of homomorphisms
$\Br^+(B_2)\to \Br^+(D_{n+1})$, $n\ge 3$
described in Theorem~\partref{thm:Hom B2.B2Dn} is
$$
\begin{cases}
\tfrac13(34\cdot 2^{n-1}-(5-\overline{\frac12(n-1)})2^{\frac12(n+1)}),&\bar n=1\\
\vphantom{\dfrac13}\tfrac13(29\cdot 2^{n-1}-(4+\overline{\frac12 n})2^{\frac12 n}),&\bar n=0.
\end{cases}
$$

\item \label{thm:combinatoric std.e}
The number of homomorphisms~$\Br^+(B_2)\to \Br^+(A_{2n-1})$ 
described 
in Theorem~\ref{thm:B2 A2n-1 spec} is equal to
\begin{align*}
\begin{cases}
2^{\frac n2}(\frac32n-2),&\bar n=0,\\
2^{\frac12(n+1)}(n-1),&\bar n=1.
\end{cases}
\end{align*}
\end{enmalph}
In all cases, the sequence grows exponentially as a function of the rang of the codomain.
\end{theorem}
\begin{proof}
Given~$M\in\Cox I$ and~$J,K\subset I$, write~$J\perp^M K$ (respectively, $J\perp^M_w K$) if~$J$ and~$K$ are orthogonal (respectively, weakly orthogonal).
For~$M=A_n$ and~$J,K\subset [1,n]$, $J\perp_M K$ if and only 
if~$|j-k|>1$ for all~$j\in J$ and~$k\in K$.
We abbreviate~$\perp$ and~$\perp_w$ in that case.
We need the following
\begin{lemma}\label{lem:weakly orthogonal comb}
Define
$$
\mathcal W_r=\{ (J,K)\in\mathscr P([1,r-1])\times \mathscr P([1,r-1])\,:\,J\perp_w K\},\quad r\in\ZZ_{\ge 0}.
$$
Then~$|\mathcal W_r|=\boldsymbol w_r$ for all~$r\in\ZZ_{\ge 0}$ with~$\boldsymbol{w}_r$
defined as in Theorem~\partref{thm:combinatoric std.a}.
\end{lemma}
\begin{proof}
Note that~$\mathcal W_{r+1}$, $r\ge 1$
is the disjoint union of the following sets:
\begin{alignat*}{2}
&\mathcal W_{r+1}^{(0)}=\{ (J,K)\in\mathcal W_{r+1}\,:\, J,K\subset [1,r-1]\},&\quad &
\mathcal W_{r+1}^{(1)}=\{ (J,K)\in\mathcal W_{r+1}\,:\, r\in J,\,K\subset [1,r-1]\},\\
&\mathcal W_{r+1}^{(2)}=\{ (J,K)\in\mathcal W_{r+1}\,:\, (K,J)\in\mathcal W_{r+1}^{(1)}\},&&
\mathcal W_{r+1}^{(3)}=\{ (J,K)\in\mathcal W_{r+1}\,:\,
r\in J\cap K\}.
\end{alignat*}
Clearly, $\mathcal W^{(0)}_{r+1}=\mathcal W_r$ and~$
|\mathcal W^{(2)}_{r+1}|=|\mathcal W^{(1)}_{r+1}|$.
We claim that 
\begin{align}\label{eq: W r+1 1}
\mathcal W_{r+1}^{(1)}&=
\bigsqcup_{1\le a\le r} \{ (J\cup [a,r],K)\,:\, (J,K)\in \mathcal W_{a-1}\},\\
\mathcal W_{r+1}^{(3)}&=\bigsqcup_{1\le a\le r} \{ (J\cup [a,r],K\cup [a,r])\,:\, (J,K)\in \mathcal W_{a-1}\},
\label{eq: W r+1 3}
\end{align}
whence $|\mathcal W_{r+1}^{(i)}|=\sum_{0\le a\le r-1} |\mathcal W_a|$, $1\le i\le 3$ and so
\begin{equation}\label{eq:w r long recursion}
|\mathcal W_{r+1}|=|\mathcal W_r|+3\sum_{0\le a\le r-1} |\mathcal W_a|.
\end{equation}
To prove the claim, note that the right hand side of~\eqref{eq: W r+1 1} (respectively, \eqref{eq: W r+1 3}) is obviously contained in~$\mathcal W^{(1)}_{r+1}$
(respectively, in~$\mathcal W^{(3)}_{r+1}$). To prove
the opposite inclusion in~\eqref{eq: W r+1 1},
let~$(J,K)\in\mathcal W_{r+1}^{(1)}$. Then
$J=J'\cup [a,r]$ for some~$a\in[1,r]$
and~$J'\subset [1,a-2]$. If~$a-1\in K$
then $a-1\in K\setminus J$ which is a contradiction as~$(K\setminus J)\perp J$
and~$a\in J$. Thus, $K=K'\cup K''$ with
$K'\subset [1,a-2]$ and~$K''\subset [a,r-1]$
as~$r\notin K$. Suppose that~$K''\not=\emptyset$ and let~$k=\max K''$.
Then~$k+1\le r$ and so~$k+1\in J\setminus K$
which is a contradiction as~$(J\setminus K)\perp K$ and~$k\in K$. Thus, $K=K'\subset [1,a-2]$ and so $(J',K)\in\mathcal W_{a-1}$.

To prove the opposite inclusion in~\eqref{eq: W r+1 3},
let~$(K_1,K_2)\in\mathcal W_{r+1}^{(3)}$ and write 
$K_i=K'_i\cup [a_i,r]$ where~$a_i\in[1,r]$ and $K'_i\subset [1,a_i-2]$, $i\in\{1,2\}$. If say~$a_1<a_2$ then $a_2-1\in K_1\setminus K_2$ which is a contradiction as~$(K_1\setminus K_2)\perp K_2$ and~$a_2\in K_2$. Thus~$a_1=a_2$.

We now prove that~$|\mathcal W_r|=\boldsymbol w_r$ for all~$r\ge0$. 
Since $\mathcal W_0=\mathcal W_1=\{(\emptyset,\emptyset)\}$, $|\mathcal W_r|=\boldsymbol{w}_r$, $r\in\{0,1\}$, while
for~$r\ge 1$ we have by~\eqref{eq:w r long recursion}
\begin{align*}
|\mathcal W_{r+1}|-|\mathcal W_r|=
|\mathcal W_r|+3 \sum_{0\le a\le r-1}|\mathcal W_a|-
|\mathcal W_{r-1}|-3\sum_{0\le a\le r-2}
|\mathcal W_a|
=|\mathcal W_r|+2|\mathcal W_{r-1}|,
\end{align*}
and so the~$|\mathcal W_r|$
and the~$\boldsymbol w_r$, $r\ge 0$ satisfy the same recursion.
The first equality in~\eqref{eq:w_r seq} 
is obtained by elementary linear algebra while the second is immediate from the first.
\end{proof}
Note the following elementary fact.
\begin{lemma}\label{lem:2 step rec}
Suppose that~$x_{r+1}=a x_r+b x_{r-1}$, $r\ge 1$. Then the sequence~$y_r:=\sum_{0\le i\le r-1} x_i$ satisfies the recursion $y_{r+1}=ay_r+by_{r-1}+c$, $r\ge 1$ with $y_0=0$, $y_1=x_0$ and
$c=x_1+(1-a)x_0$.
\end{lemma}
\begin{lemma}\label{lem:K sigma K w.o.}
Let~$\mathcal U_n=\{ K\subset [1,n]\,:\,
K\perp_w (n+1-K)\}$. Then~$|\mathcal U_n|=\boldsymbol u_{\lfloor\frac12(n+3)\rfloor}$
where
\begin{equation}\label{eq:w_r->u_r}
\boldsymbol u_r=\sum_{0\le a\le r-1} \boldsymbol w_a,\quad r\ge 0.
\end{equation}
In particular, $\boldsymbol u_0=0$, $\boldsymbol u_1=1$, $\boldsymbol u_{r+1}=2(\boldsymbol u_r+\boldsymbol u_{r-1})$, $r\ge 1$ whence
\begin{equation}\label{eq:u_r sequence}
\boldsymbol u_r=
\frac{(1+\sqrt 3)^r-
(1-\sqrt 3)^r}{2\sqrt 3}
=\sum_{k\ge 0}\binom{r}{2k+1}3^k,\qquad r\ge 0.
\end{equation}
\end{lemma}
\begin{proof}
Let~$k=\lfloor \frac12(n+1)\rfloor$, $\mathcal U_n'=
\{ K\in\mathcal U_n\,:\, k\in K\}$ and $\mathcal U_n''
=\mathcal U_n\setminus \mathcal U'_n$. 
We claim that~$\mathcal U'_n=\{ K\in\mathcal U_n\,:\, \{k,n+1-k\}\subset K\}$ and~$\mathcal U''_n=\{K\in\mathcal U_n\,:\,k,n+1-k\notin K\}$. Indeed, for~$n$
odd there is nothing to prove as~$k=n+1-k$.
Suppose that~$n$ is even and so~$n+1-k=k+1$. If $K\in\mathcal U_n''$ and~$k+1\in K$ then~$k\in n+1-K$ and~$k+1\notin n+1-K$, whence~$k\in (n+1-K)\setminus K$ which
is a contradiction as~$((n+1-K)\setminus K)\perp K$ and~$k+1\in K$.
Likewise, if~$K\in\mathcal U'_n$ and~$k+1\notin K$ then
$k+1\in (n+1-K)\setminus K$ which is again
a contradiction as~$((n+1-K)\setminus K)\perp K$ and~$k\in K$.

Thus, if~$K\in \mathcal U''_n$ then~$K=K'\cup (n+1-K'')$ where~$K',K''\subset [1,k-1]$. Then~$K\perp_w (n+1-K)$ if and only if
$(K',K'')\in\mathcal W_{k-1}$ and so
$|\mathcal U''_n|=\boldsymbol w_k$.

Suppose now that~$K\in \mathcal U'_n$ and let~$[a,b]$
be the connected component of~$K$ containing~$k$. Then $a\le k\le n+1-k\le b$ and $K=K_1\cup [a,b]\cup K_2$ with~$K_1\subset [1,a-2]$ and~$K_2\subset [b+2,n]$. We claim that~$b=n+1-a$. Indeed,
if $n+1-b<a$ then~$n+1-b\in n+1-K$. Since~$a\le k$
we have~$a-1\in [n+1-b,n+1-a]\subset (n+1-K)$.  
Then~$a-1\in (n+1-K)\setminus K$ which is a contradiction
as~$((n+1-K)\setminus K)\perp K$ and~$a\in K$. Similarly, if~$n+1-b>a$
then~$b+1\in[n+1-b,n+1-a]$ and so~$b+1\in (n+1-K)\setminus K$ which is also a contradiction. 

Thus, $K=K_1\cup [a,n+1-a]\cup K_2$ with~$K_2\subset n+1-[1,a-2]$ and so we can write~$K_2=n+1-K'_1$ where~$K'_1\in [1,a-2]$. It is now immediate that~$K\perp_w (n+1-K)$
if and only if $K_1\perp_w K'_1$. Thus,
$$
\mathcal U'_n=\bigsqcup_{1\le a\le k}
\{ K_1\cup [a,n+1-a]\cup (n+1-K'_1)\,:\, (K_1,K'_1)\in\mathcal W_{a-1}\}
$$
in the notation of Lemma~\ref{lem:weakly orthogonal comb}.
Therefore, $|\mathcal U'_n|=\sum_{0\le a\le k-1} \boldsymbol w_a$
and so
\begin{equation*}
|\mathcal U_n|=|\mathcal U'_n|+|\mathcal U''_n|=
\sum_{0\le a\le k}\boldsymbol w_a=\sum_{0\le a\le \lfloor \frac12 (n+3)\rfloor-1} \boldsymbol w_a,    
\end{equation*}
which yields the first assertion of the Lemma.
The recursion for~$\boldsymbol u_r$, $r\ge 0$
is immediate from that for~$\boldsymbol w_r$, $r\ge 0$ and Lemma~\ref{lem:2 step rec}, and the equalities in~\eqref{eq:u_r sequence}
are routine.
\end{proof}
We now have all necessary ingredients to 
finish the proof of part~\ref{thm:combinatoric std.a}. 
By Theorem~\partref{thm:Hom B2.B2An}
$$
H_n=2\sum_{0\le m\le \lfloor\frac14(n+1)\rfloor}
|\{ (J,K)\,:\, J\subset [1,m-1],\, K\subset [2m+1,n-2m],\,
K\perp_w n+1-K\}|,
$$
where the first factor accounts for the diagram
automorphism of~$\Br^+(B_2)$. Since~$\lfloor\frac14(n+1)\rfloor=\lfloor\frac12\lfloor\frac12(n+1)\rfloor\rfloor$, it follows from Lemma~\ref{lem:K sigma K w.o.}
that~$H_n=\boldsymbol h_{\lfloor \frac12(n+1)\rfloor}$ where
\begin{equation}\label{eq:hr seq}
\boldsymbol h_r=2 \boldsymbol u_r+\sum_{1\le m\le \lfloor \frac12 r\rfloor} 2^m \boldsymbol u_{r-2m},\qquad r\ge 0.
\end{equation}
Then~\eqref{eq:H n} is equivalent to $\boldsymbol h_r=\tfrac12 \boldsymbol{w}_{r+1}-\bar r\,2^{\frac12(r-1)}$, $r\ge 0$.
By~\eqref{eq:hr seq}, $\boldsymbol h_0=2\boldsymbol u_0=2=\frac12\boldsymbol{w}_1$ and
$\boldsymbol h_1=2\boldsymbol u_1=4=\frac12\boldsymbol{w}_2-1$.
Furthermore, for~$r\ge 0$,
\begin{align*}
\boldsymbol h_{r+2}&=2 \boldsymbol u_{r+2}+\sum_{1\le m\le \lfloor\frac12 r\rfloor+1}2^m \boldsymbol u_{r+2-2m}=2 \boldsymbol u_{r+2}+2\boldsymbol u_r+2\sum_{1\le m\le \lfloor\frac 12r\rfloor}
2^m \boldsymbol u_{r-2m}\\
&=2 \boldsymbol u_{r+2}+2 \boldsymbol h_r-2\boldsymbol u_r
=2\boldsymbol{h}_r+2(\boldsymbol w_{r+1}+
\boldsymbol{w}_r)
=2\boldsymbol h_r+\boldsymbol{w}_{r+2}\\
&=\boldsymbol{w}_{r+2}+\boldsymbol{w}_{r+1}-
2\bar r\cdot 2^{\frac12(r-1)}=\tfrac{1}{2}\boldsymbol{w}_{r+3}-\overline{r+2}\cdot
2^{\frac12(r+1)},
\end{align*}
where we used~\eqref{eq:w_r->u_r}, the recursion for the $\boldsymbol{w}_k$, $k\ge 2$ and
the induction hypothesis.
Part~\ref{thm:combinatoric std.a} is proven.

To prove part~\ref{thm:combinatoric std.b}, note 
that this family of homomorphisms is 
parametrized by $(J,K)$ with $J\subset [1,m-1]$
and~$K\subset [2m+1,n]$. Thus, taking into
account the diagram automorphism of~$\Br^+(B_2)$
we conclude that the total number of such homomorphisms is
$$
2\cdot 2^{n}+\sum_{1\le m\le \lfloor\frac12n\rfloor} 2^{n-m}=3\cdot 2^n-2^{\lceil\frac12 n\rceil}.
$$

We now prove~\ref{thm:combinatoric std.c}.
We need 
the following
\begin{lemma}\label{lem:comb}
Let~$M=D_{n+1}$ and~$a\in[1,n+1]$. Let~$\mathcal T_{n,a}=
\{ K\subset [a,n+1]\,:\, K\perp^M_w \tau(K)\}$. Then
\begin{equation}\label{eq:card T}
|\mathcal T_{n,a}|=\begin{cases}
3\cdot 2^{n-a},&1\le a\le n-1,\\
4,&a=n,\\
2,&a=n+1.
\end{cases}
\end{equation}
\end{lemma}
\begin{proof}
Clearly, if~$K\subset [a,n+1]$ satisfies ~$K=\tau(K)$ then~$K\in\mathcal T_{n,a}$. Furthermore, $\tau(K)=K$
if either~$K\subset [a,n-1]$ or~$K=K'\cup\{n,n+1\}$
with~$K'\subset [a,n-1]$. Thus, every subset of~$[a,n-1]$
yields precisely two $\tau$-invariant subsets of~$[a,n+1]$
and so
$$
|\{ K\subset [a,n+1]\,:\,\tau(K)=K\}|=\begin{cases}
2^{n+1-a},&1\le a\le n,\\
1,& a=n+1.
\end{cases}
$$
Let~$\mathcal T'_{n,a}=\mathcal T_{n,a}\setminus \{ K
\subset [a,n+1]\,:\, K=\tau(K)\}$. Clearly
$\mathcal T'_{n,n+1}=\{\{n+1\}\}$ and~$\mathcal T'_{n,n}=\{\{n\},\{n+1\}\}$. Let~$1\le a\le n-1$ and let~$K\in\mathcal T'_{n,a}$. Then~$|K\cap \{n,n+1\}|=1$. If say~$n\in K$ then~$n+1\in\tau(K)\setminus K$ and, since~$(\tau(K)\setminus K)\perp^M K$, it follows that~$n-1\notin K$.
Thus, $K=K'\cup\{n\}$ with~$K'\subset [a,n-2]$ and so~$\mathcal T'_{n,a}=
\{ K'\cup\{n\},K'\cup\{n+1\}\,:\,
K'\subset [a,n-2]\}$.
Therefore, 
$$
|\mathcal T'_{n,a}|=\begin{cases}
2^{n-a},&1\le a\le n-1,\\
2,&a=n,\\
1,&a=n+1.
\end{cases}
$$
The assertion is now immediate.
\end{proof}

By Lemma~\ref{lem:comb}, the number of homomorphisms from Theorem~\partref{thm:Hom B2.B2Dn} is equal to
\begin{alignat}{2}\label{eq: Dn+1 n odd}
&\sum_{0\le k\le \lfloor
\frac14(n-3)\rfloor} 2^{2k+\delta_{k,0}} |\mathscr P([4k+1,n+1])|
+\sum_{1\le k\le \lfloor\frac12(n+3)\rfloor} 2^{2k-1} |
\mathcal T_{n,4k-1}|,&&\bar n=1,
\\
\label{eq: Dn+1 n even}
&\sum_{0\le k\le \lfloor 
\frac14 n\rfloor} 2^{2k+\delta_{k,0}} 
|\mathcal T_{n,4k+1}|
+\sum_{1\le k\le \lfloor\frac14(n+2)\rfloor} 2^{2k-1}   
|\mathscr P([4k-1,n+1])|,&\qquad&\bar n=0,
\end{alignat}
for~$n\ge 3$.
By~\eqref{eq:card T} the expression~\eqref{eq: Dn+1 n odd}
is equal to
\begin{align*}
\sum_{0\le k\le \lfloor\frac14(n-3)\rfloor}&2^{\delta_{k,0}+n+1-2k}+3\sum_{1\le k\le \lfloor \frac14 (n-1)\rfloor} 2^{n-2k}+4\cdot 2^{\frac12(n-1)}
\\&=2^{n+2}+\tfrac13 (2^{n+1}-2^{\frac12(n+1)+3-\overline{\frac12(n-1)}})+
2^{n}+2^{\frac12(n+1)}(1-\overline{\tfrac12(n-1)})\\
&=\tfrac13(34\cdot 2^{n-1}-(5-\overline{\tfrac12(n-1)})2^{\frac12(n+1)}).
\end{align*}
Similarly, \eqref{eq: Dn+1 n even} becomes
\begin{align*}
3\cdot 2^{n}+3 \sum_{1\le k\le \lfloor\frac14(n-2)\rfloor}
2^{n- 2 k - 1}+2^{\frac12n+1}+\sum_{1\le k\le \lfloor\frac14n\rfloor}
2^{n+2-2k}=\tfrac13(29\cdot 2^{n-1}-2^{\frac12n}(4+\overline{\tfrac12n})).
\end{align*}

To prove part~\ref{thm:combinatoric std.e} we need the following
\begin{lemma}\label{lem:no inv subsets}
There are~$2^{\lfloor\frac12 n\rfloor}$
subsets of~$[1,n-1]$ which are invariant 
with respect to the diagram automorphism of~$\Br^+_n$, $n\ge 2$.
\end{lemma}
\begin{proof}
For~$n=2$ the assertion is obvious.
Clearly, $J\subset [1,n-1]$ is invariant
with respect to the diagram automorphism
of~$\Br^+_n$ if and only if~$J=J'\cup J''\cup (n-J')$
where~$J'$, $n-J'$ and~$J''$ are pairwise orthogonal
and~$J''=n-J''$ is connected. If~$J''=\emptyset$ then~$J'\subset [1,\lfloor\frac12 n\rfloor-1]$ and
so there are~$2^{\lfloor \frac12 n\rfloor-1}$ of them. Otherwise, $J''=[i,n-i]$ with~$1\le i\le \lfloor\frac12n\rfloor$ and
then~$J'\subset [1,i-2]$. Therefore, the total number of invariant subsets is
\begin{equation*}
1+\sum_{2\le i\le \lfloor\frac12n\rfloor}2^{i-2}+2^{\lfloor \frac12 n\rfloor-1}=2^{\lfloor\frac12 n\rfloor}.\qedhere
\end{equation*}
\end{proof}
Since homomorphisms from Theorem~\ref{thm:B2 A2n-1 spec} are parametrized by pairs~$(J,K)$ with~$J=k+1-J\subset [1,k]$ and~$K=3n+k+1-K\subset [n+k+2,2n-1]$, $0\le k\le n-2$, by Lemma~\ref{lem:no inv subsets} their total number is
equal to
\begin{align}\label{eq:cnt B2 A2n-1}
2\sum_{0\le k\le n-2} &2^{\lfloor\frac12(k+1)\rfloor+
\lfloor\frac12(n-k-1)\rfloor}.
\end{align}
Suppose first that~$n$ is even. Then
$\lfloor\frac12(k+1)\rfloor+\lfloor\frac12(n-k-1)\rfloor
=\frac12n+\lfloor\frac12(k+1)\rfloor-\lceil\frac12(k+1)\rceil=
\frac12n-\overline{k+1}=r-1+\bar k$, whence
$$
2\sum_{0\le k\le n-2} 2^{\lfloor\frac12(k+1)\rfloor+
\lfloor\frac12(n-k-1)\rfloor}=2^{\frac n2}
\sum_{0\le k\le n-2} 2^{\overline k}
=2^{\frac n2}(\tfrac12n+2(\tfrac12 n-1))=2^{\frac n2}(\tfrac32n-2).
$$
If~$n$ is odd then 
$\lfloor\frac12(k+1)\rfloor+\lfloor\frac12(n-k-1)\rfloor
=\lfloor\frac12(k+1)\rfloor+\frac12(n-1)-\lceil\frac12 k\rceil=\frac12(n-1)$,
and so the sum in~\eqref{eq:cnt B2 A2n-1} is equal to~$2^{\frac12(n+1)}(n-1)$.
\end{proof}
\begin{remark}\label{rem:OEIS}
The sequences~$\boldsymbol w_n$ and~$\boldsymbol u_n$, $n\ge 0$ coincide  with, respectively, \OEIS{A026150} and~\OEIS{A002605}
and admit a number of combinatorial interpretations (see e.g.~\cites{GK,Leh,BoMo}).
The sequence~$\boldsymbol h_n+\bar n\, 2^{\frac12(n-1)}$, $n\ge 0$ (cf.~\eqref{eq:hr seq}) coincides with~\OEIS{A052945}, up to the first term of the latter.  It is remarkable that all these
sequences satisfy the same Fibonacci-type recursion $x_{n+1}=2(x_n+x_{n-1})$, $n\ge 1$, albeit with different initial data.
The even (respectively, odd) numbered terms in the sequence from part~\ref{thm:combinatoric std.e}
form the sequence~\OEIS{A130129}
(respectively, \OEIS{A058922}).
\end{remark}

\subsection{Sporadic standard homomorphisms}\label{subs:non-disj Hecke}
First we describe all homomorphisms from $\Br^+(A_2)$ and $\Br^+(B_2)$
to Artin monoids of finite exceptional types.
\begin{proposition}\label{prop:Hom A2|B2 EF}
Let~$m\in \{3,4\}$ and~$M\in\Cox I$ be either~$F_4$ or~$E_n$, $n\in\{6,7,8\}$. Furthermore, let~$J_2=I\setminus\{1\}$ and let~$J_1
\subset I$ be 
as in the following table
$$
\begin{array}{c|c|c}
m&M&J_1\\
\hline
3&E_6&I\setminus\{5\}\\
\hline
4&F_4&I\setminus\{4\}\\
\hline
4&E_7&I\setminus\{5,6\}\\
\hline
4&E_8&I\setminus\{7\}
\end{array}
$$
\begin{enmalph}
   \item\label{prop:Hom A2|B2 EF.a}
The assignments~$\wh T_i\mapsto T_{w_{J_1\cap J_2;J_i}}$, $i\in\{1,2\}$ define a Coxeter type homomorphism $\Phi_{m,M}:\Br^+(I_2(m))\to\Br^+(M)$;
 \item\label{prop:Hom A2|B2 EF.b}
The assignments~$\wh T_1\mapsto T_{w_\circ^{J_1\cup K}}$, $\wh T_2\mapsto T_{w_\circ^{J_2}}$,
where~$K\in\{\emptyset,\{6\}\}$ if~$m=4$ and~$M=E_7$
and~$K=\emptyset$ otherwise, define a standard
homomorphism~$\wh\Phi_{m,M,K}:\Br^+(I_2(m))\to\Br^+(M)$;
\item \label{prop:Hom A2|B2 EF.c}
The only optimal fully supported standard homomorphisms~$\Br^+(I_2(m))\to\Br^+(M)$ where
$m\in\{3,4\}$ and~$M=F_4$ or~$M=E_n$, $n\in\{6,7,8\}$ are, up to the diagram automorphism of~$I_2(m)$, the 
homomorphisms~$\wh \Phi_{m,M,K}$ listed in part~\ref{prop:Hom A2|B2 EF.b}, together with
the composition of~$\wh\Phi_{4,F_4,\emptyset}$ with
the standard
unfolding~\eqref{eq:unfold F4 E6}, and homomorphisms
$\Br^+(B_2)\to \Br^+(E_6)$ given by $\wh T_1\mapsto T_{w_\circ^{J}}$, $\wh T_2\mapsto T_{w_\circ^{[1,6]}}$
where either~$J=[1,6]$ or~$J\not=\tau(J)$ are weakly orthogonal for
$\tau=(1,5)(2,4)$.
\end{enmalph}
\end{proposition}
\begin{proof}
To prove~\ref{prop:Hom A2|B2 EF.a}, it is easy to check, for example using our Python program for calculations in Coxeter groups and Hecke monoids, that the assignments $\wh s_i\mapsto w_{J_1\cap J_2;J_i}$
define a homomorphism $\phi=\phi_{m,M}:W(I_2(m))\to W(M)$ and
that $\ell(\phi(\wh w_\circ))=3\ell(\phi(\wh s_1))=
3\ell(\phi(\wh s_2))$
if~$m=3$ while $\ell(\phi(\wh w_0))=2(\ell(\phi(\wh s_1))+\ell(\phi(\wh s_2)))$ if~$m=4$. Then 
the assignments from part~\ref{prop:Hom A2|B2 EF.a} define a homomorphism $\Br^+(I_2(m))\to\Br^+(M)$
by Lemma~\ref{lem:lifting to Cox-Hecke}. 
We need the following
\begin{lemma}\label{lem:centrality}
For all pairs~$(m,M)$ listed in Proposition~\ref{prop:Hom A2|B2 EF}, $T_{w_\circ^{J_1\cap J_2}}$
commutes with the $T_{w_\circ^{J_i}}$, $i\in\{1,2\}$.
\end{lemma}
\begin{proof}
For~$m=3$ and~$M=E_6$ (respectively, $m=4$ and~$M=F_4$), the $\Br^+_{J_i}(M)$, $i\in\{1,2\}$ 
are of type~$D_4$ (respectively, $B_3$) and so the~$T_{w_\circ^{J_i}}$, $i\in\{1,2\}$ are central
in the corresponding parabolic submonoids by Proposition~\partref{prop:fund elts BrSa.e}. Let~$m=4$ and~$M=E_7$. Then~$\Br^+_{J_2}(M)$ is of type~$D_6$
and so~$T_{w_\circ^{J_2}}$ is central in~$\Br^+_{J_2}(M)$. On the other hand, $\Br^+_{J_1}(M)$ is of type~$D_5$ and so $xT_{w_\circ^{J_1}}=T_{w_\circ^{J_1}}\Sigma_{J_1}(x)$
by Proposition~\partref{prop:fund elts BrSa.e}. Since~$J_1\cap J_2=\{2,3,4,7\}$ is invariant with respect to the diagram automorphism of~$\Br^+_{J_1}(M)$ which corresponds to the transposition $(4,7)$, it follows that~$T_{w_\circ^{J_1}}$ commutes with~$T_{w_\circ^{J_1\cap J_2}}$. Finally, for~$m=4$ and~$M=E_8$, $\Br^+_{J_1}(M)$ is of type~$E_7$ and 
so~$T_{w_\circ^{J_1}}$ is central in~$\Br^+_{J_1}(M)$
by Proposition~\partref{prop:fund elts BrSa.e}. On the other hand, $\Br^+_{J_2}(M)$ is of type~$D_7$, the 
corresponding diagram automorphism being the transposition~$\tau=(2,8)$. Since~$J_1\cap J_2=\{2,3,4,5,8\}$ and hence is $\tau$-invariant,
it follows 
from Proposition~\partref{prop:fund elts BrSa.e} that~$T_{w_\circ^{J_1\cap J_2}}$
commutes with~$T_{w_\circ^{J_2}}$
\end{proof}
It follows from Lemma~\ref{lem:centrality} that the~$w_{J_1\cap J_2;J_i}$, $i\in\{1,2\}$ are products
of commuting involutions and so the~$\Phi_{m,M}$ are of Coxeter type by Theorem~\partref{thm:Main Thm Cox Heck.a}.

To prove part~\ref{prop:Hom A2|B2 EF.b}, assume first that~$K=\emptyset$. Since~$\wh T_{w_\circ^{J_i}}=
T_{w_\circ^{J_1\cap J_2}}T_{w_{J_1\cap J_2;J_i}}=
\Phi_{m,M}(\wh T_i)T_{w_\circ^{J_1\cap J_2}}$
by Lemma~\ref{lem:centrality}, it follows from Lemmata~\ref{lem:centrality} and~\ref{lem:cent decor} that $\mathbf z=(T_{w_\circ^{J_1\cap J_2}},T_{w_\circ^{J_1\cap J_2}})$
is a decoration of~$\Phi_{m,M}$ and~$\wh\Phi_{m,M,\emptyset}=(\Phi_{m,M})_{\mathbf z}$. Finally, if~$m=4$, $M=E_7$
and~$K=\{6\}$, note that~$T_6$ commutes with~$T_{w_\circ^{J_2}}$ which is central in~$\Br^+_{J_2}(M)\cong \Br^+(D_6)$ and 
with~$T_{w_\circ^{J_1}}$ since~$J_1$ and~$K$ are 
orthogonal. Then by Lemma~\ref{lem:cent decor},
$\wh\Phi_{4,E_7,\{6\}}$ is obtained as the decoration
of~$\wh\Phi_{4,E_7,\emptyset}$ by
to~$\mathbf z=(T_6,1)$. 

To prove part~\ref{prop:Hom A2|B2 EF.c},
recall that, by Lemma~\partref{lem:sqf Hecke hom.a}, a standard $\Phi\in\Hom_{\Art}(I_2(m),M)$
is uniquely determined by the~$J_i:=[\Phi](i)$,
$i\in\{1,2\}$
and induces homomorphisms of respective Coxeter groups and of Hecke monoids. In addition, for~$m=3$ we must have 
$\ell(w_\circ^{J_1})=\ell(w_\circ^{J_2})$ by Lemma~\partref{lem:sqf Hecke hom.b}. 

If say~$J_1\subset J_2$ then, since~$\Phi$ is fully supported, we must have~$J_2=I$.
Then, by the optimality of~$\Phi$, either~$J_1=I$, or
$J_1\subsetneq I$ and~$J_1\not=\emptyset$, whence~$m=4$, and~$T_{w_\circ^{J_1}}$ cannot commute with~$T_{w_\circ^I}$. Since~$T_{w_\circ^I}$
is central in~$\Br^+(M)$ for~$M=F_4$, $E_7$ or~$E_8$,
it follows that~$M=E_6$ and~$\tau(J_1)\not=J_1$. Then~$\Phi$
is the decoration with~$\mathbf z=(T_{w_\circ^{J_1}},1)$ of the character 
homomorphism $\Br^+(B_2)\to \Br^+(E_6)$, $\wh T_1\mapsto 1$, $\wh T_2\mapsto T_{w_\circ^{[1,6]}}$,
and so by Theorem~\ref{thm:decoration sufficient} and Corollary~\ref{cor:dec iff}
we conclude that~$T_{w_\circ^{J_1}}$ must commute 
with~$T_{w_\circ^{\tau(J_1)}}$ hence either~$J_1=[1,6]$ or~$J_1\not=\tau(J_1)$ are weakly orthogonal. 

Otherwise, if $J_1,J_2\not=I$,
the only pairs $J_1$, $J_2$, up to renumbering, satisfying all conditions discussed above are 
the ones listed in Proposition~\ref{prop:Hom A2|B2 EF} together with
$J_1=\{1, 2, 3, 4, 5, 7\}$, $J_2=[2,7]$ for~$m=4$
and~$M=E_7$. Yet one can verify, for example in the reflection representation of the Hecke algebra of~$W(E_7)$  (cf.~\cite{CIK}*{Proposition~9.8}) with say $q=17$, that~$(T_{w_\circ^{J_1}}T_{w_\circ^{J_2}})^2\not=
(T_{w_\circ^{J_2}}T_{w_\circ^{J_1}})^2$ in that case.
\end{proof}
\begin{remark}
One can easily verify that there are no fully supported optimal standard homomorphisms $\wh\Phi_{m,H_k}:\Br^+(I_2(m))\to \Br^+(H_k)$,
$m,k\in\{3,4\}$ except $\wh T_i\mapsto T_{w_\circ^I}$,
$i\in \{1,2\}$. Indeed, if say~$\wh T_1\mapsto T_{w_\circ^I}$ then for~$m=3$ Lemma~\partref{lem:elem Artin hom.c} forces the image 
of~$\wh T_2$ to be also~$T_{w_\circ^I}$, while 
for~$m=4$, since~$T_{w_\circ^I}$ is central in~$\Br^+(H_k)$, such a homomorphism will not be optimal unless~$\wh T_2$ is also mapped to~$T_{w_\circ^I}$.
So, we may assume that~$\wh T_i\mapsto T_{w_\circ^{K_i}}$, $i\in\{1,2\}$ with~$K_i\not=[1,k]$.
Let~$\Psi_3:\Br^+(H_3)\to
\Br^+(D_6)$ and~$\Psi_4:\Br^+(H_4)\to \Br^+(E_8)$
be standard unfoldings~\eqref{eq:unfold H3D6} and~\eqref{eq:unfold H4E8}, respectively and let~$J_i=[\Psi_k](K_i)$, $i\in \{1,2\}$. 
Then $\Psi_k\circ\wh\Phi_{m,H_k}$ would 
be a fully supported standard homomorphism
$\Br^+(I_2(m))\to \Br^+(D_6)$ for~$k=3$ (respectively, $\Br^+(I_2(m))\to \Br^+(E_6)$ for $k=4$) with the following
property
\begin{equation}
J_i\cap [\Psi_k](j)\not=\emptyset\implies [\Psi_k](j)\subset J_i,\qquad i\in\{1,2\},\,j\in[1,k].
\label{eq:H34 cnd}
\end{equation}
Yet the unique standard homomorphism
$\wh\Phi_{4,E_8}:\Br^+(B_2)\to \Br^+(E_8)$ described in Proposition~\ref{prop:Hom A2|B2 EF} does not satisfy~\eqref{eq:H34 cnd} since~$J_1=[1,7]\setminus\{7\}$ while $[\Psi_4](1)=\{1,7\}$. 
For~$k=3$, there are~$45$, up to renumbering, pairs  
$J_1,J_2\subsetneq [1,6]$
such that $(w_\circ^{J_1}w_\circ^{J_2})^4=1$,
$(w_\circ^{J_1}\star w_\circ^{J_2})^{\star 2}=(w_\circ^{J_2}\star w_\circ^{J_1})^{\star 2}$.
Yet none of them satisfies~\eqref{eq:H34 cnd}.
\end{remark}

\begin{remark}\label{rem:undercoration}
In our proof of Theorems~\ref{thm:adm I2m}, 
\ref{thm:monomial brd} and Proposition~\ref{prop:Hom A2|B2 EF} we used the ``undecoration'' procedure which in rank~2 boils down to the following algorithm. Given a standard homomorphism~$\Phi\in\Hom_{\Art}(I_2(m),M)$, $M\in\Cox I$
which by Lemma~\partref{lem:sqf Hecke hom.a} is uniquely determined by the~$K_i:=[\Phi](i)\subset I$, $i\in\{1,2\}$, we take~$J_i\subset K_i$ to be maximal
such that~$T_{w_\circ^{J_i}}$ commutes with~$T_{w_\circ^{K_j}}$, $\{i,j\}=\{1,2\}$. Quite surprisingly, it so happened in all aforementioned cases that $\mathbf z=(z_1,z_2)$
where~$z_i=T_{w_\circ^{J_i}}^{-1}$, $i\in\{1,2\}$
was a decoration of~$\Phi$ regarded as a homomorphism~$\Br^+(I_2(m))\to \Br(M)$. Yet
$\Phi_{\mathbf z}$ turned out to be a homomorphism
of monoids $\Br^+(I_2(m))\to\Br^+(M)$.
We expect that this picks up all ``missing'' homomorphisms.
\end{remark}

We now describe some irregular families of non-disjoint standard homomorphisms.

\begin{proposition}\label{prop:exotic homs}
Let $\mathsf M$ be a multiplicative monoid. Suppose that $t_0, t_1, \tau, S \in \mathsf M$ satisfy $t_iS=St_{1-i}$, $t_i \tau = \tau t_i$, $i\in\{0,1\}$, $(\tau S)^{m_1} = (S\tau)^{m_1}$ and $\brd{t_0t_1}{m_2}=\brd{t_1t_0}{m_2}$. Then the assignments $T_1 \mapsto t_1\tau$, $T_2 \mapsto S$ define a homomorphism $\Br^+(I_2(2\operatorname{lcm}(m_1,m_2)) \to \mathsf M$.
\end{proposition}
\begin{proof}
We need the following
\begin{lemma} \label{lem:weakly admissible lemma t1tauS to power m}
\label{lem:weakly admissible lemma St1tau to power m}  
Suppose that $t_0, t_1, \tau, S \in \mathsf{M}$ satisfy $t_iS=St_{1-i}$, $t_i \tau = \tau t_i$, $i\in\{0,1\}$. Then for any~$m\in\ZZ_{>0}$
$$
(t_1 \tau S)^m = \brd{  t_1t_0}m (\tau S)^m,
\quad (St_1\tau)^m=\brd{t_0t_1}m(S\tau)^m.
$$
\end{lemma}

\begin{proof}
We only prove the first equality, the argument for the second one being similar. The case~$m=1$ is trivial.
Suppose that the identity holds for some~$m\ge 1$. Then
\begin{align*}
(t_1\tau S)^{m+1}&=t_1\tau S\brd{t_1t_0}m(\tau S)^m=t_1\brd{t_0t_1}m (\tau S)^{m+1}
=\brd{t_0t_1}{m+1} (\tau S)^{m+1}.\qedhere
\end{align*}
\end{proof}

Let~$m=\operatorname{lcm}(m_1,m_2)$.
Then $(\tau S)^m=(S\tau)^m$ and
$\brd{t_1t_0}{m}=\brd{t_0t_1 }{m}$
by Lemma~\partref{lem:taut homs.a} and so
$(t_1\tau S)^m=(S t_1\tau)^m$ by
Lemma~\ref{lem:weakly admissible lemma t1tauS to power m}.
\end{proof}

\begin{corollary}
The assignments $\wh T_1 \mapsto T_1 T_3$, $\wh T_2 \mapsto T_{w_0^{[2,n-1]}}$ define a homomorphism $I_2(12) \to \Br^+_{n}$. 
\end{corollary}
\begin{proof}
Let~$\mathsf M=\Br^+_n$, $S=T_{w_\circ^{[2,n-2]}}$, $t_1=T_3$, $t_0=T_{n-2}$ and~$\tau=T_1$. 
Then the relation~$(S\tau)^{m_1}=(\tau S)^{m_1}$
with $m_1=3$ follows from Theorem~\ref{thm:adm I2m} while the remaining relations of Proposition~\ref{prop:exotic homs} are
immediate, with $m_2=2$.
\end{proof}

\begin{proposition}
Let~$M\in\Cox I$ be of finite type, let~$i\in I$ and suppose that $T_i T_{w_\circ^I}=T_{w_\circ^I}T_j$
for some~$j\not=i\in I$. The assignments $\wh T_1\mapsto T_i$, $\wh T_2\mapsto T_{w_\circ^I}$ define an optimal
homomorphism $\Br^+(I_2(2m))\to \Br^+(M)$ if and only if $m=\operatorname{lcm}(2,m_{ij})$.
\end{proposition}
\begin{proof}
Since~$T_j T_{w_\circ^I}=T_{w_\circ^I} T_i$ by Proposition~\partref{prop:fund elts BrSa.c},
using Lemma~\ref{lem:weakly admissible lemma t1tauS to power m} with~$\mathsf M=\Br^+(M)$, $t_1=T_i$, $t_0=T_j$, $S=T_{w_\circ^I}$ and~$\tau=1$
we obtain
$
(T_i T_{w_\circ^I})^m=\brd{T_iT_j}{m} T_{w_\circ^I}^m
$
whence by Proposition~\partref{prop:fund elts BrSa.a}\ref{prop:fund elts BrSa.d}
$$
((T_iT_{w_\circ}^I)^m)^{op}=T_{w_\circ^I}^m\brd{T_jT_i}{m}=\begin{cases}
\brd{T_iT_j}m T_{w_\circ^I}^m,&\bar m=1,\\
\brd{T_jT_i}m T_{w_\circ^I}^m,&\bar m=0.
\end{cases}
$$
Thus, if $m$ is odd, $(T_iT_{w_\circ^I})^m$ is automatically ${}^{op}$-invariant, while for even~$m$, since~$\Br^+(M)$ is cancellative, $(T_iT_{w_\circ^I})^m$ is ${}^{op}$-invariant
if and only if $\brd{T_iT_j}{m}=\brd{T_jT_i}m$, which
by Lemma~\partref{lem:taut homs.b} happens if and only if~$m_{ij}$ divides~$m$. The assertion follows by Lemma~\ref{lem:I2m iff cnd}.
\end{proof}
\begin{proposition}
For any~$n\in\mathbb Z_{\ge 4}$, the assignments
$\wh T_1\mapsto T_1T_n, \wh T_2\mapsto  T_{w_\circ^{[1,n-1]}}$
define a standard homomorphism $\Br^+(I_2(10))\to \Br^+_{n+1}$.
\end{proposition}
\begin{proof} We need the following
 \begin{lemma} \label{lem:TUS5=UST5}
 Let $\mathsf M$ be a multiplicative monoid and suppose that $S,T,U,V\in \mathsf M$ satisfy $(ST)^3=(TS)^3$, $UV=VU$, $US=SU$, $TU=VT$, $TV=UT$ and $SVS=VSV$. Then $(TUS)^5=(UST)^5$.
\end{lemma}
\begin{proof}We have
\begin{align*}(T&US)^5=(VTS)(TSU)^4
=VTSTSTVSUTSUTSU
=VTSTSTVSVUTSTSU\\
&=VTSTSTSVSUTSTSU
=VSTSTSTVSUTSTSU
=VSTSVTSTSUTSTSU\\
&=VSTSVTSTSTVSVTS
=VSTSVTSTSTSVSTS
=VSTSVSTSTSTVSTS\\
&=VSTVSVTSTSTVSTS
=UVSTSVTSVTSTSTS
=UVSTSVTSVSTSTST\\
&=UVSTSVTVSVTSTST
=UVSTSUTUSVTSTST
=UVSVTSTUSVTSTST\\
&=USVSTSTUSVTSTST
=USVSTSTSTVSTVST
=USVTSTSTSVSTVST\\
&=USVTSTSTVSVTVST
=USTUSTSTVSVTVST
=USTUSTSTVSUTUST\\
&=USTUSTUSTSUTUST=(UST)^5.\qedhere
\end{align*}
\end{proof}
Let $\mathsf M=\Br_n^+$, $S=T_n$, $U=T_1$, $V=T_{n-1}$
and $T=T_{w_\circ^{[1,n-1]}}$. All relations in the Lemma
are immediate except $(ST)^3=(TS)^3$. But
for~$J=\{1,n-1,n+1\}$
we have $\tau_1(J)=T_{w_\circ^{[1,n-1]}}=T$, $\tau_0(J)=T_n=S$
and then $(\tau_1(J)\tau_0(J))^3=(\tau_0(J)\tau_1(J))^3$
by Theorem~\ref{thm:adm I2m}.
\end{proof}

\subsection{Some conjectural families of non-disjoint standard homomorphisms}\label{subs:conj families}
In this section, we list several yet conjectural infinite families of standard homomorphisms from Artin monoids
of type~$I_2(N)$ to $\Br^+(A_n)$, $\Br^+(B_n)$ and $\Br^+(D_{n+1})$. So far we have verified these conjectures for $n \le 15$. 
\begin{conjecture}\label{conj:all homs A2 An}
Let~$M\in\Cox I$ be irreducible of finite type and let~$\Phi\in\Hom_{\Art}(A_2,M)$ be fully supported and 
standard with~$[\Phi](1)\not=[\Phi](2)$. Then, up to diagram automorphisms,  either~$M=A_{3m-1}$ for some~$m\ge 1$ and~$\Phi=\Phi_{m,J}$, $J\subset[1,m-1]$ from Theorem~\ref{thm:Hom A2}
or~$M=E_6$ and~$\Phi=\wh \Phi_{3,E_6}$ from Proposition~\partref{prop:Hom A2|B2 EF.b}.
\end{conjecture}
\begin{conjecture}\label{conj:all homs B2 M}
 Homomorphisms from either Theorem~\ref{thm:Hom B2} or~\ref{thm:B2 A2n-1 spec} or Proposition~\partref{prop:Hom A2|B2 EF.b},
exhaust, up to diagram automorphisms, all fully supported standard homomorphisms from $\Br^+(B_2)$ to any Artin monoid of irreducible finite type.
\end{conjecture}

\begin{conjecture}
Let~$1\le b\le a<n-1$ and~$a>1$. The assignments~$\wh T_1\mapsto T_{w_\circ^{[1,a]}}$, $\wh T_2\mapsto T_{w_\circ^{[b,n-1]}}$ define a homomorphism
$\Br^+(I_2(N))\to \Br^+_{n}$ if one of the 
following holds
\begin{itemize}
\item $N=2n/(n+b-a-2)\in 2\ZZ_{>0}$;
\item
$a+b=n$ and~$N=n/(b-1)\in1+2\ZZ_{>0}$;
\item 
$a+1=2(b-1)$, $a+b< n$ and~$N=6$.
\end{itemize}
\end{conjecture}
 
 \begin{conjecture}
For any~$i\in [2,n-2]$, the assignments $\wh T_1 \mapsto T_i T_n$, $\wh T_2 \mapsto T_{w_0^{[1,n-1]}}$, $i \in [2,n-2]$, define a standard homomorphism $\Br^+(I_2(12/d)) \to \Br^+_{n+1}$, where $d=1+\delta_{2i,n}+\delta_{2i,n-1}+\delta_{2i,n+1}$.
\end{conjecture}

We conclude with a list of conjectural families of standard
homomorphisms $\Br^+(I_2(N))\to 
\Br^+(D_{n+1})$. Here we use the following labeling of the Coxeter graph of type~$D_{n+1}$
$$
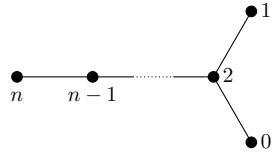
\dynkin[Coxeter,expand labels={n,n-1,2,1,0},label directions={,,right,,},make indefinite edge={2-3}]D5
$$
\begin{conjecture}
\begin{enmalph}
\item 
The assignments
$\wh T_i\mapsto T_{w_\circ^{K_i\setminus(2i+4\ZZ_{\ge 0})}}$, $i\in\{1,2\}$
define a standard homomorphism from $\Br^+(I_2(n+1))$ to $\Br^+(D_{n+1})$ for the following
 $K_1,K_2\subset [0,n]$:
\begin{enumerate}[leftmargin=*,label={$\bullet$}]

\item $K_1 =[1,n]$, $K_2 = \{0\} \cup [2,n-2]$ if $n\in 1+4\ZZ_{>0}$,

\item $K_1=[1,n]$, $K_2 = \{0\} \cup [2,n]$ if $n\in 1+ 2\ZZ_{> 0}$,

\item  $K_1=[1,n-2]$, $K_2 = \{0\} \cup [2,n]$ if $n\in 3+4\ZZ_{\ge 0}$.

\end{enumerate}
\item The assignments $\wh T_i\mapsto T_{w_\circ^{K_i\setminus(2i-1+4\ZZ_{\ge 0})}}$, $i\in\{1,2\}$
define a standard homomorphism from $\Br^+(I_2(2n))$ to $\Br^+(D_{n+1})$ for the following
pairs $K_1,K_2\subset[0,n]$:
\begin{enumerate}[leftmargin=*,label={$\bullet$}]

\item $K_1=[0,n]$,
$K_2= \{1\} \cup [4,n-2]$ if $n\in 4\ZZ_{>0}$,

\item $K_1=[0,n]$, $K_2= \{1\} \cup [4,n]$ if $n\in 2\ZZ_{>1}$,

\item  $K_1=[0,n-2]$, $K_1= \{1\} \cup [4,n]$ if $n\in 2+4\ZZ_{>0}$.

\end{enumerate}
\end{enmalph}

\end{conjecture}

\section*{List of symbols}
\def\bqq{{\setbox0\hbox{$\widehat{U}_q^+$}\setbox2\null\ht2\ht0\dp2\dp0\box2}}
\def\hr#1{\hyperlink{#1}{\pageref*{page:#1}}}

\noindent
{
\scriptsize
\begin{tabular}{p{1.55in}@{\bqq}l@{\hskip.25in}p{1.55in}@{\bqq}l@{\hskip.25in}p{1.55in}@{\bqq}l}
$\bar s$&p.~\hr{bar s}&$[a,b]_2$&p.~\hr{[a,b]2}&$\mathscr P(S)$&p.~\hr{PS}\\
$\ascprod$, $\dscprod$&p.~\hr{ascp}&
$\brd{xy}{m}$&p.~\hr{brd}&$B(x,y)$&p.~\hr{B(x,y)}\\
$\Cox I$&p.~\hr{Cox I}&$\Gamma(M)$&p.~\hr{Gamma(M)}&
$\Br^+(M)$, $\Br(M)$&p.~\hr{Br+(M)} \\
$\ell$&p.~\hr{ell}&${}^{op}$&p.~\hr{op}&$W(M)$&p.~\hr{W(M)}\\
$\pi_M$&p.~\hr{piM}&
$\SQF^+(M)$&p.~\hr{SQF}&
$\Br^+_J(M)$, $W_J(M)$&p.~\hr{Br+J(M)}\\
$\iota_J$&p.~\hr{iotaJ}&$\mathscr F(M)$&p.~\hr{F(M)}&
$\supp$&p.~\hr{supp}\\
$w_\circ^J$&p.~\hr{w0J}&$w_{J;K}$&p.~\hr{wJ;K}&
$\pi^\star_M$&p.~\hr{pi*M}\\
$(W(M),\star)$&p.~\hr{HeMon}&
$\times$&p.~\hr{times}&
$D_L(w)$,$D_R(w)$&p.~\hr{DL(w)}\\
$\cx ab$, $\cxr ab$, $\Cx ab$, $\Cxr ab$&p.~\hr{cab}&
$D_L(X)$&p.~\hr{I(X)}&
$h(M)$&p.~\hr{h(M)}\\
$\Art$, $\CoxCat  $, $\Heck $&p.~\hr{A C H}&
$[\Phi]$&p.~\hr{[Phi]}&
$\Xi_{\mathbf X}$&p.~\hr{char hom}\\
$\Phi_{\mathbf z}$&p.~\hr{Phi z}&
$\mathscr{AH}$, $\mathscr{AC}$, $\mathscr{ACH}$&p.~\hr{AH AC}&
$\mathsf H$, $\mathsf C$&p.~\hr{H C}\\
$\overline\Phi$, $\overline\Phi_\star$&p.~\hr{barPhi}&$\Ast$&p.~\hr{Ast}&
$P_J$&p.~\hr{PJ}\\
$M^{\varpi}$&p.~\hr{MII}&
$\mathbf F_{\varpi}$&p.~\hr{F wpi}&
$M(\mathbf d)$&p.~\hr{M(d)}\\
$\mathbf T_{\mathbf d}$&p.~\hr{T d}\\
$\mathbf T_{i,d}$&p.~\hr{T i,d}&
$\Br^+_n$, $\Br_n$&p.~\hr{Brn}\\
$T_{(i,j)}$&p.~\hr{T(i,j)}&
$T_J$,$\tilde\tau_k(J)$&p.~\hr{TJ}&
$\tau_k(J)$&p.~\hr{tauk(J)}\\
$\Cx ij^{(a)}$, $\Cxr ij^{(a)}$&p.~\hr{Cij(a)}&
$g(J)$&p.~\hr{gJ}&
$U(J)$&p.~\hr{U(J)}\\
$e_i$&p.~\hr{ei}&
$v_{[i,j]}$&p.~\hr{v[i,j]}&
$u_i$, $w_{[i,j]}^{(a)}$&p.~\hr{ui}\\
$\la\cdot|\cdot\ra$&p.~\hr{<.|.>}&
$\beta_\pm(J)$&p.~\hr{betapm}&
$q_s$&p.~\hr{qs}\\
$\Phi^{(m)}_n$,
$\wh \Phi^{(m)}_n$&p.~\hr{Phi(m)n}
\end{tabular}
}

\renewcommand{\PrintDOI}[1]{    DOI: \href{https://dx.doi.org/#1}{#1}}

\renewcommand{\MR}[1]{\relax}
\renewcommand{\eprint}[1]{\href{http://arxiv.org/abs/#1}{arXiv:#1}
}

\end{document}